\documentclass[
11pt,  reqno]{amsart}
\usepackage[margin=1.2in,marginparwidth=1.5cm, marginparsep=0.5cm]{geometry}

\usepackage{amsmath,amssymb,amscd,amsthm,amsxtra, esint}

\usepackage{mathabx}

 \usepackage{slashed}
 \usepackage{float}

\usepackage{booktabs} 
\usepackage{microtype}
\usepackage{amssymb}
\usepackage{mathrsfs} 

\usepackage{upgreek} 

\usepackage{color}
\usepackage[implicit=true]{hyperref}

\usepackage{xsavebox}
\usepackage{orcidlink}

\usepackage{cases}

\allowdisplaybreaks[2]

\definecolor{gr}{rgb}   {0.,   0.69,   0.23 }
\definecolor{bl}{rgb}   {0.,   0.5,   1. }
\definecolor{mg}{rgb}   {0.85,  0.,    0.85}
\definecolor{yl}{rgb}   {0.8,  0.7,   0.}
\definecolor{or}{rgb}  {0.7,0.2,0.2}

\usepackage{tikz} 
\usetikzlibrary {positioning}
\usepackage{marginnote}
\usepackage{scalerel} 

\usetikzlibrary{shapes.misc}
\usetikzlibrary{shapes.symbols}
\usetikzlibrary{decorations}
\usetikzlibrary{decorations.markings}

\tikzset{
	dot/.style={circle,fill=black,draw=black,inner sep=0pt,minimum size=0.5mm},
	>=stealth,
	}

\tikzset{
	dot2/.style={circle,fill=black,draw=black,inner sep=0pt,minimum size=0.2mm},
	>=stealth,
	}

\tikzset{
	ddot/.style={circle,fill=black,draw=black,inner sep=0pt,minimum size=0.8mm},
	>=stealth,
	}

\tikzset{decision/.style={ 
        draw,
        diamond,
        aspect=1.5
    }}

\tikzset{dia2/.style
={diamond,fill=white,draw=black,inner sep=0pt,minimum size=1mm},
	>=stealth,
	}

\tikzset{dia/.style
={star,fill=black,draw=black,inner sep=0pt,minimum size=1mm},
	>=stealth,
	}

\tikzset{dia/.style
={diamond,fill=black,draw=black,inner sep=0pt,minimum size=1.3mm},
	>=stealth,
	}

\makeatletter
\def\DeclareSymbol#1#2#3{\xsavebox{#1}{\tikz[baseline=#2,scale=0.15]{#3}}}
\def\<#1>{\xusebox{#1}}
\makeatother

\newsavebox{\peA}
\newsavebox{\pneA}
\newsavebox{\plA}
\newsavebox{\pgA}
\newsavebox{\pleA}
\newsavebox{\pgeA}
\newsavebox{\pezA}

\savebox{\peA}{\tikz \draw (0,0) node[shape=circle,draw,inner sep=0pt,minimum size=8.5pt] {\scriptsize  $=$};}
\savebox{\pneA}{\tikz \draw (0,0) node[shape=circle,draw,inner sep=0pt,minimum size=8.5pt] {\footnotesize $\neq$};}
\savebox{\plA}{\tikz \draw (0,0) node[shape=circle,draw,inner sep=0pt,minimum size=8.5pt] {\scriptsize $<$};}
\savebox{\pgA}{\tikz \draw (0,0) node[shape=circle,draw,inner sep=0pt,minimum size=8.5pt] {\scriptsize $>$};}
\savebox{\pleA}{\tikz \draw (0,0) node[shape=circle,draw,inner sep=0pt,minimum size=8.5pt] {\scriptsize $\leqslant$};}
\savebox{\pgeA}{\tikz \draw (0,0) node[shape=circle,draw,inner sep=0pt,minimum size=8.5pt] {\scriptsize $\geqslant$};}
\savebox{\pezA}{\tikz \draw (0,0) node[shape=circle,draw,
fill=white, 
inner sep=0pt,minimum size=8.5pt]{} ;}

\def \peB{\mathchoice
{\scalebox{.7}{{\usebox{\peA}}}}
{\scalebox{.7}{{\usebox{\peA}}}}
{\scalebox{.7}{{\usebox{\peA}}}}
{}
}

\def \pezB{\mathchoice
{\scalebox{.7}{{\usebox{\pezA}}}}
{\scalebox{.7}{{\usebox{\pezA}}}}
{\scalebox{.7}{{\usebox{\pezA}}}}
{}
}

\newcommand{\pe}{\mathbin{{\peB}}}

\newcommand{\pez}{\mathbin{{\pezB}}}

\usetikzlibrary{decorations.pathmorphing, patterns,shapes}

\DeclareSymbol{dot}{-2.7}{\draw (0,0) node[dot]{};}

\DeclareSymbol{nx}{0}{\draw (0, 0.85) node[]{$\nabla$}; \draw[line width=0.7pt] (-0.7,-0.45) -- (0.5, 1.85);}

\DeclareSymbol{1}{0}{\draw[white] (-.4,0) -- (.4,0); \draw (0,0)  -- (0,1.2) node[dot] {};}
\DeclareSymbol{2}{0}{\draw (-0.5,1.2) node[dot] {} -- (0,0) -- (0.5,1.2) node[dot] {};}
\DeclareSymbol{3}{0}{\draw (0,0) -- (0,1.2) node[dot] {}; \draw (-.7,1) node[dot] {} -- (0,0) -- (.7,1) node[dot] {};}
\DeclareSymbol{30}{-3}{\draw (0,0) -- (0,-1); \draw (0,0) -- (0,1.2) node[dot] {}; \draw (-.7,1) node[dot] {} -- (0,0) -- (.7,1) node[dot] {};}
\DeclareSymbol{31}{-3}{\draw (0,0) -- (0,-1) -- (1,0) node[dot] {}; \draw (0,0) -- (0,1.2) node[dot] {}; \draw (-.7,1) node[dot] {} -- (0,0) -- (.7,1) node[dot] {};}
\DeclareSymbol{32}{-3}{\draw (0,0) -- (0,-1) -- (1,0) node[dot] {}; \draw (0,0) -- (0,-1) -- (-1,0) node[dot] {}; \draw (0,0) -- (0,1.2) node[dot] {}; \draw (-.7,1) node[dot] {} -- (0,0) -- (.7,1) node[dot] {};}
\DeclareSymbol{20}{-3}{\draw (0,0) -- (0,-1);\draw (-.7,1) node[dot] {} -- (0,0) -- (.7,1) node[dot] {};}
\DeclareSymbol{22}{-3}{\draw (0,0.3) -- (0,-1) -- (1,0) node[dot] {}; \draw (0,0.3) -- (0,-1) -- (-1,0) node[dot] {};\draw (-.7,1) node[dot] {} -- (0,0.3) -- (.7,1) node[dot] {};}
\DeclareSymbol{31p}{-3}{\draw (0,0) -- (0,-1) -- (1,0) node[dot] {}; \draw (0,0) -- (0,1.2) node[dot] {}; \draw (-.7,1) node[dot] {} -- (0,0) -- (.7,1) node[dot] {}; 
\draw (0,-1) node{\scaleobj{0.5}{\pez}}; 
\draw (0,-1) node{\scaleobj{0.5}{\pe}}; }
\DeclareSymbol{32p}{-3}{\draw (0,0) -- (0,-1) -- (1,0) node[dot] {}; \draw (0,0) -- (0,-1) -- (-1,0) node[dot] {}; \draw (0,0) -- (0,1.2) node[dot] {}; \draw (-.7,1) node[dot] {} -- (0,0) -- (.7,1) node[dot] {}; 
\draw (0,-1) node{\scaleobj{0.5}{\pez}};
\draw (0,-1) node{\scaleobj{0.5}{\pe}};}
\DeclareSymbol{22p}{-3}{\draw (0,0.3) -- (0,-1) -- (1,0) node[dot] {}; \draw (0,0.3) -- (0,-1) -- (-1,0) node[dot] {};\draw (-.7,1) node[dot] {} -- (0,0.3) -- (.7,1) node[dot] {}; 
\draw (0,-1) node{\scaleobj{0.5}{\pez}};
\draw (0,-1) node{\scaleobj{0.5}{\pe}};}
\DeclareSymbol{21p}{-3}{\draw (0,0.3) -- (0,-1) -- (1,0) node[dot] {}; 
\draw (-.7,1) node[dot] {} -- (0,0.3) -- (.7,1) node[dot] {}; 
\draw (0,-1) node{\scaleobj{0.5}{\pez}};
\draw (0,-1) node{\scaleobj{0.5}{\pe}};}

\DeclareSymbol{21}{-3}{\draw (0,0.3) -- (0,-1) -- (1,0) node[dot] {}; 
\draw (-.7,1) node[dot] {} -- (0,0.3) -- (.7,1) node[dot] {}; 
}

\DeclareSymbol{11p}{-3}{\draw (0,0) node[dot2] {} -- (0,-1) -- (1,0) node[dot] {}; 
\draw (0,0) -- (0,1.2) node[dot] {};  
\draw (0,-1) node{\scaleobj{0.5}{\pez}}; 
\draw (0,-1) node{\scaleobj{0.5}{\pe}}; }
\DeclareSymbol{100}{-3}
{\draw[white] (-.4,0) -- (.4,0); 
\draw (0,0) node[dot2] {} -- (0,-1)  ; 
\draw (0,0) -- (0,1.2) node[dot] {};}  
\usetikzlibrary{mindmap,backgrounds,trees,arrows,decorations.pathmorphing,decorations.pathreplacing,positioning,shapes.geometric,shapes.misc,decorations.markings,decorations.fractals,calc,patterns}

\tikzset{>=stealth',
         cvertex/.style={circle,draw=black,inner sep=1pt,outer sep=3pt},
         vertex/.style={circle,fill=black,inner sep=1pt,outer sep=3pt},
         star/.style={circle,fill=yellow,inner sep=0.75pt,outer sep=0.75pt},
         tvertex/.style={inner sep=1pt,font=\scriptsize},
         gap/.style={inner sep=0.5pt,fill=white}}
\tikzstyle{mybox} = [draw=black, fill=blue!10, very thick,
    rectangle, rounded corners, inner sep=10pt, inner ysep=20pt]
\tikzstyle{boxtitle} =[fill=blue!50, text=white,rectangle,rounded corners]

\tikzstyle{decision} = [diamond, draw, fill=blue!20,
    text width=4.5em, text badly centered, node distance=2.5cm, inner sep=0pt]
\tikzstyle{block} = [rectangle, draw, fill=blue!20,
    text width=2.7cm, text centered, rounded corners, minimum height=3em]

\tikzstyle{line} = [draw, very thick, color=black!50, -latex']

\tikzstyle{cloud} = [draw, ellipse,fill=red!40, 
node distance=2.5cm,
    minimum height=1em]

\tikzstyle{cloud2} = [draw, ellipse,fill=red!30, text=white,text width=10em, node distance=2.5cm, text centered, minimum height=4em]

\tikzstyle{cloud3} = [draw, ellipse, fill=cyan!30, 
node distance=2.5cm,
    minimum height=3em, 
    text width=1.7cm]

\tikzstyle{cloud4} = [draw, ellipse,fill=orange!70, node distance=2.5cm,
   text width=2.1cm,  
           minimum height=3em]

\tikzstyle{cloud5} = [draw, ellipse,fill=red!20, node distance=2.5cm,
    text centered, 
    text width=3.0cm, 
    minimum height=3em]

\tikzstyle{cloud6} = [draw, ellipse,fill=red!20, node distance=2.5cm,
    text width=1.5cm, 
    text centered, 
    minimum height=3em]

\tikzset{
    position/.style args={#1:#2 from #3}{
        at=(#3.#1), anchor=#1+180, shift=(#1:#2)
    }
}

\newtheorem{theorem}{Theorem} [section]

\newtheorem{lemma}[theorem]{Lemma}
\newtheorem{proposition}[theorem]{Proposition}
\newtheorem{remark}[theorem]{Remark}

\newtheorem{definition}[theorem]{Definition}
\newtheorem{corollary}[theorem]{Corollary}

\newtheorem{conjecture}[theorem]{Conjecture}

\DeclareMathOperator*{\intt}{\int}

\DeclareMathOperator*{\supp}{supp}

\newcommand{\1}{\hspace{0.2mm}\text{I}\hspace{0.2mm}}
\newcommand{\II}{\text{I \hspace{-2.8mm} I} }
\newcommand{\III}{\text{I \hspace{-2.9mm} I \hspace{-2.9mm} I}}
\newcommand{\noi}{\noindent}
\newcommand{\Z}{\mathbb{Z}}
\newcommand{\R}{\mathbb{R}}
\newcommand{\T}{\mathbb{T}}
\newcommand{\bul}{\bullet}

\let\P= \undefined
\newcommand{\P}{\mathbf{P}}

\newcommand{\E}{\mathbb{E}}

\newcommand{\K}{\mathcal{K}}

\newcommand{\F}{\mathcal{F}}

\newcommand{\qf}{\mathfrak{q}}

\newcommand{\al}{\alpha}
\newcommand{\be}{\beta}
\newcommand{\dl}{\delta}

\newcommand{\nb}{\nabla}

\newcommand{\Dl}{\Delta}
\newcommand{\eps}{\varepsilon}
\newcommand{\kk}{\kappa}
\newcommand{\g}{\gamma}
\newcommand{\G}{\Gamma}
\newcommand{\ld}{\lambda}
\newcommand{\Ld}{\Lambda}
\newcommand{\s}{\sigma}

\newcommand{\ft}{\widehat}
\newcommand{\Ft}{{\mathcal{F}}}
\newcommand{\wt}{\widetilde}
\newcommand{\cj}{\overline}
\newcommand{\dx}{\partial_x}
\newcommand{\dt}{\partial_t}

\newcommand{\LLRA}{\Longleftrightarrow}

\newcommand{\ta}{\theta}

\renewcommand{\l}{\ell}
\renewcommand{\o}{\omega}
\renewcommand{\O}{\Omega}

\newcommand{\les}{\lesssim}
\newcommand{\ges}{\gtrsim}

\newcommand{\jb}[1]
{\langle #1 \rangle}

\newcommand{\jbb}[1]{\bigl\langle #1 \bigr\rangle}

\newcommand{\fbb}[1]
{[\hspace{-0.6mm}[ #1 ]\hspace{-0.6mm}]}

\newcommand{\ind}{\mathbf 1}

\renewcommand{\S}{\mathcal{S}}

\newcommand{\M}{\mathbf{M}}

\newcommand{\N}{\mathbb{N}}

\renewcommand{\H}{\mathcal{H}}

\newtheorem*{ackno}{Acknowledgements}

\newcommand{\I}{\mathcal{I}}
\newcommand{\J}{\mathcal{J}}

\newcommand{\A}{\mathcal{A}}

\newcommand{\C}{\mathcal{C}}

\numberwithin{equation}{section}
\numberwithin{theorem}{section}

\newcommand{\Q}{\mathbf{Q}}
\newcommand{\PP}{\mathbb{P}}

\DeclareMathOperator{\Law}{Law}

\newcommand{\ZZ}{\mathfrak{Z}}

\newcommand{\muu}{\vec{\mu}}

\newcommand{\rhoo}{\vec{\rho}}

\newcommand{\D}{\mathcal{D}}

\newcommand{\too}{\longrightarrow}

\newcommand{\Pii}{\mathbf{\Pi}}

\usepackage{comment}

\newcommand{\Ta}{\Theta}
\newcommand{\mf}[1]
{\mathfrak{#1}}
\newcommand{\TT}{\mathbf{T}}

\newcommand{\mc}[1]
{\mathcal{#1}}

\newcommand{\mb}[1]
{\mathbb{#1}}

\newcommand{\mbf}[1]
{\mathbf{#1}}

\newcommand{\Id}{\mathrm{Id}}

\makeatletter
\@namedef{subjclassname@2020}{%
  \textup{2020} Mathematics Subject Classification}
\makeatother

\begin{document}

\baselineskip = 14pt

\title[Hyperbolic sine-Gordon model beyond the first threshold]
{Hyperbolic sine-Gordon model beyond  \\ the first threshold}

\author[Y.~Zine]
{Younes Zine\orcidlink{0009-0001-7752-1205}}

\address{
Younes Zine\\
 \'Ecole Polytechnique F\'ed\'erale de Lausanne\\
1015 Lausanne\\ Switzerland}

\email{younes.zine@epfl.ch}

\subjclass[2020]{35L71, 60H15}

\keywords{
sine-Gordon equation; 
hyperbolic sine-Gordon model; 
dynamical sine-Gordon model;
Gibbs measure; 
cone multiplier;
Gaussian multiplicative chaos}

\begin{abstract}


We  study the hyperbolic sine-Gordon model, 
with a parameter $\be^2 > 0$,  
 and its associated Gibbs dynamics
on the two-dimensional torus.
By introducing a physical space approach
to the Fourier restriction norm method
and establishing nonlinear dispersive smoothing for
the imaginary multiplicative Gaussian chaos, 
we construct 
invariant Gibbs dynamics
for the hyperbolic sine-Gordon model
beyond the first threshold $\be^2 = 2\pi$. The deterministic step of our argument hinges on establishing key bilinear estimates, featuring weighted bounds for cone multipliers. Moreover, the probabilistic component involves a careful analysis of the imaginary Gaussian multiplicative chaos and reduces to integrating singularities along space-time light cones.
As a by-product of our proof, we identify $\be^2 = 6\pi$ as a critical threshold for the hyperbolic sine-Gordon model, which is quite surprising given that the associated parabolic model has a critical threshold at $\be^2 =8\pi$.

\end{abstract}

\maketitle
\tableofcontents
\section{Introduction}
\label{SEC:1}

\subsection{Hyperbolic sine-Gordon model}\label{SUBSEC:1-1}

\noi
We consider
the following stochastic damped sine-Gordon equation (SdSG) on 
$\T^2 = (\R/2\pi\Z)^2$:
\begin{align}
\begin{cases}
\dt^2 u + \dt u + (1- \Dl)  u   +  \g \sin(\be u) = \sqrt{2}\xi\\
(u, \dt u) |_{t = 0} = (u_0, v_0) , 
\end{cases}
\qquad (t, x) \in \R_+\times\T^2,
\label{SdSG}
\end{align}

\noi
where $u$ is a real-valued unknown, $\g$ and $\be$ are non-zero real numbers and $\xi$ denotes space-time white noise on $\R \times \T^2$ with the space-time covariance formally given by
\[ \E [\xi(x_1, t_1 ) \xi(x_2, t_2)] =  \dl(x_1- x_2) \dl(t_1-t_2).\]

\noi
The Gibbs measure associated with \eqref{SdSG} formally reads
\begin{align}\label{Gibbs1}
``d\rhoo(u,v) = Z^{-1}e^{-E(u,v)}dudv".
\end{align}

\noi
Here,  $Z = Z(\be)$
denotes a normalization constant
 and 
 \begin{align}
E(u,v)= \frac12\int_{\T^2}\big(u(x)^2+|\nabla u(x)|^2 + v(x)^2 \big)dx -\frac\g\be \int_{\T^2}\cos\big(\be u(x)\big)dx
\label{Hamil1}
\end{align}

\noi
 denotes the energy (= Hamiltonian) of the (deterministic undamped) sine-Gordon equation:
\begin{align}
\dt^2 u +  (1- \Dl)  u   +  \g \sin(\be u) = 0.
\label{SdSG2}
\end{align}

The Gibbs measure $\rhoo$ in \eqref{Gibbs1} 
arises in various  physical contexts such as two-dimensional Yukawa and Coulomb gases in statistical mechanics and  the quantum sine-Gordon model in Euclidean quantum field theory. 
We refer the readers to \cite{BEMS, CHS, HS, Fro2, LRV, McKean81,McKean94, PS, LRV2} and the references therein for more physical motivations and interpretations of the measure~$\rhoo$. 
The dynamical model \eqref{SdSG} then corresponds to the 
so-called ``canonical" stochastic quantization~\cite{RSS} of the quantum sine-Gordon model represented by the measure $\rhoo$ in \eqref{Gibbs1}. See the works \cite{Bara, BB, DH1, DH2, DH3, Fro1, Fro2, GHOZ, GM, ORSW2} for constructions of the sine-Gordon model for various ranges of the parameter $\be^2$.

In \cite{ORSW2}, Oh, Robert, Sosoe, and Wang constructed the dynamics \eqref{SdSG} in the range $0 < \be^2 < 2 \pi$.\footnote{See also \cite{ORSW1} for a pathwise well-posedness result on the stochastic hyperbolic undamped sine-Gordon equation with deterministic initial data.} We review this argument in Subsection \ref{SUBSEC:ideas} below. In the present work, our main goal is to further extend the well-posedness theory for \eqref{SdSG} beyond the threshold $\be^2 = 2 \pi$.
\subsection{Setup and main result}\label{SUBSEC:1-2}
Here, we state our main result regarding the construction of the dynamics \eqref{SdSG}
associated with the Gibbsian initial data $\rhoo$ \eqref{Gibbs1} for $\be^2 \ge 2\pi$.
To this end, we first fix some notations. 
Given $ s \in \R$, 
let $\mu_s$ denote
a Gaussian measure, formally defined by
\begin{align}
 d \mu_s 
   = Z_s^{-1} e^{-\frac 12 \| u\|_{{H}^{s}}^2} du
& =  Z_s^{-1} \prod_{n \in \Z^2} 
 e^{-\frac 12 \jb{n}^{2s} |\ft u_n|^2}   
 d\ft u_n , 
\label{gauss0}
\end{align}

\noi
where $Z_s$ is a normalization constant,
  $\jb{\,\cdot\,} = \big(1+|\,\cdot\,|^2\big)^\frac{1}{2}$
and $\ft u_n$  denotes the  Fourier coefficient of $u$
at the frequency $n \in \Z^2$.
We define 
\begin{align}
\muu_s = \mu_s \otimes \mu_{s-1}.
\label{gauss1}
\end{align}

\noi
In particular, when $s = 1$, 
 the measure $\muu_1$ is defined as 
   the induced probability measure
under the map:
\begin{equation*}
\o \in \O \longmapsto (u_0^\o, v_0^\o),
 \end{equation*}

\noi
where $u_0^\o$ and $v_0^\o$ are given by
\begin{equation}\label{series}
u_0^\o = \sum_{n \in \Z^2} \frac{g_n(\o)}{\jb{n}}e_n
\qquad\text{and}\qquad
v_0^\o = \sum_{n \in \Z^2} h_n(\o)e_n.
\end{equation}

\noi
Here, 
 $e_n=(2\pi)^{-1}e^{i n\cdot x}$
 and $\{g_n,h_n\}_{n\in\Z^2}$ denotes  a family of independent standard 
 complex-valued  Gaussian random variables such that $\cj{g_n}=g_{-n}$ and $\cj{h_n}=h_{-n}$, 
 $n \in \Z^2$.
It is easy to see that $\muu_1 = \mu_1\otimes\mu_0$ is supported on
\begin{align*}
\H^{s}(\T^2) := H^{s}(\T^2)\times H^{s - 1}(\T^2)
\end{align*}

\noi
for $s < 0$ but not for $s \geq 0$.

With  
 \eqref{Hamil1}, 
 \eqref{gauss0}, and 
 \eqref{gauss1},  
 we can formally write $\rhoo$ in \eqref{Gibbs1} as 
\begin{align}\label{Gibbs2}
d\rhoo(u,v) \sim e^{\frac{\g}{\be}\int_{\T^2}\cos(\be u)dx}
d\muu_1 (u, v).
\end{align} 

\noi
In view of the roughness of the support of $\muu_1$, 
the nonlinear term in \eqref{Gibbs2} is not well-defined
and thus a proper renormalization is required to give a meaning to \eqref{Gibbs2}.

Let  $\Pii_{\le N}$  be a smooth frequency projector
onto the frequencies  $\{n\in\Z^2:|n|\leq N\}$
given by the following Fourier multiplier:
\begin{align}
\ft{ \Pii_{\le N} f}(n) =   \chi_N(n) \ft f (n).
\label{chi}
\end{align}

\noi
Here, $\ft f$ denotes the spatial Fourier transform of $f$ and $\chi_N(n) = \chi(N^{-1}n)$ for 
some  fixed non-negative radial function 
 $\chi \in C^\infty_0(\R^2)$ 
such that $\chi$ is non-increasing on $\R_+$, 
$\supp \chi \subset \{\xi\in\R^2:|\xi|\leq 1\}$,  and $\chi\equiv 1$ 
on $\{\xi\in\R^2:|\xi|\leq \tfrac12\}$.  
Given  $u_0 = u_0^\o$ as in \eqref{series}, 
i.e. $\operatorname{law}(u_0) = \mu_1$, 
set $\s_N$, $N \in \N$, by setting
 \begin{align}\label{sN}
 \s_N =  \E\Big[\big(\Pii_{\le N}u_0(x)\big)^2\Big] =\frac1{4\pi^2}\sum_{n\in\Z^2}\frac{\chi_N(n)^2}{\jb{n}^2}
 = \frac1{2\pi}\log N + o(1),
 \end{align}

\noi
as $N \to \infty$, 
uniformly in $x\in\T^2$.
Given $N \in \N$, 
define 
 the truncated renormalized density:
\begin{align}\label{RN}
R_N(u) = \frac{\g_N}{\be}\int_{\T^2}\cos\big(\be \Pii_{\le N} u(x)\big)dx,
\end{align}

\noi
where $\g_N = \g_N(\be)$ is given by 
 \begin{equation}\label{CN}
 \g_N(\be) = e^{\frac{\be^2}{2}\s_N}. 
 \end{equation}

\noi
In particular, we have  $\g_N \to \infty$ as $N \to \infty$.
We then  define the truncated renormalized Gibbs measure:
\begin{align}\label{GibbsN}
d\rhoo_N(u,v)= Z_N^{-1}e^{R_N(u)}d\muu_1(u,v)
\end{align}

\noi
for some normalization constant $Z_N = Z_N(\be) \in (0,\infty)$. One then proves the existence of a measure $\rhoo$ such that
\begin{align}
\lim_{N \to \infty} \rhoo_N = \rhoo,
\label{Gibbs10}
\end{align}

\noi
in the sense of total variation. See Lemma \ref{LEM:Gibbs} in Section \ref{SEC:4} below.

We now consider the following renormalized truncated SdSG dynamics:
\begin{align}
\dt^2 u_N   + \dt u_N  +(1-\Dl)  u_N 
+\g_N \Pii_{\le N} \big\{\sin (\be \Pii_{\le N}u_N)\big\}   = \sqrt{2} \xi , 
\label{RSdSGN}
\end{align} 

\noi
with the truncated Gibbs measure initial data $\rhoo_N$ \eqref{GibbsN}. Here, $\g_N$ is as in  \eqref{CN}. Our main result below proves, for some range of parameters $\be^2$, the convergence of the sequence $(u_N, \dt u_N)_{N \in \N}$ to a non-trivial stochastic process $(u, \dt u)$ whose law is given by $\rhoo$ \eqref{Gibbs10} at every time marginal. This process $u$ is hence formally interpreted as the solution to the following renormalized SdSG equation

\noi
\begin{align}
\dt^2 u + \dt u  +(1-\Dl)  u
+ \infty \cdot \sin (\be u)  = \sqrt{2} \xi , 
\label{RSdSG}
\end{align}

\noi
at statistical equilibrium.

\begin{theorem}\label{THM:main}
Let  $0 < \be^2 < 2\pi\Big(1 + \frac{3 \sqrt{241} - 41}{122}\Big)
\approx 2.046\pi$. Then,
the stochastic damped sine-Gordon equation~\eqref{RSdSG} is almost surely globally well-posed with respect to the renormalized Gibbs measure~$\rhoo$ defined in~\eqref{Gibbs10} and the renormalized Gibbs measure $\rhoo$ is invariant under the dynamics. More precisely, there exists a process $(u,\dt u) \in C(\R_+; \mc H^{-\eps}(\T^2))$\footnote{Here, $C(\R_+, X)$ for a Banach space $X$ is the space of continuous functions from $\R_+$ to $X$, endowed with the compact-open topology.} for any small $\eps > 0$ such that the solution $(u_N, \dt u_N)$ to \eqref{RSdSGN} converges to $(u, \dt u)$ in $C(\R_+; \mc H^{-\eps}(\T^2))$ $\rhoo$-almost surely as $N \to \infty$. Moreover, for each $t \ge 0$, the law of $(u(t), \dt u(t))$ is given by $\rhoo$.
\end{theorem}

Theorem~\ref{THM:main} is proved in Section~\ref{SEC:WP}. It constitutes a first step towards building a physical space approach to study random wave equations. See Remarks~\ref{RMK:progress} and~\ref{RMK:others} below.

\subsection{Prior works}\label{SUBSEC:prior} In this subsection, we give a brief overview of the literature relevant to our problem.
%

\subsubsection{Random wave equations with polynomial nonlinearities} For power type nonlinearities, there has been spectacular progress in the study of the well-posedness issue for random wave equations in the recent years. In \cite{GKO2}, Gubinelli, Koch and Oh studied the following quadratic wave equation in three dimensions:
\begin{align} 
\dt^2 u  +(1- \Dl) u   + u^2 = \xi, \quad (t,x) \in \R_+ \times \T^3,
\label{SNLW}
\end{align}
where $\xi$ is space-time white noise on $\R_+ \times \T^2$. By adapting the paracontrolled approach of Gubinelli, Imkeller and Perkowski \cite{GIP}, developed for parabolic equations, to the wave setting together with the random operator perspective of Bourgain \cite{BO96}, they proved well-posedness (for smooth enough initial data) of \eqref{SNLW}. A key ingredient in their argument is to prove the so-called {\it multilinear smoothing} for (a renormalized version of) the square of the stochastic convolution $\Psi$ solving
\[ \dt^2 \Psi -\Dl \Psi  = \xi.\]
More precisely, they prove that (a renormalized version of)\footnote{Here and in the rest of this subsection, we omit renormalization issues for the sake of simplicity.} $\Psi^2$ belongs to $C(\R_+; W^{-\frac12-\eps, \infty}(\T^3))$ for any $\eps >0$, beating the initial guess $\Psi^2 \in C(\R_+; W^{-1-\eps, \infty}(\T^3))$ obtained by naive ``parabolic power counting"; see \cite{Hairer, MWX}. This was achieved by using a simple, but crucial observation: 
\begin{align}
\F(\Psi^2) = \F(\Psi) * \F(\Psi),
\label{nonFourier}
\end{align}
where $\F$ denotes the spatial Fourier transform; reducing the argument to a Fourier-space analysis.

In \cite{Bring2a, Bring2}, Bringmann further developed the Fourier-based methodology of \cite{GKO2} and considered the following Hartree cubic nonlinear wave equation in three dimensions:
\begin{align}
\begin{cases}
\dt^2 u + (1-\Dl) u   - (\jb \nb ^{-\al} * u^2)u = 0 \\
(u, \dt u) |_{t = 0} = (u_0, v_0) , 
\end{cases}
\qquad (t, x) \in \R_+\times\T^3,
\label{SNLW2}
\end{align}
for $\al >0$ and where the rough random initial data $(u_0, v_0)$ is sampled from the Hartree $\Phi^4_3$ Gibbs measure. By adapting the Fourier norm restriction norm method of Bourgain and Kleinerman-Machedon \cite{BO93a, BO93b, KM} to the random wave context and reducing the multilinear smoothing discussed above to counting estimates, he proved almost sure global well-posedness for \eqref{SNLW2} and invariance of the Hartree $\Phi^4_3$ Gibbs measure under the dynamics for any $\al >0$.

The developments in the polynomial setting eventually culminated in the breakthrough work \cite{BDNY}, where Bringmann, Deng, Nahmod and Yue proved almost sure global well-posedness for the hyperbolic $\Phi^4_3$-model (namely, \eqref{SNLW2} with $\al =0$) and invariance of the corresponding $\Phi^4_3$-measure under the dynamics, by mixing the paracontrolled approach together with inputs from the theory of random tensors \cite{DNY2} and the molecule analysis of \cite{DH}. 

See also \cite{Bring1, BT1, BT2, Deya1, GKO, GKOT, LTW, ORTz, OWZ, OOTol, OOTol2, OPTz, OZ, Tolo, Zine1} and references therein for other works on the well-posedness issue for other random wave models.

\subsubsection{Parabolic sine-Gordon model.} In \cite{CHS}, Chandra, Hairer and Shen considered the parabolic counterpart to \eqref{SdSG}:
\begin{align}
\begin{cases}
 \dt u + (1- \Dl)  u   +  \g \sin(\be u) = \sqrt{2}\xi\\
u |_{t = 0} = u_0 , 
\end{cases}
\qquad (t, x) \in \R_+\times\T^2.
\label{pSG}
\end{align}
They proved local well-posedness for \eqref{pSG} in the full subcritical range $0 < \be^2 < 8\pi$ in \cite{CHS,HS} by adapting the theory of regularity structures \cite{Hairer} to the sine nonlinearity setting (see also \cite{HS} for a partial result). In \cite{BC}, Bringmann and Cao globalized the solutions constructed in \cite{CHS} in the restricted range $0 < \be^2 < 6\pi$. See also \cite{CFW}.

We refer the reader to Remark~\ref{RMK:wave_para} below for a discussion on differences between the wave and heat sine-Gordon models.

\begin{remark}\label{RMK:poly}\rm
It is tempting to adapt the Fourier-based methods of the works \cite{Bring2, BDNY, GKO2} on random wave equations with polynomial nonlinearities discussed in Subsection \ref{SUBSEC:prior} to the sine nonlinearity setting of \eqref{vN3}. However, formulas of the form \eqref{nonFourier}, which are a cornerstone of the aforementioned approaches, do not hold in the non-polynomial setup. Namely, we cannot directly rely the Fourier transform of $\Ta_N$ to that of $\Psi^{\textup{wave}}_N$. Furthermore, taking inspiration from the literature \cite{CHS, HS} on the parabolic counterpart \eqref{pSG} to \eqref{SdSG}, it is natural to develop a physical-side framework to study the wave sine-Gordon dynamics in order to take advantage of the key properties of the sine nonlinearity (boundedness and Lipschitz continuity).
\end{remark}

\begin{remark}\label{RMK:wave_para}\rm
Let us highlight key differences between the hyperbolic and parabolic sine-Gordon model. First of all, parabolic flows enjoys a much stronger smoothing property than wave flows. Furthermore, while on the one hand, heat equations are compatible with $L^\infty$ type spaces, wave equations on the other hand, are only compatible with a $L^2$ analysis. This leads to integrability issues; see for instance Remark~\ref{RMK:C}. From a more technical perspective, implementing a physical space approach for the hyperbolic sine-Gordon requires to handle singularities along light cones as opposed to singularities at single points in the parabolic case; see the discussion in the next subsection. These reasons explain why the analysis of the hyperbolic model is much harder than its parabolic counterpart. 
\end{remark}

\subsection{Main challenges and ideas}\label{SUBSEC:ideas}

Here, we discuss the proof of Theorem \ref{THM:main}. In view of the absolute continuity of the Gibbs measure $\rhoo$ with respect to the Gaussian measure $\muu_1$, we consider \eqref{RSdSG} with the Gaussian random data $(u_0, v_0)$ and $\Law(u_0, v_0) = \muu_1$ as in \eqref{series}. In particular, for $N \in \N$, we consider the solution $u_N$ to \eqref{RSdSGN} with initial data given by $(u_0, v_0)$. 

\subsubsection{First order expansion} We first proceed with the following first order expansion (\cite{BO96, DPD, ORSW2}):
\noi
\begin{align}
u_N = \Psi^{\textup{KG}} + v_N,
\label{expa1}
\end{align}

\noi
where $\Psi^{\text{KG}}$ is the solution to the following linear damped wave equation:
\begin{align}\label{SdLW}
\begin{cases}
\dt^2 \Psi^{\textup{KG}} + \dt\Psi^{\textup{KG}} +(1-\Dl)\Psi^{\textup{KG}}  = \sqrt{2}\xi\\
(\Psi^{\textup{KG}},\dt\Psi^{\textup{KG}})|_{t=0}=(u_0,v_0),
\end{cases}
\end{align}

\noi
where $\Law (u_0,v_0) = \muu_1$. 
Define the linear damped Klein-Gordon propagator $\D(t)$ by 
\begin{equation}\label{D}
\D(t) = e^{-\frac{t}2}\frac{\sin(t \fbb{\nb})}{\fbb \nabla},
\end{equation} 
where
\[ \fbb n = \Big(\frac34 + |n|^2\Big)^{\frac12} , \quad n \in \Z^2, \]

\noi
as a Fourier multiplier operator.
Then, we have 
\begin{align} 
\Psi^{\textup{KG}} (t) 
 = \dt\D(t)u_0 + \D(t)(u_0+v_0)+ \sqrt{2}\int_0^t\D(t - t')d \mc W(t'), 
\label{Psi}
\end{align}

\noi
where  $\mc W$ denotes a cylindrical Wiener process on $L^2(\T^2)$:
\begin{align}
\mc W(t) =  \sum_{n \in \Z^2 } B_n (t) e_n,
\label{W1}
\end{align}

\noi
and  
$\{ B_n \}_{n \in \Z^2}$ 
is defined by 
$B_n(0) = 0$ and 
$B_n(t) = \jb{\xi, \ind_{[0, t]} \cdot e_n}_{ t, x}$.
Here, $\jb{\cdot, \cdot}_{t, x}$ denotes 
the duality pairing on $\R_+ \times \T^2$.
As a result, 
we see that $\{ B_n \}_{n \in \Z^2}$ is a family of mutually independent complex-valued\footnote
{In particular, $B_0$ is  a standard real-valued Brownian motion.} 
Brownian motions such that $B_{-n} = \cj{B_n}$, $n \in \Z^2$. 
By convention, we normalize $B_n$ such that $\text{Var}(B_n(t)) = t$.

For $N \in \N$, let $\Psi^{\textup{KG}}_N$ be the truncated stochastic convolution:
\begin{align}
\Psi^{\textup{KG}}_N=\Pii_{\le N}\Psi^{\textup{KG}}.
\label{Psi_trunc1}
\end{align}
A direct computation shows that $\Psi^{\textup{KG}}_N(t,x)$
 is a mean-zero real-valued Gaussian random variable with variance
\begin{align}
 \E \big[\Psi^{\textup{KG}}_N(t,x)^2\big] = \s_N
 \label{sig1}
\end{align}

\noi
for any $t\ge 0$, $x\in\T^2$ and $N \in \N$,
where $\s_N$ is as in \eqref{sN}. Moreover, one can show that $\{ \Psi^{\textup{KG}}_N\}_{N\in \N}$ is a Cauchy sequence in $C( [0,T] ; W^{-\eps, \infty}(\T^2))$, almost surely for any $T, \eps >0$; see Lemmas \ref{LEM:psi} and \ref{LEM:diff_psi}. Hence, it converges to $\Psi^{\textup{KG}}$ in $C( [0,T] ; W^{-\eps, \infty}(\T^2))$, almost surely.

For reasons discussed in Remark~\ref{RMK:phycov} below, we actually work with the following wave stochastic convolution:
\begin{align} 
\Psi^{\text{wave}} (t) 
 & = \dt\S(t)u_0 + \S(t)(u_0+v_0)+ \sqrt{2}\int_0^t\S(t - t')d \mc W(t'), \label{Psi_S} \\
 \Psi^{\textup{wave}}_N & =\Pii_{\le N}\Psi^{\textup{wave}}, \quad N \in \N. \label{Psi_trunc2}
\end{align}
Here, $\S$ is the linear propagator associated to the damped wave equation. Namely, $\S$ is given by

\noi
\begin{align}\label{S}
\S(t) = e^{-\frac{t}2}\frac{\sin(t |\nb|)}{|\nb|}.
\end{align} 

\noi
We also define for $N \in \N$, the truncated stochastic convolution $\Psi^{\text{wave}}_N = \Pii_{\le N}\Psi^{\text{wave}}$ and observe that
\begin{align}
\E \big[\Psi^{\textup{wave}}_N(t,x)^2\big] = \s_N + O(1),
\label{sig2}
\end{align}
where $O(1)$ is a constant which is uniform in $N$. We also show in Lemma~\ref{LEM:diff_psi} that the difference $\Psi^{\textup{KG}}_N - \Psi^{\textup{wave}}_N$ is a smooth enough function (uniformly in $N$) for our purposes.

The nonlinear remainder $v_N = u_N - \Psi^{\textup{KG}}$ satistfies the following equation:

\noi
\begin{align}
\dt^2 v_N   + \dt v_N  +(1-\Dl)  v_N  = - \g_N \Pii_{\le N} \big\{\sin(\be \Pii_{\le N} v_N + \Psi^{\textup{KG}}_N)\big\}, 
\label{vN1}
\end{align} 

\noi
with the zero initial data. By expanding the nonlinearity using trigonometric formulas, the mild formulation for \eqref{vN1} reads

\noi
\begin{align}
 v_N = -  \sum_{\eps_0, \eps_1 \in \{+,-\}} c_{\eps_0, \eps_1} \Pii_{\le N} \I \Big( e^{i \eps_1 \be \Pii_{\le N} v_N} e^{i \be (\Psi^{\textup{KG}}_N - \Psi^{\textup{wave}}_N)} \cdot  \Theta^{\eps_0}_N\Big),
\label{vN2}
\end{align}
where $c_{\eps_0, \eps_1} \in \mb C$, $\Ta_N^{\eps_0}$ is the imaginary Gaussian multiplicative chaos 

\noi
\begin{align}
\Ta_N^{\eps_0}= \g_N e^{i \eps_0 \be \Psi^{\textup{wave}}_N} = e^{\frac{\be^2}{2}\s_N} e^{i \eps_0 \be \Psi^{\textup{wave}}_N}.
\label{t1}
\end{align}
and $\I$ is the Duhamel operator
\begin{align}
\I(F)(t) = \int_{0}^t \mc D(t-t') F(t') dt', \quad t \ge 0.
\label{duha}
\end{align}
By proceeding as in 
\cite{HS,ORSW1, ORSW2}, 
 we  establish the regularity property of  $\Ta^{\eps_0}_N$; 
see Lemma~\ref{LEM:sto1}.
 In particular, given $0 < \be^2 < 4\pi$, 
$\{\Ta^{\eps_0}_N\}_{N \in \N}$ 
forms a Cauchy sequence  in $L^p(\O;L^q([0,T];W^{-\al,\infty}(\T^2)))$
 for any finite $p,q\ge 1$ and  $\al>\frac{\be^2}{4\pi}$; see Lemma \ref{LEM:sto1}.

\subsubsection{The first threshold $\be^2 = 2\pi$} We quickly describe the argument in \cite{ORSW2} which leads to the restriction $\be^2 < 2\pi$. In what follows, we work with the following simplified equation for $v_N$ for convenience:
\begin{align}
v_N = - \I \big( e^{i v_N + \Psi_N} \cdot \Ta_N \big),
\label{vN3}
\end{align} 
with $\Psi_N = \Psi^{\textup{KG}}_N - \Psi^{\textup{wave}}_N$ and $\Ta_N = \Ta^+_N$ or $\Ta^-_N$. Our goal is to solve \eqref{vN3} in $L_t^{\infty}H^s(\T^2)$ for some $s >0$ to be determined. To this end, we analyze the different frequency interactions on the right-hand-side of \eqref{vN3}:
\begin{align}
\I \big( \P_{N_{0}} \big( \P_{N_1} \big(e^{i v_N + \Psi_N}\big) \cdot \P_{N_2} \Ta_N \big)\big),
\label{freqpro}
\end{align}
where $(N_{0}, N_1, N_2) \in (2^{\N})^3$ and $\P_K$ denotes a smooth spatial projection onto frequencies $\{ n \in \Z^2 :|n| \sim K\}$; see \eqref{proj1} below. In view of the regularity\footnote{In the rest of this subsection, all stated regularities are understood to be on a set of full probability and uniform in $N$.}\begin{align}
\Ta_N \in L^{\infty}_t W_x^{-\frac{\be^2}{4\pi} -\eps,\infty}
\label{nonlinchaos0}
\end{align}
for small $\eps >0$,
the inhomogeneous estimate
\begin{align}
\|\I(F)\|_{L^\infty_t H^{s+1}_x} \les \|F\|_{L_t^1 H^s_x}
\label{inho0}
\end{align}
and standard product estimates, analyzing the frequency localized product \eqref{freqpro} leads to the following observations:
\begin{itemize}
\item[(LH)] \underline{low $\times$ high $\to$ high interaction: $N_1 \ll N_0 \sim N_2$.}\quad We need $s - 1 + \frac{\be^2}{4\pi} +\eps < 0$ to handle this case;

\medskip

\item[(HH)] \underline{high $\times$ high interaction: $N_1 \sim N_2$.} We need $s >  \frac{\be^2}{4\pi} + \eps$ to handle this case.
\end{itemize}
\smallskip
Therefore, combining the cases (LH) and (HH) yields the condition $\be^2 < 2\pi$ for $\eps$ small enough. 

\subsubsection{An interpolation argument} Without loss of generality, we fix $2\pi \le \be^2 < 4\pi$. Our main idea is to improve on the restriction $s - 1 + \frac{\be^2}{4\pi} +\eps < 0$ via an interpolation argument which we outline next. We further divide the (LH) interaction case into two subcases. Fix $0 <\g < 1$ (to be chosen small later) and consider the cases (LH1) and (LH2) as follows:
\begin{itemize}
\item[(LH1)] \underline{(not too low) $\times$ high $\to$ high interaction: $N_0^{\g} \le N_1 \ll N_0 \sim N_2$.}

\medskip

\item[(LH2)] \underline{(very low) $\times$ high interaction: $N_1 < N_0^{\g}$.}
\end{itemize}

\smallskip

\noi
By taking advantage of the high space-time integrability of the chaos $\Ta_N$ (see \eqref{nonlinchaos0}), we may borrow derivatives from $e^{i v_N + \Psi_N}$ in the (LH1) case. Consequently, this interaction can be placed in $L^\infty_t H^s(\T^2)$ for 
\begin{align}
s - 1 - \g s + \frac{\be^2}{4\pi} +\eps < 0;
\label{condo1}
\end{align}
 see Lemma \ref{LEM:prod1}.

The crucial step in our argument is to improve on the restriction on $s$ in the case (LH2): note that using the information \eqref{nonlinchaos0} as in the case (LH) would yield the condition $s - 1 + \frac{\be^2}{4\pi} +\eps < 0$ again (because of the scenario $N_1 \sim 1$). We instead (essentially) use the new information
\begin{align}
\Ta_N \in Y^{-\frac{\be^2}{4\pi}+\frac12-3\eps, -\frac12-\eps}_{-\frac12-3\eps},
\label{nonlinchaos}
\end{align}
where $Y^{s,b}_{a}$ for $(s,b,a) \in \R^3$ is the $L^2$-based space associated with a weighted variant of the usual Fourier restriction norm:
\[ \|u\|_{Y_a^{s,b}(\R \times \T^2)}  := \big\| \jb t^{a} \, \F^{-1}_{t,x} \big( \jb \zeta ^s \, ||\tau| - |\zeta||^b  \ft u (\tau, \zeta) \big)  \big\|_{L^2_{t,x}(\R \times \T^2)},\]
where $\ft u$ and $\F_{t,x}^{-1}[u]$ respectively denote the space-time Fourier transform and its inverse. See Subsection \ref{SUBSEC:spaces}.
\begin{remark}\label{RMK:nonlin}\rm
The bound \eqref{nonlinchaos} represents a $\frac12$-gain of spatial derivatives as compared to \eqref{nonlinchaos0}. This is similar in spirit to the multilinear smoothing phenomenon in the polynomial case discussed in Subsection \ref{SUBSEC:prior}. We thus refer to it as nonlinear smoothing for the imaginary Gaussian multiplicative chaos $\Ta_N$.
\end{remark}
We defer the discussion on the ideas behind the proof of \eqref{nonlinchaos} to the end of the subsection and explain how to use it to study (LH2). By using \eqref{nonlinchaos} together with the Fourier restriction norm (namely $X^{s,b}$-spaces; see Subsection \ref{SUBSEC:spaces}) and a duality argument, the contribution of the interaction (LH2) is bounded by an expression of the form
\begin{align}
\big\| \P_{N_0} w \cdot \P_{N_1} \big(e^{i v_N + \Psi_N}\big)\big\|_{Y^{-\frac{\be^2}{4\pi}+\frac12-3\eps, -\frac12-\eps}_{-\frac12-3\eps}},
\label{prodo}
\end{align}
where $w$ belongs to $X^{1-s, \frac12-\eps}$. Morally speaking, the bound \eqref{nonlinchaos} essentially allowed us to trade a $\frac12$-modulation derivatives for $\frac12$-spatial derivatives in \eqref{prodo}. The key upside of such a trade lies in the following observation: when estimating a product of two functions $u_1$ and $u_2$, modulation derivatives\footnote{Namely, the mixed symbol $||\tau|-|\zeta||$, where $(\zeta,\tau)$ is the space-time Fourier variable of the product $u_1 u_2$.} cost in general a lot less than spatial derivatives when estimating products; see Lemma \ref{LEM:hyprule}. This essentially leads to the bound
\begin{align*}
\eqref{prodo} \les N_0^{s-1 + \frac{\be^2}{4\pi} - \frac12 + \eps + C \g} \cdot \|v_N\|_{X^{s,\frac12+\eps}} \|w\|_{X^{1-s,\frac12-\eps}}
\end{align*}
for some constant $C>0$, which in turn gives the restriction 
\begin{align}
s- \frac32 + \frac{\be^2}{4\pi}  + \eps + C \g < 0.
\label{condo2}
\end{align}
For $\g$ small enough and $\be^2 < 3\pi$, the condition \eqref{condo1} is more restrictive than \eqref{condo2}. Therefore, the case (LH) can be handled under \eqref{condo1}, which together with (HH) yields
\[ \be^2 < \frac{4\pi}{2-\g}, \]
which is an improvement over the restriction $\be^2 < 2\pi$.

In practice, we optimize the value of $\g$ which leads to the specific improved range of parameters $\be^2$ in Theorem~\ref{THM:main}. The rigorous interpolation argument is implemented by proving bilinear estimates which follow from a careful multi-parameter analysis; see Section~\ref{SEC:3}.

\begin{remark}[On the constant $C$ in \eqref{condo2}]\label{RMK:C} \rm
In the strategy described above, an improvement on the constant $C$ (making it as small as possible) directly leads to an improvement on the range of parameters $\be^2$. In our approach, the main issue in minimizing $C$ comes from the fact that we are trying to estimate a product of two $L^2$ functions in a $L^2$ space. Therefore, by H\"older's inequality and Sobolev's embedding we necessarily lose a full power of $N_1$, which implies $C\ge 1$. In fact, within our framework, $C$ has to be much larger than one, as we need to place $e^{i v_N + \Psi_N}$ in a $L^p$-based anisotropic Sobolev space, with $1 < p< 2$, which enjoys a fractional chain rule, to prove relevant difference estimates in our well-posedness argument; see Proposition \ref{PROP:gwp} in Section \ref{SEC:WP}.
\end{remark}

\begin{remark}[Analysis in weighted $Y^{s,b}_a$ spaces]\rm \label{RMK:space_weights}
The presence of (time) weights in $Y^{s,b}_a$-norms renders our nonlinear analysis particularly challenging. We employ techniques from harmonic analysis to overcome this issue and, in particular, prove a weighted $L^2$ estimate for a cone multiplier; see Subsection~\ref{SUBSEC:weighted}.
\end{remark}

\subsubsection{Nonlinear smoothing for the imaginary Gaussian multiplicative chaos} We now discuss the proof of the nonlinear smoothing bound \eqref{nonlinchaos}; which is the main probabilistic step of our work and requires a careful analysis on the physical side. The main step reduces to showing the following second moment estimate:
\begin{align}
\sup_{x \in \T^2}   \E_{\muu_1 \otimes \PP} \Big[\big|\big( \Box^{-\frac12-\eps} \, |\nb_x| (\P_{N_0} \ind_{[0,1]} \Ta_N) \big)(t,x)\big|^2\Big] & \les_{\eps} N_0^{\frac{\be^2}{2\pi} + \eps} \jb t^{4\eps}
\label{variance}
\end{align}
for any $N_0 \in 2^\N$ and $t \in \R$. Here, $\Box^{-b}$ is the so-called hyperbolic Riesz potential; see \eqref{Rhyp1} and \eqref{Ysghyp}. The main advantage of the expression \eqref{variance} lies in the fact that all the convolution kernels of the multipliers on its left-hand-side have an explicit physical side representation. This is crucial to exploit the cancellation properties of the chaos $\Ta_N$. 

\begin{remark}\label{RMK:whyweight}\rm
In reducing \eqref{nonlinchaos} to \eqref{variance}, we need to take the $L^2_t$-norm of the square root of the right-hand-side of \eqref{variance}. The presence of the weight in the $Y^{s,b}_a$-norm in \eqref{nonlinchaos} ensures that this $L^2_t$-norm converges.
\end{remark}

The bound \eqref{variance} follows from three new ingredients:
\smallskip
\begin{itemize}
\item[(i)] \underline{Sharp estimates on the space-time covariance of $\Psi^{\textup{wave}}_N$.}

\medskip

\item[(ii)] \underline{A multi-variate Sobolev argument.}

\medskip

\item[(iii)] \underline{Integrating singularities along light cones.}
\end{itemize}
\medskip

\noi
We briefly discuss (i), (ii) and (iii). Let us start with (i). Since the chaos $\Ta_N$ involves the stochastic convolution $\Psi_N^{\textup{wave}}$, proving \eqref{variance} requires a fine understanding of the space-time covariance of $\Psi_N^{\textup{wave}}$ given by
\begin{align}
\G_N (t_1,t_2, x_1,x_2) = \E \big[ \Psi^{\textup{wave}}_N(t_1,x_1) \Psi^{\textup{wave}}_N(t_2,x_2) \big]
\label{cov}
\end{align}
for any $(t_1, x_1), (t_2, x_2) \in \R_+\times\T^2$. In Proposition \ref{PROP:cov}, we prove the following two-sided bound on $\G_N$:
\begin{align}
\G_N (t_1,t_2, x_1,x_2) = -\frac{1}{2\pi} \log \big( |t_1 - t_2| + |x_1 - x_2| + N^{-1}\big) + R_N(t_1, t_2, x_1, x_2),
\label{cov100}
\end{align}
where $R_N$ is bounded uniformly in $N$. Note that the singularity on the right-hand-side of \eqref{cov100} is of {\it elliptic} type, in the sense that it is singular at the space-time origin (in the limit $N \to \infty$). In view of the derivative term on the left-hand-side of \eqref{variance}, we also need to estimate the spatial derivatives of $\G_N$. When differentiating the remainder term $R_N$, {\it hyperbolic singularities} along light cones of the form
\begin{align}
(t,x) \in \R \times \T^2 \mapsto ||t| - |x||^{-c}, \quad c \in \N,
\label{hyposing}
\end{align}
show up. This is in sharp contrast with the parabolic case \cite{CHS, HS} where the remainder term is smooth; see Remarks~\ref{RMK:cov} and~\ref{RMK:dercov} for a more thorough discussion on this point. Proving the required bounds \eqref{cov100} and particularly its variant with derivatives (see Proposition \ref{PROP:cov2}) is very challenging as one needs to keep track of subtle cancellations within the covariance function $\G_N$ and handle the hyperbolic singularities \eqref{hyposing} effectively; see Section~\ref{SEC:ker} and Subsection~\ref{SUBSEC:3-1}.

\begin{remark}\label{RMK:phycov}\rm
We analyze the covariance of the stochastic convolution $\Psi^{\textup{wave}}_N$ (as opposed to $\Psi^{\textup{KG}}_N$) as it is constructed from a kernel with an explicit formula on the physical side; see \eqref{poisson2}.
\end{remark}

The ingredient (ii) comes from the following observation: in bounding the left-hand-side of \eqref{variance}, we need to estimate a quantity of the form
\begin{align}
 |\nabla_{x_1}| |\nabla_{x_2}| \P^{x_1}_{N_0} \P^{x_2}_{N_0} \operatorname{Cov}(\Ta_N, \Ta_N)(t_1, x_1, t_2, x_2),
 \label{covo}
\end{align}
where $|\nabla_{x_\l}|$ and $\P_{N_0}^{x_\l}$ are the multipliers $|\nb|$ and $\P_{N_0}$ along the variable $x_\l$ for $\l = 1,2$ and $\operatorname{Cov}(\Ta_N, \Ta_N)$ is the space-time covariance of $\Ta_N$ and is given by $e^{\be^2 \G_N}$. In estimating \eqref{covo}, one may move both derivatives either (a) to the (kernels of) $\P^{x_\l}_{N_0}$ for $\l = 1,2$ or (b) to the covariance function $\operatorname{Cov}(\Ta_N, \Ta_N)$. Scenario (a) gives a factor $N_0^2$, which is not allowed in \eqref{variance} (since $\be^2 < 4\pi$) and Scenario (b) is also problematic since it outputs second order derivatives of the covariance $\G_N$, which are not locally integrable functions; see Proposition \ref{PROP:cov2}. 

Our argument (ii) overcomes this issue by interpolating between the two cases (a) and (b). It yields the correct power of $N_0$ allowed on the right-hand-side of \eqref{variance} and a locally integrable function made of a singularity of the form \eqref{hyposing} mixed with an elliptic singularity; see Subsection~\ref{SUBSEC:sto4} for more details.

Lastly, the third ingredient (iii) allows us to integrate the product of the singularity output in the Sobolev argument (ii) and the kernel of the operator $\Box^{-\frac12-\eps}$ which is also a mix of a hyperbolic and an elliptic singularity; see \eqref{kerhyp}. Bounding the resulting integrals boils down to carefully estimating the volume of the intersection of transverse tubes in $\R^4$; see Subsection~\ref{SEC:sing}. This is the physical side counterpart of the counting arguments in the Fourier-based literature on random wave equations; see for instance \cite{Bring2, BDNY, GKO2}.

A computation (see Remark~\ref{REM:div}) with \eqref{cov100} shows that $\Ta_N$ does not converge as a space-time distribution in the limit $N \to \infty$ for $\be^2 \ge 6\pi$. This is an instance of the so-called ``variance blowup"; see \cite{Deya2, OOcomp, Hairer3}. This suggests the following conjecture.

\begin{conjecture}\label{CONJ:main}
The renormalized sine-Gordon model \eqref{RSdSG} is globally well-posed on the support of the Gibbs measure $\rhoo$ for $\be^2 < 6\pi$. 
\end{conjecture}

\subsection{Final remarks} We conclude this section with a few remarks.

\begin{remark}[Further progress on Conjecture \ref{CONJ:main}]\label{RMK:progress}\rm
In the forthcoming work \cite{Zine3}, we make a further progress to~\ref{CONJ:main} by going beyond the $L^2$ analysis of the present paper. Our method uses insights from recent developments in the Fourier restriction theory for the cone \cite{GWZ}.
\end{remark}

\begin{remark}[Thresholds for the hyperbolic sine-Gordon model]\label{RMK:criticality}\rm
We emphasize that in this work, the “first threshold” for the hyperbolic sine-Gordon equation refers to the {\it analytical} threshold $\beta^2 = 2\pi$, where the basic first-order expansion argument in \cite{ORSW2} provides a starting point for the well-posedness theory of \eqref{RSdSG}. This is different from the first {\it physical} threshold $\be^2 = 4\pi$ at which another further renormalization of the stochastic objects is needed to define the dynamics; see \cite{HS}. It is not clear at this point what is the range of $\be^2$ for which one would need to go beyond a first order expansion to solve \eqref{RSdSG}, although our analysis seems to suggest that $\be^2 = 3\pi$ is a natural candidate in view of \eqref{condo2} for example.
\end{remark}

\begin{remark}[Physical space methods for other models]\label{RMK:others}\rm
It would be of interest to apply the physical space methods developed in this paper to other hyperbolic models. Besides other equations with non-polynomial nonlinearities such as the Liouville model considered in \cite{ORW}, wave equations in non-homogeneous settings are natural candidates for this endeavor. For instance, it would be interest to study \eqref{SNLW} aas in \cite{GKO2}, but with $\T^3$ replaced with a general three dimensional manifold; see the work \cite{ORTz} for an example of the analysis of a singular wave equations with a general input manifold.
\end{remark}

The paper is organized as follows. In Section~\ref{SEC:2}, we introduce our set of notations, function spaces and state basic estimates. In Section~\ref{SEC:ker}, we prove estimates on elliptic and hyperbolic kernels that are needed later in the paper. Next, in Sections~\ref{SEC:3} and~\ref{SEC:4}, we respectively prove key bilinear estimates and construct the stochastic objects that are used in our fixed point argument. In Section~\ref{SEC:WP}, we state a global well-posedness statement and prove Theorem\ref{THM:main}.

\section{Preliminaries}\label{SEC:2}

\subsection{Notations}\label{SUBSEC:2-1} 
In this subsection, 
we introduce some notations.
We then set our conventions for the Fourier transforms
and state some basic facts.

\medskip

\noi
{\bf $\bullet$ Preliminary notations.} We write $ A \les B $ to denote an estimate of the form $ A \leq CB $. 
Similarly, we write  $ A \sim B $ to denote $ A \les B $ and $ B \les A $ and use $ A \ll B $ 
when we have $A \leq c B$ for small $c > 0$. We may write $A \les_\ta B$ for $A \leq C B$ with $C = C(\ta)$ 
if we want to emphasize the dependence of the implicit constant on some parameter $\ta$. 
We use $C, c > 0$, etc.~to denote various constants whose values may change line by line.

Given two functions $f$ and $g$ on $\R_+ \times \mc M$, with $\mc M = \R^2$ or $\T^2$, we write 
\begin{align}
f \approx g,
\label{approx1}
\end{align}

\noi
if there exist  $c_1, c_2 \in \R$ such that $f(t,x) + c_1 \le g(x) \le f(t,x) + c_2$ for any $(t,x) \in \R_+ \times \mc M \setminus \{0\}$.
Similarly,  given two sequences $\{f_N\}_{N\in \N}$ and $\{g_N \}_{N\in \N}$
 of functions, we  write 

\noi
\begin{align}
f_N \approx g_N,
\label{approx2}
\end{align}

\noi
if there exist  $c_1, c_2 \in \R$, independent of $N \in \N$,
 such that $f_N(t,x) + c_1 \le g_N(t,x) \le f_N(t,x) + c_2$ for any $(t,x) \in \R_+ \times \mc M \setminus \{0\}$.


Given a set $A \subset \R^d$ for $d \in \N$, we denote by $\ind_A$ the indicator function of $A$, by $\#A$ its cardinality and by $|A|$ its Lebesgue measure. Given a metric space $X$, we use  $B(x_0, r) \subset X$ to denote the open ball of radius
$r> 0$ centered at $x_0 \in X$.

For $x, y \in \T^2 \cong [-\pi, \pi)^2$, 
we set
\[ |x - y|_{\T^2} = \min_{k \in 2\pi \Z^2} |x - y + k|_{\R^2},\]

\noi
where $|\cdot |_{\R^2}$ denotes the standard Euclidean norm on $\R^2$.
When there is no confusion, 
we simply use $|\cdot|$ for both $|\cdot|_{\T^2}$ and $|\cdot |_{\R^2}$.

%
%

We set $\Z_{\ge 0} = \N \cup \{0\}$ and use the shorthand notation $\Z^d_{\ge 0}$ for $(\Z_{\ge 0})^d$ and $d \in \N$. Capital letters will sometimes denote dyadic numbers; namely, we write
$N \in 2^{\Z_{\ge 0}}$ and $L \in 2^\Z$, for example.

Given $a, b \in \R$, 
we set $a \vee b = \max(a, b)$
and $a \wedge b = \min (a, b)$.

Let $t >0$. We denote by $\mb S^1(t)$ the circle of centre $0$ and radius $t$ in $\R^2$ (or in $\T^2$, depending on the context). We will also use the notation $\mb S^1$ for $\mb S^1(1)$.

Throughout the paper, we use the standard multi-index notation. Namely, we call multi-index a vector of the form $\al = (\al_1, \al_2) \in \Z^2_{\ge 0}$ and write $|\al| = \al_1 + \al_2$ for its $\l^1$ norm. We also use the standard notation $\partial_{x}^\al$ for the derivative $\partial_{x_1}^{\al_1} \partial_{x_2}^{\al_2}$ in the canonical coordinate system $x = (x_1, x_2)$ in $\mc M$, with $\mc M = \T^2$ or $\R^2$.

\medskip

\noi
{\bf $\bullet$ Fourier transforms.} 
Due to the nature of our analysis, constants play important roles.
In particular, the sharp bounds on the space-time covariance of 
the truncated stochastic convolution \eqref{Psi_trunc2} 
(see Proposition \ref{PROP:cov} below)
 require us to carefully fix our conventions for Fourier transforms. 

We denote by $\Ft_{\R^2}$ and $\Ft_{\R^2}^{-1}$
the spatial Fourier transform on $\R^2$ and its inverse, respectively,
which  are given by
\noi
\begin{align}
\Ft_{\R^2} (f)(\xi) = \frac{1}{2\pi} \int_{\R^2} f(x) e^{-i \xi \cdot x} dx, \qquad \Ft_{\R^2} ^{-1}(f)(x) = \frac{1}{2 \pi} \int_{\R^2} f(\xi) e^{i \xi \cdot x} d \xi.
\label{F1}
\end{align}

\noi
We then define the convolution product on $\R^2$ by
\begin{align}
(f \ast g)(x) = \frac{1}{2 \pi} \int_{\R^2} f(y) g(x-y) dy
\label{F2}
\end{align}

\noi
such that $\Ft_{\R^2} (f \ast g) = \Ft_{\R^2} (f) \Ft_{\R^2}(g)$. Similarly, 
the Fourier transform $\Ft_{\T^2}$  on the torus $\T^2$ is given by
\begin{align}
\Ft_{\T^2} (f)(n) = \int_{\T^2} f(x) \cj{e_n (x)} dx, \quad n \in \Z^2,
\label{F3}
\end{align}

\noi
where 
\begin{align}
e_n (x)= \frac 1{2 \pi}e^{i n \cdot x}. 
\label{exp0}
\end{align}

\noi
Then, the Fourier inversion formula reads as 
\begin{align}
f(x) = \sum_{n \in \Z^2} \Ft_{\T^2} (f)(n)  e_n (x),
\label{F4}
\end{align}

\noi
We define 
the convolution product on $\T^2$ by 
\begin{align}
(f \ast g)(x) = \frac{1}{2 \pi} \int_{\T^2} f(y) g(x-y) dy,
\label{F5}
\end{align}

\noi
such that $\Ft_{\T^2} (f \ast g) = \Ft_{\T^2} (f) \Ft_{\T^2}(g)$.
We also define the space-time Fourier transform  on $\R \times \R^2$ 
by setting
\begin{align*}
\Ft_{t,x} (u)(\tau, \zeta) = \frac{1}{(2\pi)^{\frac32}} \int_{\R \times \R^2} u(t,x) e^{-i (t \tau +\zeta \cdot x)} d t d x.
\end{align*}

\noi
Then, 
the inverse Fourier transform is given by 
\begin{align*}
\Ft^{-1}_{t,x} (u)(t, x) = \frac{1}{(2\pi)^{\frac32}} \int_{\R \times \R^2} u(\tau,\zeta) e^{i (t \tau +\zeta \cdot x)} d \tau d \zeta.
\end{align*}

\noi
In the following, 
when it is clear from the context, 
 we write $\Ft(f)$ and $\ft f$ for the Fourier transform of a function $f$ defined either on 
  $\R^2$,  $\T^2$, and $\R \times \R^2$. 
  A similar comment applies to $\F^{-1}(f)$ and $\widecheck f$.

\medskip

Next, we recall the Poisson summation formula; see \cite[Theorem 3.2.8]{Grafakos1}. 
Let $f \in L^1 (\R^2)$ such that (i)
there exists  $\eta >0$ such that 
 $|f(x)| \les \jb{x}^{-2 - \eta}$ for any $x \in \R^2$, and (ii)~$\sum_{n \in \Z^2} |\F_{\R^2} (f) (n)| < \infty$. Then, we have
\begin{align}
\sum_{n \in \Z^2} \F_{\R^2} (f) (n) e_n (x) = \sum_{k \in \Z^2} f(x + 2 \pi k)
\label{poisson}
\end{align}

\noi
for any $x \in \R ^2$.

\medskip

Let $d \s$ denote the normalized surface measure on $\mathbb{S}^1$ and
 $\widecheck{d \s}$ denotes its inverse Fourier transform defined by
\begin{align}
\widecheck{d \s}(x) = \frac 1{2\pi} \int_{\mathbb{S}^1} e^{i \o \cdot x} d \s(\o).
\label{sphere}
\end{align}
Then, it follows from \cite[Theorem 1.2.1]{Sogge} that 
\begin{align}
\widecheck{d \s}(x) 
= e^{i |x|} a_+(x)  + e^{-i |x|} a_-(x), 
\label{sphere3}
\end{align}

\noi
where the functions $a_{\pm}$ are smooth and
\begin{align}
|\dx^\al  a_\pm (x)|\les \jb x^{-\frac 12 - |\al|}
\label{sphere4}
\end{align}

\noi
for any multi-index $\al \in \Z_{\ge 0}^2$.
See also \cite[Appendix B.8]{Grafakos1}.

\medskip

\medskip
\noi
{\bf $\bullet$ Sobolev  spaces.} Given $s \in \R$, the $L^2$-based Sobolev space $H^s(\T^2)$ is defined by the norm
\begin{align*}
\|f\|_{H^s} = \| \jb \nb ^s f   \|_{L^2_x} = \| \jb n ^s \ft f (n)\|_{\l ^2_n}.
\end{align*}
We also use the notation $\mc H^s(\T^2)$ for $H^s(\T^2) \times H^{s-1}(\T^2)$. Given $s \in \R $ and $1 \le p  \le \infty$, 
the $L^p$-based Sobolev $W^{s,p}(\T^2)$ is defined by the norm
\begin{align*}
\|f\|_{W^{s,p}} = \| \jb \nb ^s f   \|_{L^p_x} = \big\| \F^{-1} \big[ \jb n ^s \ft f (n) \big] \! \big\|_{L^p_x}.
\end{align*}
We define similarly the (time) Sobolev spaces $W^{s,p}(\R)$.

\subsection{Multiplier operators 
and  frequency projectors}
\label{SUBSEC:hyp}

In this subsection, we

the hyperbolic Riesz potentials

\medskip

\noi
{\bf $\bullet$ Green's functions and Bessel potentials.} 
The Green's function $G_{\R^2}$ for $1 - \Dl$ on~$\R^2$, satisfying
 $(1- \Dl) G_{\R^2} = \dl_0^{\R^2}$, where $\dl_0^{\R^2}$ is the Dirac delta function on $\R^2$, 
is given by
\begin{align}
\ft G_{\R^2}(\xi) = \frac{1}{2 \pi \jb \xi^2 }, \quad \xi \in \R^2
\label{green2}
\end{align} 

\noi
on the Fourier side.
Recall from \cite[Proposition 1.2.5]{Grafakos2}
that $G_{\R^2}$ 
 is a smooth function on $\R^2 \setminus \{ 0\}$ and decays exponentially as $|x | \to \infty$.
  Furthermore, it satisfies
\begin{align}
G_{\R^2} (x) = - \frac{1}{2\pi} \log |x| + o(1),
\label{green1}
\end{align}

\noi
as $x \to 0$; see \cite[(4,2)]{AS} and we have the estimate
\begin{align}
|\partial^\al_x G(x) | \les_\al \begin{cases} \jb{\, \log(|x|)} \ind_{|\al| = 0} + |x|^{-|\al|}\ind_{|\al| >0} & \quad \text{if $x \in B(0,2) \setminus \{0\}$}, \\
e^{-c |x|} & \quad \text{if $|x| \ge 2$}, \end{cases}
\label{Ysg29}
\end{align}
for some constant $c >0$ and any multi-index $\al \in \Z^2_{\ge 0}$.

Now, let $G$ be the Green's
function for $1-\Dl$ on $\T^2$.
In view of our normalization \eqref{exp0}, we have 
\begin{align}
\ft G (n) = \F_{\T^2}(\big(1-\Dl)^{-1} \dl_0\big)(n)
= \frac{1}{2 \pi \jb n ^2 }, \quad n \in \Z^2, 
\label{green3}
\end{align} 

\noi
where $\dl_0$ denotes the Dirac delta function on $\T^2$.
Moreover, by applying the Poisson summation formula \eqref{poisson}, we obtain
\begin{align}
G (x) =  - \frac{1}{2\pi} \log |x| + R(x), 
\quad  x \in \T^2 \setminus \{0\}, 
\label{green4}
\end{align} 

\noi
for  a smooth function $R$ on $\T^2$. 

Given  $\al >0$, let  $\jb \nb ^{-\al}$ 
be the Bessel potential of order $\al$ on $\T^2$ given by 
\begin{align}
\jb{\nb}^{-\al} f = J_\al * f,
\label{bessel0}
\end{align}

\noi
where the convolution kernel $J_\al$ is given by 
\begin{align}
J_{\al} (x) = \frac{1}{2 \pi} \sum_{n \in \Z^2} \frac{1}{\jb n ^\al} e_n (x), \quad x \in \T^2.
\label{bessel1}
\end{align}

\noi
Then, given  $0 < \al < 2$,
it follows from  \cite[Lemma 2.2]{ORSW1} that 
there exists 
 a smooth function $R_\al$ on $\T^2$ such that 
\begin{align}
J_{\al} (x) = c_\al |x|^{\al - 2} + R_\al(x),
\label{bessel2}
\end{align}

\noi
for any $x \in \T^2 \setminus \{ 0\} \cong [-\pi, \pi)^2 \setminus \{0\}$.

In Subsection \ref{SUBSEC:sto3}, we  need the Bessel potential in the temporal variable;
let  $\jb{\dt}^{-\al}$, $\al > 0$,  
be 
the Fourier multiplier operator
with the multiplier $\jb{\tau}^{-\al}$
whose convolution kernel $J_\al^{(t)}$ is given by 
\begin{align}
J_{\al}^{(t)} (t) = 
\frac{1}{2 \pi} 
\int_\R
 \frac{1}{\jb \tau ^\al} e^{i t \tau} d\tau,  \quad t \in \R.
\label{bessel2a}
\end{align}

\noi
Recall from 
\cite[Proposition 1.2.5]{Grafakos2}
that 
$J_{\al}^{(t)} $ is a smooth, strictly positive function on $\R \setminus \{0\}$.
Moreover, 
for 
 $0 < \al < 1$, there exists $c > 0$ such that 
\begin{align}
J_{\al}^{(t)}(t)  
\les 
\begin{cases}
e^{-|t|}, & \text{for } |t| \ge 2, \\
|t|^{\al - 1},
& \text{for } |t| < 2 .
\end{cases}
\label{bessel3}
\end{align}

\medskip

\noi
{\bf $\bullet$ Fractional derivation.} Consider the Fourier multiplier $(-\partial^2_{\tau})^s$, for $0 < s <\frac12$, on $\R$:
\[ \F_{\tau} [(-\partial^2_{\tau})^s f](t) = |t|^{2s} \ft f(t), \quad \text{for $t \in \R$}.  \]
The operator $(-\partial^2_{\tau})^s$ has the following integral representation; see \cite[Theorem 1]{Stinga}:
\begin{align}
(-\partial^2_{\tau})^s f(\tau) = c_s \int_{\R} \frac{ f(\tau+h) - f(\tau) }{|h|^{1+2s}} dh, \quad \text{for $\tau \in \R$}.
\label{fracder}
\end{align}
We use the representation \eqref{fracder} in Section \ref{SUBSEC:weighted}.

\medskip

\noi
{\bf $\bullet$ Poisson's formula.} Consider the Fourier multiplier on $\R^2$ given by $ \frac{\sin( t |\nb|)}{|\nb|}$ for $t \in \R_+$. Then, from \cite[(27) on p.\,74]{Evans}, it admits the following physical space representation as a convolution kernel:
\begin{align}
\frac{\sin( t |\nb|)}{|\nb|} f  =   W(t, \cdot) * f
\label{poisson2}
\end{align}
\noi
for any $t \in \R_+$ and where the wave kernel $W$ is defined as
\begin{align}
W(t,x) = \frac{\ind_{B(0,t)}(x)}{\sqrt{t^2 - |x|^2}}
\label{poisson3}
\end{align}
for any $(t,x) \in \R_+ \times \R^2$. The identity \eqref{poisson2} is often referred to as Poisson's formula. Note that for a fixed function $f$, the function $g = \frac{\sin( t |\nb|)}{|\nb|} f$ is the solution to Cauchy problem for the linear wave equation:
\begin{align*}
\begin{cases}
\dt^2 g  - \Dl g   = 0,\\
(g, \dt g) |_{t = 0} = (0, f), 
\end{cases}
\qquad (t, x) \in \R_+\times\R^2.
\end{align*}

\medskip

\noi
{\bf $\bullet$ Hyperbolic Riesz potentials.} 
Next, we introduce the {\it hyperbolic Riesz potential}
which plays a fundamental role in our analysis.
Let $\Box$ be 
 the d'Alembertian given by 
\begin{align}
\Box = \dt^2 -  \Dl.
\label{box1}
\end{align}

\noi
Then, 
given $b \in \R$, 
we define
the hyperbolic Riesz potential $\Box^{b}$
to be 
 the following space-time Fourier multiplier operator
with the following multiplier:
\begin{align}
\F_{t,x}\big[ \Box^{b} u \big] (\tau, \zeta) = 
\qf_b(\tau, \zeta) \big| \tau^2 - |\zeta|^2 \big|^b \ft u (\tau, \zeta),
\quad (\tau, \zeta) \in \R \times \R^2.
\label{Rhyp1}
\end{align}

\noi
Here,  the space-time multiplier  $\qf_b(\tau, \zeta)$ is given by
\begin{align}
\qf_b (\tau, \zeta) = 
\begin{cases} 
 e^{-b\pi i \cdot \text{sgn}(\tau)},  & \text{if } |\tau| \geq |\zeta|, \\
1,  & \text{if } |\tau| < |\zeta|,  \end{cases}
\quad \text{where } \ 
\text{sgn}(\tau) = 
\begin{cases}
1, & \tau \ge 0, \\
-1, & \tau < 0.
\end{cases}
\label{Rhyp2}
\end{align}

\noi
See \cite[(28.28)]{SKM} with $\al = -2b$.\footnote{Note 
that our sign conventions are slightly different from 
\cite[Subsection~28.1]{SKM}.
Moreover, there is a sign mistake in \cite[(28.28)]{SKM}.}
For example, when $b = 1$, 
it follows from \eqref{Rhyp2} that
\[  \qf_{1}(\tau, \zeta)\big| \tau^2 - |\zeta|^2 \big|
=- \tau^2 +|\zeta|^2\]

\noi
for $\tau \ne 0$, 
which corresponds almost everywhere to the symbol for  the standard d'Alembertian~$\Box$ in \eqref{box1}.
Note that the hyperbolic Riesz potential satisfies
the semigroup property:
\[ \Box^{b_1}\Box^{b_2} = \Box^{b_1+b_2}\]

\noi
for any $b_1, b_2 \in \R$.
From \eqref{Rhyp2}, 
we see that  the multiplier $\qf_b$  and its inverse $\qf_b^{-1}$ 
can  be written as a linear combination of the form
\begin{align}
\ld_1 \H \C + \ld_2 \C + \ld_3
\label{Rhyp3}
\end{align}

\noi
for some inessential constants $\ld_1, \ld_2, \ld_3 \in \mathbb{C}$,
where $\H$ and $\C$ are respectively the (temporal) Hilbert transform and the (sharp) cone multipliers defined by
\begin{align}
\F_{t,x} \big( \H u \big)(\tau, \zeta)&  = -i \text{sgn}(\tau) \ft u(\tau, \zeta),
\label{ht}\\
\F_{t,x} \big( \C u \big)(\tau, \zeta) & =  \ind_{|\tau| > |\zeta|} \ft u(\tau, \zeta).
\label{cone}
\end{align}

\noi
\begin{remark}[unboundedness of the cone multiplier]\label{RMK:cone}\rm
We note that the cone multiplier $\mathcal C$ in~\eqref{cone}
is unbounded 
 in $L^p(\R^3)$ for $1 < p \neq 2 < \infty $.
 Indeed, the unboundedness of 
the cone multiplier $\mathcal C$ 
follows from the 
unboundedness 
   of the (sharp) ball multiplier $\mathcal B$, 
  defined by 
\begin{align*}
\F_x \big( \mathcal B f \big) (\zeta) = \ind_{B(0,1)} (\zeta) \ft f (\zeta), 
\end{align*}

\noi
in   $L^p(\R^2)$  for $1 < p \neq 2 < \infty $
due to Fefferman \cite{Fefferman};
see \cite{Mock, LV}.
See also \cite[Proposition 3.2 on p.374]{deLeeuw}\footnote{Strictly speaking, 
Proposition 3.2 
in \cite{deLeeuw} is not directly applicable 
but one can proceed with a limiting argument.}
 for such an argument. We thus need to proceed with care when estimating objects involving the symbol $\qf_b$ in \eqref{Rhyp2}.
\end{remark}

For $b < - \frac12$, the hyperbolic Riesz potential  $\Box^{b}$ 
admits the following physical side representation as a convolution operator
(see \cite[(28.21)]{SKM}):

\noi
\begin{align}
 ( \Box^{b} u )(t,x) =  \int_{\R \times \R^2}  \mf K_b (t-t', x-y) u(t',y) dt' dy,
\label{Rhyp4}
\end{align}

\noi
where the kernel $\mf K$ is given by

\noi
\begin{align}
\mf K_b (t, x) = c_b \frac{\ind_{t\ge0} \ind_{B(0, t)}(x)}{( |t|^2 - |x|^2 )^{\frac32 + b}}, \quad (t, x) \in \R \times \R^2.
\label{kerhyp}
\end{align}

\noi
See also \cite[(28.19)]{SKM}, where the condition $b <- \frac 12$ appears.
Compare \eqref{Rhyp4} (when $b = -1$) with  Poisson's formula 
for a solution to the wave equation on $\R^2$; see \eqref{poisson2}-\eqref{poisson3} above.
See  \cite[Subsection~28.1]{SKM}
for a further discussion. 

Now, consider the Fourier multiplier on $\R \times \T^2$ given by
\begin{align}
\F_{t,x}\big( \Box_{\T^2}^{b} u \big) (\tau, n) = 
\qf_b(\tau, n) \big| \tau^2 - |n|^2 \big|^b \ft u (\tau, n),
\quad (\tau, n) \in \R \times \Z^2.
\label{Ysghyp}
\end{align}
When there is no possible confusion, we also write $\Box^{b}$ for $\Box^{b}_{\T^2}$. From the Poisson formula \eqref{poisson} and an approximation argument (see for instance \cite[Lemma 2.5]{ORW}), the convolution kernel $\mf K^{\T^2}_b$ of $\Box_{\T^2}^{b}$ is given by
\begin{align}
\mf K^{\T^2}_b (t,x) = \sum_{m \in \Z^2} \mf K_b(t, x + 2\pi m)
\label{Ysgker}
\end{align}
for all $(t,x) \in \R \times \T^2 \cong \R \times [-\pi, \pi)^2$. Note that the sum in \eqref{Ysgker} is finite since the spatial support of $\mf K_b$ is included in $B(0,|t|)$. In particular, if $u \in \mc S(\R \times \T^2; \R)$ then we have
\begin{align}
(\Box^{b}_{\T^2} u)(t,x) = \int_{\R \times \R^2} \mf K_b (t-t', x-y) u(t',y) dt' dy
\label{Ysgker2}
\end{align}
for all $(t,x) \in \R \times [-\pi, \pi)^2$ and where the function $u(t', \cdot)$ is viewed as a $2\pi$-periodic function on $\R^2$.

\medskip

\noi
{\bf $\bullet$ Frequency projectors and paraproducts.} 
In the following, we define various frequency projectors and paraproducts.
 Let $\varphi \in C^{\infty}_c(\R; [0,1])$ be a smooth and symmetric bump function such that
\[  \varphi(\tau) = \begin{cases} 1 \quad & \text{if $|\tau| \le \frac{5}{4}$}, \\
0 \quad & \text{if $|\tau| > \frac{8}{5}$}. \end{cases} \]
Then, we define  $\phi \in C^{\infty}_c(\R^2; [0,1])$ and $(\eta, \psi) \in C^{\infty}_c(\R; [0,1])^2$
by setting 
\begin{align}
\begin{split}
\phi ( \zeta ) &  = \varphi(|\zeta|) - \varphi(2|\zeta|),\\
\eta(\tau) & = \varphi (\tau) - \varphi(2 \tau),\\
\psi(\tau) & = \varphi (\tau) - \varphi(2\tau).
\end{split}
\label{eta1}
\end{align}

\noi
Obviously, we have $\eta = \psi$.
We however, introduce these two functions
since we use $\eta$ for localization
in temporal frequencies, while
we use $\psi$ for localization in modulation (namely, the variable $|\tau| - |\zeta|$).

For any dyadic numbers $N,R,L \in 2^\Z$, we define the following 
Littlewood-Paley frequency projectors:
\begin{align}
\F_{t,x} \big(\P_N u \big) (\tau, n ) 
& = \phi \Big(\frac{\zeta}{N}\Big) \ft u (\tau, n), 
\label{proj1}\\
\F_{t,x} \big(\mathbf{T}_R u \big) (\tau, n )
&  = \eta \Big(\frac{\tau}{R}\Big) \ft u (\tau,n), 
\label{proj2}
\\
\F_{t,x} \big( \M_{N,R,L} u \big) (\tau, n)
&  =  \phi \Big(\frac{\zeta}{N}\Big) \eta \Big(\frac{\tau}{R}\Big) \psi \Big(\frac{|\tau| - |\zeta|}{L}\Big) \ft u (\tau, n)
\label{proj3}
\end{align}

\noi
for $(\tau,n) \in \R \times \Z^2$. By construction, we have
\begin{align*}
\sum_{N \in 2^\Z} \P_N = \sum_{R \in 2^\Z} \mathbf{T}_R  = \sum_{(N,R, L) \in (2^\Z)^3} \M_{N,R,L} 
= \Id .
\end{align*}
Let $\K_N$ and $\mc T_R$ be the respective convolution kernels of $\P_N$ and $\mbf T_R$ defined above. By integration by parts, it is easy to see that for any $\al \in \Z_{\ge 0}^2$, $k \in \Z_{\ge 0}$
snd  $A \ge 1$, we have 
\begin{align}
\begin{split}
& |\partial^{\al}_x \K_N(x) | \les_{\al, A} N^{|\al|+2} \jb{N x}^{-A}, \\
& |\partial^{k}_t  \mc T_R(t) | \les_{k, A} N^{k+1} \jb{R t}^{-A}
\end{split}
\label{ker1}
\end{align}
for all $x \in \T^2$ and $t \in \R$.

We  set 
\begin{align}
\P_{\textup{lo}} = \sum_{\substack{N \in 2^\Z \\ N \le 1}} \P_N
\quad \text{and}\quad  
\P_{\textup{hi}} = \sum_{\substack{N \in 2^\Z \\ N > 1}} \P_N.
\label{proj4a}
\end{align}

\noi
We also introduce the following space-time frequency projectors,
allowing us to compare spatial and temporal frequencies:
\begin{align}
\begin{split}
&  \Q^{\text{hi,hi}}  = \sum_{\substack{j, k \in \Z^2 \\ |j-k| \le 2}} \mathbf{T}_{2^k} \P_{2^j}  ,
\quad 
  \Q^{\text{hi,lo}}  = \sum_{\substack{j, k \in \Z^2 \\ k > j+2}} \mathbf{T}_{2^k} \P_{2^j}  ,\\
 \text{and} \quad   & \Q^{\text{lo,hi}}  = \sum_{\substack{j, k \in \Z^2 \\ k < j-2}} \mathbf{T}_{2^k} \P_{2^j}  , 
\end{split}
\label{proj4}
\end{align}


\noi
As a direct consequence of the H\"ormander-Mihlin multiplier theorem, the operators $\P_N$, $\mathbf{T}_R$, $\Q^{\text{hi,hi}}$, $\Q^{\text{lo,hi}}$ and  $\Q^{\text{hi,lo}}$ bounded in $L^p(\R^3)$ for any $1 < p < \infty$. Since the symbol $\psi( |\tau| - |\zeta| )$ does not decay when $(\tau, \zeta)$ is close to the light cone $\{(\tau, \zeta): |\tau| = |\zeta|\}$, the H\"ormander-Mikhlin multiplier theorem is {\it not} applicable to $\M_{N,R,L}$.

Finally, 
for $\g >0$, we define the following $\g$-dependent (spatial) paraproducts:
\begin{align}
& \mathcal{P}^{>}_{\g}(u,v) = \sum_{\substack{(N_1,N_2) \in (2^\Z)^2 \\ N_1 \ge N_2^{\g}}} \P_{N_1} u \cdot \P_{N_2}v \label{para1}, \\
& \mathcal{P}^{<}_{\g}(u,v) = \sum_{\substack{(N_1,N_2) \in (2^\Z)^2 \\ N_1 < N_2^{\g}}} \P_{N_1} u \cdot \P_{N_2}v \label{para2}.
\end{align}

\noi
Note  that 
we have $u v = \mathcal{P}^{>}_{\g}(u,v) + \mathcal{P}^{<}_{\g}(u,v)$.
 for any space-time functions $u$ and $v$.
 
 \subsection{Function spaces and linear estimates}\label{SUBSEC:spaces}
In this subsection, we define the function spaces used in this work and study their properties.
\begin{definition} \label{DEF:spaces1}
\rm
Let $s, b \in \R $ and $1 \le p, q \le \infty$. We define the spaces $X^{s,b}(\R \times \T^2)$ and $Y^{s,b}_{p,q}(\R \times \T^2)$ as the completions of $\mathcal{S}(\R \times \T^2)$ under the norms

\noi
\begin{align}
\|u\|_{X^{s,b}(\R \times \T^2)} & := \| \jb \zeta ^s \jb{ |\tau| - |\zeta|} ^b  \ft u (\tau, \zeta)  \|_{ L^2_\zeta L^2_\tau(\R \times \T^2)}, \label{X1} \\
\|u\|_{Y_a^{s,b}(\R \times \T^2)} & := \big\| \jb t^{a} \, \F^{-1}_{t,x} \big( \jb \zeta ^s \, ||\tau| - |\zeta||^b  \ft u (\tau, \zeta) \big)  \big\|_{L^2_{t,x}(\R \times \T^2)} \label{Y1}.
\end{align}

\noi
where $\Box^{b}$ is as in \eqref{Ysghyp}. We also use the shorthand notation $Y^{s,b}_p$ and $Y^{s,b}_a$ for $X^{s,b}$ and $Y^{s,b}_a$, respectively.
\end{definition}

\noi
\begin{remark}\rm \label{REM:scaling1}
Our choice of having a homogeneous modulation symbol $\big||\tau|- |\zeta|\big|$ in the $Y^{s,b}_a$-norm 
is motivated by  the fact that 
the hyperbolic Riesz potential 
 $\Box^{b}$ in \eqref{Rhyp1} depends on the homogeneous symbol $\big||\tau|^2- |\zeta|^2\big|$; see \eqref{Rhyp2} and \eqref{Ysghyp}. Working with a homogeneous modulation weight is also useful in our nonlinear analysis, see the proof of Lemma \ref{LEM:wcone}.
\end{remark}

In this work it is also convenient to work with the following anisotropic Sobolev spaces.

\noi
\begin{definition} \label{DEF:spaces2}
\rm 

Let $(s, b) \in \R^2 $ and $1 \le p \le \infty$. We define the space $\Ld^{s,b}_{p}(\R \times \T^2)$ as the completion of $\mathcal{S} (\R \times \R^2)$ under the norm

\noi
\begin{align}
\|u\|_{\Ld^{s,b}_{p}(\R\times \T^2)} := \big\| \jb \nb ^s \jb \dt ^b u \big\|_{L^p_{t,x}(\R\times \T^2)}.
 \label{ld1}
\end{align} 
\end{definition}

For an interval $I \subset \R$,  we define the restriction $X^{s,b}(I)$ 
of the space $X^{s,b}(\R \times \T^2)$
 onto interval 
$I$ via the norm:
\begin{equation}
\| u \|_{X^{s,b}(I)} := \inf \big\{ \| v \|_{X^{s,b}}: v|_{I\times \T^2} = u \big\}.
\label{loc1}
\end{equation}
When $I = [0,T]$ for $T>0$, we use the shorthand notation $X^{s,b}_T$ for $X^{s,b}([0,T])$.

\begin{remark}\label{RMK:loc}\rm
Let $B(\R \times \T^2)$ be a space of space-time functions and $I \subset \R$ an interval. In Section \ref{SEC:4}, we use the notation $B(I)$ to denote the subspace
\begin{align}
\{ u \in B(\R \times \T^2) : \|\ind_I(t) u\|_{B(\R \times \T^2)} < \infty \}.
\label{loc10}
\end{align}
If $B = X^{s,b}$, then the two spaces \eqref{loc1} and \eqref{loc10} coincide for $b < \frac12$; see Lemma \ref{LEM:restri} (ii).
\end{remark}

We borrow the following gluing lemma from \cite[Lemma 4.5]{Bring2}.
\begin{lemma}\label{LEM:gluing}
Let $s \in \R$, $\frac12 < b < 1$ and $I_1, I_2 \subset \R$ be bounded intervals such that $I_1 \cap I_2 \neq \emptyset$. Then, we have
\begin{align}
\|u\|_{X^{s,b}(I_1 \cup I_2)} \les |I_1 \cap I_2|^{\frac12 - b}\big( \|u\|_{X^{s,b}_{I_1}} + \|u\|_{X^{s,b}_{I_2}}\big).
\end{align}
\end{lemma}

Consider the nonhomogeneous linear damped wave equation:
\begin{align}
\begin{cases}
\dt^2 u + \dt u + (1- \Dl)  u   = F \\
(u, \dt u) |_{t = 0} = (u_0, v_0), 
\end{cases}
\qquad (t, x) \in \R_+\times\T^2.
\label{wdamped}
\end{align}
The solution to \eqref{wdamped} is given by
\begin{align}
u(t) = \mc U(t)(u_0, v_0) + \I(F)(t), \quad t \in \R_+,
\end{align}
where $\mc U(t)$ is the linear operator 
\begin{align}
\mc U(t)(u_0,v_0) = \dt\D(t)u_0 + \D(t)(u_0+v_0)
\end{align}
and $\I$ is as in \eqref{duha}. We now state linear estimates in $X^{s,b}$-spaces for the problem \eqref{wdamped}. To this end, we need to extend the definitions of the linear operators $\mc U$ and $\I$ to the whole real line in an appropriate way. See \cite{MR, LTW, Zine2} for similar issues. Define the operators $\wt{\mc{U}}$ and $\wt \I$ by
\begin{align}
 \wt{\mc{U}}(t) (u_0, v_0) = \Big( e^{-\frac{|t|}{2}} \cos(t \fbb \nb)u_0 + e^{-\frac{|t|}{2}}  \frac{\sin(t \fbb \nb)}{2\fbb \nb}\Big) u_0 + e^{-\frac{|t|}{2}} \frac{\sin (t \fbb \nb)}{\fbb \nb} v_0, \quad t \in \R,
\label{lin2}
\end{align}
and
\begin{align}
\wt \I(F)(t) = \frac{1}{2 i} \big( \I_+(F)(t) - \I_-(F)(t) \big), \quad t \in \R,
\label{duha2}
\end{align}
where 
\begin{align*}
\I_{\pm}(F)(t) = \sum_{n \in \Z^2} \frac{e_n(x)}{\fbb n} \int_\R \frac{e^{it \mu} - e^{-\frac{|t|}{2} \pm it \fbb n}}{\frac12 + i \mu \mp i \fbb n} \ft F(\mu,n) d \mu.
\end{align*}
One then observes that $ \wt{\mc{U}}(t) (u_0, v_0) = \mc U(t) (u_0, v_0)$ and $\wt \I(F)(t) = \I(F)(t)$ for any $t \ge 0$ (see \cite[page 16]{LTW}).

We first state the linear homogeneous estimate for $\wt{\mc{U}}(t)$. See \cite[Lemma 2.7]{LTW} for a proof.
\begin{lemma}\label{LEM:lin}
Let $s \in \R$, $b < \frac32$ and $I$ be an interval. Then, we have
\begin{align*}
\big\|\wt{\mc{U}}(t)(u_0, v_0)\big\|_{X^{s,b}(I)} \les (1 + |I|) \|(u_0, v_0)\|_{\mc H^s}.
\end{align*}
\end{lemma}
Next, we recall the linear nonhomogeneous estimate for the modified Duhamel operator $\wt \I$. See \cite[Lemma 2.8]{LTW} for a proof in the identical three-dimensional case.

\begin{lemma}\label{LEM:inho}
Fix $s \in \R$, $\frac12 < b < 1$ and an interval $I \subset \R$ such that $0 \le |I| \le 1$. Then, we have
\begin{align*}
\big\|\wt \I(F)\big\|_{X^{s,b}(I)} \les \|F\|_{X^{s-1,b-1}(I)}.
\end{align*}
\end{lemma}
Lastly, we record the following time localization estimate. See \cite[Lemma 2.9]{LTW}.
\begin{lemma}\label{LEM:timeloc}
Let $s \in \R$, $-\frac12<b_1 < b_2 < \frac12$ and $I \subset \R$ be a closed interval. Then, we have
\begin{align}
\|u\|_{X^{s,b_1}(I)} \les |I|^{b_2 - b_1}\|u\|_{X^{s,b_2}(I)}.
\end{align}
\end{lemma}

We now recall the Strichartz estimates
for the  linear wave equation.
Given  $0 < s < 1$, 
we say that a pair $(q, r)$ is $s$-admissible 
(a pair $(\wt q, \wt r)$ is dual $s$-admissible,\footnote{Here, we define
the notion of dual $s$-admissibility for the convenience of the presentation.
Note that $(\wt q, \wt r)$ is dual $s$-admissible
if and only if $(\wt q', \wt r')$ is $(1-s)$-admissible.}
 respectively)
if $1 \leq \wt q < 2 < q \leq \infty$, 
 $1< \wt r \le 2 \leq r < \infty$, 
\begin{align}
 \frac{1}{q} + \frac 2r  = 1-  s =  \frac1{\wt q}+ \frac2{\wt r} -2, 
\qquad
\frac 2q + \frac{1}{r} \leq \frac 1 2, 
\qquad \text{and} 
\qquad  
\frac2{\wt q}+\frac1{\wt r} \geq \frac52   .
\label{admis1}
\end{align}

\noi
We refer to the first two equalities as the scaling conditions
and the last two inequalities as the admissibility conditions.

Let us now state a lemma, providing a more direct description of the admissible exponents;
see \cite[Lemma 3.1]{GKO}

\begin{lemma} \label{LEM:pair} 
Let $0< s<1$.
A pair $(q,r)$ is $s$-admissible if 
\begin{align}
  \frac{1}{q} + \frac 2r  = 1-  s 
\qquad  
\text{and} \qquad
2\le r \le 
\begin{cases}
 \frac6{3-4s}, & \text{if } s < \frac34, \\ 
\infty, & \text{otherwise}. 
\end{cases}
\label{admis2}
\end{align}

\noi
A pair $(\wt  q,\wt  r)$ is dual $s$-admissible if 
\begin{align} 
\frac1{\wt q} + \frac2{\wt r} = 3 - s 
\qquad  
\text{and} \qquad
\max\bigg\{1+, \frac6{7-4s} \bigg\}  \le   \wt r \le \frac2{2-s}. 
\label{admis3}
\end{align}
\end{lemma}

The Strichartz estimates on $\R^d$ have been studied by many
mathematicians.  See Ginibre-Velo \cite{GV}, Lindblad-Sogge \cite{LS},
and Keel-Tao \cite{KeelTao}. 
and   the finite speed of propagation for the wave equation.

The transference principle (\cite[Theorem 3.2]{KS}).
See also \cite[Lemma 2.9]{TAO}.

\noi
\begin{lemma}\label{LEM:stri}
Given $0< s <1$, let $(q,r)$ be $s$-admissible. Fix $0 < T \le 1$. Then, for $b>\frac12$, we have

\noi
\begin{align*}
\|u\|_{L^q_t([0, T];  L^r_x(\T^2))} \les \|u\|_{X^{s,b}_T}.
\end{align*}
\end{lemma}

In particular by \eqref{admis2} and Lemma \ref{LEM:stri}, we have the following estimate for $0 < \dl < \frac14$:

\noi
\begin{align}
\|u\|_{L_t^{\frac{6}{1+2\dl}}([0, T]; L^{\frac{6}{1-4\dl}}_x(\T^2))} \les \|u\|_{X^{\frac12+\dl, \frac12+\eps}_T},
\label{stri1}
\end{align}

\noi
for any $\eps >0$.

\smallskip

Lastly, we state a result on the boundedness of the multiplication with smooth functions of the time variable on the spaces defined above.

\begin{lemma}\label{LEM:restri}
For any $\ld \in C^{\infty}_c(\R; \R)$ and $s \in \R$. Then, the following bounds hold:

\smallskip

\noi
\textup{(i)} For any $b\in\R$ and $1 < p < \infty$, we have
\begin{align}
\|\ld(t)  u\|_{\Ld^{s,b}_{p}} 
& \les \|u\|_{\Ld^{s,b}_{p}}, \label{restri1} \\
\|\ld(t) u\|_{X^{s,b}}  & \les  \|u\|_{X^{s,b}}. \label{restri2}
\end{align}

\smallskip

\noi
\textup{(ii)} For any $-\frac12 < b<\frac12$ and interval $I \subset \R$, we have
\begin{align}
\| u\|_{X^{s,b}(I)}
 \sim \|\ind_{I}(t) u\|_{X^{s,b}}. \label{restri3}
 \end{align}

\end{lemma}
\begin{proof} See \cite[Lemma 2.11]{TAO} and \cite[Lemma 4.4]{Bring2} for proofs of \eqref{restri2} and \eqref{restri3}. The bound \eqref{restri1} follows from standard product estimates.
\end{proof}

\section{Elliptic and hyperbolic singular kernels}\label{SEC:ker} 

In this section, we prove technical lemmas on various convolutions with singular kernels which exhibit singularities at either a point or on a light cone. These results are crucial in Section \ref{SEC:4} to study regularity properties of the stochastic convolution and the imaginary Gaussian multiplicative chaos.

Let us note that in this section, some results are either stated for the periodic or the full space settings (or for both).

\subsection{Elliptic singularities} We first study singularities which are of elliptic type; namely, functions which are singular at a single point. Many of the results which follow deal with estimates on convolutions of such singularities with bump functions $\nu_N$ which satisfy a decay condition, see \eqref{CC000} and \eqref{Gd0} below. We note that in the case when $\nu_N$ is actually compactly supported, some of these results are already essentially proved in the literature on parabolic singular stochastic partial differential equations; see for instance \cite{Hairer}.  

\smallskip

We first prove an estimate on smoothed elliptic singularities.
\begin{lemma}\label{LEM:t0}
Let $\ta \in (0,2)$. Fix $N \in \N$ and let $\nu_N : \T^2 \to \R$ be a function satisfying the bound

\noi
\begin{align}
|\nu_N(x)| \les_A N^{2} \jb{Nx}^{-A}
\label{CC000}
\end{align}
for any $x \in \T^2$ and any finite $A \ge 1$. Then, we have 
\begin{align}
\int_{\T^2} \nu_N (x-y) |y|^{-\ta}dy  \les \min(N, |x|^{-1})^{\ta}, 
\label{Ta1}
\end{align}

\noi
for any $x \in \T^2$, with an implicit constant independent of $N$.
\end{lemma}

\begin{proof}
Fix $N \in \N$ and $x \in \T^2$. We separately estimate
the contributions from  $|x-y| \les N^{-1}$ 
and $|x-y| \gg N^{-1}$ to the integral
\[  \1(x) := \int_{\T^2} \nu_N (x-y) |y|^{-\ta}dy. \]

\smallskip

\noi
$\bullet$ {\bf Case 1:  $|x-y| \les N^{-1}$.}\quad We first consider the case  $|x| \les N^{-1}$.
Then by \eqref{CC000}, the contribution to $\1(x)$ in this case is bounded by 
\begin{align}
\1(x) \les N^2 \int_{B(0,c N^{-1})} |y|^{-\ta} dy \les N^{\ta}
 = \min(N, |x|^{-1})^{\ta}. 
\label{Ta2}
\end{align}

Next, suppose that $|x| \gg N^{-1}$.
Since $|x- y| \les N^{-1}$, 
we have  $|x| \sim |y| \gg N^{-1}$. 
Thus, from \eqref{CC000}, the contribution to $\1(x)$ in this case is bounded by 
\begin{align}
\1(x) \les N^2 |x|^{-\ta} \int_{B(x,10 N^{-1})} dy \les |x|^{-\ta} = \min(N, |x|^{-1})^{\ta}.
\label{Ta3}
\end{align}

\smallskip

\noi
$\bullet$ {\bf Case 2:  $|x-y| \gg N^{-1}$.}\quad In this case, we have
$\jb{N(x-y)} \sim N|x-y|$.
We first consider the case  $|x| \ll N^{-1}$.
By \eqref{CC000} with $|y| \sim |x-y| \gg N^{-1}$, 
the contribution to $\1(x)$ in this case is bounded by 
\begin{align}
\1(x) \les N^{-8} \int_{|y|\gg N^{-1}} |y|^{-10 -\ta} dy \les N^{\ta}
 = \min(N, |x|^{-1})^{\ta}. 
\label{Ta4}
\end{align}

Next, we consider  the case  $|x| \ges N^{-1}$.
By \eqref{CC000} and estimating separately the cases
(i)~$|x| \les |y|$
and (ii)~$|x| \gg |y|$ (which implies $|x|\sim |x-y|$), 
we bound the contribution to $\1(x)$ in this case by 
\begin{align}
\begin{split}
\1(x) &  \les N^{-8}|x|^{-\ta} \intt_{|x-y|\gg N^{-1} }|x-y|^{-10}  dy 
+ N^{-8} |x|^{-10} \intt_{|y| \ll |x|} |y|^{-\ta} dy \\
& \les |x|^{-\ta}  = \min(N, |x|^{-1})^{\ta}. 
\end{split}
\label{Ta5}
\end{align}

By putting \eqref{Ta2}, 
 \eqref{Ta3},  \eqref{Ta4},  and \eqref{Ta5} together, 
 we obtain  \eqref{Ta1}. 
\end{proof}

Note that in view of \eqref{Ysg29}, the Green's function \eqref{green2} has an elliptic singularity at the origin. In the next lemma, we prove various bounds on smoothed Green's functions. 

\begin{lemma}\label{LEM:green_der}
Fix $N \in \N$ and let $\nu_N \in C^{\infty}(\R^2; \R)$ be a function satisfying the bound
\begin{align}
|\partial_x^{\al}\nu_N(x)| \les_A N^{|\al|+2} \jb{Nx}^{-A}
\label{Gd0}
\end{align}
for all $A \ge 1$ and $\al \in \Z^2_{\ge 0}$. Let $G$ be the Green's function \eqref{green2}. Then, the following bounds hold.

\smallskip

\noi
\textup{(i)} Set $G_N = G * \nu_N$. Then,  we have

\noi
\begin{align}
| G_{N} (x)| \les \begin{cases} \jbb{\,\log \big(|x| + N^{-1}\big)} \quad & \text{for $|x| < 2$}, \\
\jb x^{-A} \quad & \text{for $|x| \ge 2$} \\
  \end{cases}
\label{Gd1}
\end{align}
for any $A \ge 1$ and 

\noi
\begin{align}
|\partial^\al_x G_{N} (x)| \les \begin{cases}  \big(|x| + N^{-1}\big)^{-|\al|} \big(1 + \jbb{\,\log \big(|x| + N^{-1}\big)} \ind_{|\al|=2}\big) \quad & \text{for $|x| < 2$}, \\
\jb x^{-A} \quad & \text{for $|x| \ge 2$} \\
  \end{cases}
\label{Ysg33}
\end{align}
for any $A \ge 1$ and $\al \in \Z^2_{\ge 0}$ with $1 \le |\al| \le 2$. Here, the implicit constants are independent of $N$.

\smallskip

\noi
\textup{(ii)} Set $\wt G_N = G * G * \nu_N$. Then, we have

\noi
\begin{align}
\big|\partial^\al_x \wt G_{N} (x)\big| \les \jb x^{-A}
\label{Ysg34}
\end{align}
for any $A \ge 1$ and $\al \in \Z^2_{\ge 0}$ with $0 \le |\al| \le 1$ and 

\noi
\begin{align}
\big| \partial_x ^\al \wt G_{N} (x) \big| \les \begin{cases} \jbb{\,\log \big(|x| + N^{-1}\big)} \quad & \text{for $|x| < 2$}, \\
\jb x^{-A} \quad & \text{for $|x| \ge 2$} \\
  \end{cases}
\label{Ysg35}
\end{align}
for any $A \ge 1$ and $\al \in \Z^2_{\ge 0}$ with $|\al| = 2$. Here, the implicit constants are independent of $N$.
\end{lemma}

\begin{proof}
Fix $N \in \N$. We first prove (i). Let $\ld \in C^{\infty}_c(\R^2; \R)$ be a smooth bump function such that 
\begin{align*}
\ld(x) = \begin{cases}1 \quad & \text{if $|x| < 1$}, \\
 0 \quad & \text{if $|x| \ge 2$}.  \end{cases}
\end{align*} 
We decompose $G_N$ as follows:

\noi
\begin{align}
G_N= ( \ld G) * \nu_N + \big( (1-\ld) G \big) * \nu_N =: \1 + \II.
\label{Gd2}
\end{align}
Let us focus on \eqref{Gd1}. By \eqref{Gd2}, it suffices to prove \eqref{Gd1} with $G_N$ replaced with $\1$ and $\II$.

Fix $x \in \R^2$. By \eqref{Ysg29}, we have that
\begin{align}
|\1(x)| \les \int_{\R^2} \jbb{\, \log(|x - y |)} \, \ind_{|x-y| \le 2} \, |\nu_N(y)| dy.
\label{Ysg36}
\end{align}
If $|x| \les 1$, then \eqref{Ysg36} and arguments similar to those in the proof of Lemma \ref{LEM:t0} show
\begin{align}
|\1(x)| \les  \int_{B(0,10)} \jbb{\,\log \big(|x - y |\big)} |\nu_N(y)| dy \les \jbb{\,\log \big(|x| + N^{-1}\big)}.
\label{Ysg37}
\end{align}
If $|x| \gg 1$, then we note that $|x| \sim |y| \ges 1$ for each $y$ in the support of the integrand of $\1(x)$. Hence, by \eqref{Ysg36} and \eqref{Gd0}, we have
\begin{align}
|\1(x)| & \les N^{2 -A} \int_{\R^2} \jbb{\,\log \big(|x - y |\big)} \, \ind_{|x-y| \le 2} \, \jb y^{-A} dy \les \jb x^{-A}
\label{Ysg38}
\end{align}
for each $A \ge 2$. Thus, \eqref{Ysg37} and \eqref{Ysg38} show \eqref{Gd1} for $\1$. By \eqref{Ysg29} and similar arguments, we also have that 
\begin{align*}
|\II(x)| \les  \jb x^{-A}
\end{align*}
for any $x \in \R^2$ and $A \ge 2$. This shows \eqref{Gd1} for $\II$ and finishes the proof of \eqref{Gd1}.

Next, we prove \eqref{Ysg33}. Fix a multi-index $\al \in \Z^2_{\ge 0}$. If $|\al| = 1$, then \eqref{Ysg33} follows from the equality
\[ \partial^\al_x G_N = ( \partial_x ^\al G) * \nu_N, \]
the bound \eqref{Ysg29} and by arguing as in the proof of \eqref{Gd1}.

Now we prove \eqref{Ysg33} for $\al \in \Z^2_{\ge 0}$ and $|\al| = 2$. In this case, we have to proceed with care since the distribution $\partial_x ^{\al} G$ is not locally integrable near the origin; see \eqref{Ysg29}. Let $\al_1, \al_2 \in \Z^2_{\ge 0}$ such that $\al = \al_1 + \al_2$ and $|\al_1| = |\al_2| = 1$. In particular, we have 
\begin{align}
\partial_x ^\al \1 = \partial_x ^{\al_2} \big\{ \partial^{\al_1} _x \1 \big\} =\partial_x ^{\al_2} \big\{   \big( \partial_x^{\al_1} ( \ld G) \big) * \nu_N \big\}.
\label{Ysg38b}
\end{align}
Next, we decompose $\big( \partial_x^{\al_1} ( \ld G) \big) * \nu_N$ as follows:
\begin{align}
 \big( \partial_x^{\al_1} ( \ld G) \big) * \nu_N = \big( \ld(N \cdot) \partial_x^{\al_1} ( \ld G) \big) * \nu_N + \big((1-\ld)(N\cdot) \partial_x^{\al_1} ( \ld G) \big) * \nu_N =: \1_1^{\al_1} + \1_2^{\al_1}.
\label{Ysg39}
\end{align}
Therefore, by \eqref{Gd2}, \eqref{Ysg38b} and \eqref{Ysg39}, \eqref{Ysg33} (with $|\al| = 2$) follows from the bound
\begin{align}
|\partial_x^{\al_2} \1^{\al_1}_1 (x) | + |\partial_x^{\al_2} \1^{\al_1}_2 (x) | + |\partial^{\al}_x \II(x)| \les \begin{cases}  \big(|x| + N^{-1}\big)^{-2} \jbb{\,\log \big(|x| + N^{-1}\big)} \quad & \text{for $|x| < 2$}, \\
\jb x^{-A} \quad & \text{for $|x| \ge 2$} \\
  \end{cases}
\label{Ysg40}
\end{align}
for any $A \ge 1$.

We first consider the contribution of $\partial^{\al_2}_x \1_1^{\al_1}$. Fix $x \in \R^2$. By \eqref{Ysg29}, we have that

\noi
\begin{align}
\begin{split}
|\partial^{\al_2}_x \1_1^{\al_1} (x) | & = \big| \big( \big( \ld(N \cdot) \partial_x^{\al_1} ( \ld G) \big) *  \partial^{\al_2}_x \nu_N \big)(x) \big| \\
& \les  \int_{\R^2} |x-y|^{-1} \, \ind_{|x-y| \les N^{-1}} | \partial_x^{\al_2} \nu_N(y) | dy.
\label{Gd3}
\end{split}
\end{align}
If $|x| \les N^{-1}$, then we have $|y| \les N^{-1}$ for each $y$ in the support of the integrand of $\partial_x^{\al_2} \1_1^{\al_1}(x)$. Hence, by \eqref{Gd0} and \eqref{Gd3}, we have
\begin{align}
|\partial^{\al_2}_x \1_1^{\al_1} (x) | \les N^3 \int_{\R^2} |x-y|^{-1} \, \ind_{|x-y| \les N^{-1}} dy \les N^2 \sim \big( |x| + N^{-1}\big)^2.
\label{Ysg41}
\end{align}
Otherwise, $|x| \gg N^{-1}$ and $|y| \sim |x|$ for each $y$ in the support of the integrand of $\partial_x^{\al_2} \1_1^{\al_1}(x)$. In this case, \eqref{Gd0} and \eqref{Gd3} imply that 
\begin{align}
\begin{split}
|\partial^{\al_2}_x \1_1^{\al_1} (x) | & \les N^{3-A} |x|^{-A} \int_{\R^2} |x-y|^{-1} \, \ind_{|x-y| \les N^{-1}}   dy \\
&  \les N^{2-A} |x|^{-A} \\
& \les \big(|x| + N^{-1}\big)^{-2} \ind_{N^{-1} \les |x| \les 1} + \jb x ^{-A} \ind_{|x| \gg 1}
\end{split}
\label{Ysg42}
\end{align}
for any $A \ge 2$. Combining \eqref{Ysg41} and \eqref{Ysg42} yields the bound
\begin{align}
|\partial_x^{\al_2} \1^{\al_1}_1 (x) | \les \begin{cases}  \big(|x| + N^{-1}\big)^{-2} \quad & \text{for $|x| < 2$}, \\
\jb x^{-A} \quad & \text{for $|x| \ge 2$} \\
  \end{cases}
\label{Ysg43}
\end{align}
for any $A \ge 1$.

We now consider the contribution of $\partial^{\al_2}_x \1_2^{\al_1}$. Fix $x \in \R^2$. By \eqref{Ysg29}, we have that
\begin{align}
\begin{split}
|\partial^{\al_2}_x \1_2^{\al_1} (x) | & = \big| \big( \partial^{\al_2}_x \big\{ (1- \ld)(N \cdot) \partial_x^{\al_1} ( \ld G) \big\} \big) *   \nu_N \big)(x) \big| \\
& \les  \int_{\R^2} |x-y|^{-2} \, \ind_{N^{-1} \les |x-y| \les 1 } |  \nu_N(y) | dy.
\label{Ysg44}
\end{split}
\end{align}
In the above, we used the chain rule and the fact that the support of $\partial_x^{\al_2}( 1- \ld)$ is included in the set $\{x \in \R^2 : |x| \sim 1\}$ so that
\[ \big| \partial_x^{\al_2} \big\{ (1-\ld)(N \cdot) \big\}(z)\big| \sim N \big|\partial_x ^{\al_2} \ld(Nz)\big| \sim |z|^{-1}  \ind_{|z| \sim N^{-1}}  \]
for any $z$ in the support of $\partial_x^{\al_2}(1- \ld)$. If $|x| \les N^{-1}$, then by \eqref{Gd0} and \eqref{Ysg44}, we have
\begin{align}
\begin{split}
|\partial^{\al_2}_x \1_2^{\al_1} (x) | & \les N^2 \int_{\R^2} |x-y|^{-2} \, \ind_{N^{-1} \les |x-y| \les 1 } dy \\
&  \les N^2  \jb{\, \log(N)} \\
& \les \big( |x| + N^{-1} \big)^{-2} \jbb{\,\log \big(|x|+N^{-1}\big)}.
\label{Ysg45}
\end{split}
\end{align}
Otherwise, we have $|x| \gg N^{-1}$. Let $y$ be in the support of the integrand of $\partial_x ^{\al_2} \1_2(x)$. Assume that $|y| \ges |x|$. Let $A \ge 2$. From \eqref{Ysg44} and \eqref{Gd0}, we then have
\begin{align}
\begin{split}
|\partial^{\al_2}_x \1_2^{\al_1} (x) | & \les N^{2-A} \int_{\R^2} |x-y|^{-2} \, \ind_{N^{-1} \les |x-y| \les 1 } \, \ind_{|y| \ges |x|} \, |y|^{-A}  dy \\
& \les N^{2-A} |x|^{-A} \int_{\R^2} |x-y|^{-2} \, \ind_{N^{-1} \les |x-y| \les 1 } \, \ind_{|y| \ges |x|}  dy.
\label{Ysg45}
\end{split}
\end{align}
If $|y| \sim |x|$ for $y$ in the support of the integrand of $\partial_x ^{\al_2} \1_2(x)$, then we throw away the factor $|x-y|^{-2}$ in \eqref{Ysg45} via the bound $|x-y| \ges N^{-1}$ and get
\begin{align}
\int_{\R^2} |x-y|^{-2} \, \ind_{N^{-1} \les |x-y| \les 1 } \, \ind_{|y| \ges |x|}  dy \les N^2 \int_{\R^2} \ind_{|y| \sim |x|} dy \les N^2 |x|^2
\label{Ysg46}
\end{align}
Thus, by \eqref{Ysg45} and \eqref{Ysg46}, we deduce that
\begin{align}
|\partial^{\al_2}_x \1_2^{\al_1} (x) | \les N^{4-A} |x|^{2-A} \les \big(|x| + N^{-1}\big)^{-2} \ind_{N^{-1} \les |x| \les 1} + \jb x ^{2-A} \ind_{|x| \gg 1}
\label{Ysg47}
\end{align}
in that case, since $|x| \gg N^{-1}$. If $|y| \gg |x|$ for $y$ in the support of the integrand of $\partial_x ^{\al_2} \1_2(x)$ then $|x - y| \sim |y| \gg |x|$ and hence, by \eqref{Ysg45}, we have
\begin{align}
\begin{split}
|\partial^{\al_2}_x \1_2^{\al_1} (x) | & \les N^{2-A} |x|^{-A} \int_{\R^2} |y|^{-2} \, \ind_{|x| \les |y| \les 1}  dy \\
& \les N^{2-A} |x|^{-A} \big( \jb{\, \log x} \, \ind_{N^{-1} \les |x| \les 1} + \ind_{|x| \ges 1} \big) \\
& \les \big(|x| + N^{-1}\big)^{-2} \jbb{\,\log \big(|x|+N^{-1}\big)} \, \ind_{N^{-1} \les |x| \les 1} + \jb x ^{-A} \ind_{|x| \gg 1}
\label{Ysg48}
\end{split}
\end{align}
in that case, where we used $|x| \gg N^{-1}$. The case $|y| \ll |x|$ for $y$ in the support of the integrand of $\partial_x ^{\al_2} \1_2(x)$ is treated via similar arguments and we have
\begin{align}
|\partial^{\al_2}_x \1_2^{\al_1} (x) | \les \big(|x| + N^{-1}\big)^{-2} \ind_{N^{-1} \les |x| \les 1} + \jb x ^{-A} \ind_{|x| \gg 1}
\label{Ysg49}
\end{align}
in that case as well.

Therefore, combining \eqref{Ysg47}, \eqref{Ysg48}, \eqref{Ysg49} gives
\begin{align}
|\partial_x^{\al_2} \1^{\al_1}_2 (x) | \les \begin{cases}  \big(|x| + N^{-1}\big)^{-2} \jbb{\,\log \big(|x|+N^{-1}\big)} \quad & \text{for $|x| < 2$}, \\
\jb x^{-A} \quad & \text{for $|x| \ge 2$}. \\
  \end{cases}
\label{Ysg50}
\end{align}
Lastly, from \eqref{Ysg29} and the definition of the smooth function $\II$, it is easy to see that
\begin{align}
|\partial_{x}^{\al_2} \II(x)| \les  \jb x^{-A}.
\label{Ysg51}
\end{align}
Thus, \eqref{Ysg40} follows from \eqref{Ysg43}, \eqref{Ysg50} and \eqref{Ysg51}. This concludes the proof of \eqref{Ysg33} with $|\al| = 2$.

The bound \eqref{Ysg35} in (ii) is a consequence of arguments similar to those in the proof of \eqref{Ysg33} and we omit details.

\end{proof}

\subsection{Hyperbolic singularities}

We now consider functions which are singular along circles. Recall that for $t>0$, $\mb S^1(t)$ denotes the circle of centre $0$ and radius $t$.

In the first result of this subsection, we prove estimates on smoothed hyperbolic singularities. This is essentially the hyperbolic counterpart of Lemma \ref{LEM:t0}.

\begin{lemma}\label{LEM:green_wave0}
Fix $N \in \N$ and $0 <\g \le \frac12$. Let $\mc M = \T^2$ or $\R^2$ and $\nu_N : \mc M \to \R$ be a function satisfying the bound

\noi
\begin{align}
|\nu_N(x)| \les_A N^{2} \jb{Nx}^{-A}
\label{CC0}
\end{align}
for all $x \in \mc M$ and any finite $A \ge 1$.  Let $0 < t \le 1$, $H_t$ and $\wt H_t$ be the functions given by 

\noi
\begin{align*}
H_{t,\g}(x) & = |t - |x||^{-\g}, \\
\wt H_{t,\g}(x) & = \big|t^2 - |x|^2 \big|^{-\g} 
\end{align*}
for any $x \in \mc M \setminus \mb S ^1(t)$. Set $H_{N,t, \g} := H_{t,\g} * \nu_N$ and $\wt H_{N,t, \g} := \wt H_{t,\g} * \nu_N$. Then, the following bounds hold:

\noi
\begin{align}
\begin{split}
|H_{N,t,\g} (x)| & \les \min \! \big\{ N^{\g}, |t - |x||^{-\g} \big\} \, \jbb{ \, \log \! \big( \! \min \! \big\{N, |t- |x||^{-1}\big\}\big) }, \\
 |\wt H_{N,t,\g} (x)|& \les  \min \! \big\{ N^{2\g}, \big| t^2 - |x|^2 \big|^{-\g} \big\} \, \jbb{ \, \log \big(  \! \min \! \big\{N, |t- |x||^{-1}\big\}\big) }
\end{split}
\label{CC1}
\end{align}
for any $x \in \mc M \cap B(0,10)$. Here, the implicit constants are independent of $N$.\footnote{Here, we chose the ball $B(0,10)$ because it contains a copy of the torus $\T^2 \cong [-\pi, \pi)^2$ and is bounded, but this choice is otherwise arbitrary.}
\end{lemma}

The proof of Lemma \ref{LEM:green_wave0} is significantly more challenging than that of Lemma \ref{LEM:t0}. This is due to the fact the set of singular points of the functions $H_t$ (and $\wt H_t$) is not the singleton $\{0\}$, as in the case of functions of the form $x \in \mc M \mapsto |x|^{-\ta}$, $0 < \ta <2$ in Lemma \ref{LEM:green_der}, but consists of points lying on the circle $\mb S^1(t)$. 

Before proceeding with the proof of Lemma \ref{LEM:green_wave0}, we introduce a convenient spatial localization procedure. By \eqref{CC0}, we have\footnote{Note that the sum over $k \ge 1$ in is finite for $\mc M = \T^2$. Namely, we have $k \les \jb{\,\log N}$ since $|y|_{\T^2} \le 2\sqrt{2} \pi$ for any $y \in \T^2$.}
\begin{align}
\begin{split}
 |H_{N,t,\g}(x)| &  =  \int_{\mc M} \ind_{0 \le |y|_{\mc M} < N^{-1}} \, |\nu_N(y)| H_t(x-y) dy \\
& \qquad  + \sum_{k \ge 1} \int_{\T^2} \ind_{2^{k-1} N^{-1} \le  |y|_{\mc M} < 2^k N^{-1}} \, |\nu_N(y)| H_t(x-y) dy \\
& \les  \sum_{k \ge 0} 2^{-200k} \, H^{k}_{N,t,\g}(x),
\end{split}
\label{Ysg9}
\end{align}
where 
\begin{align}
H^k_{N,t,\g}(x) :=  N^2 \int_{\mc M} \ind_{|y|_{\mc M} < 2^k N^{-1}} \, H_t(x-y) dy
\label{Ysg10}
\end{align}
for all $x \in \mc M$ and $k \in \Z_{\ge 0}$. Lastly, set 
\begin{align}
\wt H_{N,t,\g}(x) :=  N^2 \int_{\mc M} \ind_{|y|_{\mc M} < 2^k N^{-1}} \, \wt H_t(x-y) dy
\label{Ysg16}
\end{align}
for all $x \in \mc M$ and $k \in \Z_{\ge 0}$. 

Lemma \ref{LEM:green_wave0} is an immediate consequence of the following result.
\begin{lemma}\label{LEM:green_wave1}
Fix $N \in \N$, $k \in \Z_{\ge 0}$ and $0 \le \g \le \frac12$. Let $0 < t \le 1$, $H_{N,t,\g}^{k}$ and $\wt H^k_{N,t,\g}$ be as in \eqref{Ysg10} and \eqref{Ysg15}, respecively. Then, the following estimates holds
\begin{align}
|H^k_{N,t,\g} (x)| & \les 2^{100k} \min \! \big\{ N^{\g}, |t - |x||^{-\g} \big\} \, \jbb{ \, \log \! \big( \! \min \! \big\{N, |t- |x||^{-1}\big\}\big) } , \label{CC3} \\
 |\wt H^k_{N,t,\g} (x)|& \les 2^{100k}  \min \! \big\{ N^{2\g}, \big| t^2 - |x|^2 \big|^{-\g} \big\} \, \jbb{ \, \log \! \big( \! \min \!  \big\{N, |t- |x||^{-1}\big\}\big) } \label{Ysg17}
\end{align}
for any $x \in \mc M \cap B(0,10)$. Here, the implicit constants are independent of $N$.

\end{lemma}

\begin{proof} We first consider \eqref{CC3} and \eqref{Ysg17} in the case $\mc M = \T^2$ and prove \eqref{CC3}. Fix $N \in \N$ and $0 < t \le 1$.

\medskip
\noi
{\bf $\bul$ Step $\1$: preliminary reductions and estimates.}\quad Fix $x \in \T^2$ and let $y \in \T^2$ be in the support of the integrand of $H^k_{N,t,\g}$. If $|x-y| \gg t$ or $|x-y| \ll t$, it is easy to see that \eqref{CC3} holds. For instance, let us assume that $|x-y| \ll t$. Then, we have

\noi
\begin{align}
|H^k_{N,t,\g} (x)| \les N^2 t^{-\g} \min(N^{-2}, t^2) \les  \min(N^\g, t^{-\g}).
\label{CC4}
\end{align}

\noi
If $t \ges |x|$, then \eqref{CC4} clearly implies \eqref{CC3}. Otherwise, $t \ll |x|$ and this implies that $|x| \les |y| \les 2^k N^{-1}$ in view of the condition $|x-y| \ll t$. Thus, the bound 
\[N^\g \les 2^{\g k} |x|^{-\g} \les 2^{\g k} | t - |x||^{-\g}\] 
holds and \eqref{CC4} implies \eqref{CC3}. The case $|x-y| \gg t$ follows from Lemma \ref{LEM:t0} and similar arguments. Therefore, we henceforth assume the extra condition $|x-y| \sim t$ in the integrand of $H^k_{N,t,\g}$ in what follows. 


Now, we assume that the condition $\big| t^2 - |x-y|^2 \big| < N^{-10}$ holds in the integrand of $H^k_{N,t,\g}$, which implies that
\[ |t - |x-y|| \les t^{-1} N^{-10}, \]
under the condition $t \sim |x-y|$. Hence, by a polar change of coordinate, we have

\noi
\begin{align*}
|H^k_{N,t,\g}(x)|& \les  N^2 \int_{0}^{10 t} \frac{\ind_{0 < |t - r| \les t^{-1} N^{-10}}}{|t-r|^{\g}}rdr  \les t^\g N^{-2},
\end{align*}

\noi
which is a stronger estimate than \eqref{CC3}. 

Therefore, we may assume that the conditions

\noi
\begin{align}
\begin{split}
 t  \sim |x-y| \qquad \text{and} \qquad \big|t^2 - |x-y|^2\big|  \ge N^{-10},
\end{split}
\label{CC6}
\end{align}

\noi
hold in the integrand of $H^k_{N,t,\g}$ for the rest of the proof. To sum up, we have 
\begin{align}
| H^k_{N,t,\g} (x)| \sim N^2 t^\g \int_{\T^2} \ind_{|x-y|\sim t} \, \ind_{|t^2 - |x-y|^2| \ge N^{-10}|} \, \ind_{|y| < 2^k N^{-1}} \big| t^2 - |x-y|^2 \big|^{-\g} dy.
\label{CC5}
\end{align}

By H\"older's inequality and a polar change of coordinates, we obtain the following basic estimate on $H^k_{N,t,\g}$:

\noi
\begin{align}
\begin{split}
|H^k_{N,t,\g}(x)|& \les N^2 t^\g |B(0,2^{k} N^{-1})|^{1-\g} \Big( \int_{\T^2}  \frac{\ind_{|t^2 - |x-y|^2| \ge N^{-10}}}{\big|t^2 - |x-y|^2\big|}  \ind_{|x-y| \sim t} dy \Big)^{\g} \\
& \les 2^{2k}  t^\g N^{2\g} \jb{\, \log N}^{\g}.
\end{split}
\label{CC7}
\end{align}

\noi
The bound \eqref{CC7} will be useful in several instances later in the proof. 

\medskip
\noi
{\bf $\bul$ Step $\II$: analysis close to the radial singularity.}\quad Let $y \in \T^2$ be in the support of the integrand of $H^k_{N,t,\g}$.\footnote{If $\operatorname{Supp}(H^k_{N,t,\g}) = \emptyset$, there is nothing to show. We will discard such cases without further mention in this proof.} We expand the expression $t^2 - |x-y|^2$ depending on the sign of $t^2 - |x|^2$ as follows:

\noi
\begin{align}
t^2 - |x-y|^2 = \big|t^2 - |x|^2\big| \bigg( \text{sgn}( t^2 - |x^2|)+ \frac{Q_x(y)}{\big|t^2 - |x|^2\big|} \bigg),
\label{CC7b}
\end{align}
where 
\begin{align}
Q_x(y) = 2 \langle x, y \rangle - |y|^2
\label{Ysg12}
\end{align}
for all $y \in \T^2$ and $\text{sgn}$ is the sign function in \eqref{Rhyp2}. We now divide our analysis into several cases. 

\medskip

\noi
{\bf $\bul$ Case 1: $|Q_x(y)| \ll |t^2 - |x|^2|$.}\quad In this case, we have that $\big|t^2 - |x-y|^2\big| \sim \big|t^2 - |x|^2\big|$ by \eqref{CC7b}. Hence, by \eqref{CC6}, this leads to the bound

\noi
\begin{align*}
|H^k_{N,t,\g}(x)| \les  2^{2k} |t - |x||^{-\g},
\end{align*}

\noi
which is acceptable in view of \eqref{CC3} if $N \ges |t-|x||^{-1}$. Otherwise, $N \ll |t-|x||^{-1}$ and we proceed as follows. 

If $t \gg |x|$ then we have $t \sim |y|$ as $t \sim |x-y|$ by \eqref{CC6} and hence $t \les 2^k N^{-1}$. This directly gives \eqref{CC3} from the bound \eqref{CC7}. Similarly, if $t \les N^{-1}$, we get \eqref{CC3} from the bound \eqref{CC7}. We thus assume that $N^{-1} \les t \les |x|$ in the rest of the proof of this case.

By \eqref{CC7b}, we have that 

\noi
\begin{align*}
\big| t^2 - |x-y|^2 \big| \sim \big| t^2 - |x|^2 \big|   \gg |Q_x(y)|.
\end{align*}

\noi
Hence, by \eqref{CC5} and polar changes of variables, we have

\noi
\begin{align}
\begin{split}
| H^k_{N,t,\g} (x)| & \les  N^2 t^\g \int_{\T^2} \ind_{|y| < 2^k N^{-1}} \, |Q_x(y)|^{-\g} \ind_{|Q_x(y)| \ge N^{-100}} dy \\
& \quad  + N^2 t^\g \big| t^2 - |x^2| \big|^{-\g} \cdot  N^{-25} \int_{\T^2} \ind_{|y| < 2^k N^{-1}} \, |Q_x(y)|^{-\frac14} \ind_{|Q_x(y)| < N^{-100}}  dy  \\
& \les N^2 t^\g \int_{0}^{2^k N^{-1}} r dr \int_{0}^{2\pi} \frac{d \ta}{|2 \cos(\ta) |x| r - r^2 |^{\g}} \ind_{| 2 \cos(\ta) |x| r - r^2 | \ge N^{-100}}  \\
& \quad +  N^{-23 + 10\g} t^\g   \int_{0}^{2^k N^{-1}} r dr \int_{0}^{2\pi} \frac{d \ta}{|2 \cos(\ta) |x| r - r^2 |^{\frac14}}  \\
& \les  N^2 t^\g \int_{0}^{2^k N^{-1}} r^{1-\g} dr \int_{0}^{2\pi} \frac{d\ta}{|2 \cos(\ta) |x| - r |^{\g}} \ind_{|2 \cos (\ta)|x| - r | \ge N^{-100}}\\
& \quad +  N^{-15} t^\g   \int_{0}^{2^k N^{-1}} r^{\frac34} dr \int_{0}^{2\pi} \frac{d \ta}{| 2\cos(\ta) |x|  - r |^{\frac14}}.
\end{split}
\label{CC100}
\end{align}
Note that if $\ta \in [\frac{\pi}{2}, \frac{3\pi}{2}]$ then $\ta-\pi \in [-\frac{\pi}{2}, \frac{\pi}{2}]$ and $\cos(\ta-\pi) >0$. Therefore, we have
\[|2 \cos(\ta) |x| - r|^{-\g} = ( 2\cos(\ta-\pi) |x|+ r)^{-\g}\]
for any $r>0$. The last fact together with the symmetry $\cos(-\ta) = \cos(\ta)$ for any $\ta \in \R$, shows that the contribution of the angular integral over $\ta \in [0,2\pi]$ can be bounded by that over the range $\ta \in [0,\frac{\pi}{2}]$. Thus, from a change of variables and recalling $0< \g \le \frac12$, we then get

\noi
\begin{align}
\begin{split}
& \int_{0}^{2\pi} \frac{d\ta}{| 2 \cos(\ta) |x| - r |^{\g}} \ind_{| 2\cos (\ta)|x| - r | \ge N^{-100}} \\
& \qquad \quad \les \int_{0}^{\frac{\pi}{2}} \frac{d\ta}{| 2 \cos(\ta) |x| - r |^{\g}} \ind_{| 2\cos (\ta)|x| - r | \ge N^{-100}} \\
& \qquad \quad \les |x|^{-\g} \int_0^1 \frac{d u}{\sqrt{1-u^2} \big|2u - r |x|^{-1} \big|^{\frac12}}\ind_{|2u- r |x|^{-1}| \ge N^{-100}} \\
 & \qquad \quad \les |x|^{-\g} \jb{\, \log N},
 \end{split}
 \label{CC101}
\end{align}
uniformly in $r>0$. Similarly, we have

\noi
\begin{align}
&  \int_{0}^{\frac{\pi}{2}} \frac{d\ta}{| \cos(\ta) |x| - r |^{\frac14}} \les |x|^{-\frac14},
\label{CC101b}
\end{align}

\noi
uniformly in $r >0$. Thus, combining \eqref{CC100}, \eqref{CC101} and \eqref{CC101b} together with the condition $N^{-1} \les t \les |x|$ gives

\noi
\begin{align*}
|H^k_{N,t,\g} (x)| \les 2^{2k} N^\g \jb{\, \log N} + 2^{2k} N^{-10} \les 2^{2k} N^\g \jb{\, \log N},
\end{align*}

\noi
as required in \eqref{CC3}.

\medskip

\noi
{\bf $\bul$ Case 2: $|Q_x(y)| \ges \big|t^2 - |x|^2\big|$.}\quad By the Cauchy-Schwarz inequality and the fact that $|y| \les 2^k N^{-1}$ in the support of the integrand of $H^k_{N,t,\g}$, we deduce that 

\noi
\begin{align}
\big|t^2 - |x|^2\big| \les \max\big( 2^k  N^{-1} |x|, 2^{2k} N^{-2} \big).
\label{CC8}
\end{align}
If $2^{2k}N^{-2} \gg 2^k N^{-1} |x|$ then we have $|x| \ll 2^{k}N^{-1}$. Hence, we have $t \sim |x-y| \les 2^k N^{-1}$ as $|y| \le 2^k N^{-1}$. By \eqref{CC7}, \eqref{CC8} and the simple estimate $|t + |x|| \ges |t -|x||$, this shows that

\noi
\begin{align*}
|H^k_{N,t,\g}(x)| & \les 2^{2k} N^{\g} \jb{\, \log N}^{\g} \\
& \les 2^{4k}  \big|t^2 - |x|^2 \big|^{\frac{\g}{2}} \, \jbb{\, \log \! \big(\big|t^2 - |x|^2 \big|\big)}^{\g} \\
& \les 2^{4k}|t - |x||^{\g}  \,  \jbb{\, \log \! \big(|t - |x||\big)}^{\g},
\end{align*}

\noi
which is acceptable in view of \eqref{CC3}. Otherwise, we have $2^{2k}N^{-2} \les 2^k N^{-1}|x|$ and the condition \eqref{CC8} reduces to 

\noi
\begin{align}
N \les 2^k |t - |x||^{-1}.
\label{CC9}
\end{align}

\noi
Note that by \eqref{CC9}, it suffices to show the bound 

\noi
\begin{align}
|H^k_{N,t,\g} (x)| \les 2^{10k} N^\g \jb{\, \log N}
\label{CC9b}
\end{align}

\noi
in order to get \eqref{CC3}. By using \eqref{CC7} as in Case 1, we may assume that $t \les |x|$ for the remainder of the proof.

\medskip

\noi
{\bf $\bul$ Subcase 2.1: $|Q_x(y)| \gg \big|t^2 - |x|^2\big|$.}\quad By \eqref{CC7b}, we infer that 

\noi
\begin{align*}
\big| t^2 - |x-y|^2 \big| \sim |Q_x(y)|.
\end{align*}

\noi
Hence, by \eqref{CC6} with \eqref{CC5} and by arguing as in \eqref{CC100} - \eqref{CC101b} in Case 1, we have

\noi
\begin{align*}
|H^k_{N,t,\g} (x)| & \les  N^2 t^\g \int_{\T^2} \ind_{|y|<2^k N^{-1}} \, |Q_x(y)|^{-\g} \ind_{|Q_x(y)| \ges N^{-10}} dy \\
& \les  2^{2k} N^\g \jb{\, \log N},
\end{align*}

\noi
as required in \eqref{CC9b}.

\medskip

\noi
{\bf $\bul$ Subcase 2.2: $|Q_x(y)| \sim \big|t^2 - |x|^2\big|$.}\quad In this case, instead of doing an explicit computation as in Subcase 2.1, we rely on a simple geometric observation combined with a dyadic localization argument. 

Let $\mu \in 2^{\Z}$ and $\eps = \text{sgn}(t^2 - |x^2|) \in \{+1,-1\}$. Define the sets 

\noi
\begin{align*}
E^{\eps}_{k, \mu} = \Big\{ y \in \R^2: \frac{\mu}{2} < \Big| \eps +  \frac{Q_x(y)}{\big|t^2 - |x|^2\big|} \Big| \le \mu \Big\}.
\end{align*}

\noi
Let $\mu$ be such that $E^{\eps}_{k, \mu}$ is non-empty. Then by \eqref{CC7b} and \eqref{CC6}, we necessarily have that 

\noi
\begin{align}
\mu \ges \big|t^2 - |x|^2\big|^{-1} N^{-10} \ges N^{-10}.
\label{Ysg15}
\end{align}

\noi
With these notations and \eqref{CC5}, we have

\noi
\begin{align}
\begin{split} 
|H^k_{N,t,\g}(x)| & \les \sum_{\substack{ \mu \in 2^{\Z} \\ N^{-10} \les \mu \les 1}} N^2 t^\g  \int_{\T^2} \ind_{|y|<2^k N^{-1}} \, \ind_{E^\eps_{k,\mu}}(y) \, \ind_{|x-y| \sim t} \, \big| t^2 - |x-y|^2 \big|^{-\g} dy \\
& =:  \sum_{\substack{ \mu \in 2^{\Z} \\ N^{-10} \les \mu \les 1}} \1^{\eps}_{k, \mu}
\end{split}
\label{CC10}
\end{align}

\noi
for all $k \in \Z_{\ge 0}$. In order to obtain \eqref{CC9b}, it thus suffices to prove the following estimate:

\noi
\begin{align}
\max_{\substack{ \mu \in 2^{\Z} \\ N^{-10} \les \mu \les 1}} | \1^{\eps}_{k, \mu} | \les 2^{10k} N^\g.
\label{CC12}
\end{align}
in view of the condition \eqref{Ysg15}.

Fix $\mu \les 1$. Then by definition of the set $E^{\eps}_{k,\mu}$ and \eqref{CC10}, we have 

\noi
\begin{align}
|\1^{\eps}_{k,\mu}| \les N^2 \mu^{-\g}|t-|x||^{-\g} \int_{\T^2} \ind_{B(0,2^kN^{-1}) \cap E^\eps_{k,\mu}}(y) \, \ind_{|x-y| \sim t} \, dy.
\label{CC15}
\end{align}

\noi
It is easy to see that $B(0,2^kN^{-1}) \cap E^{\eps}_{k,\mu}$ is included in a dilate of $B(0,2^kN^{-1}) \cap (R_+ \cup R_{-}) $, where $R_\s$ is the rectangle given by

\noi
\begin{align*}
R_\s :=  & \bigg\{ y \in \R^2: \big(-\s \eps + \frac{\mu}2\big) \frac{\big| t^2 -|x|^2 \big|}{2|x|} + O\Big( \frac{2^{2k}N^{-2}}{|x|} \Big) < \langle \s e(x), y \rangle \\
& \qquad \qquad \qquad \le \big(- \s \eps + \mu\big) \frac{\big| t^2 -|x|^2 \big|}{2|x|} + O\Big( \frac{2^{2k}N^{-2}}{|x|} \Big)\bigg\},
\end{align*}

\noi
with $e(x) = \frac{x}{|x|} \in \mb S^1$ and $\s \in \{+,-\}$. The rectangle $R_\s$ has dimensions about
\[\bigg( \mu \frac{| t^2 -|x|^2|}{|x|} + O\Big( \frac{2^{2k}N^{-2}}{|x|} \Big) \bigg) \times 1\]
in the directions $\R e(x)$ and $\R e(x)^\perp$, respectively. Noting that the area of the intersection of any any ball of radius $r>0$ and any rectangle of dimensions $r_1, r_2 >0$ is bounded by $\min(r,r_1) \min(r,r_2)$, we have that

\noi
\begin{align*}
| B(0, 2^kN^{-1}) \cap (R_+ \cup R_{-}) | \les 2^k N^{-1}  \bigg(  \mu \frac{\big| t^2 -|x|^2 \big|}{|x|} + O\Big( \frac{2^{2k}N^{-2}}{|x|} \Big) \bigg).
\end{align*}

\noi
Thus, plugging the above bound in \eqref{CC15} together with \eqref{CC9} and the condition $t \les |x|$ yields

\noi
\begin{align}
\begin{split}
|\1^{\eps}_{k,\mu}| & \les 2^k N \mu^{1-\g}| t -|x| |^{1-\g} + 2^{3k} |t-|x||^{-\g}  \mu^{-\g} \frac{N^{-1}}{|x|} \\
& \les 2^{2k} \mu^{1-\g} N^{\g} + 2^{3k} |t-|x||^{-\g}  \mu^{-\g} \frac{N^{-1}}{|x|}.
\end{split}
\label{CC16} 
\end{align}

\noi
Hence, the bound \eqref{CC16} is enough to get \eqref{CC12} if $|t-|x||^{-\g} \mu^{-\g} \frac{N^{-1}}{|x|} \les N^{\g}$. Otherwise, we have

\noi
\begin{align}
N^{1+\g} \ll \frac{\mu^{-\g}}{|x|}|t-|x||^{-\g}.
\label{CC17}
\end{align}

Let us assume that \eqref{CC7} holds. We go back to \eqref{CC15} and bound the integral on the right-hand-side so as to get an estimate with appropriate decay in the parameter $\mu$. If $y \in E^{\eps}_{k,\mu}$ then we have $\big|t^2 - |x-y|^2\big| \sim \mu \big|t^2 - |x|^2 \big|$ by definition. Hence, by a polar change of variables, we have

\noi
\begin{align*}
\int_{\T^2} \ind_{B(0,2^kN^{-1}) \cap E^\eps_{k,\mu}}(y) \, \ind_{|x-y| \sim t} \,dy & \les \int_{\T^2} \ind_{|x-y|\sim t} \, \ind_{|t^2 - |x-y|^2| \sim \mu |t^2 - |x|^2|} \, dy \\
& \les \mu \big|t^2-|x|^2\big|,
\end{align*}

\noi
which, from the condition $t \les |x|$, gives the estimate 

\noi
\begin{align}
|\1^{\eps}_{k,\mu}| \les N^2 \mu^{1-\g} |x| \cdot | t-|x||^{1-\g}.
\label{CC18}
\end{align}

\noi
Now, by \eqref{CC17} and \eqref{CC9} successively, we bound the factor $N^2$ in \eqref{CC18}, leading to

\noi
\begin{align*}
|\1^{\eps}_{k,\mu}| & \les N^{1-\g} \mu^{1-2\g} |t-|x||^{1-2\g} \\
& \les 2^{(1-2\g)k} N^\g \mu^{1-2\g},
\end{align*} 

\noi
where we used the condition $0 < \g \le \frac12$. This shows \eqref{CC12} and proves \eqref{CC3}.

The proof of \eqref{Ysg17} follows from similar arguments as that of \eqref{CC3} and we omit details.

Let us briefly explain how to obtain \eqref{CC3} in the case $\mc M = \R^2$. The bound \eqref{Ysg17} for $\mc M = \R^2$ then also follows from similar arguments. We decompose $H^k_{N,t,\g}$ as follows 
\begin{align}
\begin{split}
H^k_{N,t,\g}(x) & = N^2 \int_{B(0,20)} \ind_{|y| < 2^k N^{-1}} \,  H_t(x-y) dy \\
& \qquad \qquad  + N^2 \int_{\R^2 \setminus B(0,20)} \ind_{|y| < 2^k N^{-1}} \,  H_t(x-y) dy \\
& =: H^{k,1}_{N,t, \g}(x) + H^{k,2}_{N,t, \g}(x)
\end{split}
\label{Ysg18}
\end{align}
for any $x \in B(0,10)$. Fix $x \in B(0,10)$ and let $y$ be in the integrand of $H^{k,2}_{N,t, \g}(x)$, which is contained in $\R^2 \setminus B(0,20)$. Then we have
\[ H_{t, \g}(x-y) \sim |y|^{-\al} \ges 1, \]
since $0<t\le 1$ and hence
\begin{align} |H^{k,2}_{N,t, \g}(x)| \les 2^{2k}.\label{Ysg19} \end{align}
On the other hand, since $H^{k,1}_{N,t,\g}(x)$ is an integral over a compact domain, by arguing as in the proof of \eqref{CC3} in the periodic case, we have
\[| H^{k,1}_{N,t, \g}(x)| \les  2^{100k} \min \! \big\{ N^{\g}, |t - |x||^{-\g} \big\} \, \jbb{ \, \log \! \big( \! \min \! \big\{N, |t- |x||^{-1}\big\}\big) } \]
for any $x \in B(0,10$), which together with the bounds \eqref{Ysg18} and \eqref{Ysg19} proves \eqref{CC3} in the case $\mc M = \R^2$.
\end{proof}

Next, we state a result allows to differentiate functions of the form $\big(\ind_{B(0,t)} f \big) * g$ for smooth enough functions $f$ and $g$, where $\ind_{B(0,t)}$ is the indicator function of the ball $B(0,t)$ for some $t >0$. This result is crucial to study convolutions with the kernel $W$ defined in \eqref{poisson3}; see Lemma \ref{LEM:green_wave2} and \ref{LEM:green_wave3} below.

\begin{lemma}\label{LEM:Dder}
Fix $\al \in \Z^2_{\ge 0}$ with $|\al|=1$ and $0<t\le 1$. Let $f\in W^{1,1} (B(0,2)) \cap C^0(\R^2; \R)$\footnote{Here, $W^{1,1} (O)$ for an open set $O \subset \R^2$ is the space $W^{1,1} (O) = \{ f \in \mc D'(O): (f, \nb f) \in (L^1(O))^2 \}$ endowed with the norm $\|f\|_{W^{1,1}(O)} = \|f\|_{L^1(O)} + \|\nb f\|_{L^1(O)}$.} and $g \in W^{1,\infty}(\R^2) \cap C^0(\R^2; \R)$. Consider the function $T$ given by \[T = \big(\ind_{B(0,t)} f \big) * \varphi.\]
Then, the following formula holds:

\noi
\begin{align}
\partial^\al_x T(x) = \big(\big(\ind_{B(0,t)} \partial_x ^\al f \big)* g \big) (x) - \int_{\mb S^1 (t)} f(y) g(x-y) \al \cdot y \, d\s_t(y)
\label{YD1}
\end{align}
for any $x \in \R^2$ and where $d\s_t$ denotes the normalized surface measure on $\mb S^1(t)$.
\end{lemma}

Note that under the smoothness assumptions on $f$ and $g$ in Lemma \ref{LEM:Dder}, all the terms in \eqref{YD1} make sense.

\begin{proof} We fix $t = 1$ in the proof for convenience. Let $\{\nu_\eps \}_{\eps \in (0,1)}$ be a sequence of smooth functions such that 

\noi
\begin{align}
0 \le  \nu_\eps \le 1 \ \ \text{on} \ \R^2, \ \nu_\eps \equiv 1 \ \text{on} \ B(0,1) \ \text{and} \ \text{Supp}(\nu_\eps) \subset B(0,1+\eps) \ \text{for} \ \eps \in (0,1).
\label{Ycond}
\end{align}
Let $T_\eps = \big(\nu_\eps f \big) * g$ for $\eps \in (0,1)$. Then, by \eqref{Ycond} and H\"older's inequality, we have that

\noi
\begin{align}
\begin{split}
\|T_\eps - T\|_{L^\infty(\R^2)} & = \sup_{x \in \R^2} \Big| \int_{\T^2} ( \nu_\eps(y) - \ind_{B(0,1)}(y) ) f(y) g(x-y) dy\Big| \\
& \le \sup_{x \in \R^2} \int_{\R^2 \setminus B(0,1)} |\nu_\eps(y) - \ind_{B(0,1)}(y) | |f(y)| | g(x-y)| dy \\
& \le  \|f\|_{L^{1}(B(0,1+\eps)\setminus B(0,1))}  \| g \|_{L^{\infty}(\R^2)} \too 0,
\end{split}
\label{YD2}
\end{align}

\noi
as $\eps \to 0$, by dominated convergence. Thus, we have $T_\eps \to T$ in $L^\infty(\R^2)$ as $\eps \to 0$ and hence $\partial_x ^\al T_\eps \to \partial_x ^\al T$ in $\D'(\R^2)$ as $\eps \to 0$. Moreover, since $|\al| = 1$, we have 

\noi
\begin{align}
\begin{split}
\partial_x ^\al T_\eps & = \big( \partial_x ^\al \nu_\eps f \big) * g + \big(\nu_\eps \partial_x^\al f \big) * g \\
& =: \1_\eps + \II_\eps.
\end{split}
\label{YD3}
\end{align}

\noi
By an immediate modification of \eqref{YD2}, we get

\noi
\begin{align}
\II_\eps \to \big(\ind_{B(0,1)} \partial_x ^\al f \big) * g
\label{YD4}
\end{align}

\noi
in $L^\infty(\R^2)$ as $\eps \to 0$. On the other hand, by \eqref{Ycond}, we note that $\partial_x \nu_\eps$ is supported in $B(0,1+\eps) \setminus B(0,1)$ and $\nu_\eps \equiv 0$ on $\mb S^1(1+\eps)$. Hence, by Green's formula
\cite[Theorem 2 on p.\,712]{Evans}, we have 

\noi
\begin{align}
\begin{split}
\1_\eps(x) & = \int_{B(0,1+\eps) \setminus B(0,1)} \partial_y^\al \nu_\eps(y)  f(y) g(x-y) dy  \\
& = \int_{\mb S^1(1+\eps)} \nu_\eps (y) f(y) g(x-y) \frac{\al \cdot y}{1+\eps} \, d\s_{1+\eps}(y) \\
& \qquad \qquad - \int_{\mb S^1(1)} \nu_\eps (y) f(y) g(x-y) \al \cdot y \, d\s_{1}(y) \\
& \qquad \qquad - \int_{B(0,1+\eps) \setminus B(0,1)}  \nu_\eps(y) \partial_y^\al \big( f(y) g(x-y) \big) dy \\
& =  - \int_{\mb S^1(1)}  f(y) g(x-y) \al \cdot y \, d\s_{1}(y) \\
& \qquad \qquad  -  \int_{B(0,1+\eps) \setminus B(0,1)}  \nu_\eps(y) \partial_y^\al \big( f(y) g(x-y) \big) dy \\
& =: - \1^1(x) - \1^2_\eps(x).
\end{split}
\label{YD5}
\end{align}
By \eqref{Ycond}, H\"older's inequality and the dominated convergence theorem, we have

\noi
\begin{align}
\begin{split}
\|\1^2_\eps\|_{L^\infty(\R^2)} & \le \|\partial^\al_x f \|_{W^{1,1}(B(0,1+\eps) \setminus B(0,1))} \|g\|_{W^{1,\infty}(\T^2)}  \too 0,
\end{split}
\label{YD7}
\end{align}

\noi 
as $\eps \to 0$.

Combining \eqref{YD3}, \eqref{YD4}, \eqref{YD5} and \eqref{YD7} shows that $\partial_x T_\eps$ converges to the right-hand-side of \eqref{YD1} as $\eps \to 0$. Recall that $\partial_x T_\eps$ also converges to $\partial_x T$ as $\eps \to 0$. Thus, \eqref{YD1} follows from uniqueness of the limit in $\D'(\T^2)$.
\end{proof}

%
%
%


In the following lemma, we obtain bounds on smoothed functions whose derivatives exhibit hyperbolic singularities.

\begin{lemma}\label{LEM:green_wave2}
Fix $N \in \N$ and let $\nu_N \in C^{\infty}(\R^2; \R)$ satisfying the estimate \eqref{Gd0}. Fix $0 < t \le 1$ and let $W_t \in C^2(\R^2 \setminus \mb S^1(t); \R)$ be a function so that there exists a constant $C_t >0$ depending only on $t$ such that
\begin{align}
\begin{split}
|W_t(x)| &\les C_t, \\
| \partial_{x}^{\al} W_t(x)| &\les (t + |x|)^{-\frac12} |t-|x||^{\frac12-|\al|}
\end{split}
\label{hyplike222}
\end{align}
for all $x \in B(0,10) \setminus \mb S^1(t)$ and $\al \in \Z^2_{\ge 0}$ with $1 \le |\al| \le 2$. Set ${W_{N,t} =  (\ind_{B(0,t)} W_t) * \nu_N}$. Then, the following bound holds:

\noi
\begin{align}
\begin{split}
|\partial^\al_x W_{N,t} (x)|  & \les C_t \cdot \min \! \big\{ N^{|\al|}, |t - |x||^{-|\al|} \big\}
\end{split}
\label{YW1222}
\end{align}
for all $x \in B(0,10)$, $\al \in \Z^2_{\ge 0}$ with $1\le |\al| \le 2$. Here, the implicit constant is independent of $N$
\end{lemma}

Note that we state Lemma \ref{LEM:green_wave2} in the full space setting and not on $\T^2$. This allows us to avoid technical issues related to the specific spatial localization that is used in the proof; see Remark \ref{RMK:R2} for a more thorough discussion on this point. 

\begin{remark}\rm \label{RMK:indi}
Note that for $\al \in \Z^2_{\ge 0}$ with $1 \le |\al| \le 2$, the pointwise bound on $\partial_x^{\al} W_{N,t}$ is no matching that on $\partial_x^{\al} W_{t}$. This is because the worse contribution to $\partial_x^{\al} W_{N,t}$ comes from the (formal) scenario when the $\partial_x^{\al}$ hits the indicator function $\ind_{B(0,t)}(x)$, which roughly speaking gives the normalized measure $d \s_t$ on the sphere $\mb S^1(t)$; see Lemma \ref{LEM:Dder}. Then from \eqref{Gd0}, one can prove the estimate 
\begin{align}
\begin{split}
| \nu_N * d\s_t (x)| & = \Big| \int_{\mb S^1(t)} \nu_N(x-y) d \s_t(y) \Big| \\
&  \les N \cdot \jb{N \operatorname{dist}(x, \mb S^1(t))}^{-10} \\
& \les \min \! \big\{ N, |t - |x||^{-1} \big\}
\end{split}
\label{sphere_bdd}
\end{align}
for all $x \in \R^2$. Hence, the bound \eqref{sphere_bdd} (formally) justifies the form of the right-hand-side of \eqref{YW1222} for $|\al| = 1$.
\end{remark}

\begin{proof} Fix $N \in \N$ and $0 <t \le 1$. Let $\varphi$ and $\phi$ be as in \eqref{eta1}. We define $\{\phi_{N,k}\}_{k \in \Z_{\ge 0}}$ as follows:
\begin{align*}
\phi_{N,0}(x)= \varphi( N x )
\end{align*}
and
\begin{align}
\phi_{N,k}(x)= \phi(2^{-k} N x)
\label{phikkk}
\end{align}
for any $x \in \R^2$ and $k \in \N$. By construction, we have $\phi_{N,k} \in C_c^{\infty}(\R^2; [0,1])$ for all $k \in \Z_{\ge 0}$,
\[ \sum_{k = 0}^{\infty} \phi_{N,k}(x) = 1\]
for any $x \in \R^2$. Moreover, we have $\operatorname{Supp}(\phi_{N,0}) \subset \{ x \in \R^2 : |x| \les N^{-1}\}$ and ${\operatorname{Supp}(\phi_{N,k}) \subset \{ x \in \R^2 : |x| \sim 2^k N^{-1}\}}$ for all $k \in \N$.

For all $k \in \Z_{\ge 0}$, set 
\begin{align} W^{k}_{N,t} = (\ind_{B(0,t)} W_t) *( \phi_{N,k} \, \nu_N), \label{Ysg20a} \end{align}
such that 
\begin{align}
 W_{N,t} =\sum_{k = 0}^{\infty} W^{k}_{N,t}.
 \label{Ysg20ba}
\end{align}
Fix $\al \in \Z^2_{\ge 0}$ with $1 \le |\al| \le 2$. From \eqref{Ysg20ba}, we deduce that \eqref{YW1222} follows from the estimate
\begin{align}
\begin{split}
 |\partial^\al_x W^{k}_{N,t} (x)| & \les 2^{-100k} C_t \min \! \big\{ N^{|\al|}, |t - |x||^{-|\al|} \big\}
\end{split}
\label{Ysg20}
\end{align}
for all $x \in B(0,10)$ and $k\in \Z_{\ge 0}$. 

We further break $W^{k}_{N,t}$ into two parts, depending on whether we want to distribute the derivative $\partial^\al_x$ to the first or second factor in the convolution \eqref{Ysg20a}. To this end, let $M \gg1$ be a large constant and define $\ld \in C^{\infty}_c(\R; [0,1])$ such that 
\begin{align}
\ld(s) = \begin{cases}1 \quad & \text{if $|s| < M$}, \\
 0 \quad & \text{if $|s| \ge M + 1$}.  \end{cases}
 \label{ldYsg}
\end{align}
Now, we write
\begin{align}
\begin{split}
 W^{k}_{N,t}  & = \big(\ind_{B(0,t)} W_t  \, \ld\big(2^{-k}N (t- |\cdot|) \big)  \big) *( \phi_{N,k} \, \nu_N) \\ 
&  \qquad \qquad + \big(\ind_{B(0,t)} W_t  \, (1- \ld)\big(2^{-k}N (t- |\cdot|)\big)  \big) *( \phi_{N,k} \, \nu_N)\\
 & =: W^{k,1}_{N,t} + W^{k,2}_{N,t}.
 \end{split}
 \label{Ysg20b}
 \end{align}
Therefore, \eqref{Ysg20} is a consequence of the bound
\begin{align}
\begin{split}
 |\partial^\al_x W^{k,j}_{N,t} (x)| & \les 2^{-100k} C_t  \min \! \big\{ N^{|\al|},|t - |x||^{-|\al|} \big\}
\end{split}
\label{Ysg21}
\end{align}
for all $x \in B(0,10)$ and $j \in \{1,2\}$. We focus on proving \eqref{Ysg21} in what follows.

\medskip

\noi
{\bf $\bul$ Proof of \eqref{Ysg21} for $j=1$.}\quad By \eqref{Ysg20b} and \eqref{hyplike222}, we have 
\begin{align}
\begin{split}
\partial^\al_x W^{k,1}_{N,t}(x) & = \big( \big(\ind_{B(0,t)} W_t  \, \ld\big(2^{-k}N (t- |\cdot|) \big)  \big) * \big( \partial_x^{\al} \{ \phi_{N,k} \, \nu_N \} \big) \big) (x) \\
& = \int_{B(x,t)} W_t(x-y) \ld \big( 2^{-k} N (t - |x-y|) \big)  \partial_x^{\al} \{ \phi_{N,k} \, \nu_N\}(y) dy.
\end{split}
\label{YW103}
\end{align}
Fix $x$ in $B(0,10) \setminus \mb S^1(t)$ and in the support of $\partial^\al_x W^{k,1}_{N,t}$ and let $y$ be in the support of the integrand $\partial^\al_x W^{k,1}_{N,t}(x)$.\footnote{If $\operatorname{supp}(\partial^\al_x W^{k,1}_{N,t} ) \cap B(0,10) = \emptyset$, there is nothing to show. We will discard such cases without further mention in this proof.} In view of the definition of $\ld$, we have
\begin{align}
0 < |t - |x-y|| \le (M+1) \cdot 2^{k} N^{-1}.
\label{Ysg22a}
\end{align}
Since $|y| \les 2^k N^{-1}$, we deduce that
\begin{align}
|t - |x|| \les M \cdot 2^{k}N^{-1}.
\label{YW103b}
\end{align}
By definition of $\phi_{N,k}$ and \eqref{Gd0}, the following bound holds:
\begin{align}
 \big|\partial_x^{\al} \{ \phi_{N,k} \, \nu_N\}(z)\big| \les 2^{-200k} N^{|\al|} \ind_{|z| \les 2^k N^{-1}}
 \label{Ysg22ab}
\end{align}
for all $z \in \R^2$. Therefore, by \eqref{hyplike222}, \eqref{YW103}, \eqref{Ysg22ab} and \eqref{YW103b}, we have
\begin{align}
\begin{split}
|\partial_x^{\al} W^{k,1}_{N,t}(x)| & \les 2^{-200k} C_t  N^{2+|\al|}  \int_{\R^2} \ind_{|y| \les 2^k N^{-1}} dy \\
&  \les 2^{-198k} C_t  N^{|\al|}\\
& \les_M 2^{-197k} C_t  \min \! \big\{ N^{|\al|}, | t - |x||^{-|\al|} \big\}
\end{split}
\label{Ysg22}
\end{align}
for any $x \in B(0,10)$. Therefore, \eqref{Ysg20} for $j=1$ follows from \eqref{Ysg22}.

\medskip

\noi
{\bf $\bul$ Proof of \eqref{Ysg21} for $j=2$.}\quad Note that the functions 
\begin{align*}
 f : x \in \R^2 & \mapsto W_t(x) \, (1-\ld)\big( 2^{-k}N (t - |x|) \big), \\
 g : x \in \R^2 & \mapsto \phi_{N,k}(x) \, \nu_N(x).
\end{align*}
satisfy the assumptions in Lemma \ref{LEM:Dder} by definition of the bump function $\ld$ in \eqref{ldYsg}. Moreover, by Lemma \ref{LEM:Dder}, we have
\begin{align}
\begin{split}
\partial^\al_x W^{k,2}_{N,t}(x) & =\big(\ind_{B(0,t)} \partial_x^{\al}\big\{ W_t \, (1-  \ld)\big(2^{-k}N (t- |\cdot|)\big) \big\} \big) *( \phi_{N,k} \, \nu_N)\big)(x) \\
& =  \int_{B(x,t)} \partial_x^{\al} \big\{ W_t \, (1-  \ld)\big(2^{-k}N (t- |\cdot|)\big) \big\} (x-y) \, \phi_{N,k}(y) \, \nu_N(y) dy.
\end{split}
\label{Ysg23}
\end{align}
Note that the boundary term in \eqref{YD1} vanishes in the current setting as $f \equiv 0$ on $\mb S^1(t)$. Fix $x$ in $B(0,10) \setminus \mb S^1(t)$ and in the support of $\partial^\al_x W^{k,2}_{N,t}$ and let $y$ be in the support of the integrand $\partial^\al_x W^{k,2}_{N,t}(x)$. In view of the definition of $\ld$, we have 
\begin{align*}
 |t - |x-y|| \ge M \cdot 2^k N^{-1}.
\end{align*}
Thus, since $|y| \les 2^k N^{-1}$, we deduce that
\begin{align}
| t + \eps |x-y|| \sim | t + \eps |x| | \ges M \cdot  2^{k}N^{-1} 
\label{YW108}
\end{align}
for any $\eps \in \{+1, -1\}$, upon choosing $M$ large enough. Besides, by \eqref{hyplike222} and definition of $\ld$, we have
\begin{align}
\big| \partial_x^{\al} \big\{ W_t \, (1-  \ld)\big(2^{-k}N (t- |\cdot|)\big) \big\} (z) \big| \les C_t  |t-|z||^{-|\al|}
\label{Ysg23b}
\end{align}
for all $z \in \R^2$. Note that here, we used the fact that the support of spatial derivative of $\ld$ is included in $\{z \in \R^2 : |z| \sim 1\}$ so that we may exchange the factor $2^{-k}N$ for the term $|t - |z||^{-1}$ when a derivative hits the function $(1-\ld)\big(2^{-k} N(t-|\cdot|)\big)$. Therefore, from \eqref{Gd0}, \eqref{Ysg23}, \eqref{Ysg23b} and \eqref{YW108}, we have
\begin{align}
\begin{split}
|\partial^\al_x W^{k,2}_{N,t}(x)| & \les 2^{-200k} C_t  |t - |x||^{-|\al|} \\
& \les_M 2^{-198k} C_t  \min \! \big\{ N^{|\al|}, | t - |x||^{-|\al|} \big\}
\end{split}
\label{Ysg25}
\end{align}
for any $x \in B(0,10)$. Therefore, \eqref{Ysg20} for $j=2$ follows from \eqref{Ysg25}.
\end{proof}

\begin{remark}\rm \label{RMK:R2}
In order to prove a periodic version (i.e. on $\T^2$ and not $\R^2$) of Lemma \ref{LEM:green_wave2} with the same arguments as in the proof above, one would need to construct a periodic function $\wt \phi_{N,k}$ which (i) essentially coincides with $\phi_{N,k}$ on $[-\pi, \pi)^2$ and (ii) is smooth (since we need to differentiate $\phi_{N,k}$ in our argument; see \eqref{YW103}). However, for $k$ large enough, the support of the bump function $\phi_{N,k}$ in \eqref{phikkk} is strictly larger than the box $[-\pi, \pi)^2 \cong \T^2$. Hence, it is not immediate to construct a function $\wt \phi_{N,k}$ which satisfies both (i) and (ii) at the same time. This is the reason why we work in the full space setting in Lemma \ref{LEM:green_wave2}. In practice, Lemma \ref{LEM:green_wave2} will be used to study multipliers defined in the periodic setting via the Poisson formula \ref{poisson}; see Subsection \ref{SUBSEC:3-1}.
\end{remark}

Next, we consider smoothed an analogue of Lemma \ref{LEM:green_wave2} when $W_t$ is replaced with a smooth function and a variant of the Green's function $G$ \eqref{green2} of the form: $x \in \R^2 \setminus \{0\} \mapsto \ind_{B(0,t)^c}(x) G(x)$ for some fixed $t >0$.\footnote{The set $B(0,t)^c$ denotes the complement of the ball $B(0,t)$ in $\R^2$.}

\begin{corollary}\label{COR:green_wave}
Fix $N \in \N$ and let $\nu_N \in C^{\infty}(\R^2; \R)$ satisfying the estimate \eqref{Gd0}. Fix $0 < t \le 1$ and let $G$ be the Green's function \eqref{green2} and $F \in C^{2}(\R^2; \R)$. Then, the following bounds hold.

\smallskip

\noi
\textup{(i)} Set $G_{N,t} = \big( \ind_{B(0,t)^c} G \big)*\nu_N$. Then, we have

\noi
\begin{align}
& |\partial^\al_x G_{N,t} (x)| \les  \jbb{\, \log \! \big(  t + |x| + N^{-1} \big)} \min \! \big\{ N^{|\al|},| t - |x||^{-|\al|} \big\}
\label{Ysg30}
\end{align}
for any $x \in B(0,10)$ and $\al \in \Z^2_{\ge 0}$ with $1 \le |\al| \le 2$. Here, the implicit constant is independent of $N$.

\smallskip

\noi
\textup{(ii)} Set $F_{N,t} = \big( \ind_{B(0,t)} F \big)*\nu_N$. Then, we have
\begin{align}
|\partial^\al_x F_{N,t} (x)| & \les \min \! \big\{ N^{|\al|},| t - |x||^{-|\al|} \big\}
\label{Ysg30b}
\end{align}
for any $x \in B(0,10)$ and $\al \in \Z^2_{\ge 0}$ with $1 \le |\al| \le 2$. Here, the implicit constant is independent of $N$.
\end{corollary}
\begin{proof}
Fix $N \in \N$ and $0 < t \le 1$. Let $\ld_1, \ld_2 \in C^{\infty}_c(\R^2; [0,1])$ be such that 
\begin{align*}
\ld_1(x) = \begin{cases}1 \quad & \text{if $|x| < 10^{-10}$}, \\
 0 \quad & \text{if $|x| \ge 10^{-9}$}  \end{cases}
\end{align*}
and
\begin{align*}
\ld_2(x) = \begin{cases}1 \quad & \text{if $|x| < 10^{10}$}, \\
 0 \quad & \text{if $|x| \ge 10^{10}+1$}.  \end{cases}
\end{align*}
Define the functions $\wt G_t$ and $\wt F_t$ via the formulas
\begin{align*}
\wt G_t(x) & = G(x) \cdot (1-\ld_1) \Big( \frac x t \Big), \\
\wt F_t(x) & = F(x) \cdot \ld_2  \Big( \frac x t \Big).
\end{align*}
Then, we have
\begin{align}
\begin{split}
\big( \ind_{B(0,t)^c} G \big) * \nu_N & = \big( \ind_{B(0,t)^c} \wt G_t \big) * \nu_N, \\
\big( \ind_{B(0,t)} F \big) * \nu_N & = \big( \ind_{B(0,t)} \wt F_t \big) * \nu_N.
\end{split}
\label{Ysg31}
\end{align}
Moreover, by \eqref{Ysg29}, it is easy to check that $\wt G_t$ and $\wt F_t$ satisfy the bounds \eqref{hyplike222} for $C_t = \jb{\, \log t}$ and $C_t = 1$, respectively. Therefore, \eqref{Ysg30b} is a direct consequence of \eqref{Ysg31} and Lemma \ref{LEM:green_wave2}. As for \eqref{Ysg30}, a weaker version (with $\jbb{\, \log \! \big(  t + |x| + N^{-1} \big)}$ replaced with $\jb{\, \log (t)}$) essentially follows from Lemma \ref{LEM:green_wave2}.\footnote{Note that replacing $\ind_{B(0,t)}$ with $\ind_{B(0,t)^c}$ does not change the bounds in Lemma \ref{LEM:green_wave2}.} The improved factor $\jbb{\, \log \! \big(  t + |x| + N^{-1} \big)}$ comes from a slight modification of the proof of Lemma \ref{LEM:green_wave2} in order to take into account the singularity of $G$ at the origin. We omit details.
\end{proof}

\begin{remark}\rm We note that there is no reason for the function $\big(\ind_{B(0,t)} G \big) * \nu_N$ to also satisfy the bound \eqref{Ysg30}. Indeed, otherwise $G * \nu_N$ would also satisfy \eqref{Ysg30}, which is not compatible (uniformly in $N$) with the estimates in Lemma \ref{LEM:green_der} (i) in the regime of parameters $|x| \ll t$.
\end{remark}

The following result is a variant of Lemma \ref{LEM:green_wave2}.

\begin{lemma}\label{LEM:green_wave3}
Fix $N \in \N$ and let $\nu_N \in C^{\infty}(\R^2; \R)$ satisfying the estimate \eqref{Gd0}. Fix $0 < t \le 1$ and let $W^1_t \in C^1(\R^2 \setminus \mb S^1(t); \R)$ and $W^2_t \in C^1(\R^2 \setminus \mb \{0\}; \R)$ be functions such that
\begin{align}
\begin{split}
|W^1_t(x)| &\les \big| t^2 - |x| ^2\big|^{-\frac12}, \\
| \partial_{x}^{\al} W^1_t(x)| &\les (t + |x|)^{-\frac12} |t-|x||^{-\frac32}
\end{split}
\label{hyplike222b}
\end{align}
and
\begin{align}
\begin{split}
|W^2_t(y)| &\les |y|^{-1}, \\
| \partial_{y}^{\al} W^2_t(y)| &\les |y|^{-2}
\end{split}
\label{hyplike222c}
\end{align}
for any $x \in B(0,10) \setminus \mb S^1(t)$, $y \in B(0,10) \setminus \{0\}$ and $\al \in \Z^2_{\ge 0}$ with $|\al|=1$. Then, the following bounds hold.

\smallskip

\noi
\textup{(i)} Set ${W^1_{N,t} =  (\ind_{B(0,t)} W^1_t) * \nu_N}$. Then, we have

\noi
\begin{align}
\begin{split}
|\partial^\al_x W^1_{N,t} (x)|  & \les  \min \! \big\{ N^2, (t+|x|)^{-\frac12} |t - |x||^{-\frac32} \big\} \jbb{ \, \log \! \big( \! \min \!  \big\{N, |t- |x||^{-1}\big\}\big) }
\end{split}
\label{YW1222b}
\end{align}
for any $x \in B(0,10)$ and $\al \in \Z^2_{\ge 0}$ with $|\al| = 1$. Here, the implicit constant is independent of $N$.

\smallskip

\noi
\textup{(ii)} Set ${W^2_{N,t} =  (\ind_{B(0,t)} W^2_t) * \nu_N}$. Then, we have

\noi
\begin{align}
\begin{split}
|\partial^\al_x W^2_{N,t} (x)|  & \les  \big( \! \min \! \big\{ N^2, |x|^{-1} |t - |x||^{-1} \big\} + \min \! \big\{ N^2, |x|^{-2} \big\}  \big) \\
& \qquad \quad \times \jbb{ \, \log \! \big( \! \min \!  \big\{N, |x|^{-1}, |t- |x||^{-1}\big\}\big) }
\end{split}
\label{YW1222c}
\end{align}
for any $x \in B(0,10) \setminus \{0\}$ and $\al \in \Z^2_{\ge 0}$ with $|\al|=1$. Here, the implicit constant is independent of $N$.
\end{lemma}

\begin{proof} The proof follows from arguments which are similar to those in the proof of Lemma \ref{LEM:green_wave2} and we omit details.
\end{proof}

Lastly, in the next lemma, we look at (smoothed) spatial convolutions of the wave kernel \eqref{poisson3} with the Green's function \eqref{green2}. 

\begin{lemma}\label{LEM:wave_conv_green}
Fix $N \in \N$ and let $\nu_N \in C^{\infty}(\R^2; \R)$ satisfying the estimate \eqref{Gd0}. Fix $0 < t \le 1$ and let $G$ and $W(t, \cdot)$ be the Green's function \eqref{green2} and the wave kernel \eqref{poisson3}, respectively. Set $Q_{N,t} = W(t, \cdot) *  G * \nu_N$. Then, the following bounds hold:
\begin{align}
|Q_{N,t} (x)| \les 1
\label{Ysg32}
\end{align}
and
\begin{align}
|\partial^{\al}_x Q_{N,t} (x)| \les \min \! \big\{N^{|\al|-1}, \big|t^2- |x|^2\big|^{-\frac12(|\al|-1)}\big\} \, \jbb{ \, \log \! \big( \! \min \!  \big\{N, |t- |x||^{-1}\big\} \big) }^2
\label{Ysg32b}
\end{align}
for any $x \in B(0,10)$ and multi-index $\al \in \Z^2_{\ge 0}$ with $1 \le |\al| \le 2$. Here, the implicit constant is independent of $N$. 
\end{lemma}
\begin{proof} Fix $N \in \N$ and $0 <t \le 1$. Set $G_N = G * \nu _N$. Note that \eqref{Ysg32} is immediate from the bounds on $G_N$ in Lemma \ref{LEM:green_der} (i). Thus, it remains to prove \eqref{Ysg32b}.

Let $\varphi, \phi \in C^{\infty}_c(\R^2; [0,1])$ be as in \eqref{eta1}. We define $\chi_0$ and $\chi_\l$ for $\l \in \N$ as follows:
\begin{align*}
\chi_{0}(x)=  1 -\varphi(2x)
\end{align*}
and
\begin{align*}
\chi_{\l}(x)= \phi(2^\l x)
\end{align*}
for any $x \in \R^2$ and $\l \in \N$. Note that $\chi_0 \in C^{\infty}(\R^2; [0,1])$, $\chi_{\l} \in C_c^{\infty}(\R^2; [0,1])$ for all $\l \in \N$ and
\[ \sum_{\l = 0}^{\infty} \chi_{\l}(x) = 1\]
for any $x \in \R^2 \setminus \{0\}$. Moreover, we have $\operatorname{Supp}(\chi_{0}) \subset \{ x \in \R^2 : |x| \ges 1\}$ and ${\operatorname{Supp}(\chi_{\l}) \subset \{ x \in \R^2 : |x| \sim 2^{-\l}\}}$ for all $\l \in \N$.

Set $Q^{\al, \l}_{N,t} = W(t, \cdot) *( \chi_\l \partial_x^{\al} G_N)$ for each $\l \in \Z_{\ge 0}$ and $\al \in \Z^2_{\ge 0}$. Then, by construction, we have
\begin{align}
\partial_x^{\al} Q_{N,t} = \sum_{\l=0}^{\infty} Q^{\al, \l}_{N,t}.
\label{Ysg52}
\end{align}
From \eqref{Ysg33} in Lemma \ref{LEM:green_der} (i) and since the support of $\chi_0$ is away from the origin, we have that 
\begin{align}
\begin{split}
|Q_{N,t}^{\al, 0}(x)| &\les \int_{\R^2} \frac{\ind_{B(0,t)}(x-y)}{\sqrt{ t^2 - |x-y|^2 }} \chi_0(y)  |\partial_x^{\al} G_N(y) |  dy \\
& \les \int_{\R^2} \frac{\ind_{B(0,t)}(x-y)}{\sqrt{ t^2 - |x-y|^2 }} \jb y^{-10}  dy \\
&  \les 1
\end{split}
\label{Ysg54}
\end{align}
for all $x \in B(0,10)$ and all multi-index $\al \in \Z^2_{\ge 0}$ with $|\al| \le 2$. 

Fix $\al \in \Z^2_{\ge 0}$ with $|\al|=1$ and $\l \in \N$. Then, from \eqref{Ysg33} in Lemma \ref{LEM:green_der} (i) and since ${\operatorname{supp}(\chi_{\l}) \subset \{ x \in \R^2 : |x| \sim 2^{-\l}\} \subset B(0,2)}$, we have 
\begin{align}
\begin{split}
|Q_{N,t}^{\al, \l}(x)| & \les \int_{\R^2} \frac{\ind_{B(0,t)}(x-y)}{\sqrt{ t^2 - |x-y|^2 }} \chi_\l(y)  |\partial_x^{\al} G_N(y) |  dy \\
& \les  \int_{\R^2} \frac{\ind_{B(0,t)}(x-y)}{\sqrt{ t^2 - |x-y|^2 }} \chi_\l(y) \big(|y| + N^{-1}\big)^{-1}  dy  \\
& \les 2^{\l}  \int_{\R^2} \frac{\ind_{B(0,t)}(x-y)}{\sqrt{ t^2 - |x-y|^2 }} \, \ind_{|y| \les 2^{-\l}}  dy \\
& \les  2^{-\l} \cdot \wt H^{0}_{2^\l, t, \frac12}(x)
\end{split}
\label{Ysg53}
\end{align}
for all $x\in B(0,10)$. Here, $\wt H^{0}_{2^\l, t, \frac12}$ is as in \eqref{Ysg16}. Therefore, we deduce from \eqref{Ysg53} and \eqref{Ysg17} in Lemma \ref{LEM:green_wave1} that
\begin{align}
|Q_{N,t}^{\al_1, \l}(x)| \les 2^{-\l} \min \! \big\{2^\l, \big|t^2- |x|^2\big|^{-\frac12}\big\} \, \jbb{ \, \log \! \big( \! \min \!  \big\{2^\l, |t- |x||^{-1}\big\} \big)}
\label{Ysg56}
\end{align}
for all $x \in B(0,10)$. Hence, by \eqref{Ysg52}, \eqref{Ysg54} and summing \eqref{Ysg56} over $\l \in \N$, we deduce that
\begin{align}
|\partial_x^{\al} Q_{N,t}(x)| \les \jbb{\, \log\big(|t- |x| |^{-1}\big)}^2
\label{Ysg57}
\end{align}
for any $x \in B(0,10)$.

An integration by parts argument with \eqref{Gd0} shows
\begin{align}
| \ft \nu_N (\xi)| \les \jbb{\frac{\xi}{N}}^{-A}
\label{Ysg58}
\end{align}
for any $\xi \in \R^2$ and $A \ge 1$. Hence, by \eqref{poisson2}, \eqref{green2}, \eqref{Ysg58} and the Hausdorff-Young inequality, we have
\begin{align}
|\partial_x ^{\al} Q_{N,t}(x) | \les \jb{\, \log N}.
\label{Ysg59}
\end{align}
for all $x \in \R^2$. The desired bound \eqref{Ysg32b} in the case $|\al| = 1$ thus follows from \eqref{Ysg57} and \eqref{Ysg59}. 

We now prove \eqref{Ysg32b} for $|\al|=2$. Here, we need to proceed with care in view of the lack of integrability of $\partial_x^{\al} G_N$ near the origin (uniformly in $N$); see Lemma \ref{LEM:green_der} (i). Fix $\al \in \Z^2_{\ge 0}$ with $|\al| = 2$ and let $\al_1, \al_2 \in \Z^2_{\ge 0}$ such that $\al = \al_1 + \al_2$ and $|\al_1| = |\al_2| = 1$. Then, with the same notations as above (i.e. as in \eqref{Ysg52}), we have
\begin{align}
\partial_x ^{\al} Q_{N,t} = \sum_{\l =0}^{\infty} \partial_x^{\al_2} \big\{ W(t,\cdot) * (\chi_\l \partial_x^{\al_1} G_N) \big\} = \sum_{\l = 0}^{\infty} \partial_x^{\al_2} Q^{\al_1, \l}_{N,t} .
\label{Ysg60}
\end{align}
Fix $\l \in \N$ and $\kk \in (0,1)$ to be chosen later. Let $\ld \in C^{\infty}_c(\R; [0,1])$ be as in \eqref{ldYsg} and decompose $Q^{\al_1, \l}_{N,t}$ as follows:
\begin{align}
\begin{split}
Q^{\al_1, \l}_{N,t} & = \big( W(t, \cdot) \, \ld\big( 2^{\l \kk} (t - |\cdot|) \big) \big) *  (\chi_\l \partial_x^{\al_1} G_N) \\
& \qquad   + \big( W(t, \cdot) \, (1-\ld)\big(2^{\l \kk} (t - |\cdot|) \big) \big) *  (\chi_\l \partial_x^{\al_1} G_N) \\
& =: \1_\l + \II_\l.
\end{split}
\label{Ysg61}
\end{align}

We first bound $\partial_x^{\al_2} \1_\l$. Note that $\operatorname{supp}(\chi_{\l}) \subset \{ x \in \R^2 : |x| \sim 2^{-\l}\}$ and hence by the chain rule and \eqref{Ysg33} in Lemma \ref{LEM:green_der} (i), we have
\begin{align}
\begin{split}
\big| \partial_x^{\al_2} \{ \chi_\l \partial_x^{\al_1} G_N\}(z) \big|  \les 2^{2\l} \ind_{|z| \sim 2^{-\l}}
\end{split}
\label{Ysg62}
\end{align}
for all $z \in \R^2$. Therefore, by \eqref{Ysg62} and moving the derivative to the second factor in the convolution $\1_\l$, we have
\begin{align}
\begin{split}
\big| \partial_x^{\al_2} \1_{\l} (x) \big| & \les \big| \big( \big( W(t, \cdot) \ld\big( 2^{\l \kk} (t - |\cdot|) \big) \big) *  \partial_x^{\al_2}\{\chi_\l \partial_x^{\al_1} G_N\} \big)(x)\big| \\
& \les 2^{2\l}  \int_{B(x,t)} \frac{\ld\big( 2^{\l \kk} (t - |x-y|) \big)}{\sqrt{ t^2 - |x-y|^2 }} \, \ind_{|y| \les 2^{-\l}}  dy \\
& \les \wt H^{0}_{2^\l, t, \frac12}(x) \\
& \les \min \! \big\{2^\l, \big|t^2- |x|^2\big|^{-\frac12}\big\} \, \jbb{ \, \log \! \big( \! \min \!  \big\{2^\l, |t- |x||^{-1}\big\} \big)}
\end{split}
\label{Ysg63}
\end{align}
for all $x\in B(0,10)\setminus \mb S^1(t)$. Here, $\wt H^{0}_{2^\l, t, \frac12}$ is as in \eqref{Ysg16}. Now, fix $x \in \big(B(0,10) \setminus \mb S^1(t)\big) \cap \operatorname{supp}(\partial_x^{\al_2} \1_\l)$ (if such a $x$ does not exist, then there is nothing to show)  and write
\[ \partial_x^{\al_2} \1_{\l} (x) = \int_{B(x,t)} \frac{\ld\big( 2^{\l \kk} (t - |x-y|) \big)}{\sqrt{ t^2 - |x-y|^2 }} \partial_x^{\al_2} \{ \chi_\l \partial_x^{\al_1} G_N\}(y)  dy. \]
Therefore, there must exists $y$ in the support of the integrand of $\partial_x^{\al_2} \1_{\l} (x)$ such that
\[ | t - |x-y|| \les 2^{-\l \kk}, \]
by definition of $\ld$. Since $|y| \les 2^{-\l}$ and $\kk \in (0,1]$, we must have
\begin{align}
| t- |x||\les 2^{-\l \kk}.
\label{Ysg64}
\end{align} 
Therefore, combining \eqref{Ysg63} and \eqref{Ysg64} yields
\begin{align}
\big| \partial_x^{\al_2} \1_{\l} (x) \big| \les \l 2^{\l} \, \ind_{2^{\l} \les |t^2 - |x|^2|^{-\frac12}} + \big|t^2 - |x|^2\big|^{-\frac12} \jbb{\, \log\big(|t- |x| |^{-1}\big)} \ind_{2^\l \les |t-|x||^{-\frac 1 \kk}}
\label{Ysg65}
\end{align}
for any $x \in B(0,10) \setminus \mb S^1(t)$.

Now, we consider the contribution of the term $\partial_x^{\al_2} \II_\l$. Note that the functions 
\begin{align*}
 f : x \in \R^2 & \mapsto \frac{W(t,\cdot)  (1-\ld)\big( 2^{\l \kk} (t - |x|) \big)}{\sqrt{t^2 - |x|^2}}, \\
 g : x \in \R^2 & \mapsto G_N(x).
\end{align*}
satisfy the assumptions in Lemma \ref{LEM:Dder} by definition of the bump function $\ld$ in \eqref{ldYsg}. Thus, by moving the $\partial_x^{\al_2}$ derivative to the first factor in the convolution $\II_\l$ and Lemma \ref{LEM:Dder}, we have
\begin{align}
\begin{split}
 \partial_x^{\al_2} \II_{\l}(x) & =  \big(\big( \ind_{B(0,t)}(\cdot)  \partial_x^{\al_2} \big\{ (t^2 - |\cdot|^2)^{-\frac12} \, (1-\ld)\big( 2^{\l \kk} (t - |\cdot|) \big) \big\} \big) *  (\chi_\l \partial_x^{\al_1} G_N)\big)(x) \\
 & = \int_{B(x,t)} \partial_x^{\al_2} \big\{ (t^2 - |\cdot|^2)^{-\frac12} \, (1-\ld)\big( 2^{\l \kk} (t - |\cdot|) \big) \big\}(x-y)  (\chi_\l \partial_x^{\al_1} G_N)(y) dy.
 \end{split}
 \label{Ysg67}
\end{align}
Note that the boundary term in \eqref{YD1} vanishes since $f \equiv 0$ on $\mb S^1(t)$. Fix $x$ in $B(0,10) \setminus \mb S^1(t)$ and in the support of $\partial^\al_x \II^{\l}$ and let $y$ be in the support of the integrand $\partial^\al_x \II^{\l}(x)$. In view of the definition of $\ld$, we have 
\begin{align*}
 |t - |x-y|| \ges 2^{-\l \kk}.
\end{align*}
Thus, since $|y| \les 2^{-\l}$ and $\kk \in (0,1)$ small enough, we deduce that
\begin{align}
| t + \eps |x-y|| \sim | t + \eps |x| | \ges 2^{-\l \kk} 
\label{Ysg69}
\end{align}
for any $\eps \in \{+1, -1\}$. Besides, by definition of $\ld$, we have
\begin{align}
\big| \partial_x^{\al_2} \big\{ (t^2 - |\cdot|^2)^{-\frac12} \, (1-\ld)\big( 2^{\l \kk} (t - |\cdot|) \big) \big\} (z) \big| \les |t + |z||^{-\frac12} |t-|z||^{-\frac32}
\label{Ysg70}
\end{align}
for all $z \in \R^2$. Hence, by \eqref{Ysg67}, \eqref{Ysg69}, \eqref{Ysg70} and Lemma \ref{LEM:green_der} (i), we have
\begin{align}
\begin{split}
|\partial_x^{\al_2} \II_\l(x)| & \les \ind_{|t-|x||^{\frac 1\kk} \ges 2^{-\l}} \, |t + |x||^{-\frac12} |t-|x||^{-\frac32} \int_{\R^2} \ind_{|y| \les 2^{-\l}} |y|^{-1} dy \\
& \les |t + |x||^{-\frac12} |t-|x||^{-\frac32} \, 2^{-\l} \, \ind_{|t-|x||^{\frac 1\kk} \ges 2^{-\l}} 
\end{split}
\label{Ysg73}
\end{align}
for all $x \in B(0,10)\setminus \mb S^1(t)$.

Therefore, from \eqref{Ysg60}, \eqref{Ysg61}, \eqref{Ysg54} and summing \eqref{Ysg65} and \eqref{Ysg73}, we get
\begin{align}
\begin{split}
\big| \partial_x ^{\al} Q_{N,t}(x)\big| & \les_\kk \big| t^2 - |x|^2 \big|^{-\frac12}  \jbb{\, \log\big(|t- |x| |^{-1}\big)}^2 + |t + |x||^{-\frac12} |t-|x||^{-\frac32 + \frac1 \kk} \\
& \les_\kk \big| t^2 - |x|^2 \big|^{-\frac12}  \jbb{\, \log\big(|t- |x| |^{-1}\big)}^2
\end{split}
\label{Ysg76}
\end{align}
for all $x \in B(0,10)\setminus \mb S^1(t)$ and upon choosing $\kk$ small enough. By working on the Fourier side and arguing as in \eqref{Ysg57}-\eqref{Ysg59} above, we also get
\begin{align}
\big| \partial_x ^{\al} Q_{N,t}(x)\big| \les  N.
\label{Ysg77}
\end{align}
Thus, the bound \eqref{Ysg32b} for $|\al|=2$ follows from \eqref{Ysg76} and \eqref{Ysg77}.
\end{proof}

\section{Nonlinear analysis}\label{SEC:3}

In this section, we state key bilinear estimates for our well-posedness argument in Section \ref{SEC:WP}.

\subsection{Basic product estimates and fractional chain rules}

Here, we recall standard product estimates and the fractional chain rule.

Our first estimate is a product estimate in Sobolev spaces. See \cite{GKO} for a proof.

\begin{lemma}\label{LEM:prod}
Let $d \in \N$ and $\mc M = \R^d$ or $\T^d$. Fix $0 < s < 1$ and $1<p_j,q_j,r<\infty$ with $\frac1{p_j}+\frac1{q_j}=\frac1r$, $j=1,2$. Then, we have 
\begin{align*}
\big\|\jb{\nabla}^\al(fg)\big\|_{L^r(\T^d)} 
\les \big\|\jb{\nabla}^s f\big\|_{L^{p_1}(\mc M)}\|g\|_{L^{q_1}(\mc M)} + \|f\|_{L^{p_2}(\mc M)}\big\|\jb{\nabla}^s g\big\|_{L^{q_2}(\mc M)}.
\end{align*}
\end{lemma}

Next, we recall the fractional chain rule in Sobolev spaces.

\begin{lemma}\label{LEM:fcr}
Let $d \in \N$ and $\mc M = \R^d$ or $\T^d$. Fix $0 < s < 1$ and $F$ a Lipschitz function on $\R$ such that $\|F'\|_{L^{\infty}(\R)}\le L$.
Then, for any    $1<p<\infty$, we have 
\begin{align*}
\big\||\nabla|^\al F(f)\big\|_{L^p(\mc M)}\les L\big\||\nabla|^\al f\big\|_{L^p(\mc M)}.
\end{align*}

\smallskip

\end{lemma}

The fractional chain rule on $\R^d$ was essentially proved in \cite{CW}.\footnote{As pointed out in \cite{Staffilani}, 
the argument in \cite{CW} needs
a small correction, which yields the fractional chain rule in a 
less general context.
See \cite{Kato, Staffilani, Taylor}.}
As for the estimate on $\T^d$, see~\cite{Gatto}.

\subsection{Weighted estimates}\label{SUBSEC:weighted}

In this subsection, we record several weighted estimates. Namely, we study here the boundedness properties of convolution operators on spaces of the form $L^p(\R^d; w(z) dz)$,\footnote{That is, the space of functions whose $p^{\text{th}}$ power is integrable with respect to the measure $w(z) dz$.} where $d \in \N$, $1 < p < \infty$ and $w$ is non-negative function on $\R^d$. We will also consider spaces of the form $L^p_{t,x}(\R \times \T^2; w(t,x) dt dx)$ for space-time weights $w = w(t,x)$.

Although in this work we mainly consider specific (time) weights of the form $w : t \in \R \mapsto \jb t^a$, with $a \in \R$, we introduce next a general class of weights which is standard in the literature on harmonic analysis.

\begin{definition}[$A_p$ weights]\label{DEF:weight}
Let $d \in \N$ and $1 < p < \infty$. We denote by $A_p$ the set of non-negative locally integrable functions $w$ on $\R^d$ for which there exists a constant $C >0$ such that
\[ \Big( \frac{1}{|B|} \int_B w \Big) \cdot \Big( \frac{1}{|B|} \int_B w^{-\frac{p'}{p}}\Big)^{\frac{p}{p'}} \le C \]
for all balls $B \subset \R^d$. Here, $p'$ denotes the dual Lebesgue exponent to $p$; i.e. $\frac{1}{p} + \frac{1}{p'} = 1$.
\end{definition}

We subsequently state a result regarding weighted estimates for singular integrals. See \cite[Theorem 2 in Chapter V]{Stein} for a proof.

\begin{lemma}\label{LEM:singAP}
Fix $d \in \N$ and $1 < p < \infty$. Let $T$ be a convolution operator with distribution kernel $K$ on $\R^d$. Namely, $T(f) = K * f$ for any $f \in C^{\infty}_c(\R^d)$. We assume that the kernel $K$ satisfies the following:
\begin{align*}
 \textup{(i)}&  \  \  |\partial_x^{\al} K(x)| \les |x|^{-d-|\al|} \quad \text{for all $x \in \R^d \setminus \{0\}$ and $\al \in \Z_{\ge 0}^d$ with $|\al| \le1$}, \\
\textup{(ii)}&  \ \  \ft K \in L^{\infty}(\R^d).
\end{align*}
Then, the operator $T$ maps $L^p(\R^d; w(z) dz)$ into itself.
\end{lemma}

In practice, we use the following specific version of Lemma \ref{LEM:singAP}.
\begin{corollary}\label{COR:AP}
Fix $a \in (-1,1)$, $r \in \R$, $N_0 \in \N$ and a bump function $\ld \in C^{\infty}_c(\R; \R)$. Let $\mc H_0$ and $T$ be the Fourier multiplier given by
\begin{align}
\begin{split}
\ft{\mc H_0(f)}(\tau)  & = - i \, \textup{sgn}(\tau) \cdot \ft f(\tau), \\
\ft{T(f)}(\tau)  & = \ld\Big(\frac{|\tau| + r}{N_0}\Big) \cdot \ft f(\tau)
\end{split}
\label{wope}
\end{align}
for any $\tau \in \R$. Here, $\textup{sgn}$ is as in \eqref{Rhyp2}.\footnote{The projection onto the $\tau$-coordinate of the Fourier multiplier $\mc H$ defined in \eqref{ht} basically gives $\mc H_0$.} Then, the operators $\mc H_0$ and $T$ map $L^2(\R; \jb t^a dt)$ into itself.
\end{corollary}
\begin{proof} We first observe that for $a \in (-1,1)$, $t \mapsto \jb t^a$ is a $A_2$ weight following Definition \ref{DEF:weight}. The convolution kernel $K_0$ of the Hilbert transform $\mc H_0$ coincides with the function $t \in \R \setminus \{0\} \mapsto \frac{c}{t}$, where $c>0$ is a constant, on $\R \setminus \{0\}$. See for instance \cite[Chapter III]{SW}. Hence, $K_0$ satisfies the assumptions in Lemma \ref{LEM:singAP} and $\mc H_0$ is bounded on $L^2(\R; \jb t^a dt)$.

On the other hand, note that the symbols $\ind_{\tau >0}$ and $\ind_{\tau<0}$ are of the form
\[ c_1 + c_1 m(\tau), \]
where $m$ is the symbol of the Hilbert transform $\mc H_0$ and $c_1, c_2 \in \mb C$. Therefore, the boundedness of $T$ on $L^2(\R; \jb t^a dt)$ reduces to that of the multipliers
\begin{align*}
\ft{T_+(f)}(\tau)  & = \ld\Big(\frac{\tau + r}{N_0}\Big) \cdot \ft f(\tau), \\
\ft{T_-(f)}(\tau)  & = \ld\Big(\frac{-\tau + r}{N_0}\Big) \cdot \ft f(\tau)
\end{align*}
on $L^2(\R; \jb t^a dt)$. For each $\eps \in \{+, -\}$, an integration by parts argument shows that the kernel $K_{\eps}$ of $T_{\eps}$ satisfies the bound
\begin{align*}
|K_\eps' (t)| \les N_0^{1+|\al|} \jb{N_0 t}^{-10} \les |t|^{-1-|\al|}
\end{align*}
for any $t \in \R$ and $\al \in \Z_{\ge 0}$, with implicit constants uniform in $r$. Hence, by Lemma \ref{LEM:singAP}, $T_{\eps}$ is bounded on $L^2(\R; \jb t^a dt)$.
\end{proof}

Next, for a dyadic triplet $(N,R,L) \in (2^{\N})^2 \times 2^{\Z}$ and $b \in \R$, we consider the Fourier multiplier $\C^b_{N,R,L}$ on $\R \times \T^2$ given by 
\begin{align}
\F_{t,x}\big[\C^b_{N,R,L} u\big](\tau, n) = \phi \Big(\frac{n}{N}\Big) m^b_{R,L}(\tau, n) \cdot \ft u(\tau, n), \quad (\tau, n) \in \R \times \Z^2,
\label{C1}
\end{align}
where $m_b$ is the symbol
\begin{align}
m^b_{R,L}(\tau, n) = \big||\tau|- |n|\big|^b  \eta \Big(\frac{\tau}{R}\Big) \psi \Big(\frac{|\tau| - |n|}{L}\Big).
\label{C1b}
\end{align}
Here, the bump functions $\phi$, $\eta$ and $\psi$ are as in \eqref{eta1}. 
Recall in particular that the supports of $\eta$ and $\psi$ are away from the origin. Hence, the map 
\begin{align} m^0_{R,L}(\cdot,n): \tau \mapsto \eta \Big(\frac{\tau}{R}\Big) \psi \Big(\frac{|\tau| - |n|}{L}\Big)\label{C1bb} \end{align}
is smooth on $\R$ for each $n \in \Z^2$, with derivative bounded by
\begin{align}
|\partial_{\tau} m^0_{R,L} (\tau,n)| \les (R^{-1} + L^{-1}) \cdot \ind_{||\tau|-|n|| \sim L}.
\label{C1bbb}
\end{align}
We also note that $\C^b_{N,R,L}$ is related to the multiplier $\M_{N,R,L}$ defined in \eqref{proj3} via the formula:
\begin{align}
\mathcal{C}^b_{N,R,L} = \big| |\dt| - |\nb| \big|^b \M_{N,R,L}.
\label{C111b}
\end{align}

Operators of the form $\C^b_{N,R,L}$ are referred to as {\it cone multipliers} in the Euclidean harmonic analysis literature and have been heavily studied; see for instance \cite{BO95, GWZ, HNS, LV, Mock, TV1, TV2, Wolff2} and references therein.

The following weighted $L^2$ estimate on $\C^b_{N,R,L}$ (and variants) plays a crucial role in our bilinear analysis; see Section \ref{SUBSEC:bilin} below.

\begin{lemma}\label{LEM:wcone}
Fix $a \in (0,1)$, $b \in \R$ and $\A \in \{\Id, \C, \H\C\}$, where $\H$ and $\C$ are as in 
\eqref{ht} and \eqref{cone}, respectively.\footnote{Here, with a slight abuse of notation, we also denote by $\H$ and $\C$ the natural spatially periodic versions of the multipliers in \eqref{ht} and \eqref{cone}.} Fix $(N,R,L) \in (2^{\N})^2 \times 2^{\Z}$ and let $\C^b_{N,R,L}$ be as in \eqref{C1}-\eqref{C1b}. Then, $\mc A\C^b_{N,R,L}$ maps ${L^2(\R \times \T^2; \jb t^a dt dx)}$ into itself and we have the bound 
\begin{align}
\begin{split}
\big\|\jb t^a \mc A \C^b_{N,R,L} u\big\|_{L^2_{t,x}} & \les L^b \big(1 + L^{-(1-\dl_\circ) (a + \frac{\dl_\circ}{10})}\big) \cdot \|\jb t^a u\|_{L^2_{t,x}} \\
& \qquad + L^b \big(1+L^{-a(1-\dl_{\circ}) - \dl_{\circ}}\big) \cdot \big\|\F_{t,x}^{-1}[\ind_{||\tau|-|n|| \les L^{1-\dl_{\circ}}} \, \ft u (\tau, n)]\big\|_{L^2_{t,x}}
\end{split}
\label{Ysg140}
\end{align}
for any small enough $\dl_{\circ}>0$.
\end{lemma}

\begin{proof}
The boundedness of $\mc H_0$ defined in \eqref{wope} on $L^2(\R; \jb t^a dt)$ immediately implies that of $\mc H$ on $L^2(\R \times \T^2; \jb t^a dt dx)$ by Plancherel's identity since 
\[ \F_x[ \mc H(u)](t, n) = \mc H_0( \F_x [u](\cdot, n) )(t) \]
for each $(t,n) \in \R \times \Z^2$. Therefore, by Corollary \ref{COR:AP}, it suffices to prove the statement with $\mc A = \Id$ or $\mc A = \C$. Next, noting that the symbol $(\tau, n) \mapsto \ind_{|\tau| > |n|}$ of $\C$ is a $0$-homogeneous, we have
\[ \psi \Big(\frac{|\tau| - |n|}{L}\Big) \ind_{|\tau| > |n|} = \wt \psi \Big(\frac{|\tau| - |n|}{L}\Big),   \]
where $\wt \psi(\tau) = \psi(\tau)\ind_{\tau>0}$ is bump function which is smooth on $\R \setminus \{0\}$ and whose support is away from the origin. Thus, the map \eqref{C1bb} where $\psi$ is replaced with $\wt \psi$ is smooth on $\R$. Hence, up to changing the bump function $\psi$, it suffices to prove \eqref{Ysg140} for $\mc A = \Id$. Similarly, since $(\tau, n) \mapsto \big||\tau|- |n|\big|^b$ is $b$-homogeneous it suffices to prove \eqref{Ysg140} for $b = 0$. Our goal is thus to prove
\begin{align}
\begin{split}
\big\|\jb t^a  \C^0_{N,R,L} u\big\|_{L^2_{t,x}} & \les \big(1 + L^{-(1-\dl_\circ) (a + \frac{\dl_\circ}{10})}\big) \cdot \|\jb t^a u\|_{L^2_{t,x}} \\
& \qquad + \big(1+L^{-a(1-\dl_{\circ}) - \dl_{\circ}}\big) \cdot \big\|\F_{t,x}^{-1}[\ind_{||\tau|-|n|| \les L^{1-\dl_{\circ}}} \, \ft u (\tau, n)]\big\|_{L^2_{t,x}}.
\end{split}
\label{Ysg141}
\end{align}

By Plancherel's identity, we have
\begin{align}
\big\|\jb t^a  \C^0_{N,R,L} u\big\|_{L^2_{t,x}} \les \|u\|_{L^2_{t,x}} + \Big\| \phi \Big(\frac{n}{N}\Big) \big\|(-\partial^2_{\tau})^{\frac a2} \{ m^0_{R,L} (\cdot, n) \ft u(\cdot, n) \}(\tau)\big\|_{L^2_{\tau}} \Big\|_{\l^2_n}.
\label{Ysg142}
\end{align}
Thus, since $\|\phi\|_{L^\infty} \les 1$, \eqref{Ysg141} follows from \eqref{Ysg142} and the following estimates:
\begin{align}
& \big\| (-\partial^2_{\tau})^{\frac a2} \{ m^0_{R,L} (\cdot, n) \ft u(\cdot, n) \}(\tau) \big\|_{L^2_{\tau}} \les \|\ft u(\tau,n)\|_{L^2_\tau} + \big\|(-\partial^2_{\tau})^{\frac a2} \ft u(\cdot, n)\big\|_{L^2_\tau} 
\label{Ysg143}
\end{align}
for $L \ges 1$ and 
\begin{align}
\begin{split}
 \big\| (-\partial^2_{\tau})^{\frac a2} \{ m^0_{R,L} (\cdot, n) \ft u(\cdot, n) \}(\tau) \big\|_{L^2_{\tau}} & \les L^{-(1-\dl_\circ) (a + \frac{\dl_\circ}{10})} \cdot \|\ft u (\tau, n)\|_{L^2_{\tau}}  + \big\|(-\partial^2_{\tau})^{\frac a2} \ft u(\cdot, n)\big\|_{L^2_\tau} \\
 & \qquad  +  L^{-a(1-\dl_{\circ}) - \dl_{\circ}} \cdot \|\ind_{||\tau|-|n|| \les L^{1-\dl_{\circ}}} \, \ft u (\tau, n)\|_{L^2_\tau} 
\end{split}
\label{Ysg143b}
\end{align}
for $L \ll 1$ and any small enough $\dl_{\circ} >0$.

From \eqref{fracder}, we have
\begin{align}
(-\partial^2_{\tau})^{\frac a2} \{ m^0_{R,L} (\cdot, n) \ft u(\cdot, n) \}(\tau) = \1(\tau) + \II(\tau),
\label{Ysg144}
\end{align}
where
\begin{align*}
\1(\tau) & = c_a \int_{\R} \frac{m^0_{R,L}(\tau+h,n) - m^0_{R,L}(\tau,n)}{|h|^{1+a}} \ft u (\tau + h, n) dh, \\
\II(\tau) & = c_a \int_{\R} \frac{\ft u (\tau+h, n) - \ft u(\tau, n)}{|h|^{1+a}} dh \cdot m^0_{R,L}(\tau, n). 
\end{align*}
We immediately note the estimate:
\begin{align}
\|\II(\tau)\|_{L^2_\tau} \les \big\|(-\partial^2_{\tau})^{\frac a2} \ft u(\cdot, n)\big\|_{L^2_\tau},
\label{Ysg144b}
\end{align}
since $\|m^0_{R,L}(\tau, n)\|_{L^{\infty}_{\tau}} \les 1$. Next, we bound the $L^2_{\tau}$-norm of the term $\1(\tau)$ in different regimes of the integrand parameter $h$, depending on the size of $L$.

\medskip

\noi
{\bf $\bul$ Case 1: $L \ges 1$.}\quad If $|h| \ges 1$ on the integrand of $\1(\tau)$, then by Minkowski's inequality, we have
\begin{align}
\|\1(\tau)\|_{L^2_{\tau}} \les \|\ft u (\tau, n)\|_{L^2_{\tau}},
\label{Ysg145}
\end{align}
since $\|m^0_{R,L}(\tau, n)\|_{L^{\infty}_{\tau}} \les 1$ and $a >0$. By the mean value theorem, the smoothness of $m^0_{R,L}$ and \eqref{C1bbb}, we have
\begin{align*}
|m^0_{R,L}(\tau+h,n) - m^0_{R,L}(\tau,n)| = |h| \cdot \Big|\int_0^1 \partial_{\tau} m^0_{R,L}(\tau + s h, n) ds \Big| \les |h|.
\end{align*}
Therefore, if $|h| \les 1$ on the integrand of $\1(\tau)$, we deduce that
\begin{align}
\|\1(\tau)\|_{L^2_{\tau}} \les \|\ft u (\tau, n)\|_{L^2_{\tau}},
\label{Ysg146}
\end{align}
since $a < 1$.

\medskip

\noi
{\bf $\bul$ Case 2: $L \ll 1$.}\quad If $|h| \ges 1$ on the integrand of $\1(\tau)$ then the bound \eqref{Ysg145} holds. Otherwise, fix $0 < \dl_{\circ} \ll 1$. In the case where $L^{1-\dl_{\circ}} \les  |h| \ll 1$ holds on the integrand of $\1(\tau)$, we use the bound 
\[|h|^{-1-a} \les L^{-(1-\dl_\circ) (a + \frac{\dl_\circ}{10})} \cdot |h|^{-1+\frac{\dl_{\circ}}{10}}\]
and Minkowski's inequality to get
\begin{align}
\|\1(\tau)\|_{L^2_{\tau}} \les L^{-(1-\dl_\circ) (a + \frac{\dl_\circ}{10})} \cdot \|\ft u (\tau, n)\|_{L^2_{\tau}}.
\label{Ysg147}
\end{align}
Otherwise, we have $|h| \ll L^{1-\dl_{\circ}} \ll 1$ on the integrand of $\1(\tau)$. From the mean value theorem and \eqref{C1bbb}, we have
\begin{align}
\begin{split}
|m^0_{R,L}(\tau+h,n) - m^0_{R,L}(\tau,n)| & = |h| \cdot \Big|\int_0^1 \partial_{\tau} m^0_{R,L}(\tau + s h, n) ds \Big| \\
&  \les L^{-1} |h| \int_0^1 \ind_{||\tau + sh|-|n|| \sim L} ds.
\end{split}
\label{Ysg148}
\end{align}
Let $s \in [0,1]$. Note that in view of the estimate
\[\big|||\tau + sh| - |n|| – | |\tau| -|n||\big| \les |h| \les L^{1-\dl_{\circ}},\]
we deduce the bound 
\begin{align}
\ind_{||\tau + sh|-|n|| \sim L} \les \ind_{||\tau|-|n|| \les L^{1-\dl_{\circ}}}.
\label{Ysg149}
\end{align}
Hence, from \eqref{Ysg148} and \eqref{Ysg149}, we get 
\begin{align}
\begin{split}
\|\1(\tau)\|_{L^2_\tau} & \les L^{-1} \int_{\R} |h|^{-a} \ind_{|h| \les L^{1-\dl_{\circ}}} dh \cdot \|\ind_{||\tau|-|n|| \les L^{1-\dl_{\circ}}} \, \ft u (\tau, n)\|_{L^2_\tau} \\
& \les L^{-a(1-\dl_{\circ}) - \dl_{\circ}} \cdot \|\ind_{||\tau|-|n|| \les L^{1-\dl_{\circ}}} \, \ft u (\tau, n)\|_{L^2_\tau}.
\end{split}
\label{Ysg150}
\end{align}
\end{proof}

We conclude this section with weighted variants of Bernstein-type inequalities; see \cite[Appendix A]{TAO}.

\begin{lemma}\label{LEM:bern}
Fix $R \in \N$ and let $\mathbf T_R$ be as in \eqref{proj2}. Fix $1 \le p,q,r \le \infty$ with $p \le q$ and $s \in \R$.  Then, the following estimates holds:
\begin{align}
\| \jb t  \mbf T_R f\|_{L^q_t} & \les R^{\frac1p - \frac1q} \|\jb t  f\|_{L^p_t}, \label{bern1} \\
R^s \| \jb t  \mbf T_R f\|_{L^r_t} & \les \| f \|_{W_t^{s,r}} + \| t \cdot f \|_{W^{s,r}_t}. \label{bern2}
\end{align}
\end{lemma}
\begin{proof}
Let us first prove \eqref{bern1}. We note that
\begin{align}
\F_t [ t \cdot \mbf T_R f](\tau) = i \partial_\tau \Big\{  \eta\Big( \frac{\tau}{R}\Big) \ft f(\tau) \Big\} = i R^{-1} (\partial_\tau \eta)\Big(\frac{\tau}{R}\Big) \ft f(\tau) +  \eta\Big( \frac{\tau}{R}\Big) \big(i \partial_\tau \ft f \, \big)(\tau).
\label{Ysg1501}
\end{align}
Thereofore, denoting by $\wt{\mbf T}_R$ the multiplier defined as $\mbf T_R$, but with $\eta$ replaced with $\partial_\tau \eta$, from \eqref{Ysg1501} and the standard Bernstein inequality, we have that
\begin{align*}
\| \jb t  \mbf T_R f\|_{L^q_t} & \les \| \mbf T_R f\|_{L^q_t} + \|  t \cdot \mbf T_R f\|_{L^q_t} \\
& \les \| \mbf T_R f\|_{L^q_t} + R^{-1} \big\| \wt{\mbf T}_R f \big\|_{L^q_t} + \| \mbf T_R( t \cdot f) \|_{L^q_t} \\
& \les R^{\frac1p - \frac1q} \|\jb t  f\|_{L^p_t}.
\end{align*}

As for \eqref{bern2}, we use a dyadic decomposition as follows:
\begin{align}
\|  t \cdot  \mbf T_R f\|_{L^r_t} \le \sum_{R_0 \in 2^{\Z}} \| \mbf T_{R_0}( t \cdot  \mbf T_R f)\|_{L^r_t}.
\label{Ysg1501b}
\end{align}
For any fixed $R_0 \in 2^{\Z}$, we have
\begin{align}
\begin{split}
\F_t \big[ \mbf T_{R_0} (t \cdot \mbf T_R f) \big](\tau) & = \eta\Big( \frac{\tau}{R_0}\Big) \cdot i \partial_\tau \Big\{  \eta\Big( \frac{\tau}{R}\Big) \ft f(\tau) \Big\} \\
&  = i R^{-1} \eta\Big( \frac{\tau}{R_0}\Big) (\partial_\tau \eta)\Big(\frac{\tau}{R}\Big) \ft f (\tau) + \eta\Big( \frac{\tau}{R_0}\Big)  \eta\Big( \frac{\tau}{R}\Big) \big(i \partial_\tau \ft f \, \big)(\tau).
\end{split}
\label{Ysg1502}
\end{align}
Therefore, $\mbf T_{R_0} (t \cdot \mbf T_R f) \neq 0$ if and only if $R \sim R_0$. Hence, from \eqref{Ysg1501b}, \eqref{Ysg1502} and using the notation $\wt{\mbf T} _R$ as in the above, we have
\begin{align*}
R^s \| \jb t  \mbf T_R f\|_{L^r_t} & \le R^s \| \mbf T_R f \|_{L^r_t} + R^s \|  t \cdot  \mbf T_R f\|_{L^r_t} \\
& \les  \|f\|_{W^{s,r}_t} + \sum_{\substack{R_0 \in 2^{\Z} \\ R_0 \sim R}}  R_0^s \, \| \mbf T_{R_0}( t \cdot  \mbf T_R f)\|_{L^r_t} \\
& \les  \|f\|_{W^{s,r}_t} + \sum_{\substack{R_0 \in 2^{\Z} \\ R_0 \sim R}} R^{-1} R_0^s \, \big\| \mbf T_{R_0} \wt{\mbf T}_R f \big\|_{L^r_t} + R_0^s \, \| \mbf T_{R_0} \mbf T_R (t \cdot f) \|_{L^r_t} \\
& \les  \| f \|_{W_t^{s,r}} + \| t \cdot f \|_{W^{s,r}_t},
\end{align*} 
as claimed. This finishes the proof.
\end{proof}

\subsection{Bilinear estimates}
\label{SUBSEC:bilin}

In this section, we establish key bilinear estimates for products of the from $\ld(t) u v$, where $\ld$ is a smooth bump function of the time variable and $u$ and $v$ are space-time functions. Roughly speaking, our strategy is to estimate this product in different space-time regions: either $(\1)$ close or $(\II)$ far away from the light cone $\{ (\tau, n) \in \R \times \Z^2: |\tau| = |n| \}$, where $\tau$ and $n$ are respectively the output time and spatial frequencies of $\ld(t) u v$. Whilst the analysis of $(\1)$ and $(\II)$ are markedly different as we use different norms for these regions, the main goal in both cases is to place $u$ in a space of low integrability which satisfies the fractional chain rule and $v$ in a Fourier restriction norm space.

In the next result, we deal with product estimates in the space-time region close to the light cone ${(\1): \{(\tau, n) \in \R \times \Z^2 = |\tau| \sim |n|\}}$.
\begin{proposition}\label{PROP:prod1}
Let  $0 < \al < \frac{3\sqrt{241}-41}{244}$
and  
 $\Q^{\textup{hi,hi}}$, 
 $\mathcal{P}^{>}_{\g}$ and  $\mathcal{P}^{<}_{\g}$ be 
as  in \eqref{proj4}, \eqref{para1} and \eqref{para2}, respectively. Fix $\ld \in C^{\infty}_c(\R; \R)$.
Then, there exists $0<\g <1$ and small $\eps = \eps(\al) >0$ 
 such that, \noi
with  $\dl = \al + 10 \eps$, $\dl_1 = \al + 5 \eps$ and $\dl_2 = \al + 15\eps$, we have  
\begin{align}
 \|\Q^{\textup{hi,hi}}
\mathcal{P}^{>}_{\g}(\ld(t) u,v)\|_{L^{1}_t W_x^{\al + \frac12, 1}} & \les \|u\|_{\Ld^{\frac12+\dl_1,0}_{\frac{3}{2(1-\dl_1)}}}  \| v \|_{X^{\frac12-\dl, \frac12-\eps}}, \label{prod_est1} \\
 \|\A  \Q^{\textup{hi,hi}}\mathcal{P}^{<}_{\g}(\ld(t) u, v)\|_{Y_{\frac12+3\eps}^{\al,\frac12 + \eps}} 
& \les \big( \| u \|_{\Ld^{\frac12+\dl_1,0}_{\frac{3}{2(1-\dl_1)}}}  +\|u\|_{\Ld^{0,\frac12-\eps}_{\frac{3}{2+\dl_2}}} \big) \| v \|_{X^{\frac12-\dl, \frac12-\eps}} \label{prod_est2}
\end{align}
for any  $\A \in \{\Id, \C, \H\C\}$, 
where
 $\H$ and $\C$ are as in 
\eqref{ht} and \eqref{cone}, respectively. Here, the implicit constants may depend on the bump function $\ld$.
\end{proposition}

Before proceeding to the proof of Proposition \ref{PROP:prod1}, 
we first recall the following hyperbolic Leibniz rule;
see \cite[Subsection 4.2]{DFS2}.  See also \cite[Proof of Lemma 3.4]{KS}.

\begin{lemma}[hyperbolic Leibniz rule]\label{LEM:hyprule} 
Let $\tau, \tau_1, \tau_2 \in \R$ and $\xi, \xi_1, \xi_2 \in \R^2$ such that $\tau = \tau_1 + \tau_2$ and $\xi = \xi_1 + \xi_2$. Let $\pm_1$ and $\pm_2$ be the signs of $\tau_1$ and $\tau_2$. Then, we have

\smallskip

\noi
\begin{align*}
\big| |\tau| - |\xi|\big| \les \big|-\tau_1 \pm_1 |\xi_1| \big| + \big|-\tau_2 \pm_2 |\xi_2| \big| + b_{\pm_1, \pm_2}(\xi, \xi_1, \xi_2),
\end{align*}

\noi
where 

\noi
\begin{align*}
| b_{\pm_1, \pm_2}(\xi, \xi_1, \xi_2) | \les \min( |\xi_1|, |\xi_2|).
\end{align*}
\end{lemma}

In order to prove Proposition \ref{PROP:prod1}, we consider bounds of the form \eqref{prod_est1} and \eqref{prod_est2} in the two next lemmas. First, we find a range of $\g$ for which the estimate \eqref{prod_est1} holds (for fixed $\al$).

\begin{lemma}\label{LEM:prod1}
Let $0 < \al < \frac14$ and $\g >0$ such that
\begin{align}
\g > \frac{12\al}{1+14\al}.
\label{cond1}
\end{align}

\noi
Let 
 $\Q^{\textup{hi,hi}}$ and 
 $\mathcal{P}^{>}_{\g}$ be 
as  in \eqref{proj4} and \eqref{para1}, respectively. Fix $\ld \in C^{\infty}_c(\R; \R)$.
Then, given small  $ \eps = \eps(\al,\g) > 0 $ satisfying
\begin{align}
 \dl := \al + 10\eps <\frac14\qquad 
\text{and}\qquad 
\g > \frac{12 \al + 60 \eps}{1 + 14\al + 140 \eps},
\label{cond1a}
\end{align}

\noi 
we have 
\begin{align}
& \|\Q^{\textup{hi,hi}}
\mathcal{P}^{>}_{\g}(\ld(t) u,v)\|_{L^{1}_t W_x^{\al + \frac12, 1}} \les \|u\|_{\Ld^{\frac12+\dl_1,0}_{\frac{3}{2(1-\dl_1)}}}  \| v \|_{X^{\frac12-\dl, \frac12-\eps}},
\label{bd1}
\end{align}
where $\dl_1 = \al + 5\eps$.
\end{lemma}

\begin{proof} 
By a dyadic decomposition together with \eqref{proj4}, we have
\begin{align}
\begin{split}
\Q^{\textup{hi,hi}}\mathcal{P}_\g ^{>}(\tilde u,v)&  
= \sum_{\substack{(N,R) \in (2^\Z)^2 \\ R \sim N \ge 1}} 
\sum_{\substack{N_1, N_2 \in (2^\Z)^2 \\ N_1 \ge N_2^{\g}}} 
\P_N \mathbf{T}_R (\P_{N_1} \tilde u \cdot \P_{N_2} v), 
\label{z0}
\end{split}
\end{align}

\noi
where  $\tilde u = \ld(t) u$. 
Thus, 
it suffices to prove that there exists $\ta > 0$ such that  
\begin{align}
\begin{split}
 \| \P_N \mathbf{T}_R (\P_{N_1} \tilde u \cdot \P_{N_2} v)\|_{L^{1}_t W_x^{\al + \frac12, 1}}  
\les N_1^{-\ta} \|\tilde u \|_{\Ld^{\frac12+\dl_1,0}_{\frac{3}{2(1-\dl_1)}}}  \| v \|_{X^{\frac12-\dl, \frac12-\eps}}
\end{split}
\label{z1}
\end{align}

\noi
for any  $N_1, R_1, N_2, R_2, L_2 \ge 1$ with $N_1 \ge N_2^{\g}$ and 
$N \sim R \ges \max(L, 1)$, 
where $\dl$ and $\g$ satisfy~\eqref{cond1a}.
Indeed, 
 the conditions 
$N \sim R$
and  $N_1 \ge N_2^\g$,
the decaying factor  
$N_1^{-\ta}$
 in \eqref{z1} allows
us to sum over dyadic numbers
$N$, $R$, $N_1$ and $N_2$.
Hence, the desired bound \eqref{bd1} follows
from \eqref{z1}, \eqref{z0} and Lemma~\ref{LEM:restri} (i) (in order to remove the function $\ld(t)$ on the right-hand-side of \eqref{z1}).
\medskip

\noi
{\bf $\bul$ Case 1:} $N_2 \le N_1$.
\quad 
By the bound $N \les N_1$, H\"older's and 
 Sobolev's inequalities (with $\dl < \frac 14$) 
 and $(\dl, \dl_1) = (\al + 10\eps, \al + 5\eps)$, 
 we have 
\begin{align*}
  \| \P_N \mathbf{T}_R (\P_{N_1} \tilde u \cdot \P_{N_2} v)\|_{L^{1}_t W_x^{\al + \frac12, 1}} & \les N^{\al + \frac12}  
\| \P_{N_1} \tilde u \cdot \P_{N_2} v \|_{L^{1}_{t,x}} \\
& 
\les N_1^{\al + \frac12} N_1^{-\frac12-\dl}  \| \P_{N_1} \tilde u \|_{\Ld^{\frac12+\dl_1,0}_{\frac{3}{2(1-\dl_1)}}}  \| \P_{N_2}v \|_{L_{t,x}^{\frac{3}{1+2\dl}}}\\
& 
\les N_1^{\al -\dl}  \| \tilde u \|_{\Ld^{\frac12+\dl_1,0}_{\frac{3}{2(1-\dl_1)}}}  
\| v \|_{X^{\frac{1-4\dl}{3}, \frac{1-4\dl}{6}}}\\
& \les N_1^{-5\eps}\| \tilde u \|_{\Ld^{\frac12+\dl_1,0}_{\frac{3}{2(1-\dl_1)}}}  
\| v \|_{X^{\frac12-\dl, \frac12-\eps}},
\end{align*}

\noi
since 
 $\frac{1-4\dl}{3}\le  \frac12-\dl$ 
 and $\frac{1-4\dl}{6} \le \frac12-\eps$.

\medskip

\noi
{\bf $\bul$ Case 2:} $N_2^\g \le N_1 \le N_2$.
\quad Proceeding as in  Case 1 with $N \les N_2$, we have 
\begin{align*}
  \| \P_N \mathbf{T}_R (\P_{N_1} \tilde u \cdot \P_{N_2} v)\|_{L^{1}_t W_x^{\al + \frac12, 1}}
& 
 \les N^{\al+\frac12} 
 \|  \P_{N_1} \tilde u \cdot \P_{N_2} v \|_{L^1_{t,x}} \\
& 
  \les  N_2^{\al + \frac12}  N_1^{-\frac12-\dl_1} N_2^{-\frac12+\dl} \|\P_{N_1} \tilde u \|_{L^{\frac{3}{2(1-\dl_1)}}_t H_x^{\frac12+\dl_1}}  \|\P_{N_2} v \|_{L_{t}^{\frac{3}{1+2\dl}} H^{\frac12-\dl}_x}\\
& 
  \les  N_1^{-\frac12-\dl_1 + \frac{1-4\dl_1}{3}} N_2^{\al +\dl} \| \tilde u \|_{\Ld^{\frac12+\dl_1,0}_{\frac{3}{2(1-\dl_1)}}}  \| v \|_{X^{\frac12-\dl,\frac12-\eps}}\\
& 
 \les N_1^{-\ta} \|\tilde u \|_{\Ld^{\frac12+\dl_1,0}_{\frac{3}{2(1-\dl_1)}}}  \| v \|_{X^{\frac12-\dl,\frac12-\eps}}
\end{align*}

\noi
for some small $\ta > 0$, provided that 
\begin{align*}
\g > \frac{12 \al + 60 \eps}{1 + 14\al + 140 \eps}.
\end{align*}
This concludes the proof of Lemma \ref{LEM:prod1}.
\end{proof}

Next, we prove that bounds of the form \eqref{prod_est1} hold for some range of parameters $\g$ (for fixed $\al$).

\begin{lemma}\label{LEM:prod2}
Let $0 < \al < \frac14$ and $0 < \g < 1$ such that
\begin{align}
\g < \min\bigg(\frac{1-4\al}{3-6\al}, \frac{1-7\al}{4+2\al}\bigg).
\label{cond2}
\end{align}

\noi
Let 
 $\Q^{\textup{hi,hi}}$ and 
 $\mathcal{P}^{<}_{\g}$ be 
as  in \eqref{proj4} and \eqref{para2}, respectively. Fix $\ld \in C^{\infty}_c(\R;\R)$.
Then, given small  $ \eps = \eps(\al,\g) > 0 $ satisfying 
\begin{align}
 \dl := \al + 10\eps <\frac14\quad 
\text{and}\quad 
\g < \min\bigg(
\frac{1-4\al -20\eps}{3-6\al -54\eps}, 
\frac{1-7\al -58\eps}{4+2\al+20\eps}\bigg), 
\label{cond2a}
\end{align}

\noi 
we have 
\begin{align}
& \|\A  \Q^{\textup{hi,hi}}\mathcal{P}^{<}_{\g}(\ld(t) u, v)\|_{Y_{\frac12+3\eps}^{\al,\frac12 + \eps}} 
\les \big( \| u \|_{\Ld^{\frac12+\dl_1,0}_{\frac{3}{2(1-\dl_1)}}}  +\|u\|_{\Ld^{0,\frac12-\eps}_{\frac{3}{2+\dl_2}}} \big) \| v \|_{X^{\frac12-\dl, \frac12-\eps}},
\label{bdd2}
\end{align}
where $\dl_1 = \al + 5\eps$, $\dl_2 = \al + 15\eps$ and for any  $\A \in \{\Id, \C, \H\C\}$, 
where
 $\H$ and $\C$ are as in 
\eqref{ht} and \eqref{cone}, respectively.
\end{lemma}

\begin{remark}\rm
We note that bilinear estimates in $X^{s,b}$ type spaces (which are closely related to the norm on the left-hand-side of \eqref{bdd2}) have appeared in the literature on dispersive PDEs; see \cite{FK, KRT, DFS1, DFS2}. They however do not seem to be helpful in the context at hand, since they place both input functions $u$ and $v$ in $L^2$-based spaces, while we wish to put $u$ in a space of low integrability which satisfies a fractional chain rule.
\end{remark}

\begin{proof}
By a dyadic decomposition together with \eqref{proj3} and \eqref{proj4}, we have
\begin{align}
\begin{split}
\A \Q^{\textup{hi,hi}}\mathcal{P}_\g ^{<}(\tilde u,v)&  
= \sum_{\substack{(N,R,L) \in (2^\Z)^3\\ R \sim N \ge 1}} 
\sum_{(N_1, R_1) \in (2^\Z)^2} \sum_{\substack{(N_2, R_2, L_2) \in (2^\Z)^3\\ N_1 < N_2^\g}} \\
& \hphantom{XXXXXX}
  \A\M_{N, R,L}(\P_{N_1}\TT_{R_1} \tilde u \cdot \M_{N_2, R_2,L_2} v ), 
\label{z00}
\end{split}
\end{align}
where $\tilde u = \ld(t) u$. Note that by the triangle inequality with  $N \sim R$, we have $N \sim R \ges L$.
In the following, 
we only consider the contributions
to \eqref{z00}
from 
 $N_1, N_2, R_1,  R_2, L_2\ge1$.
 When $\min(N_1, N_2, R_1,  R_2, L_2) < 1$, 
 we first sum over dyadic $N_1, N_2, R_1, R_2, L_2 < 1$ (on the right-hand-side of \eqref{z00})
 and apply the argument presented below.
 (Namely, the homogeneous dyadic decompositions
 in $N_1, N_2, R_1, R_2, L_2 < 1$ are not needed.) We however, point out that 
the cases $L \le 1$ and $L>1$ need to be treated separately.

It follows from 
 \eqref{Y1}, 
\eqref{C1} and \eqref{C111b}
that 
\begin{align}
\| \A\M_{N, R,L} w
\|_{Y_{\frac12+3\eps}^{\al,\frac12 +\eps}}
= 
\| \A\C_{N,R,L}^{\frac 12 + \eps} w
\|_{Y_{\frac12+3\eps}^{\al,0}}.
\label{z0a}
\end{align}
We now claim that the desired bound \eqref{bdd2} follows if we prove 
that there exists  $\ta>0$ such that 
\begin{align}
\begin{split}
&  \big\|\A  \C^{\frac12+\eps}_{N,R,L}
  ( \P_{N_1} \TT_{R_1} \tilde u \cdot \M_{N_2, R_2,L_2} v ) \big\|_{Y_{\frac12+3\eps}^{\al,0}} \\
&\quad   \les 
\min(L, 1)^\ta N_2^{-\ta} L_2^{-\ta} 
 \big( \| \tilde u \|_{\Ld^{\frac12+\dl_1,0}_{\frac{3}{2(1-\dl_1)}}} 
 + \|\tilde u\|_{\Ld^{0,\frac12-\eps}_{\frac{3}{2+\dl_2}}}\big) \| v \|_{X^{\frac12-\dl, \frac12-\eps}}
\end{split}
\label{z3}
\end{align}

\noi
for any $N, R, N_1, R_1, N_2, R_2, L_2 \ge 1$ with $N_1 < N_2^{\g}$ and 
$N \sim R \ges \max(L, 1)$ and  $\dl$ and $\g$ satisfy~\eqref{cond2a}. Indeed, 
 the conditions 
$N \sim R \ges \max(L, 1)$
and  $N_1 < N_2^\g$, the triangle inequality
\begin{align}
R_2 \les N_2 + L_2,
\label{z1b} 
\end{align}

\noi
and the decaying factor  
$\min(L, 1)^\ta  N_1^{-\ta} L_2^{-\ta}$
 in \eqref{z3} allows
us to sum over dyadic numbers
$N$, $R$, $L$, $N_1$, $N_2$, $R_2$, and $L_2$.
As for the summation in $R_1$, we 
split it into two parts.
When 
$R_1 \les R_2$, 
we use a small negative power of $R_2$ 
to sum over $R_1$.
Otherwise, we have 
$R_2 \ll R_1 \sim R \sim N \les N_2$
and 
thus we use a small negative power of $N_2$ 
to sum over $R_1$.
Hence, the desired bound \eqref{bdd2} follows
from \eqref{z3}
under the assumption that 
 $N_1, N_2, R_1,  R_2, L_2\ge1$ and Lemma \ref{LEM:restri} (in order to remove the function $\ld(t)$ on the right-hand-side of \eqref{z3}).
 
 We note that the right-hand-side of \eqref{z3} is given by $L^p_{t,x}$-based norms (for $1<p<\infty$) of $\tilde u$ and $v$. Thus, in view of the boundedness of the temporal Hilbert transform $\mc H$ in~\eqref{ht}, we may assume that the temporal frequencies $\tau_1$ and $\tau_2$ of $\tilde u$ and $v$ are signed (such that Lemma~\ref{LEM:hyprule} is applicable).
 
 \medskip
 
 \noi
 {\bf $\bul$ Case 1: $L \ge 1$.}\quad Using Lemma \ref{LEM:wcone} with the parameters $(a,b)=(\frac12+3\eps, \frac12+\eps)$ and $\dl_{\circ} \ll 1$, the bound \eqref{Ysg140} simply reads
\begin{align}
\big\|\jb t^{\frac12+3\eps}  \C^{\frac12+\eps}_{N,R,L} w\big\|_{L^2_{t,x}} & \les L^{\frac12+\eps} \|\jb t^{\frac12+3\eps} w\|_{L^2_{t,x}}
\label{z4a}
\end{align}
in this case. Since $\g < 1$, we have $N_1 < N_2^\g \le  N_2$.
Then, 
by \eqref{z4a} and Lemma \ref{LEM:hyprule}
with $L \les L_1 + L_2 + N_1$ and $L_1 \le N_1 + R_1$, 
we have 
\begin{align}
\begin{split}
& \big\|\A\mathcal C^{\frac12+\eps}_{N, R,L}( \P_{N_1} \TT_{R_1} \tilde u \cdot \M_{N_2, R_2,L_2} v ) \big\|_{Y^{\al,0}_{\frac12+3\eps}}\\
& \quad \les N^\al L^{\frac12+\eps}  \big\| \jb t ^{\frac12+3\eps}(\P_{N_1} \TT_{R_1} \tilde u \cdot \M_{N_2, R_2,L_2} v)\big\|_{L^{2}_{t,x}}\\
& \quad \les N_2^{\al}  L_{\max}^{\frac12+\eps} \|\jb t (\P_{N_1} \TT_{R_1} \tilde u \cdot \M_{N_2, R_2,L_2} v)\|_{L^{2}_{t,x}}, 
\end{split}
\label{z4}
\end{align}

\noi
where $L_{\max}$ is given by 
\[L_{\max} = \max(N_1,R_1) + L_2.\] 


\medskip

\noi
{\bf $\bul$ Subcase 1.1:} $L_{\max} \les N_1$.
\quad  By H\"older's, Sobolev's and Bernstein's inequalities in space and time, Minkowski's inequality and \eqref{bern1} in Lemma \ref{LEM:bern},
with the conditions $N_1 < N_2^\g$, $R_1 \les N_1$ and $(\dl, \dl_1) = (\al + 10 \eps, \al +5\eps)$, we have
\begin{align}
\begin{split}
& \text{RHS of }\eqref{z4} 
\les N_2^{\al}  N_1^{\frac12+\eps} \|\P_{N_1} (\jb t \TT_{R_1} \tilde u)\|_{L^{\frac{2}{1-4\eps}}_t L^{\infty}_x} \|\M_{N_2, R_2,L_2}v\|_{L_t^{\frac{1}{2\eps}} L^2_x}  \\
& \quad \les N_1^{\frac{4(1-\dl)}{3} - \dl_1+ \eps} N_2^{\al + \dl - \frac12} 
  \| \jb t \TT_{R_1} \jb{\nb_x}^{\frac12+\dl_1} \tilde u\|_{L^{\frac{2}{1-4\eps}}_t L_x^{\frac{3}{2(1-\dl_1)}}} \|\M_{N_2, R_2,L_2}v\|_{X^{\frac12-\dl,\frac12-2\eps}}  \\
 & \quad \les N_1^{\frac{4-7\dl_1}{3} + \eps} N_2^{\al + \dl - \frac12} L_2^{-\eps}
  \| \jb t \TT_{R_1} \jb{\nb_x}^{\frac12+\dl_1} \tilde u\|_{L_x^{\frac{3}{2(1-\dl_1)}} L^{\frac{2}{1-4\eps}}_t } \|\M_{N_2, R_2,L_2}v\|_{X^{\frac12-\dl,\frac12-\eps}}  \\
& \quad \les N_1^{\frac{4-7\dl_1}{3} + \eps} N_2^{\al + \dl - \frac12} R_1^{\frac{1-4\dl_1 + 12\eps}{6}} L_2^{-\eps} \| \jb t \tilde u\|_{\Ld^{\frac12+\dl_1,0}_{\frac{3}{2(1-\dl_1)}}} \|v\|_{X^{\frac12-\dl,\frac12-\eps}} \\
& \quad \les  N_2^{\al + \dl - \frac12 + \g \frac{3-6\al -54\eps}{2} } L_2^{-\eps} \|\tilde u\|_{\Ld^{\frac12+\dl_1,0}_{\frac{3}{2(1-\dl_1)}}} \|v\|_{X^{\frac12-\dl,\frac12-\eps}}\\
& \quad \les N_2^{-\ta} L_2^{-\eps} \|\tilde u\|_{\Ld^{\frac12+\dl_1,0}_{\frac{3}{2(1-\dl_1)}}} \|v\|_{X^{\frac12-\dl,\frac12-\eps}}
\end{split}
\label{z4b}
\end{align}

\noi
for some small $\ta > 0$, provided that $\al < \frac12$ and
\begin{align}
\g <  
\frac{1-4\al -20\eps}{3-6\al -54\eps}
\label{cond2b}
\end{align}

\noi
and $\eps = \eps(\al, \g) > 0$ is sufficiently small. We used the condition $\dl_1 < \frac14$ in order to apply Minkowski's inequality in going from the second to the third line of \eqref{z4b} and took advantage of the time localization of $\tilde u$ to remove the weight $\jb t$ in the second to last line of \eqref{z4b}.

\medskip
\noi
{\bf $\bul$ Subcase 1.2:} $L_{\max} \les R_1$.
\quad  
Let $\wt{\mbf T}_R$ be defined as in $\mbf T_R$ in \eqref{proj2}, but with a symbol whose support is slightly larger than $\eta(\cdot/R)$ so that $\mc C^{\frac12+\eps}_{N,R,L} \wt{\mbf T}_R \equiv \mc C^{\frac12+\eps}_{N,R,L}$. Then, from \eqref{z4a}, \eqref{bern1} in Lemma~\ref{LEM:bern}, Lemma \ref{LEM:hyprule} and Bernstein's and Minkowski's inequalities, with $L \les L_{\max} \les R_1$ and $R \sim N \les N_2$, we have
\begin{align}
\begin{split}
& \big\|\A\mathcal C^{\frac12+\eps}_{N, R,L}( \P_{N_1} \TT_{R_1} \tilde u \cdot \M_{N_2, R_2,L_2} v ) \big\|_{Y^{\al,0}_{\frac12+3\eps}} \\
& \qquad = \big\|\A\mathcal C^{\frac12+\eps}_{N, R,L} \wt{\mbf T}_R ( \P_{N_1} \TT_{R_1} \tilde u \cdot \M_{N_2, R_2,L_2} v ) \big\|_{Y^{\al,0}_{\frac12+3\eps}}\\
& \qquad \les N^\al L^{\frac12+\eps}  \big\| \jb t ^{\frac12+3\eps} \wt{\mbf T}_R (\P_{N_1} \TT_{R_1} \tilde u \cdot \M_{N_2, R_2,L_2} v)\big\|_{L^2_x L^2_t}\\
& \qquad \les N_2^{\al}  R_1^{\frac12+\eps} R^{\frac{1+2\dl_2 + 24\eps}{6}} \|\jb t (\P_{N_1} \TT_{R_1} \tilde u \cdot \M_{N_2, R_2,L_2} v)\|_{L^{2}_x L^{\frac{3}{2+\dl_2 +12\eps}}_t} \\
& \qquad \les N_2^{\al + \frac{1+2\dl_2 + 24\eps}{6}}  R_1^{\frac12+\eps} \|\jb t (\P_{N_1} \TT_{R_1} \tilde u \cdot \M_{N_2, R_2,L_2} v)\|_{L^{\frac{3}{2+\dl_2 +12\eps}}_t L^2_x}.
\end{split}
\label{Ysgcond1}
\end{align}
Note that we have $R_1 \les \max(N_2, R_2)$. Indeed, otherwise, we would have $R \sim R_1 \gg \max(N_2, R_2) \ges N$, which is 
a contradiction to $N \sim R$ under the projector $\Q^{\textup{hi,hi}}$. Therfore, we have $R_1 \les \max(N_2, L_2)$ since $R_2 \les \max(N_2,L_2)$. Next, by H\"older's, Sobolev's and Bernstein's inequalities, \eqref{bern2} in Lemma~\ref{LEM:bern}, with $N_1 < N_2^{\g}$, $R_1^{2\eps} \les N_2^{2\eps} \cdot L_2^{2\eps}$ and $(\dl, \dl_2) = (\al + 10 \eps, \al + 15\eps)$, we have
\begin{align}
\begin{split}
& \text{RHS of }\eqref{Ysgcond1} 
\les N_2^{\al + \frac{1+2\dl_2 + 24\eps}{6}} R_1^{\frac12+\eps} \| \jb t \P_{N_1} \TT_{R_1} \tilde u\|_{L_t^{\frac{3}{2+\dl_2}}L^{\infty}_x} 
\|\M_{N_2, R_2,L_2}v\|_{L_t^{\frac{1}{4\eps}}L^{2}_x} \\
& \quad \les  N_1^{\frac{4 +2\dl_2}{3}} N_2^{\al  + \frac{1+2\dl_2 + 24\eps}{6} + \dl_2 - \frac12} 
L_2^{-3\eps} R_1^{\frac12+\eps}
\| \jb t \TT_{R_1} \tilde u\|_{L_x^{ \frac{3}{2+\dl_2}} L_t^{ \frac{3}{2+\dl_2}} } \|\M_{N_2, R_2,L_2}v\|_{X^{\frac12-\dl,\frac12-\eps}} \\
& \quad \les   N_2^{\g \frac{4+2\dl_2}{3} + \frac{-1 + 7\al +58\eps}{3}} 
L_2^{-\eps}
\big( \|\tilde u\|_{\Ld^{0,\frac12-\eps}_{\frac{3}{2+\dl_2}}} + \| t \cdot \tilde u\|_{\Ld^{0,\frac12-\eps}_{\frac{3}{2+\dl_2}}}\big) \|v\|_{X^{\frac12-\dl,\frac12-\eps}}\\
& \quad \les N_2^{-\ta} 
L_2^{-\eps} \|\tilde u\|_{\Ld^{0,\frac12-\eps}_{\frac{3}{2+\dl_2}}} \|v\|_{X^{\frac12-\dl,\frac12-\eps}}
\end{split}
\label{Ysgcond2}
\end{align}

\noi
for some small $\ta > 0$, 
 provided that 
\begin{align}
\g <  
\frac{1-7\al -58\eps}{4+2\al+20\eps}
\label{cond2c}
\end{align}

\noi
and $\eps = \eps(\al, \g) > 0$ is sufficiently small. Note that in \eqref{Ysgcond2}, we used \eqref{restri1} in Lemma \ref{LEM:restri} to bound
\begin{align}
\begin{split}
\| t \cdot \tilde u\|_{\Ld^{0,\frac12-\eps}_{\frac{3}{2+\dl_2}}} & = \| (t \cdot \tilde \ld (t)) \ld(t) u\|_{\Ld^{0,\frac12-\eps}_{\frac{3}{2+\dl_2}}} \\
& \les \| \ld(t) u\|_{\Ld^{0,\frac12-\eps}_{\frac{3}{2+\dl_2}}} \\
& \les  \| \tilde u\|_{\Ld^{0,\frac12-\eps}_{\frac{3}{2+\dl_2}}},
\end{split}
\label{Ysgcond2cc}
\end{align}
where $\tilde \ld \in C^{\infty}_c(\R; \R)$ equals $1$ on the support of $\ld$.

\medskip

\noi
{\bf $\bul$ Subcase 1.3:} $\max(N_1, R_1) \ll L_{\max} \sim L_2$.
\quad  
Note that we have $L_2 \les N_2$ in this case. 
Indeed, otherwise, 
we would have  $L_2 \sim R_2 \gg \max( N_2, R_1)$.
This would in turn imply  $R \sim R_2 \gg N_2  \sim N$, which is 
a contradiction to $N \sim R$ under the projector $\Q^{\textup{hi,hi}}$. Let $\wt{\mbf T}_{R_1}$ be defined as in $\mbf T_{R_1}$ in \eqref{proj2}, but with a symbol whose support is slightly larger than $\eta(\cdot/R_1)$ so that $\wt{\mbf T}_{R_1} \mbf T_{R_1}  \equiv \mbf T_{R_1}$.
By H\"older's and Sobolev's inequalities, \eqref{bern1} and \eqref{bern2} in Lemma \ref{LEM:bern} and \eqref{Ysgcond2cc}, with $R_1\ll L_2 \les N_2$, 
 $N_1 < N_2^{\g}$ and $(\dl, \dl_2) = (\al + 10 \eps, \al + 15 \eps)$, we have
\begin{align*}
& \text{RHS of }\eqref{z4} 
\les N_2^{\al}  L_2^{\frac12+\eps} \|\jb t \P_{N_1} \TT_{R_1} \tilde u\|_{L^{\infty}_{t,x}} \|\M_{N_2, R_2,L_2}v\|_{L_{t,x}^2} \notag \\
& \quad \les N_1^{\frac{4 +2\dl_2}{3}} N_2^{\al +\dl - \frac12 + 3\eps} L_2^{-\eps} 
\| \jb t \wt{\mbf T}_{R_1} \TT_{R_1} \tilde u\|_{L^{\infty}_t L^{\frac{3}{2+\dl_2}}_x} \|\M_{N_2, R_2,L_2}v\|_{X^{\frac12-\dl,\frac12-\eps}} \notag \\
& \quad \les N_1^{\frac{4 +2\dl_2}{3}} N_2^{\al +\dl - \frac12 + 3\eps} L_2^{-\eps} R_1^{\frac{1+2\dl_2+6\eps}{6}} \cdot R_1^{\frac12-\eps} \| \jb t  \TT_{R_1} \tilde u\|_{L^{\frac{3}{2+\dl_2}}_x L^{\frac{3}{2+\dl_2}}_t} \|v\|_{X^{\frac12-\dl,\frac12-\eps}} \\
& \quad \les N_2^{\g \frac{4 +2\dl_2}{3} + \frac{-1+7\al +52\eps}{3}} L_2^{-\eps}  \big( \|\tilde u\|_{\Ld^{0,\frac12-\eps}_{\frac{3}{2+\dl_2}}} + \| t \cdot \tilde u\|_{\Ld^{0,\frac12-\eps}_{\frac{3}{2+\dl_2}}}\big) \|v\|_{X^{\frac12-\dl,\frac12-\eps}} \notag \\
& \quad \les N_2^{-\ta} L_2^{-\eps}  \| \tilde u\|_{\Ld^{0,\frac12-\eps}_{\frac{3}{2+\dl_2}}} \|v\|_{X^{\frac12-\dl,\frac12-\eps}}
\end{align*}

\noi
for some small $\ta > 0$, 
 provided that
\begin{align}
\g <  
\frac{1-7\al -52\eps}{4+2\al+20\eps}, 
\label{cond2d}
\end{align}

\medskip

\noi
and $\eps = \eps(\al, \g) > 0$ is sufficiently small.

 \medskip
 
 \noi
 {\bf $\bul$ Case 2: $N_2^{-100} \le  L < 1$.}\quad Using Lemma \ref{LEM:wcone} with the parameters $(a,b, \dl_{\circ})=(\frac12+3\eps, \frac12+\eps, 10\eps)$, the bound \eqref{Ysg140} reads
 \begin{align}
\big\|\jb t^{\frac12+3\eps} \, \C^b_{N,R,L} w\big\|_{L^2_{t,x}} & \les L^{-20\eps} \|\jb t^{\frac12+3\eps} w\|_{L^2_{t,x}} \les N_2^{200\eps} \|\jb t^{\frac12+3\eps} w\|_{L^2_{t,x}}.
\label{Ysg200}
\end{align}
Therefore, since \eqref{Ysg200} essentially corresponds to \eqref{z4a} but without the factor $L^{\frac12+\eps}$, by arguing as in Case 1 (or \eqref{Ysgcond1} without the factor $R_1^{\frac12+\eps}$), we have 
\begin{align*}
&  \big\|\A\mathcal C^{\frac12+\eps}_{N, R,L}( \P_{N_1} \TT_{R_1} \tilde u \cdot \M_{N_2, R_2,L_2} v ) \big\|_{Y^{\al,0}_{\frac12+3\eps}} \\
& \qquad \quad \les N_2^{-\ta} L^{\ta} L_2^{-\eps} \big( \|\tilde u\|_{\Ld^{\frac12+\dl_1,0}_{\frac{3}{2(1-\dl_1)}}} + \| \tilde u\|_{\Ld^{0,\frac12-\eps}_{\frac{3}{2+\dl_2}}} \big)\|v\|_{X^{\frac12-\dl,\frac12-\eps}}.
\end{align*}

\noi
for some small $\ta > 0$, provided that $\g$ satisfies the conditions \eqref{cond2b}, \eqref{cond2c}, \eqref{cond2d}and $\eps = \eps(\al, \g) > 0$ is sufficiently small.

  \medskip
 
 \noi
 {\bf $\bul$ Case 3: $L < N_2^{-100} $.}\quad Using Lemma \ref{LEM:wcone} with the parameters $(a,b, \dl_{\circ})=(\frac12+3\eps, \frac12+\eps, 10\eps)$, the bound \eqref{Ysg140} reads
 \begin{align}
\begin{split}
\big\|\jb t^{\frac12+3\eps} \C^b_{N,R,L} w\big\|_{L^2_{t,x}} & \les L^{\eps} \|\jb t^{\frac12+3\eps} w\|_{L^2_{t,x}} + L^{-20\eps}\big\|\F_{t,x}^{-1}[\ind_{||\tau|-|n|| \les L^{1-10\eps }} \, \ft w (\tau, n)]\big\|_{L^2_{t,x}}.
\end{split}
\label{Ysg201}
\end{align}
By Plancherel's identity and the Hausdorff-Young inequality, we also have
\begin{align}
\begin{split}
\big\|\F_{t,x}^{-1}[\ind_{||\tau|-|n|| \les L^{1-10\eps }} \, \ft w (\tau, n)]\big\|_{L^2_{t,x}}  & =  \|\ind_{||\tau|-|n|| \les L^{1-10\eps}} \, \ft w (\tau, n)\|_{\l^2_{n} L^2_\tau} \\
& \les \big\| \|\ind_{||\tau|-|n|| \les L^{1-10\eps}}\|_{L^2_\tau} \|\ft w (\tau, n)\|_{L^{\infty}_\tau}\big\|_{\l^2_n} \\
& \les L^{\frac12-5\eps} \| \F_x [w](t,n)\|_{\l^2_n L^1_t} \\
& \les L^{\frac12-5\eps} \| w\|_{L^1_t L^2_x}.
\end{split}
\label{Ysg202}
\end{align}
Combining \eqref{Ysg201} and \eqref{Ysg202}, with $L < N_2^{-100}$, yields
\begin{align}
\big\|\jb t^{\frac12+3\eps} \C^b_{N,R,L} w\big\|_{L^2_{t,x}} & \les L^\eps \|\jb t^{\frac12+3\eps} w\|_{L^2_{t,x}} + L^\eps N_2^{-40} \| w\|_{L^1_t L^2_x}.
\label{Ysg203}
\end{align}
Therefore, since \eqref{Ysg203} is a much better estimate than \eqref{z4a} or \eqref{Ysgcond1}, by arguing as in Case 1, we have 
\begin{align*}
&  \big\|\A\mathcal C^{\frac12+\eps}_{N, R,L}( \P_{N_1} \TT_{R_1} \tilde u \cdot \M_{N_2, R_2,L_2} v ) \big\|_{Y^{\al,0}_{\frac12+3\eps}} \\
& \qquad \quad \les N_2^{-\ta} L^\ta L_2^{-\eps} \big( \|\tilde u\|_{\Ld^{\frac12+\dl_1,0}_{\frac{3}{2(1-\dl_1)}}} + \| \tilde u\|_{\Ld^{0,\frac12-\eps}_{\frac{3}{2+\dl_2}}} \big)\|v\|_{X^{\frac12-\dl,\frac12-\eps}}.
\end{align*}

\noi
for some small $\ta > 0$, provided that $\g$ satisfies the conditions \eqref{cond2b}, \eqref{cond2c}, \eqref{cond2d}and $\eps = \eps(\al, \g) > 0$ is sufficiently small.

Putting Cases 1, 2, and 3 together, 
we obtain the desired bound \eqref{z3}, 
provided that~\eqref{cond2b}, 
\eqref{cond2c} and \eqref{cond2d}
are satisfied. 
\end{proof}

We now present the proof of Proposition \ref{PROP:prod1}. The proof basically consists of working out values the set of values $\al >0$ for which the two ranges of parameters $\g$ obtained in Lemmas \ref{LEM:prod1} and \ref{LEM:prod2} have a non-empty intersection.

\begin{proof}[Proof of Proposition \ref{PROP:prod1}]

We first note that 
\begin{align}
\frac{1 - 7\al}{4+2\al} 
< \frac{1-4\al}{3-6\al} \ \ \text{and} \ \ \al > 0
\quad \LLRA
\quad 0 < \al < \frac {13 + 3\sqrt{41} }{100} \approx 0.3221, 
\label{cond3}
\end{align}
%

\noi
Putting together the restrictions on $\al$
in Lemmas~\ref{LEM:prod1} and~\ref{LEM:prod2}
with \eqref{cond3}, we have 
\begin{align}
\frac{12\al}{1+14\al} < \frac{1 - 7\al}{4+2\al} \ \ \text{and} \ \ 0< \al <\frac14
\quad \LLRA
\quad 0 < \al < \frac{3\sqrt{241}-41}{244} \approx 0.0228. 
\label{cond3}
\end{align}

%
%
%
\noi
Hence,   
 for $0 < \al < \frac{3\sqrt{241}-41}{244}$, there exists $0 < \g <1$ such that 
both the conditions \eqref{cond1} and \eqref{cond2} hold, 
and thus 
  Proposition \ref{PROP:prod1} follows 
as a direct consequence of Lemmas~\ref{LEM:prod1} and~\ref{LEM:prod2}.
\end{proof}

In the next two propositions, we consider product estimates in the space-time region ${(\II) : \{ (\tau, n) \in \R \times \Z^2 = |\tau| \gg |n| \text{ or } |\tau| \ll |n|\}}$. First, we consider the contribution of the region $\{ |\tau| \gg |n|\}$ to $(\II)$.

\begin{proposition}\label{PROP:prod2}
Let $0< \al < \dl < \frac{1}{16}$. Fix $\ld \in C_c^{\infty}(\R; \R)$.
Then, there exists small $\eps_0 = \eps_0(\al, \dl) > 0$ such that 
\begin{align*}
\|\Q^{\textup{hi,lo}}( \ld(t) uv)\|_{\Ld^{\al+\frac12,0}_{1+\eps}} 
\les \Big( \|u\|_{\Ld^{\frac12+\dl,0}_{\frac{3}{2(1-\dl)}}} +  \|u\|_{\Ld^{0, \frac12-\eps}_{\frac{3}{2+\dl}}}+\|u\|_{L^2_{t,x}}\Big)\|v\|_{X^{\frac12-\dl,\frac12-\eps}}
\end{align*}

\noi
for any $0 < \eps < \eps_0$, 
where 
 $\Q^{\textup{hi,lo}}$ is as in \eqref{proj4}. Here, the implicit constant may depend on the bump function $\ld$.
\end{proposition}

\noi
\begin{proof} 
By a dyadic decomposition and Lemma~\ref{LEM:restri} (i), 
it suffices to show that there exists  $\ta > 0$ such that 
\begin{align}
\begin{split}
& \|\P_N \Q^{\textup{hi,lo}} (\P_{N_1} \tilde u \cdot \P_{N_2}v)\|_{\Ld^{\frac12+\al,0}_{1+\eps}} \\
& \quad  
\les \max(N_1,N_2)^{-\ta} 
\Big( \|\tilde u\|_{\Ld^{\frac12+\dl,0}_{\frac{3}{2(1-\dl)}}} +  \|\tilde u\|_{\Ld^{0, \frac12-\eps}_{\frac{3}{2+\dl}}}+\|\tilde u\|_{L^2_{t,x}}\Big)\|v\|_{X^{\frac12-\dl,\frac12-\eps}}
\end{split}
\label{u1}
\end{align}

\noi
for any dyadic  $N, N_1, N_2 \ge 1$ and where $\tilde u = \ld(t) u$.
with $N \ll R$.

\medskip

\noi
{\bf $\bul$ Case 1:} $N_1 \ges N_2$. 
\quad In this case, we have $N_1 \ges N$.
Then, by the boundedness of $\P_N$ and $\Q^{\textup{hi,lo}}$ and H\"older's and Sobolev's inequalities, we have
\begin{align*}
\text{LHS of }\eqref{u1} 
& \les N_1^{\frac12+\al} \|\P_{N_1}\tilde u\|_{L_{t,x}^{\frac{3}{2(1-\dl)}}}
\|\P_{N_2}v\|_{L_{t,x}^{\frac{3(1+\eps)}{1+2\dl - \eps(2- 2\dl)}}} \notag \\
& \les N_1^{\al-\dl} \| \tilde u\|_{\Ld^{\frac12+\dl,0}_{\frac{3}{2(1-\dl)}}}\|v\|_{X^{\frac12-\dl,\frac12-\eps}},
\end{align*}

\noi
yielding \eqref{u1}, 
provided that  $\al < \dl$, $\frac{1-4\dl}{3}  < \frac12-\dl$ 
and $\eps = \eps(\dl)>0$ is sufficiently small.

\medskip

\noi
{\bf $\bul$ Case 2:} $N_1 \ll N_2$.
\quad In this case, we have  $N \sim N_2$. 
By a further dyadic decomposition,  it suffices to show that there exists small $\ta > 0$
such that\footnote{Here, 
we only consider $R_1, R_2, L_2 \ge 1$ for simplicity.
The other cases can be handled in a similar manner;
see the proof of Lemma \ref{LEM:prod1}.
}

\begin{align}
\begin{split}
& N^{\frac12+\al} \|\P_N \TT_R (\P_{N_1} \TT_{R_1} \tilde u \cdot \M_{N_2,R_2,L_2}v)\|_{L_{t,x}^{1+\eps}}  \\
& \quad 
 \les N_2^{-\ta} \max(R_1,R_2)^{-\ta}
 \Big( 
  \|\tilde u\|_{\Ld^{0, \frac12-\eps}_{\frac{3}{2+\dl}}}+\|u\|_{L^2_{t,x}}\Big)\|v\|_{X^{\frac12-\dl,\frac12-\eps}}, 
\end{split} 
 \label{u10}
\end{align}

\noi
for any dyadic  $N,R, N_1, R_1, N_2, R_2, L_2 \ge 1$ such that ${N_1 \ll N_2 \sim N}$ and $N \ll R$. 

\medskip

\noi
{\bf $\bul$ Subcase 2.1:} $R_1 \ges R_2$. 
\quad In this case, we have $N_2 \sim N \ll R \les R_1$.
Then, by  H\"older's and Sobolev's inequalities, we have
\begin{align*}
\text{LHS of }\eqref{u10} 
& \les N_2^{\al-\dl + 2\eps} R_1^{\frac12-2\eps} \|\P_{N_1}\TT_{R_1} \tilde u\|_{L_{t,x}^{\frac{3}{2+\dl}}}\| v\|_{L_t^{\frac{3(1+\eps)}{1-\dl-\eps(2+\dl)}} W_x^{\dl, \frac{3(1+\eps)}{1-\dl-\eps(2+\dl)} }} \\
& \les N_2^{\al-\dl +2\eps} R_1^{-\eps} \|\tilde u\|_{\Ld^{0, \frac12-\eps}_{\frac{3}{2+\dl}}}\|v\|_{X^{\frac12-\dl,\frac12-\eps}}, \end{align*}

\noi
yielding \eqref{u10}, 
provided that 
 $\al < \dl$, $\frac{1+2\dl}{3} + \dl   < \frac12-\dl$ (namely, $\dl < \frac1 {16}$)
and $\eps = \eps(\dl)>0$ is sufficiently small.

\medskip

\noi
{\bf $\bul$ Subcase 2.2:} $R_1 \ll R_2$. 
\quad In this case, we have $N_2 \sim N \ll R \sim R_2$
and thus $N_2 \ll R_2 \sim L_2$.
 Thus, by the boundedness of $\TT_R$ and $\P_N$ and H\"older's and Sobolev's inequalities, we have
\begin{align*}
\text{LHS of }\eqref{u10} 
& \les N_2^{\al+ \dl- \frac 12 + 13\eps}R_2^{-\eps} L_2^{\frac12 - 10\eps} \| \tilde u\|_{L^{2}_{t,x}}
\|\M_{N_2,R_2,L_2}v\|_{L^{\frac{2(1+\eps)}{1-\eps}}_{t}W^{\frac 12 - \dl - 2\eps, \frac{2(1+\eps)}{1-\eps}}_{x}} \\
& \les N_2^{\al+ \dl- \frac 12 + 13\eps}R_2^{-\eps} \|\tilde u\|_{L^2_{t,x}}\|v\|_{X^{\frac 12 - \dl, \frac12-\eps}},
\end{align*}

\noi
yielding \eqref{u10}, 
provided that 
$\al+ \dl <  \frac 12 $
and 
 $\eps = \eps(\al, \dl)>0$ is sufficiently small.

Note that the restriction $\dl < \frac 1{16}$
comes from Subcase 2.1.
This concludes the proof of Proposition \ref{PROP:prod2}.
\end{proof}

Lastly, we consider the contribution of the region $\{ |\tau| \ll |n| \}$ to $(\II)$.

\begin{proposition}\label{PROP:prod3}
Let $0< \al < \dl \le \frac{1}{16}$. Fix $\ld \in C_c^{\infty}(\R; \R)$.
Then, there exists small $\eps_0 = \eps_0(\dl) > 0$ such that 
\begin{align}
\|\Q^{\textup{lo,hi}}(\ld(t) uv)\|_{\Ld^{\al, \frac12-2\eps}_{1}} 
\les \Big( \|u\|_{\Ld^{\frac12+\dl,0}_{\frac{3}{2(1-\dl)}}} +  \|u\|_{\Ld^{0, \frac12-\eps}_{\frac{3}{2+\dl}}}+\|u\|_{L^2_{t,x}}\Big)\|v\|_{X^{\frac12-\dl,\frac12-\eps}}
\label{prod3}
\end{align}

\noi
for any $0 < \eps < \eps_0$, 
where
$\Q^{\textup{lo,hi}}$ is as in \eqref{proj4}. Here, the implicit constant may depend on the bump function $\ld$.

\noi
\end{proposition}

\noi
\begin{proof} 
By a dyadic decomposition as in the proof of Lemma \ref{LEM:prod1} and Lemma~\ref{LEM:restri} (i), 
it suffices to prove 
that there exists $\ta > 0$ such that 
\begin{align}
\begin{split}
& \|\P_N \TT_R (\P_{N_1} \TT_{R_1} \tilde u \cdot \M_{N_2,R_2,L_2}v)\|_{\Ld^{\al, \frac12-2\eps}_{1}} \\
&  \quad \les N^\al R^{\frac 12 - 2\eps}\|\P_N \TT_R (\P_{N_1} \TT_{R_1} \tilde u \cdot \M_{N_2,R_2,L_2}v)
\|_{L^1_{t, x}} \\
&  \quad \les \max(N_1,N_2)^{-\ta} R_2^{-\ta} 
\Big( \|\tilde u\|_{\Ld^{\frac12+\dl,0}_{\frac{3}{2(1-\dl)}}} +  \|\tilde u\|_{\Ld^{0, \frac12-\eps}_{\frac{3}{2+\dl}}}+\|\tilde u\|_{L^2_{t,x}}\Big)\|v\|_{X^{\frac12-\dl,\frac12-\eps}}
\end{split}
\label{w1}
\end{align}

\noi
for any  dyadic $N,R, N_1, R_1, N_2, R_2, L_2 \ge 1$ 
such that $R \ll N$ and where $\tilde u = \ld(t) u$. 
Owing to the decaying factor $\max(N_1,N_2)^{-\ta} R_2^{-\ta}$, 
we can sum over dyadic $N,R, N_1,N_2,R_2, L_2 \ge 1$. 
If $R_1 \gg \max(N_1,N_2, R_2)$, then we would have $R \sim R_1 \gg N$,
leading to a contradiction.
Hence, we have 
$R, R_1 \les \max(N_1,N_2, R_2)$, allowing us to also sum over dyadic $R_1 \ge 1$.

\medskip

\noi
{\bf $\bul$ Case 1:} $N_1 \ges N_2$.
\quad
In this case, we have $R \ll N \les N_1$
and $R_2 \les N_1 + L_2$.
Then, 
by  H\"older's and Sobolev's inequalities, we have
\begin{align*}
\text{LHS of } \eqref{w1} 
& \les N_1^{\al+\frac12-2\eps} \|\P_{N_1}\TT_{R_1}\tilde u\|_{L_{t,x}^{\frac{3}{2(1-\dl)}}}\|\M_{N_2,R_2,L_2}v\|_{L_{t,x}^{\frac{3}{1+2\dl}}} \notag \\
& \les N_1^{\al-\dl-\eps} R_2^{-\eps} \|\tilde u\|_{\Ld^{\frac12+\dl,0}_{\frac{3}{2(1-\dl)}}}\|v\|_{X^{\frac12-\dl,\frac12-\eps}},
\end{align*}

\noi
yielding \eqref{w1}, 
provided that 
 $\al \le \dl$, $\frac{1-4\dl}{3}    \le \frac12-\dl$, 
and $\eps = \eps(\dl)>0$ is sufficiently small.

\medskip

\noi
{\bf $\bul$ Case 2:} $N_1 \ll N_2$ and $R_1 \ges R_2$. 
\quad 
In this case, we have $R_1 \ges R$
and thus 
\begin{align*}
\text{LHS of } \eqref{w1} 
 & \les N_2^{\al-\dl} R_1^{\frac12 - 2\eps} \|\P_{N_1}\TT_{R_1}\tilde u\|_{L_{t,x}^{\frac{3}{2+\dl}}}\|\M_{N_2,R_2, L_2}v\|_{L_t^{\frac{3}{1-\dl}} W_x^{\dl, \frac{3}{1-\dl}}} \notag \\
& \les N_2^{\al-\dl} R_1^{-\eps} \|\tilde u\|_{\Ld^{0, \frac12-\eps}_{\frac{3}{2+\dl}}}\|v\|_{X^{\frac12-\dl,\eps}}, \end{align*}

\noi
yielding \eqref{w1}, 
provided that 
 $\al <  \dl$, $\frac{1+2\dl}{3} + \dl \le \frac12-\dl$ 
 (namely, $\dl \le \frac1 {16}$), 
and $\eps = \eps(\dl)>0$ is sufficiently small.

\medskip

\noi
{\bf $\bul$ Case $3$:} $N_1 \ll N_2$ and $R_1 \ll R_2$.
\quad  
In this case, we have $N_2 \sim N \gg R \sim R_2$
and thus $N_2 \sim L_2 \gg R_2$. 
Then, by  H\"older's and Sobolev's inequalities, we have
\begin{align*}
\text{LHS of } \eqref{w1} 
 & \les N_2^{\al}R_2^{-\eps} L_2^{\frac12 - \eps} \|\tilde u\|_{L^{2}_{t,x}}\|\M_{N_2,R_2,L_2}v\|_{L^2_{t,x}} \notag \\
& \les N_2^{\al-\dl}R_2^{-\eps} \|\tilde u\|_{L^2_{t,x}}\|v\|_{X^{\dl,\frac12-\eps}}
\end{align*}

\noi
yielding \eqref{w1}, 
provided that 
 $\al <  \dl$ and  $\dl \le \frac 12 - \dl$.
 \end{proof}

\section{Stochastic objects}\label{SEC:4}

\subsection{Gibbs measure}\label{SUBSEC:4-0}

Here, we state the result on the construction of the Gibbs measure $\rhoo$ \eqref{Gibbs10} proved in \cite{ORSW2}.

\begin{lemma}\label{LEM:Gibbs}
Let $0<\be^2<4\pi$. 

\smallskip

\noi
\textup{(i)} The truncated renormalized density $\{R_N\}_{N\in\N}$ in \eqref{RN} 
is a Cauchy sequence in $L^p(\mu_1)$ for any finite $p\ge 1$, thus converging to some limiting random variable $R\in L^p(\mu_1)$.

\smallskip

\noi
\textup{(ii)} 
Given any finite $ p \ge 1$, 
there exists $C_p > 0$ such that 
\begin{equation}
\sup_{N\in \N} \Big\| e^{R_N(u)}\Big\|_{L^p(\mu_1)}
\leq C_p  < \infty.
\label{exp1}
\end{equation}

\noi
Moreover, we have
\begin{equation}\label{exp2}
\lim_{N\rightarrow\infty}e^{ R_N(u)}=e^{R(u)}
\qquad \text{in } L^p(\mu_1).
\end{equation}

\noi
As a consequence, 
the truncated renormalized Gibbs measure $\rhoo_N$ in \eqref{GibbsN} converges, in the sense of \eqref{exp2}, 
 to the renormalized Gibbs measure $\rhoo$ given by
\begin{align}\label{Gibbs3}
d\rhoo(u,v)= \ZZ^{-1} e^{R(u)}d\muu_1(u, v).
\end{align}

\noi
Furthermore, 
the resulting Gibbs measure $\rhoo$ is equivalent 
to the Gaussian measure $\muu_1$.
\end{lemma}

Then, a standard argument shows the invariance of the measure $\rhoo_N$ under the flow of \eqref{RSdSGN}; see for instance \cite[Subsection 5.2]{ORW} for details in the context of the hyperbolic Liouville model.

\begin{lemma}\label{LEM:invariance}
Fix $N \in \N$ and $\be \in \R$ with $0 < \be^2 < 4\pi$. The truncated sine-Gordon measure $\rhoo_N$ in \eqref{GibbsN} is invariant under the truncated dynamics \eqref{RSdSGN}.
\end{lemma}

\subsection{Stochastic convolution
and its space-time covariance}\label{SUBSEC:3-1} 

In this subsection, 
we study basic properties of the stochastic convolution $\Psi^{\text{wave}}$
defined in \eqref{Psi_S}. In particular, we establish 
sharp bounds on the space-time covariance of the $\Psi_N^{\text{wave}}$ in \eqref{Psi_trunc2} and its spatial derivatives that are uniform in the smoothing parameter $N$; see Propositions~\ref{PROP:cov} and~\ref{PROP:cov2} below.

The following lemma provides the (uniform in $N$) regularity properties for $\Psi^{\text{wave}}_N$; see \cite{GKOT} for a proof of (i).

 \begin{lemma}\label{LEM:psi}
Fix any $0 \le T \le 1$, $\eps>0$ and finite $p\ge 1$, 
$ \{\Psi^{\textup{wave}}_N\}_{N\in\N}$ is a Cauchy sequence in $L^p(\O;C([0,T]; W^{-\eps}(\T^2)))$, 
thus converging to some limiting process  $\Psi^{\textup{wave}} \in L^p(\O;C([0,T];W^{-\eps}(\T^2)))$. 
Moreover, $\{ \Psi^{\textup{wave}}_N \}_{N \in \N}$ converges almost surely to the same limit $\Psi^{\textup{wave}}$ in $C([0,T];W^{-\eps}(\T^2))$.
 \end{lemma}
 
 Next, we study the difference of the stochastic convolutions \eqref{Psi} and \eqref{Psi_S}.
 
 \begin{lemma}\label{LEM:diff_psi}
Fix $0 < s < 1$, $0 < b < \frac12$ and  $0 < T \le 1$. Let $Z_{\infty}^{s,b} = \Ld_{\infty}^{s, 0} \cap \Ld^{0,b}_{\infty}$. Then, $\{\Psi^{\textup{KG}}_N - \Psi^{\textup{wave}}_N\}_{N \in \N}$ is a Cauchy sequence in $L^p( \O;Z_{\infty}^{s, b}([0,T]))$, 
thus converging to some limiting process  ${\Psi^{\textup{KG}} - \Psi^{\textup{wave}} \in L^p( \O;Z_{\infty}^{s, b}([0,T]))}$. 
Moreover, $\{\Psi^{\textup{KG}}_N - \Psi^{\textup{wave}}_N\}_{N \in \N}$ converges almost surely to the same limit $\Psi^{\textup{KG}} - \Psi^{\textup{wave}}$ in $Z_{\infty}^{s, b}([0,T])$.
 \end{lemma}
 
\begin{proof}
Fix $(N,N_1,N_2) \in \N^3$ with $N_2 \ge N_1$, $(t,t_1, t_2) \in [0,1]^2$ and set $\Psi_M = \ind_{[0,1]}\big(\Psi^{\textup{KG}}_M -\Psi^{\textup{wave}}_M\big)$ for each $M \in \N$. Our goal is to prove the following bounds:
\begin{align}
\E_{\muu_1 \otimes \PP} \Big[ \big|\ft{\Psi_N}(t,n) \big|^2\Big] & \les \jb n^{-4}, \label{Ysg500} \\
\E_{\muu_1 \otimes \PP} \Big[ \big|\ft{\Psi_N}(t_1,n) - \ft{\Psi_N}(t_2,n)  \big|^2\Big] & \les |t_1 - t_2| \jb n^{-2}, \label{Ysg501}\\
\E_{\muu_1 \otimes \PP} \Big[ \big|\ft{\Psi_{N_1}}(t,n) - \ft{\Psi_{N_2}}(t,n)  \big|^2\Big]&  \les N_1^{-\ta} \jb{n}^{-4+\ta}, \label{Ysg502}
\end{align}
for any small constant $\ta >0$. Here, the implicit constants are uniform in the parameters $N,N_1,N_2, t , t_1, t_2$.

We start with the proof of \eqref{Ysg500}. By \eqref{D}, \eqref{Psi}, \eqref{Psi_S}, \eqref{S}, we have
\begin{align}
\ft{\Psi_N}(t,n) = \big(\1(t,n)  + \II(t,n)  + \III(t,n)\big) \chi_N(n),
\label{Ysg503}
\end{align}
where 
\begin{align*}
\1(t,n) & =\bigg(\cos(t \fbb n) - \cos(t|n|) + \frac{\sin(t \fbb n)}{2\fbb n} - \frac{\sin(t |n|)}{2|n|} \bigg) e^{-\frac t2} \frac{g_n}{\jb n}, \\
\II(t,n) & = \bigg(\frac{\sin(t \fbb n)}{\fbb n} - \frac{\sin(t |n|)}{|n|}\bigg) e^{-\frac t2} h_n, \\
\III(t,n) & = \sqrt 2 \int_0^t e^{-\frac{t-t'}{2}} \bigg(\frac{\sin((t-t') \fbb n)}{\fbb n} - \frac{\sin((t-t) |n|)}{|n|}\bigg) dB_n(t').
\end{align*}
The bound \eqref{Ysg500} is then a consequence of \eqref{Ysg503}, the formulas above, the independence of $g_n$, $h_n$ and $B_n$, Ito's isometry, the mean value theorem and the bounds
\begin{align} 
\begin{split}
\big| \fbb n - |n| \big| & \les \jb n ^{-1}, \\
\bigg|\frac{1}{\fbb n} - \frac{1}{|n|}\bigg| & \les \jb n ^{-2}
\end{split}
\label{Ysg504}
\end{align}
for $n \in \Z^2 \setminus \{0\}$. The estimate \eqref{Ysg502} then follows from \eqref{Ysg500} and observing that the left-hand-side of \eqref{Ysg502} is non-zero if and only if $N_1 \les \jb n \les N_2$.

We turn our attention to \eqref{Ysg501}. From using the mean value theorem twice, we get
\begin{align}
\begin{split}
& \big| \big(\cos(t_1 \fbb n) - \cos(t_1 |n|)\big) - \big(\cos(t_2 \fbb n) - \cos(t_2 |n|)\big)\big| \\
& \qquad = \bigg| - t_1 \big( \fbb n - |n| \big) \int_0^1 \sin\big((1-h) t_1 \fbb n + h t_1 |n|\big) dh \\
& \qquad \qquad \qquad  + t_2 \big( \fbb n - |n| \big) \int_0^1 \sin\big((1-h) t_2 \fbb n + h t_2 |n|\big) dh \bigg| \\
& \qquad \les |t_1 - t_2|.
\end{split}
\label{Ysg505}
\end{align}
Similarly, we also have
\begin{align}
 \big| \big(\sin(t_1 \fbb n) - \sin(t_1 |n|)\big) - \big(\sin(t_2 \fbb n) - \sin(t_2 |n|)\big)\big| \les |t_1 - t_2|. \label{Ysg506}
\end{align}
Therefore, \eqref{Ysg501} follows from \eqref{Ysg503}, the formulas for $\1$, $\II$ and $\III$, \eqref{Ysg504} together with the mean value theorem and \eqref{Ysg505}-\eqref{Ysg506}.

By interpolating \eqref{Ysg500}, \eqref{Ysg501} and \eqref{Ysg502}, hypercontractivity (see \cite[Theorem I.22]{Simon}) and the Kolmogorov continuity criterion (see \cite[Theorem 8.2]{Bass}) together with standard arguments,\footnote{See for instance \cite[Proposition 5]{MWX} for a detailed proof of a very similar argument.} we deduce that $\{\Psi^{\textup{KG}}_N - \Psi^{\textup{wave}}_N\}_{N \in \N}$ is a Cauchy sequence in $L^p( \O; C_t^{b} L^\infty_x \cap C_t W^{s, \infty}_x)$ and in $C_t^{b} L^\infty_x \cap C_t W^{s, \infty}_x$ almost surely and for any $0 < s < 1$ and $0 < b < \frac12$. Here, $C^b_t (\R;X)$ for a Banach space $(X, \|\cdot\|)$ denotes the space of $b$-H\"older continuous functions defined as the completion of $C_c^{\infty}(\R; X)$ under the norm 
\[ \|f\|_{C^b_t(\R; X)} = \sup_{t \in \R} \|f(t)\| + \sup_{\substack{(t_1,t_2) \in \R^2 \\ t_1 \neq t_2}} \frac{\| f(t_1) - f(t_2) \|}{|t_1 - t_2|}. \]
By \cite[Proposition 2.9.5]{Meyer}, we learn that the spaces $C^b(\R,\R)$ and $\mc B^{b}_{\infty, \infty}$ coincide for each $0 < b < 1$, where $\mc B^{b}_{\infty, \infty}$ is the usual H\"older-Besov space. The desired result hence follows from the continuous embeddings $W^{b, \infty} \hookrightarrow \mc B^{b}_{\infty, \infty} \hookrightarrow W^{b + \eps, \infty}$ for any $\eps >0$.
\end{proof}

Our main goal in this subsection is 
to study 
 the space-time covariance $\G_N$ of $\Psi^{\textup{wave}}_N$,  $N \in \N$, adefined in \eqref{cov}. Given $N_1, N_2 \in \N$, we set 
\begin{align}
\G_{N_1, N_2} (t_1,t_2, x_1,x_2) = \E \big[ \Psi^{\textup{wave}}_{N_1}(t_1,x_1) \Psi^{\textup{wave}}_{N_2}(t_2,x_2) \big].
\label{L7}
\end{align}
for any $(t_1, x_1), (t_2, x_2) \in \R_+\times\T^2$. Since $\Psi^{\text{wave}}_{N}$ is constructed from the spatially homogeneous processes\footnote{A random variable $X$ is said to be spatially homogeneous if $X$ and $X(\cdot + y)$ share the same law for any $y\in \T^2$.} $u_0$, $v_0$ and $W$ and translation invariant operators, $\Psi^{\text{wave}}_{N}$ is also spatially homogeneous and we have
\begin{align*}
\G_N(t_1, t_2, x_1, x_2) & = \G_N(t_1, t_2, x_1 - x_2, 0),\\
 \G_{N_1, N_2}(t_1, t_2, x_1, x_2) & = \G_{N_1,N_2}(t_1, t_2, x_1 - x_2, 0).
\end{align*}
In what follows we use, with a slight abuse of notations, the (spatially) ``translation-invariant" notations $\G_N(t_1, t_2, x)$ and $\G_{N_1, N_2}(t_1, t_2, x)$ for $\G_N(t_1, t_2, x,0)$ and $\G_{N_1,N_2}(t_1, t_2, x,0)$, respectively. Namely, we write
\begin{align}
\begin{split}
\G_N(t_1, t_2, x) & = \E \big[ \Psi^{\textup{wave}}_N(t_1,x) \Psi^{\textup{wave}}_N(t_2,0) \big],\\
 \G_{N_1, N_2}(t_1, t_2, x) & = \E \big[ \Psi^{\textup{wave}}_{N_1}(t_1,x) \Psi^{\textup{wave}}_{N_2}(t_2,0) \big].
\end{split}
\label{cov10}
\end{align}

The following proposition establishes
sharp bounds on the space-time covariance $\G_N$
and its variant $\G_{N_1, N_2}$, 
extending
\cite[Lemma 2.7]{ORSW1}
and 
\cite[(2.2)]{ORSW2} to the time-dependent context.

\noi
\begin{proposition}\label{PROP:cov}
Given $N \in \N$, let $\G_N$ be as in \eqref{cov}-\eqref{cov10}. Then, we have 
\begin{align}
\G_N (t_1,t_2, x) \approx - \frac{1}{2 \pi} \log \big( |t_1 - t_2 | + |x| + N^{-1} \big)
\label{cov1}
\end{align}  

\noi
for any $(t_1, t_2,x) \in [0,1]^2 \times \T^2$. Here, the notation ``\,$\approx$'' is as in \eqref{approx1}
and~\eqref{approx2}.
Given $N_1, N_2 \in \N$, 
let $\G_{N_1, N_2}$ be as in \eqref{L7}.
Then, we have 
\begin{align}
\G_{N_1, N_2} (t_1, t_2,x) \approx - \frac{1}{2 \pi} \log \big( |t_1 - t_2| + |x| + N_1^{-1} \big)
\label{L8}
\end{align}

\noi
and 
\begin{align}
\begin{split}
& |\G_{N_j} (t_1, t_2,x) - \G_{N_1, N_2} (t_1,t_2,x)| \\
& \quad \les  \Big(1 \vee \big(- \log \big( |t_1 - t_2| + |x| + N_2^{-1} \big)\big)\Big)\wedge
(N_1^{-\frac 12} |x|^{-\frac 12})
+ O(N_1^{-1})
\end{split}
\label{L10}
\end{align}

\noi
for any $(t_1, t_2,x) \in [0,1]^2 \times \T^2$, 
 $N_2 \ge N_1 \ge 1$ and $j = 1, 2$.
\end{proposition}

\begin{remark}\rm \label{RMK:cov}
The estimate \eqref{cov1} in Proposition \ref{PROP:cov} shows that the (smoothed) space-time covariance function $\G_N$ has a singularity of elliptic type in the sense of Section \ref{SEC:ker}. This is rather surprising since $\Psi^{\text{wave}}_N$ is the solution to a linear (damped) wave equation and is due to a key cancellation; see the proof of Lemma \ref{LEM:cov} below. The hyperbolic nature of $\Psi^{\text{wave}}_N$ however shows up when considering spatial derivatives of $\G_N$; see Proposition \ref{PROP:cov2} and Remark \ref{RMK:dercov} below.
\end{remark}

In our physical space approach, it is crucial to also obtain tight bounds on spatial derivatives of the space-time covariance $\G_N$; see for instance Subsection \ref{SUBSEC:sto4} where such estimates are heavily used. To this end, we introduce a conventient notation. Fix $s>0$, $N \in \N$ and define the functions

\noi
\begin{align}
\mc H_N(t,x; s) & = \min \! \big(N ^s, (|t| + |x|)^{-\frac12}| |t| - |x||^{\frac12 - s}\big).
\label{singN}
\end{align}
for any $t \in \R$ and $x \in \T^2$.

In the next proposition, we show how the functions \eqref{singN} control the size of spatial derivatives of the space-time covariance $\G_N$.

\begin{proposition}\label{PROP:cov2}
Fix $N \in \N$ and let $\G_N$ and $\mc H_N$ be as in \eqref{cov}-\eqref{cov10} and \eqref{singN} respectively. Then, we have 
\begin{align}
|\partial^\al_{x} \G_N (t_1,t_2, x) | &  \les_{\eps,s} \mc H_N(t_1 - t_2, x; s) + |x_1 - x_2|^{-\eps}
\label{Ycovder}
\end{align}  

\noi
for any $(t_1, t_2,x) \in [0,1]^2 \times \T^2$, $\al \in \Z^2_{\ge 0}$ with $1 \le |\al| \le 2$, any $\eps>0$ and $s > |\al|$. Here, the implicit constant is independent of $N$.
\end{proposition}

%

The rest of this section is devoted to the proofs of Propositions \ref{PROP:cov} and \ref{PROP:cov2}.

\begin{remark}\rm \label{RMK:dercov}
We make a few remarks.

\smallskip
\noi
(i) Note that the derivative of order $\al \in \Z^2_{\ge 0} \setminus \{0\}$ of the right-hand-side of \eqref{cov1} in Proposition \ref{PROP:cov} is essentially given by the elliptic singularity
\[  \big( |t_1 - t_2| + |x| + N^{-1} \big)^{-|\al|}, \]
which is in general much better behaved than the hyperbolic singularity $\mc H_N(t_1 - t_2,x;s)$ in \eqref{Ycovder}. The latter comes from spatial derivatives of the remainder (hidden in the symbol ``\,$\approx$'') in \eqref{cov1} and highlights the hyperbolic nature of our problem. The presence of functions $\H_N$ which are singular along light cones (as opposed to a point in the elliptic case) in \eqref{Ycovder} makes the analysis in Subsections \ref{SUBSEC:sto4} and \ref{SEC:sing} very challenging.

\smallskip
\noi
(ii) In the case of the heat equation, 
the space-time covariance of the associated
stochastic convolution is given by 
\begin{align}
\G_N^\text{heat} (t_1-t_2, x) = - \frac{1}{2 \pi} \log \big( |t_1 - t_2 |^\frac12 + |x| + N^{-1} \big) + R_N(x),
\label{covv1}
\end{align}  

\noi
where $R_N$ is smooth uniformly in $N$ in the sense that
\begin{align*}
\sup_{N \in \N} \| \partial_x ^\al R_N \|_{L^{\infty}_x} \le C_\al,
\end{align*}
for any $\al \in \Z^2_{\ge0}$. Here, $C_\al >0$ is a constant independent of $N$. See, for example, \cite[Lemmas 3.7 and 3.8]{HS}.
Therefore in the parabolic setting, we only have to deal with parabolic/elliptic singularities centered at the space-time origin when considering spatial derivatives of $\G_N^\text{heat}$. This in sharp contrast with the hyperbolic case at hand.
\end{remark}

 Proposition \ref{PROP:cov} essentially follows from 
 an analogous  estimate on 
 the following  time-dependent variants of the periodic Green function 
 $G$
defined  in \eqref{green3} and \eqref{green4}:
\begin{align}
\Ld_N(t,x) 
& =  
\cos \big( t  |\nb|) \Pii_{\le N} ^2 G(x), 
\label{L}\\
\Ld_{N_1, N_2} (t,x) 
& =  
\cos \big( t | \nb |) \Pii_{\le N_1}
\Pii_{\le N_2}  G(x)
\label{L1}
\end{align}

\noi
for any $(t,x) \in \R \times \T^2$. 

In what follows, we aim to relate $\G_N$ to $\Ld_N$. To this end, we first introduce some convenient notations. Fix an integer $k \in \N$. We denote by $C^k(\T^2)$ the usual H\"older space of $C^k$-functions from $\T^2$ to $\R$ equipped with the norm

\noi
\begin{align*}
\|f\|_{C^k_x} = \max_{0 \le |\al| \le k} \| \partial_x^{\al} f\|_{L^{\infty}(\T^2)}.
\end{align*}

\noi
Let $\mc U^{\infty, k}$ be the space $L^{\infty}\big( (\R_+)^2 ; C^k (\T^2) \big)$ endowed with the norm

\noi
\begin{align*}
\|u\|_{\mc U^{\infty, k}} = \| u(t_1,t_2) \|_{L^{\infty}( (\R_+)^2; C^k_x)}.
\end{align*}
In what follows, we write $u \asymp v$ for $u, v : \R_+^2 \times \T^2 \to \R$ if $u-v \in \mc U^{\infty, 2}$. Similarly, for $\{u_N\}_{N \in \N}$ and $\{ v_N\}_{N \in \N}$ two sequences of functions in $(\R^{\R_+^2 \times \T^2})^\N$, we write $u_N \asymp v_N$ if $u_N - v_N$ belongs to $\mc U^{\infty, 2}$ uniformly in $N \in \N$. Namely, if we have

\noi
\begin{align*}
\sup_{N \in \N} \|u_N - v_N\|_{\mc U^{\infty, 2}} < \infty.
\end{align*}
Let $F_1, F_2, F_3 : \R_+^2 \times \T^2 \to \R$ be the functions given by

\noi
\begin{align}
\begin{split}
F_1(t_1,t_2,x) & = \sum_{n \in \Z^2} \frac{\sin((t_1 - t_2) | n |)}{|n|\jb n^2} e_n(x),\\
F_2(t_1,t_2,x) & = \sum_{n \in \Z^2} \frac{\cos((t_1-t_2) |n|)}{\jb n^4} e_n(x),\\
F_3(t_1,t_2,x) & = \sum_{n \in \Z^2} \frac{\cos((t_1+t_2) |n|)}{\jb n^4} e_n(x).
\label{YFun}
\end{split}
\end{align}

\noi
Clearly, $F_1, F_2$ and $F_3$ belong to $\mc U^{\infty,1}$. We define the subspace $\mc U^{\infty, 1}(F_1, F_2, F_3)$ of $\mc U^{\infty, 1}$ given by
\begin{align*}
\mc U^{\infty, 1}(F_1, F_2, F_3) = \big\{ g_1 F_1 + g_2 F_2 + g_3 F_3 : (g_1, g_2, g_3) \in L^{\infty}(\R_+^2; \R) \big\}.
\end{align*}

Armed with these notations, we can now give a precise description of the covariance function \eqref{cov} in terms of the periodic Green function \eqref{L} modulo elements of $\mc U^{\infty, 1}(F_1, F_2, F_3)$ and $\mc U^{\infty, 2}$.

\begin{lemma}\label{LEM:cov_GL}
Fix $N \in \N$. There exists a function $F \in \mc U^{\infty, 1}(F_1, F_2, F_3)$ such that we have the following decomposition:

\noi
\begin{align}
\G_N(t_1,t_2, x) \asymp e^{-\frac{|t_1 - t_2|}{2}} \Ld_N(t_1-t_2, x) + \Pii_{\le N}^2 F(t_1, t_2, x),
\label{dec1}
\end{align}

\noi
for any $(t_1, t_2,x) \in [0,1]^2 \times \T^2$ with $|t_1 - t_2| \le 1$.
\end{lemma}

\begin{proof}
Fix $N \in \N$ and $(t_1, t_2,x) \in \R_+^2 \times \T^2$ with $|t_1 - t_2| \les 1$.
We assume $t_1 \ge t_2$ for convenience, but the proof is similar in the case $t_1 < t_2$ and leads to a slightly different function $F$. 
From \eqref{Psi_S} and the independence
of  $u_0$, $v_0$,  and the space-time white noise forcing $\xi$, we have
\begin{align}
\G_N(t_1,t_2,x) = \1_N(t_1,t_2, x)  + \II_N(t_1,t_2, x) + \III_N(t_1,t_2, x),
\label{pcov1}
\end{align}

\noi
where
\begin{align*}
 \1_N & = \E \Big[ ( \dt \S (t_1) + \S(t_1) ) \Pii_{\le N} u_0(x) 
\cdot ( \dt \S (t_1) + \S(t_2) ) \Pii_{\le N}  u_0(0)  \Big], \\
 \II_N & = \E \Big[  \S(t_1) \Pii_{\le N}  v_0(x) \cdot \S(t_2) \Pii_{\le N}  v_0(0)  \Big],\\
 \III_N & =2  \E \bigg[  \Big( \int_{0}^{t_1} \S(t_1-t) d \Pii_{\le N} \mc W(t) \Big) (x) 
\cdot \Big( \int_{0}^{t_2} \S(t_2-t) d \Pii_{\le N}  \mc W(t) \Big)(0) \bigg].
\end{align*}

\noi
From \eqref{Psi_S}, \eqref{series}, \eqref{chi},
 and the independence of $g_n$, we have
\begin{align}
\begin{split}
\1_N(t_1,t_2, x) 
& = \frac{e^{- \frac{t_1+t_2}{2} }}{2 \pi} 
\sum_{n \in \Z^2} \Big( \cos (t_1 |n| ) + \frac12 \frac{\sin ( t_1 |n| ) }{|n|} \Big) \\
&  \quad  \times 
\Big( \cos (t_2 |n| ) + \frac12 \frac{\sin ( t_2 |n| ) }{|n|} \Big)  \frac{ \chi_N^2(n)}{ \jb{n}^2} e_n(x) \\
& = \1_N^a(t_1,t_2, x) + \1_N^b(t_1,t_2, x) + \1_N^c(t_1,t_2, x),
\end{split}
 \label{pcov3}
\end{align}

\noi
with

\noi
\begin{align}
\begin{split}
\1_N^a(t_1, t_2, x) & = \frac{e^{- \frac{t_1+t_2}{2} }}{2 \pi} 
\sum_{n \in \Z^2}  \cos (t_1 |n| )  \cos (t_2 |n| )  \frac{ \chi^2_N(n)}{ \jb{n}^2} e_n(x),\\
\1_N^b(t_1,t_2, x) & = \frac{e^{- \frac{t_1+t_2}{2} }}{4 \pi} \sum_{n \in \Z^2}  \frac{\sin((t_1+t_2)|n|) \chi^2_N(n)}{|n| \jb n ^2} e_n(x) ,\\
\1_N^c(t_1,t_2,x) & = \frac{e^{- \frac{t_1+t_2}{2} }}{8 \pi} \sum_{n \in \Z^2} \frac{\sin (t_1 |n| )  \sin (t_2 |n| ) \chi^2_N(n)}{|n|^2  \jb{n}^2} e_n(x) \\
& \asymp \frac{e^{- \frac{t_1+t_2}{2} }}{16 \pi} \Pii_{\le N}^2 \big( F_2(t_1, t_2, x) - F_3(t_1, t_2,x) \big),
\end{split}
\label{Ypcov3b}
\end{align}

\noi
where $F_2$ and $F_3$ are as in \eqref{YFun}. In the expression for $\1^c_N$, we used the identities $\sin(A)\sin(B) = \frac12 (\cos(A-B) - \cos(A+B))$ and 
\begin{align}
\frac{1}{|n|^2} - \frac{1}{\jb n ^2} = \frac{1}{\jb n ^2 |n|^2}
\label{Ysg81}
\end{align}
for any $n \in \Z^2 \setminus \{0\}$, to replace the factor $|n|^{-2}$ with $\jb n ^{-2}$ in the Fourier decomposition of $\1^c_N$ (via the use of the symbol ``$\asymp$"). Similarly, we have
\begin{align}
\II_N(t_1,t_2,x) =  \frac{e^{- \frac{t_1+t_2}{2} }}{2 \pi} \sum_{n \in \Z^2}  
\frac{\sin (t_1 |n| )  \sin (t_2 |n| )}{|n|^2} \chi_N^2(n) e_n(x).
\label{pcov4}
\end{align}

\noi
By \eqref{W1}, \eqref{Psi_S}, the independence of $B_n$'s 
and the Wiener isometry and since $t_1 \ge t_2 \ge 0$, we have 

\noi
\begin{align}
\begin{split}
\III_N(t_1,t_2,x) & =  \frac{e^{- \frac{t_1+t_2}{2} }}{\pi} \sum_{n \in \Z^2} 
\chi_N^2(n) e_n(x) \int_0 ^{t_2} e^{t} \frac{\sin( (t_1-t) |n|  ) \sin( (t_2-t) |n|  )}{|n|^2}  dt  
 \\
& = \III_N^{1}(t_1,t_2,x) + \III_N^2(t_1,t_2,x),
\end{split}
\label{Ypcov4b}
\end{align}

\noi
with

\noi
\begin{align}
\III_N^1(t_1, t_1, x) = \frac{e^{- \frac{t_1+t_2}{2} }}{\pi} \sum_{n \in \Z^2} 
\frac{\chi_N^2(n)}{\jb n ^2} e_n(x) \int_0 ^{t_2} e^{t} \sin( (t_1-t) |n|  ) \sin( (t_2-t) |n|  ) dt, 
\label{Ypcov4bb}
\end{align}

\noi
and

\noi
\begin{align}
\begin{split}
\III_N^2(t_1, t_2, x) & =  \frac{e^{- \frac{t_1+t_2}{2} }}{\pi} \sum_{n \in \Z^2} \chi_N^2(n) e_n(x) \int_0 ^{t_2} e^{t} \Big( \frac{\sin( (t_1-t) |n|  ) \sin( (t_2-t) |n|  )}{|n|^2} \\
& \qquad \qquad \qquad \qquad  \qquad \qquad \qquad -\frac{\sin( (t_1-t) |n|  ) \sin( (t_2-t) |n| )}{\jb n ^2 }\Big)  dt.
\end{split}
\label{Ypcov4bbb}
\end{align}
Note that the zero$^{\text{th}}$-Fourier mode of $\III^2_N$ is smooth uniformly in $N \in \N$. Hence, by \eqref{Ysg81}, the identity $\sin(A)\sin(B) = \frac12 (\cos(A-B) - \cos(A+B))$ and integration by parts, we have that

\noi
\begin{align}
\begin{split}
\III_N^2(t_1, t_2, x) & \asymp  \frac{e^{- \frac{t_1+t_2}{2} }}{\pi} \sum_{n \in \Z^2\setminus\{0\}} \chi_N^2(n) e_n(x) \int_0 ^{t_2} e^{t} \Big( \frac{\sin( (t_1-t) |n|  ) \sin( (t_2-t) |n|  )}{|n|^2}
 \\
& \qquad \qquad \qquad \qquad  \qquad \qquad \qquad -\frac{\sin( (t_1-t) |n|  ) \sin( (t_2-t) |n| )}{\jb n ^2 }\Big)  dt \\
& \asymp  \frac{e^{- \frac{t_1+t_2}{2} }}{\pi} \sum_{n \in \Z^2\setminus\{0\}} \frac{\chi_N^2(n)}{ \jb n ^2 |n|^2 } e_n(x) \int_0 ^{t_2} e^{t}  \sin( (t_1-t) |n|  ) \sin( (t_2-t) |n|  ) dt\\
& \asymp  \frac{e^{- \frac{t_1+t_2}{2} }}{2\pi} \sum_{n \in \Z^2 \setminus \{0\}} \frac{\chi_N^2(n)}{ \jb n ^2 |n|^2 } e_n(x)  \\
& \qquad \qquad \quad \times \int_0 ^{t_2} e^{t} \big( \cos((t_1 -t_2 )|n| )  -\cos((t_1+t_2-2t)|n| ) \big)  dt \\
& \asymp \frac{e^{ \frac{t_2-t_1}{2}} - e^{- \frac{t_1 + t_2}{2}} }{2\pi} \sum_{n \in \Z^2 \setminus \{0\}}  \frac{\cos((t_1 -t_2 ) |n| )}{ \jb{n}^2 |n|^2} \chi_N^2(n) e_n(x) \\
& \asymp \frac{e^{ \frac{t_2-t_1}{2}} - e^{- \frac{t_1 + t_2}{2}} }{2\pi} \Pii_{\le N}^2 F_2(t_1, t_2, x),
\end{split}
\label{Ypcov4bbb}
\end{align}
where we used \eqref{Ysg81} again in the last line to replace $|n|^{-2}$ with $\jb n ^{-2}$. Similarly, we also have that

\noi
\begin{align}
\begin{split}
\III^1_N(t_1, t_2, x) &=  \frac{e^{- \frac{t_1+t_2}{2} }}{2 \pi} 
\sum_{n \in \Z^2}\frac{\chi_N^2(n)}{\jb n ^2} e_n(x)\\
& \quad \times  \int_0 ^{t_2} 
e^{t} \big(\cos((t_1 -t_2 )|n| )  -\cos((t_1+t_2-2t)|n| ) \big)dt \\
& =: \III_N^{1,a} (t_1, t_2, x) - \III_N^{1,b}(t_1, t_2, x). 
 \end{split}\label{pcov5}
\end{align}

\noi
By performing the $t$-integration in $\III^{1,a}_N$ and \eqref{L}, we have

\noi
\begin{align}
\III_N^{1,a}(t_1,t_2, x) = \big(e^{ \frac{t_2-t_1}{2}} - e^{- \frac{t_1 + t_2}{2}} \big) \Ld_N(t_1-t_2, x).
\label{Yc1a}
\end{align}

\noi
Hence, from  \eqref{Ypcov3b}, \eqref{pcov4}, \eqref{Yc1a}, \eqref{L} and the identity $\cos (A-B) = \cos A \cos B + \sin A \sin B$, we obtain

\noi
\begin{align}
\begin{split}
\1^a_N(t_1,t_2,x_1) + \II_N(t_1,t_2,x) + \III^{1,a}_N(t_1,t_2,x) =  e^{\frac{t_2-t_1}{2}} 
 \Ld_N(t_1-t_2, x).
 \end{split}
\label{pcov6}
\end{align}
Integrating by parts the term $\III^{1,b}_N$ gives

\noi
\begin{align*}
\III_N^{1,b}(t_1, t_1, x) & \asymp  \frac{e^{- \frac{t_1+t_2}{2} }}{4\pi} \sum_{n \in \Z^2\setminus\{0\}} \frac{\chi_N^2(n)}{\jb n^2 |n|} \sin( (t_1+t_2) |n|) e_n(x) \\  
& \quad  - \frac{e^{\frac{t_2-t_1}{2} }}{4\pi} \sum_{n \in \Z^2\setminus\{0\}} \frac{\chi_N^2(n)}{\jb n^2 |n|} \sin( (t_1-t_2) |n|) e_n(x) \\
& \quad + \frac{e^{\frac{t_2-t_1}{2} }}{8\pi} \sum_{n \in \Z^2\setminus\{0\}} \frac{\chi_N^2(n)}{\jb n^2 |n|^2} \cos( (t_1-t_2) |n|) e_n(x) \\
& \quad - \frac{e^{- \frac{t_2+t_1}{2} }}{8\pi} \sum_{n \in \Z^2\setminus\{0\}} \frac{\chi_N^2(n)}{\jb n^2 |n|^2} \cos( (t_1+t_2) |n|) e_n(x) \\
& \quad - \frac{e^{- \frac{t_2+t_1}{2} }}{8\pi} \sum_{n \in \Z^2\setminus\{0\}} \frac{\chi_N^2(n)}{\jb n^2 |n|^2} \int_0^{t_2} e^t \sin((t_1 + t_2 - 2t) |n| )  e_n(x).\\
\end{align*}
Thus, with the notations in \eqref{YFun} and using \eqref{Ysg81} as before, we have

\noi
\begin{align}
\begin{split}
\III_N^{1,b}(t_1, t_1, x)  & \asymp  \1_N^b(t_1,t_2, x)  - \frac{e^{\frac{t_2-t_1}{2} }}{4\pi} \Pii_{\le N}^2 F_1(t_1, t_2, x) \\
& \quad + \frac{e^{\frac{t_2-t_1}{2} }}{8\pi} \Pii_{\le N}^2 F_2(t_1, t_2, x) - \frac{e^{- \frac{t_2+t_1}{2} }}{8\pi} \Pii_{\le N}^2 F_3(t_1, t_2,x),
\end{split}
\label{Yc1}
\end{align}

Therefore, by \eqref{pcov1}, \eqref{pcov3}, \eqref{Ypcov3b}, \eqref{Ypcov4b}, \eqref{Ypcov4bbb}, \eqref{pcov5}, \eqref{pcov6} and \eqref{Yc1}, we deduce that

\noi
\begin{align*}
\G_N(t_1, t_2, x_1- x_2) & \asymp e^{\frac{t_2-t_1}{2}} \Ld_N(t_1-t_2, x_1-x_2) + \frac{e^{\frac{t_2-t_1}{2} }}{4\pi} \Pii_{\le N}^2 F_1(t_1, t_2, x_1-x_2) \\
&  \qquad + \frac{1}{16\pi} \big( 6  e^{\frac{t_2-t_1}{2}} - 7 e^{-\frac{t_1 + t_2}{2}} \big) \Pii_{\le N}^2 F_2(t_1, t_2, x_1-x_2) \\
& \qquad  + \frac{1}{16\pi} e^{-\frac{t_1 + t_2}{2}} \Pii_{\le N}^2 F_3(t_1, t_2, x_1-x_2),
\end{align*}

\noi
as required.
\end{proof}

Then, we have the following bound on $\Ld_N$.

\begin{lemma}\label{LEM:cov}
Given  $N \in \N$, let  $\Ld_N$ be as in \eqref{L}. 
Then, we have \begin{align}
\Ld_N (t,x) \approx - \frac{1}{2 \pi} \log \big( |t| + |x| + N^{-1} \big)
\label{cov4}
\end{align}

\noi
for any $(t, x) \in \R\times \T^2$ with  $0 \le  |t| \le 1 $.
Given  $N_1, N_2 \in \N$, let  $\Ld_{N_1, N_2}$ be as in \eqref{L1}. 
Then, we have 
\begin{align}
\Ld_{N_1, N_2} (t,x) \approx - \frac{1}{2 \pi} \log \big( |t| + |x| + N_1^{-1} \big)
\label{L2}
\end{align}

\noi
and 
\begin{align}
|\Ld_{N_j} (t,x) - \Ld_{N_1, N_2} (t,x)| 
\les  \Big(1 \vee \big(- \log \big(|t| + |x| + N_2^{-1} \big)\big)\Big)\wedge
(N_1^{-\frac 12} |x|^{-\frac 12})
\label{L3}
\end{align}

\noi
for any $(t, x) \in \R \times \T^2$
with  $0 \le  |t| \le 1$, 
 $N_2 \ge N_1 \ge 1$ and $j = 1, 2$.

\end{lemma}

See \cite[Lemma 2.3 and Remark 2.4]{ORSW1}
for analogous results  in the time-independent case.
We first present a proof of  
Proposition \ref{PROP:cov} by 
assuming  Lemma \ref{LEM:cov}.

\begin{proof}[Proof of Proposition \ref{PROP:cov}] First, note that by Lemma \ref{LEM:cov_GL}, we have that 

\noi
\begin{align*}
\G_N(t_1, t_2, x) \approx e^{\frac{t_2-t_1}{2}} \Ld_N(t_1-t_2, x).
\end{align*}

\noi
Therefore, 
by noting that 
\begin{align}
t \big|\log(t + c_0)\big|
\le t |\log t| + O(1) \les 1, 
\label{pcov8}
\end{align}

\noi
uniformly in  $0 < t \les 1$ and $0 < c_0 \les 1$, from \eqref{cov4} in Lemma~\ref{LEM:cov} with the mean value theorem
and~\eqref{pcov8}, we deduce that
\begin{align*}
\G_N(t_1,t_2,x)
& \approx  \Ld_N(t_1-t_2, x) +  \Big(  e^{\frac{t_2-t_1}{2}} -1 \Big) \Ld_N(t_1-t_2, x) \\
& \approx \bigg( - \frac{1}{2 \pi} + O( |t_1-t_2|) \bigg)  \log \big( |t_1 - t_2 | + |x| + N^{-1} \big) \\
& \approx - \frac{1 }{2 \pi} \log \big( |t_1 - t_2 | + |x| + N^{-1} \big), 
\end{align*}

\noi
provided that  $|t_1-t_2| \les 1$.
A similar computation with
\eqref{L2} in Lemma~\ref{LEM:cov}.
yields \eqref{L8}.

From 
a slight modification of the computations in the proofs of Lemma \ref{LEM:cov_GL}, we have
\begin{align}
\begin{split}
& \G_{N_j} (t_1,t_2,x) - \G_{N_1, N_2} (t_1,t_2,x)\\
& \quad = e^\frac {t_2 - t_1}{2}
\Big\{\Ld_{N_j} (t_1 - t_2,x) - \Ld_{N_1, N_2} (t_1 - t_2,x)\Big\}
+ O(N_1^{-1}).
\end{split}
\label{pcov9}
\end{align}

\noi
Then, the bound \eqref{L10} 
follows from  \eqref{pcov9}
and 
 \eqref{L3} 
in Lemma~\ref{LEM:cov}.
%
%
%
%
\end{proof}

We now prove Lemma \ref{LEM:cov}.

\begin{proof}[Proof of Lemma \ref{LEM:cov}] Fix $N \in \N$. In view of the parity of the cosine function, we fix $0 \le t \le 1$ in what follows without any loss of generality.

We first prove \eqref{cov4}. It is easy to see that 
\begin{align*}
\Ld_N(t,x) \approx (\Ld_N - \Ld_{10})(t,x)
\end{align*}
for any $x \in \T^2$. Hence, from \eqref{L} and 
the Poisson summation formula \eqref{poisson}, we have
\begin{align}
(\Ld_N-\Ld_{10}) (t,x) = \sum_{k \in \Z^2} (f_N-f_{10})(t,x + 2 \pi k)
\label{c1}
\end{align}
for any $x  \in \T^2  \cong [-\pi, \pi)^2$ and where $f_K$ is given by 
\begin{align}
f_K(t,x) = \frac{1}{(2 \pi)^2} \int_{\R^2} \cos( t |\xi| ) \frac{\chi_K^2( \xi ) }{\jb \xi ^2} e^{i  \xi \cdot x} d \xi
\label{c2}
\end{align}
for any $x \in \R^2$ and $K \in \N$.

\medskip
\noi
{\bf $\bul$ Step 1: analysis of summands with $|k| \ge 1$ in \eqref{c1}.}\quad We first prove a bound on $f_N$. Let $(t,x) \in \R_+ \times \R^2$. Then, by trigonometric identities, a polar change of variables with \eqref{sphere} and \eqref{sphere3}, we have
\begin{align}
\begin{split}
(f_N-f_{10})(t,x) & = \frac{1}{8 \pi^2} \sum_{\eps_1 \in \{+,-\}} \int_{\R^2} \frac{ (\chi_N^2- \chi_{10}^2)( \xi ) }{\jb \xi ^2} e^{i  \xi \cdot x + i \eps_1 t |\xi|} d \xi
\\
& =  \frac{1}{8 \pi^2} \sum_{\eps_1 \in \{+,-\}} \int_{\R_+ \times \mb S^1} \frac{ (\chi_N^2- \chi_{10}^2)( r) }{\jb r^2} e^{i  r x \cdot \omega + i \eps_1 t r} r dr d \s(\o)  \\
& = \frac{1}{4 \pi} \sum_{\eps_1 \in \{+,-\}} \int_{0}^{\infty} \frac{ (\chi_N^2- \chi_{10}^2)( r) }{\jb r^2} e^{i \eps_1 t r} \widecheck{d \s}(rx) r dr \\
& =  \frac{1}{4 \pi} \sum_{\eps_1, \eps_2 \in \{+,-\}} \int_{0}^{\infty} \frac{ (\chi_N^2- \chi_{10}^2)( r) }{\jb r^2} e^{i r( \eps_1 t + \eps_2 |x|)} a_{\eps_2}(rx) r dr.
\end{split}
\label{Ysg3}
\end{align}
For fixed $r \in [0,\infty)$, the function $x \in \R^2 \setminus B(0,\pi) \mapsto e^{i r \eps_2 |x|}  a_{\eps_2}(rx) $ is smooth and by the Leibniz rule, its $\al^{\textup{th}}$-order derivative, for $\al \in \Z^2_{\ge 0}$ with $|\al| \le 2$, is a finite sum (over $\al_0 \in \Z^2_{\ge 0}$) of terms of the form
\[ F_{\al_0}(x) e^{ir \eps_2 |x|} \cdot (\partial_x^{\al_0} a_{\eps_2})(r x) r^{|\al|}, \]
where $\al_0 \in \Z^2_{\ge 0}$ with $|\al_0| \le |\al|$ and $F_{\al_0}$ is a function bounded away from the origin. By integration by parts in the variable $r$ and \eqref{sphere4}, noting that $\chi_N^2- \chi_{10}^2$ (and its derivatives) vanishes near the origin, we get
\begin{align}
\begin{split}
& \Big|\int_0^\infty e^{i r( \eps_1 t + \eps_2 |x|)}  \frac{ (\chi_N^2- \chi_{10}^2)( r) }{\jb r^2} (\partial_x^{\al_0} a_{\eps_2})(r x) r^{|\al_0| +1} dr \Big| \\
& \qquad \les \big| \eps_1 t + \eps_2  |x| \big|^{-10} \int_0^{\infty} \bigg| \Big(\frac{\partial}{\partial r}\Big)^{10} \bigg\{ \frac{ (\chi_N^2- \chi_{10}^2)( r) }{\jb r^2} (\partial_x^{\al_0} a_{\eps_2})(r x) r^{|\al| +1}\bigg\} \bigg| dr \\
& \qquad \les  \big| \eps_1 t + \eps_2  |x| \big|^{-10} \int_{0}^{\infty} \jb r ^{-2} dr\\
& \qquad \les |x|^{-10}
\end{split}
\label{Ysg3bb}
\end{align}
provided that $|x| \ge \pi$ and $0 \le t \le 1$. Thus, by \eqref{Ysg3}, \eqref{Ysg3bb} and the discussion above, we have
\begin{align}
& |\partial_x^{\al}(f_N-f_{10})(t,x)| \les |x|^{-10},
\label{Ysg1}
\end{align}
for any $x \in \R^2 \setminus B(0,\pi)$, $0 \le t \le 1$ and $\al \in \Z^2_{\ge 0}$ with $|\al| \le 2$. Hence, by \eqref{Ysg1}, we deduce 
\begin{align}
\bigg\|  \sum_{k \in \Z^2\setminus\{0\}}\partial^{\al}_x(f_N-f_{10})(t,x + 2 \pi k) \bigg\|_{L_x^{\infty}([-\pi, \pi)^2)} \les 1
\label{c3}
\end{align}
for any $\al \in \Z^2_{\ge 0}$ with $|\al| \le 2$, uniformly in 
 $0\le t \le 1$.\footnote{In the current proof, we only need \eqref{c3} for $\al = 0$, but we proved \eqref{c3} for all $|\al| \le 2$ for future reference; see Lemma \ref{LEM:covYsg1}.} Hence, by \eqref{c1}, \eqref{c2} and \eqref{c3}, we have
\begin{align}
\begin{split}
\Ld_N(t,x) &\approx  (\Ld_N - \Ld_{10})(t,x) \\
& \approx (f_N - f_{10})(t,x) \\
& \approx f_N(t,x)
\end{split}
\label{c4}
\end{align}
for any $x \in \T^2 \cong [-\pi, \pi)^2$. Thus, from \eqref{c1}, \eqref{c3} and \eqref{c4}, \eqref{cov4} reduces to proving
\begin{align}
f_N(t,x) \approx - \frac{1}{2 \pi} \log \big( |t| + |x| + N^{-1} \big).
\label{Ysg2}
\end{align}
for any $x \in \T^2 \cong [-\pi, \pi)^2$.

\medskip
\noi
{\bf $\bul$ Step 2: proof of \eqref{Ysg2}.}\quad Define a function $\eta_N$  on $\R^2$ by setting
\begin{align}
\eta_N(x) = 2 N^2 \cdot  \ind_{B (0, N^{-1})} (x), \quad x \in \R^2, 
\label{c5}
\end{align}

\noi
where $B ( x , r) \subset \R^2$ denotes the  ball of radius $r > 0$
centered at  $x$. On the Fourier side, we have
\begin{align}
\ft  \eta_N (\xi ) = \frac{N^2}{\pi} \int_{B (0, N^{-1})} e^{-i \xi \cdot x} d x = \frac{1}{\pi} \int_{B (0, 1)} e^{-i  \frac{\xi}{N} \cdot x} d x, \quad \xi \in \R^2.
\label{c6}
\end{align}

\noi
We claim that 
\begin{align}
| \chi_N^2(\xi) - \ft  \eta_N  (\xi) | \les \min \bigg( \frac{|\xi|}{N}, \frac{N}{|\xi|} \bigg)
\label{c6a}
\end{align}

\noi
for any $\xi \in \R^2$. 
Indeed, when $|\xi|\sim N$, the bound \eqref{c6a} trivially follows since $|\chi_N(\xi)|, |\ft \eta_N(\xi)|\les 1$, 
uniformly in $\xi \in \R^2$ and $N \in \N$.
When $|\xi| \ll N$, it follows from~\eqref{c6} and the mean value theorem that
\begin{align*}
| \chi_N^2(\xi) - \ft  \eta_N  (\xi) |
= \frac{1}{\pi} \bigg|\int_{B (0, 1)} (1 - e^{-i  \frac{\xi}{N} \cdot x}) d x\bigg|
\les \frac{|\xi|}{N}, 
\end{align*}

\noi
yielding \eqref{c6a}.
When $|\xi|\gg N$, we have $\chi_N(\xi) = 0$
and thus \eqref{c6a} follows from Green's formula
\cite[Theorem 3 (i) on p.\,712]{Evans}.
Hence, from
\eqref{c4} and \eqref{c6a}, we obtain
%
%
\begin{align*}
f_N(t,x) \approx \frac{1}{(2 \pi)^2} \int_{\R^2} \cos( t | \xi| ) \frac{ \ft \eta_N (\xi) }{\jb \xi ^2} e^{i \xi \cdot x} d \xi.
\end{align*}

\noi
Let  $G_N = \eta_N * G_{\R^2}$, where $G_{\R^2}$ is the Green function for $1-\Dl$ on $\R^2$
defined in \eqref{green2}
and~\eqref{green1}.
Then, we have 
\begin{align}
\begin{split}
f_N(t,x) & \approx \frac{1}{(2 \pi)^2} \int_{\R^2} \cos( t |\xi| ) \frac{ \ft  \eta_N (\xi) }{\jb \xi ^2} e^{i \xi \cdot x} d \xi 
\\
& = \frac{1}{(2 \pi)^2} \int_{\R^2}  \frac{ \ft \eta_N (\xi) }{\jb \xi ^2} e^{i \xi \cdot x} d \xi +  \frac{1}{(2 \pi)^2} \int_{\R^2} \big( \cos( t |\xi| ) - 1 \big) \frac{ \ft \eta_N (\xi) }{\jb \xi ^2} e^{i \xi \cdot x} d \xi  \\
& = G_N (t,x)  -  \frac{1}{(2 \pi)^2}  \int_0^{t} \int_{\R^2}\frac{\sin( t' |\xi| )}{|\xi|}  \frac{ |\xi|^2 }{\jb \xi ^2} \ft \eta_N (\xi) e^{i \xi \cdot x} d \xi dt' \\
& \approx G_N (t,x)  -  \frac{1}{(2 \pi)^2}  \int_0^{t}\int_{\R^2} \frac{\sin( t' |\xi| )}{|\xi|} \ft  \eta_N (\xi) e^{i \xi \cdot x} d \xi dt'.
\end{split}
 \label{c8}
\end{align}

Using \eqref{poisson2}-\eqref{poisson3}, we can write the second term on the right-hand side of  \eqref{c8} as 
\begin{align}
\begin{split}
& \frac{1}{(2 \pi)^2}  \int_0^{t} \int_{\R^2} \frac{\sin( t' |\xi| )}{|\xi|} \ft  \eta_N (\xi) e^{i \xi \cdot x} d \xi dt' \\
&  \quad  
=  \frac{1}{(2 \pi)^2} \int_0^{t} \int_{B(x,t')} \frac{\eta_N (y)}{\sqrt{(t')^2 - |x-y|^2}} dy dt' \\
&  \quad  
 = \frac{1}{(2 \pi)^2}  \int_{B(x, t)} \eta_N (y)  \int_{|x-y|}^{t} \frac{1}{\sqrt{(t')^2 - |x-y|^2}} dt' 
 dy \\
&  \quad  
  = \frac{1}{(2 \pi)^2}  \int_{B(x, t)} \eta_N (y)  \Big[ \log \big( t' + \sqrt{ (t')^2 - |x-y|^2 } \big)  \Big]
  \bigg|^{t}_{|x-y|} 
  dy \\
&  \quad  
 = \frac{1}{(2 \pi)^2}  \int_{B(x, t)} \eta_N (y)  \log \big( t + \sqrt{ t^2 - |x-y|^2 } \big) dy \\
&  \quad  \quad  - \frac{1}{(2 \pi)^2}  \int_{B(x, t)} \eta_N (y)  \log  |x-y|  \, dy\\
& \quad = : A_N(t, x) - B_N(t, x).
\end{split}
\label{c9}
\end{align}

We first deal with  $G_N$ in \eqref{c8}. In view of \eqref{c5}, \eqref{green1},  and the smoothness of $G_{\R^2}$ away from the origin, we have
\begin{align}
\begin{split}
G_N(t,x) 
& = - \frac{N^2}{2 \pi^2} \int_{B(0,N^{-1})} \log  |x-y| \,  dy\\
& =  - \frac{1}{2 \pi^2} \int_{B(0,1)} \log \big( N^{-1} |N x-y| \big) dy \\
& \approx - \frac{1}{2 \pi} \log (N^{-1}) + R(N x),
\end{split}
\label{c11}
\end{align}

\noi
where 
\begin{align}
R(z) = - \frac{1}{2\pi^2} \int_{B(0,1)} \log |z-y| \, dy, 
\quad z \in \R^2.
\label{c11a}
\end{align}

\noi
For  $|z| \les 1$, we have  $| R (z)| \les 1 $. 
On the other hand, for $|z| \gg 1$, we have 
\begin{align}
R(z) = - \frac{1}{2 \pi^2} \int_{B(0,1)} \log |z| \, dy +O(1)
= - \frac{1}{2 \pi} \log |z| +O(1).
\label{c11b}
\end{align}

\noi
Hence, we conclude that
\begin{align}
G_N(t,x) \approx - \frac{1}{2 \pi} \log \big( |x| + N^{-1} \big).
\label{c12}
\end{align}

Next, we treat $A_N$ and $B_N$ in \eqref{c9}.
By noting 
$t\le t + \sqrt{ t^2 - N^{-2} |Nx-y|^2 } \le 2t$, 
it follows from \eqref{c9} and a change of variables that
\begin{align}
A_N(t,x) & \approx \log  t  \cdot \frac{ \big| B(0,1) \cap B(Nx, N t ) \big| }{2 \pi^2}.
 \label{c13}
\end{align}

\noi
From \eqref{c9} and a change of variables, we have 
\begin{align}
B_N(t, x) = \frac{1}{2\pi^2}
\int_{B(0, 1) \cap B(Nx, Nt)}
\log \big(N^{-1}|Nx - y|\big) dy.
\label{c13a}
\end{align}

\noi
Then, by arguing as in \eqref{c11}-\eqref{c12}, we have 
\begin{align}
B_N(t,x) & \approx \log\big( |x| + N^{-1} \big) \cdot \frac{ \big| B(0,1) \cap B(Nx, N t ) \big| }{2 \pi^2}.
 \label{c14}
\end{align}

\noi
In the following, 
we 
may assume that 
\begin{align}
N ( |x| - t ) \le 1,
\label{c13b}
\end{align}

\noi
 since, otherwise, we 
would have $B(0,1) \cap B(Nx, N t ) = \varnothing$ 
and thus $A_N(t, x) = B_N(t, x) = 0$. Thus, by \eqref{c8}, \eqref{c9} and \eqref{c12}, we would have
 \begin{align*}
f_N(t,x) \approx -  \frac{1}{2 \pi} \log \big( |x| + N^{-1} \big)
 \approx 
 - \frac{1}{2 \pi} \log \big( t + |x| + N^{-1} \big), 
\end{align*}
proving \eqref{Ysg2} in that case.

\medskip

\noi
{\bf $\bul$ Case 1:} $|x| + N^{-1} \ges t $. 
\quad In this case, from \eqref{c13} and \eqref{c14}, we have
\begin{align}
\begin{split}
|A_N(t,x) - B_N(t,x)| 
& \les \bigg| \log \bigg( \frac{|x| + N^{-1}}{t} \bigg) \bigg| 
\cdot \big| B(0,1) \cap B(Nx, N t ) \big| \\
& \les \bigg( \frac{ N |x| + 1 }{ N t } \bigg) \cdot \big| B(0,1) \cap B(Nx, N t ) \big|\\
&  \les \bigg( \frac{ N |x| + 1 }{ N t } \bigg) \cdot \min \big( 1, ( N t)^2 \big). 
\end{split}
\label{c16}
\end{align}

\noi
If  $N |x| \les 1$, then we have
\begin{align*}
\eqref{c16} \les (N t)^{-1} \cdot (N t) \les 1.
\end{align*}

\noi
Otherwise, i.e.~if $N |x| \gg 1$,  then we have 
\begin{align*}
N |x | \sim N t,
\end{align*}

\noindent
in view of the conditions $|x| + N^{-1} \ges t$ and \eqref{c13b}.
This shows that $\eqref{c16} \les 1$ in this case as well.
Therefore, 
 from  \eqref{c8},  \eqref{c9}, and \eqref{c12}
 with $\eqref{c16} \les 1$, we conclude that 
 \begin{align*}
 f_N(t,x) \approx -  \frac{1}{2 \pi} \log \big( |x| + N^{-1} \big)
 \approx 
 - \frac{1}{2 \pi} \log \big( t + |x| + N^{-1} \big), 
\end{align*}

\noi
yielding
 \eqref{Ysg2} in this case.

\medskip

\noi
{\bf $\bul$ Case 2:} $|x| + N^{-1} \ll t $.
\quad
In this case, we have 
 $B(0,1)\subset B(Nx, N t )$
 and thus it follows from \eqref{c12} and \eqref{c14}
 that $B_N(t, x) \approx -G_N(t, x)$.
Hence, 
 from  \eqref{c8}, \eqref{c9},  and \eqref{c13}, 
  we have
\begin{align*}
f_N(t,x) \approx - \frac{1}{2\pi} \log t \approx 
- \frac{1}{2 \pi} \log \big( t + |x| + N^{-1} \big), 
\end{align*}

\noi
yielding 
\eqref{Ysg2}

\medskip

The second bound \eqref{L2}
follows from 
 a slight modification
 of the proof of \eqref{cov4}, 
 once we replace
 \eqref{c6a} by 
\begin{align*}
| \chi_{N_1}(\xi)
\chi_{N_2}(\xi)
 - \ft  \eta_{N_1}  (\xi) |
  \les \min \bigg( \frac{|\xi|}{N_1}, \frac{N_1}{|\xi|} \bigg).
\end{align*}

Lastly, we  discuss the third bound \eqref{L3}.
The bound 
by $1 \vee \big(- \log \big( t + |x| + N_1^{-1} \big)\big)$
in~\eqref{L3}
follows from \eqref{L2}.
In the following, we briefly discuss how to obtain the  bound 
by $N_1^{-\frac 12} |x|^{-\frac 12}$ in \eqref{L3}
when $j = 1$.
When $j = 2$, a similar computation holds and thus we omit details.

As in \eqref{c1}, 
it follows from 
the Poisson summation formula \eqref{poisson} that 
\begin{align}
\Ld_{N_1} (t,x) - \Ld_{N_1, N_2} (t,x) = \sum_{k \in \Z^2} f_{N_1, N_2}(t,x + 2 \pi k),
\label{L4}
\end{align}

\noi
for $(t, x) \in \R_+ \times \T^2$,  where 
$f_{N_1, N_2}$ is given by 
\begin{align*}
f_{N_1,N_2}(t,x) 
 = \frac{1}{(2 \pi)^2} \int_{\R^2} \cos( t | \xi |) 
\frac{ \chi_{N_1}( \xi )(\chi_{N_1}( \xi )-\chi_{N_2}( \xi )) }{\jb \xi ^2} e^{i  \xi \cdot x} d \xi
\end{align*}
for any $(t,x) \in \R_+ \times \R^2$. Then, by arguing as in \eqref{Ysg3}-\eqref{Ysg1}, we have
\begin{align}
f_{N_1, N_2}(t,x) = \frac{1}{4 \pi} \sum_{\eps_1, \eps_2 \in \{+,-\}} \int_{0}^{\infty} \frac{ \chi_{N_1}( r )(\chi_{N_1}( \xi )-\chi_{N_2}( r )) }{\jb r^2} e^{i r( \eps_1 t + \eps_2 |x|)} a_{\eps_2}(rx) r dr
\label{Ysg4},
\end{align}
where $\{a_{\eps_2}\}_{\eps \in \{+,-\}}$ is as in \eqref{sphere3} and
\begin{align}
|f_{N_1, N_2}(t,x)| \les |x|^{-10} N_1^{-2}
\label{Ysg5}
\end{align}
provided $|t|\le 1$ and $|x| \ge \pi$. In \eqref{Ysg5}, the factor $N_1^{-2}$ comes from the restriction $\{r \sim N_1\}$ on the right-hand-side of \eqref{Ysg4}.

Furthermore, by \eqref{sphere4} and \eqref{Ysg3}, we get 
\begin{align}
|f_{N_1, N_2}(t,x)|  \les  \int_{r \sim N_1} r^{-\frac{3}{2}} |x|^{-\frac12} dr \les N_{1}^{-\frac12} |x|^{-\frac12}.
\label{Ysg6}
\end{align}

Hence, we have
\begin{align}
 \bigg|  \sum_{k \in \Z^2}f_{N_1, N_2}(t,x + 2 \pi k)\bigg|  \les N_1^{-\frac 12} |x|^{-\frac 12}, \label{L6}
\end{align}
where we used \eqref{Ysg6} and \eqref{Ysg5} to estimate the contributions of the summands corresponding to $k=1$ and $|k| \ge 1$, respectively. Therefore, 
the  bound 
by $N_1^{-\frac 12} |x|^{-\frac 12}$ in \eqref{L3}
follows from \eqref{L4} and \eqref{L6}. This concludes the proof of Lemma \ref{LEM:cov}.
\end{proof}

We now turn our attention to the proof of Proposition \ref{PROP:cov2} which is a consequence of the two following lemmas.

\begin{lemma}\label{LEM:covYsg1}
Let $F_1, F_2, F_3$ be as in \eqref{YFun}. Fix $N \in \N$ and let $\mc H_N$ be as in \eqref{singN}. Then, the following estimate holds:
\begin{align}
\big| \partial_x^{\al} \Pii_{\le N}^2 F_j (t_1, t_2, x) \big| \les_\eps \mc H_N (t_1 - t_2, x;s) + |x|^{-\eps}
\label{Ysg78}
\end{align}
for any $(t_1, t_2,x) \in [0,1]^2 \times \T^2$, $j \in \{1,2,3\}$, $\al \in \Z^2_{\ge 0}$ with $1 \le |\al| \le 2$ and all $s>1$ and $\eps >0$, with an implicit constant independent of $N$.
\end{lemma}
\begin{proof} Fix $N \in \N$ and $(t_1, t_2) \in [0,1]^2$. From \eqref{YFun} and 
the Poisson summation formula \eqref{poisson}, we have
\begin{align}
\big( \Pii_{\le N}^2  - \Pii^2_{\le 10} \big) F_j (t_1, t_2, x) (t_1,t_2,x) = \sum_{k \in \Z^2} (f_{N,j} - f_{10,j})(t_1, t_2 ,x + 2 \pi k)
\label{Ysg79}
\end{align}
for any $x \in \T^2  \cong  [-\pi, \pi)^2$, where $f_{K,j}$ for $j \in \{1,2,3\}$ are the functions given by 
\begin{align}
\begin{split}
f_{K,1}(t_1, t_2 ,x) & = \frac{1}{(2 \pi)^2} \int_{\R^2} \frac{\sin( (t_1 - t_2) |\xi| )}{ |\xi| \jb \xi ^2} \chi_K^2( \xi ) e^{i  \xi \cdot x} d \xi \\
f_{K,2}(t_1, t_2 ,x) & = \frac{1}{(2 \pi)^2} \int_{\R^2} \frac{\cos( (t_1 - t_2) |\xi| )}{ \jb \xi ^4} \chi_K^2( \xi ) e^{i  \xi \cdot x} d \xi \\
f_{K,3}(t_1, t_2 ,x) & =\frac{1}{(2 \pi)^2} \int_{\R^2} \frac{\cos( (t_1 + t_2) |\xi| )}{ \jb \xi ^4} \chi_K^2( \xi ) e^{i  \xi \cdot x} d \xi
\end{split}
\label{Ysg80}
\end{align}
for any $x\in\R^2$ and $K \in \N$. By arguing as in \eqref{Ysg3}-\eqref{c3} in Step 1 in the proof of Proposition \ref{PROP:cov}, we have
\begin{align}
\bigg\|  \sum_{k \in \Z^2\setminus\{0\}}\partial^{\al}_x(f_{N,j}-f_{10,j})(t_1, t_2,x + 2 \pi k) \bigg\|_{L_x^{\infty}([-\pi, \pi)^2)} \les 1
\label{Ysg82}
\end{align}
for each $j \in \{1,2,3\}$ and any multi-index $\al \in \Z^2_{\ge 0}$ with $|\al| \le 2$, uniformly in 
 $0\le t_1, t_2 \le 1$. Therefore, since the functions $\Pii_{\le 10}^2 F_j$ and $f_{10,j}$ are smooth, \eqref{Ysg78} follows from \eqref{Ysg79}, \eqref{Ysg82} and the bound
 \begin{align}
 \big| \partial_x^{\al} f_{N,j} (t_1, t_2, x) \big| \les_\eps \mc H_N(t_1 - t_2, x; s) + |x|^{-\eps}
 \label{Ysg83}
\end{align}
for any $x \in [-\pi, \pi)^2 \cong \T^2$, any $\al \in \Z^2_{\ge 0}$ with $1 \le |\al| \le 2$, $j \in \{1,2,3\}$ and all $s>1$ and $\eps >0$.

Fix $\al \in \Z^2_{\ge 0}$ with $1 \le |\al| \le 2$. We first consider the contribution of $f_{N,1}$. By Poisson's formula \eqref{poisson2}-\eqref{poisson3}, we have 
\begin{align}
f_{N,1} (t_1, t_2, \cdot) =   W(|t_1- t_2|, \cdot) * G *\nu_N,
\end{align}
where $G$ is as in \eqref{green2} and $\nu_N = \F_x^{-1}[ \chi_N^2]$. Note that by integration by parts, $\nu_N$ satisfies \eqref{Gd0}. Hence, \eqref{Ysg83} for $j=1$ follows immediately from \eqref{Ysg32} and \eqref{Ysg32b} in Lemma \ref{LEM:wave_conv_green}.

Now, we look at the contribution of $f_{N,2}$ and $f_{N,3}$. By proceeding as in \eqref{c8} (where $\ft \eta_N$ is replaced with $\chi_N^2$) and from \eqref{Ysg81} and \eqref{poisson2}-\eqref{poisson3}, we have 
\begin{align}
\begin{split}
&f_{N,2}(t_1, t_2, x) \\
& \quad  = 2\pi  (G * G * \nu_N)(x) - \frac{1}{(2\pi)^2} \int_0 ^{|t_1 - t_2|} \int_{\R^2} \frac{\sin(t' |\xi|)}{|\xi|} \frac{1}{\jb \xi^2} \frac{|\xi|^2}{\jb \xi^2} \,\chi_N^2(\xi) e^{i \xi \cdot x} d \xi dt' \\
& \quad = 2\pi   (G * G * \nu_N)(x) - \frac{1}{(2\pi)^2}\int_0 ^{|t_1 - t_2|} \int_{\R^2} \frac{\sin(t' |\xi|)}{|\xi|} \frac{1}{\jb \xi^2}\,\chi_N^2(\xi)  e^{i \xi \cdot x}d \xi dt' \\
& \qquad \qquad  + \frac{1}{(2\pi)^2}  \int_0 ^{|t_1 - t_2|} \int_{\R^2} \frac{\sin(t' |\xi|)}{|\xi| \jb \xi^4 } \,\chi_N^2(\xi) e^{i \xi \cdot x} d \xi dt' \\
& \quad = 2\pi (G * G * \nu_N)(x) - \int_0^{|t_1 - t_2|} (W(t', \cdot) * G * \nu_N )(x) dt' \\
& \qquad \qquad +2 \pi \int_0 ^{|t_1 - t_2|} \int_{\R^2} \frac{\sin(t' |\xi|)}{|\xi| \jb \xi^4 } \,\chi_N^2(\xi) e^{i \xi \cdot x} d \xi dt' \\
& =: \1(x) - \II(x) + \III(x).
\end{split}
\label{Ysg84}
\end{align}
By \eqref{Ysg35} in Lemma \ref{LEM:green_der} (ii), we have
\begin{align}
|\partial_x^{\al} \1(x)| \les_\eps |x|^{-\eps}
\label{Ysg85}
\end{align}
for any $x \in [-\pi, \pi)^2$ and $\eps >0$. From Lemma \ref{LEM:wave_conv_green}, we have
\begin{align}
|\partial_x^{\al} \II(x)| \les \int_0^{|t_1 - t_2|} |x|^{-\eps} \cdot |t' - |x||^{-1+\frac\eps2} dt'  \les_\eps |x|^{-\eps}
\label{Ysg86}
\end{align}
for any $x \in [-\pi, \pi)^2$ and $\eps >0$. Since $\III$ decays rapidly on the Fourier side, we immediately have that
\begin{align}
|\partial_x^{\al} \III(x)| \les 1
\label{Ysg87}
\end{align}
for any $x \in [-\pi, \pi)^2$. Therefore, by \eqref{Ysg84}, \eqref{Ysg85}, \eqref{Ysg86} and \eqref{Ysg87} we deduce that $f_{N,2}$ satisfies \eqref{Ysg83}. By arguing similarly, one shows that $f_{N,3}$ also satisfies \eqref{Ysg83}. This finishes the proof.
\end{proof}

\begin{lemma}\label{LEM:covYsg2}
Fix $N \in \N$ and let $\Ld_N$ and $\mc H_N$ be as in \eqref{L} and \eqref{singN}, respectively. Then, the following bound holds:
\begin{align}
\big| \partial_x^{\al} \Ld_N(t,x) \big| \les_\eps \mc H_N(t, x; s) + |x|^{-\eps}
\label{Ysg88}
\end{align}
for any $(t,x) \in [0,1] \times \T^2$, $\al \in \Z^2_{\ge 0}$ with $1 \le |\al| \le 2$, $s>|\al|$ and all $\eps >0$, with an implicit constant independent of $N$.
\end{lemma}

\begin{proof}
Fix $N \in \N$ and $t \in [0,1]$. By arguing as in the proof of Lemma \ref{LEM:covYsg1}, it suffices to prove the following estimate
 \begin{align}
 \big| \partial_x^{\al} f_{N} (t, x) \big| \les_\eps \mc H_N(t, x; s) + |x|^{-\eps} 
 \label{Ysg89}
\end{align}
for any $x \in [-\pi, \pi)^2 \cong \T^2$, $\al \in \Z^2_{\ge 0}$ with $1 \le |\al|\le 2$, $s > |\al|$ and all $\eps >0$. Here, $f_N$ is as in \eqref{c2}.

Fix $\al \in \Z^2_{\ge 0}$ with $1 \le |\al| \le 2$. By proceeding as in \eqref{c8} (where $\ft \eta_N$ is replaced with $\chi_N^2$), we have
\begin{align}
f_N(t,x) = (G * \nu_N)(t,x) - \frac{1}{(2\pi)^2}\int_0^{t}\int_{\R^2} \frac{\sin( t' |\xi| )}{|\xi|} \cdot \frac{|\xi|^2}{\jb{\xi}^2} \, \chi_N^2 (\xi) e^{i \xi \cdot x} d \xi dt',
\label{Yd2}
\end{align}
where $G$ is as in \eqref{green2} and $\nu_N = \F_x^{-1}[ \chi_N^2]$. By arguing as in \eqref{c9} and from \eqref{Ysg81} and \eqref{poisson2}-\eqref{poisson3}, we have
\begin{align}
\begin{split}
& \frac{1}{(2\pi)^2}\int_0^{t}\int_{\R^2} \frac{\sin( t' |\xi| )}{|\xi|} \cdot \frac{|\xi|^2}{\jb{\xi}^2} \chi_N^2 (\xi) e^{i \xi \cdot x} d \xi dt' \\
& \qquad = \frac{1}{(2\pi)^2}\int_0^{t}\int_{\R^2} \frac{\sin( t' |\xi| )}{|\xi|}  \chi_N^2 (\xi) e^{i \xi \cdot x} d \xi dt' -\frac{1}{(2\pi)^2}\int_0^{t}\int_{\R^2} \frac{\sin( t' |\xi| )}{|\xi| \jb \xi^2}  \chi_N^2 (\xi) e^{i \xi \cdot x} d \xi dt' \\
& \qquad =: A_N(t,x) - B_N(t,x) - \II(t,x),
\end{split}
\label{Ysg90}
\end{align}
where $\II$ is as in \eqref{Ysg84} and $A_N$ and $B_N$ are given by
\begin{align*}
A_N(t,x) & = \frac{1}{2\pi} \big( \ind_{B(0,t)} \log \big( t + \sqrt{t^2 - |\cdot|^2} \big)\big) * \nu_N, \\
B_N(t,x) & =  \frac{1}{2\pi} \big( \ind_{B(0,t)} \log |\cdot | \big) * \nu_N.
\end{align*}
Combining \eqref{Yd2} and \eqref{Ysg90} gives
\begin{align}
f_N(t,x) = G_{N}(x) - A_N(t,x) + B_N(t,x) + \II(t,x),
\label{Ysg91}
\end{align}
where $G_{N} = G * \nu_N$. Let us now fix $x \in [-\pi, \pi)^2$. We divide our analysis in several cases.

\medskip

\noi
{\bf $\bul$ Case 1: $|x| \gg t$.}\quad In this case, we write 
\[ G_N(t,x) + B_N(t,x) = G_{N,t}(x) + F_{N,t}(x),\]
where $G_{N,t}$ and $F_{N,t}$ are the functions
\begin{align*}  
G_{N,t} & =\big( \ind_{B(0,t)^c} G\big) * \nu_N, \\
F_{N,t} &= \big( \ind_{B(0,t)} \big( G + \frac{1}{2\pi}  \log |\cdot |\big) \big) * \nu_N.
\end{align*}
Note that by Corollary \ref{COR:green_wave} (ii) and since the function $G + \frac{1}{2\pi} \log |\cdot| $ is smooth on $\R^2$, we have
\begin{align}
\begin{split}
|\partial_x^{\al} G_{N,t}(x)| + | \partial_x ^\al F_{N,t}(x) | & \les  \jbb{\, \log \! \big(  t + |x| + N^{-1} \big)}  \min \! \big\{ N^{|\al|},| t - |x||^{-|\al|} \big\} \\
 & \les_{\eps}  \min \! \big\{ N^{|\al|}, (t+ |x|)^{-\frac12} | t - |x||^{\frac12-|\al|} \big\} \\
 &  \les_{s} \mc H_N(t, x; |\al|)
 \end{split}
\label{Ysg92}
\end{align}
for any $\eps >0$, as $|x| \gg t$. Similarly, by applying Lemma \ref{LEM:green_wave2} to the function $W_t =  \log \big( t + \sqrt{|t^2 - |\cdot|^2|}\,\big)$, we also have
\begin{align}
|\partial_x^{\al} A_N(t,x) | \les_{s} \mc H_N(t, x; |\al|)
\label{Ysg93}
\end{align}
for any $\eps >0$. Therefore \eqref{Ysg89} follows from \eqref{Ysg91}, \eqref{Ysg92}, \eqref{Ysg93} and \eqref{Ysg86} in this case.

\medskip

\noi
{\bf $\bul$ Case 2: $|x| \ll t$.}\quad The proof of \eqref{Ysg89} in this case is identical to that of Case 1 and we omit details.

\medskip

\noi
{\bf $\bul$ Case 3: $|x| \sim t$.}\quad Here, we first consider the case $|\al| = 1$. Then, by Lemma \ref{LEM:Dder}, we have that
\begin{align}
\begin{split}
\partial_x ^\al A_N(t,x) & = \big( \ind_{B(0,t)} \partial_x^{\al} \big\{  \log \big( t + \sqrt{|t^2 - |\cdot|^2|} \big)\big) \big\}  \big) * \nu_N \\
& \qquad \quad - \log(t) \int_{\mb S^1(t)}  \nu_N(x-y) \al \cdot y \, d \s_t(y) \\
& =: A_N^{\al}(t,x) -   \log(t) \int_{\mb S^1(t)}  \nu_N(x-y) \al \cdot y \, d \s_t(y),\\
\partial_x ^\al B_N(t,x) & = \big( \ind_{B(0,t)} \partial_x^{\al} \{  \log |\cdot| \}  \big) * \nu_N \\
& \qquad \quad - \log(t) \int_{\mb S^1(t)}  \nu_N(x-y) \al \cdot y \, d \s_t(y) \\
& =: B_N^{\al}(t,x) - \log(t) \int_{\mb S^1(t)}  \nu_N(x-y) \al \cdot y \, d \s_t(y).
\end{split}
\label{Ysg94}
\end{align}
Thus, the contribution of the boundary terms in \eqref{Ysg94} to \eqref{Ysg91} vanish and by Lemmas \ref{LEM:green_der} (i) and \ref{LEM:green_wave0}, we have
\begin{align}
\begin{split}
& |\partial_x^{\al} A_N(t,x) - \partial_x^{\al}B_N(t,x)| = |A_N^{\al}(t,x) - B_N^{\al}(t,x)|  \\
& \qquad \les \min \! \big\{ N, \big| t^2 - |x|^2\big|^{-\frac{1}{2}}\big\} \, \jbb{ \, \log \! \big( \! \min \! \big\{N, |t- |x||^{-1}\big\}\big) } + \big(|x| +N^{-1}\big)^{-1} \\
& \qquad \les_s \mc H_N(t,x;s)
\label{Ysg95}
\end{split}
\end{align}
for any $s >1$, as $|x| \sim t$. Therefore, \eqref{Ysg89} for $|\al| =1$ follows from \eqref{Ysg91}, \eqref{Ysg95}, the bound on $G_N$ provided by Lemma \ref{LEM:green_der} (i) (using $|x| \sim t$) and \eqref{Ysg86}.

Now assume $|\al| = 2$ and write $\al = \al_1 + \al_2$ for $\al_1,\al_2 \in \Z^2_{\ge 0}$ with $|\al_1| = |\al_2| = 1$. From Lemma \ref{LEM:green_wave3}, we have
\begin{align}
|\partial_x^{\al_2} A_{N}^{\al_1}(t,x) | \les_s \mc H_N(t,x;s)
\label{Ysg96}
\end{align}
and
\begin{align}
\begin{split}
|\partial_x^{\al_2} B_{N}^{\al_1}(t,x) | & \les \big( \! \min \! \big\{ N^2, |x|^{-1} |t - |x||^{-1} \big\} + \min \! \big\{ N^2, |x|^{-2} \big\}  \big) \\
& \qquad \qquad \times \jbb{ \, \log \! \big( \! \min \!  \big\{N, |x|^{-1}, |t- |x||^{-1}\big\}\big) }\\
& \les  \min \! \big\{ N^2, (t+|x|)^{-\frac12} |t - |x||^{-\frac32} \big\} \jbb{ \, \log \! \big( \! \min \!  \big\{N, |t- |x||^{-1}\big\}\big) } \\
& \les_s \mc H_N(t,x; s)
\end{split}
\label{Ysg97}
\end{align}
for any $s>2$, as $|x| \sim t$. Thus, by \eqref{Ysg94}, \eqref{Ysg95}, \eqref{Ysg96} and \eqref{Ysg97}, we have that
\begin{align}
\begin{split}
 |\partial_x^{\al} A_N(t,x) - \partial_x^{\al}B_N(t,x)|  & =  | \partial_x^{\al_2} A_N^{\al_1}(t,x) - \partial_x^{\al_2} B_N^{\al_1}(t,x)| \\
 & \les_s \mc H_N(t,x;s)
 \end{split}
 \label{Ysg98}
\end{align}
for any $s >2$. Therefore, \eqref{Ysg89} for $|\al| =2$ follows from \eqref{Ysg91}, \eqref{Ysg98}, the bound on $G_N$ provided by Lemma \ref{LEM:green_der} (i) (using $|x| \sim t$) and \eqref{Ysg86}.
\end{proof}

\subsection{Imaginary Gaussian multiplicative chaos}\label{SUBSEC:3-2a}

In this subsection, 
we establish various regularity properties
of the (truncated) imaginary Gaussian multiplicative chaos $\Ta_N^{\eps_0}$ defined in \eqref{t1}. Recall the definitions of the space-time localizations in \eqref{proj4a}, \eqref{proj4a}, \eqref{proj4}, \eqref{proj4} and \eqref{proj4}. For $\al, \eps >0$, we define the space $\mathcal Z ^{\al,\eps}([0,1])$ by the norm

\noi
\begin{align}
\begin{split}
\|\Ta\|_{\mathcal{Z}^{\al,\eps}([0,1])} &:= \|\P_{\textup{lo}}\Q^{\textup{hi,hi}}\ind_{[0,1]} \Ta\|_{L^{\infty}_{t,x}} 
+ \|\qf_{-\frac12-\eps} \P_{\textup{hi}}\Q^{\textup{hi,hi}}\ind_{[0,1]} \Ta\|_{Y_{-\frac12-3\eps}^{-\al,-\frac12-\eps}} \\
& \qquad \qquad + \|\ind_{[0,1]}  \Ta\|_{\Ld_{\infty}^{-\al,-\frac12+\eps}} 
+ \|\ind_{[0,1]}  \Ta\|_{\Ld_{\frac{1+\eps}{\eps}}^{-\al-\frac12,0}}.
\end{split}
\label{normZ}
\end{align}
We emphasize that the restriction here is as in Remark~\ref{RMK:loc}.

The main result of this subsection is as follows.

\noi
\begin{proposition}\label{PROP:sto_main}
Let $0<\be^2<4\pi$, $\eps_0 \in \{+ 1, -1\}$. Then, for any $\eps >0$ and $\al>\frac{\be^2}{4\pi} - \frac12 +2 \eps$, $\{\Ta^{\eps_0}_N\}_{N\in\N}$ is a Cauchy sequence in $\mathcal{Z}^{\al,\eps}([0,1])$, $\muu_1 \otimes \mb P$-almost surely. We denote by $\Ta^{\eps_0}$ its limit.
\end{proposition}

In the remainder of this section, we establish Proposition \ref{PROP:sto_main}. Its proof is a straightforward consequence of the following results.

\begin{lemma}\label{LEM:sto1}
Let $0<\be^2<4\pi$, $\al>\frac{\be^2}{4\pi}$ and $\eps_0 \in \{+ 1, -1\}$. Then, $\{\Ta^{\eps_0}_N\}_{N\in\N}$ is a Cauchy sequence in $C([0,1];W^{-\al,\infty}(\T^2))$, $\muu_1 \otimes \mb P$-almost surely.
\end{lemma}
The proof of Lemma~\ref{LEM:sto1} can be found in \cite[Proposition 5.7]{Zine1}. See also \cite[Lemma 2.2]{ORSW2} and \cite[Proposition 1.1]{ORSW1}.

The following two propositions establish nonlinear smoothing 
for the imaginary Gaussian multiplicative chaos.

\begin{proposition}\label{PROP:sto2}
Let $0<\be^2<6\pi$ and $\eps_0 \in \{+ 1, -1\}$. Then, for any small $\eps >0$ and 
 $\al>\frac{\be^2}{4\pi} - \frac12 + \eps$, $\{ \Ta^{\eps_0}_N\}_{N\in\N}$ is a Cauchy sequence in $\Ld^{-\al, -\frac12+\eps}_{\infty}([0,1])$, $\muu_1 \otimes \mb P$-almost surely.
\end{proposition}
Proposition~\ref{PROP:sto2} is proved in Subsection~\ref{SUBSEC:sto3}.

\begin{proposition}\label{PROP:sto3} 
Let $0<\be^2<4\pi$ and $\eps_0 \in \{+ 1, -1\}$. Then, for any small $\eps =  >0$ and ${\al>\frac{\be^2}{4\pi} - \frac12 + 2 \eps}$, $\{\qf_{-\frac12-\eps} \P_{\textup{hi}}\Q^{\textup{hi,hi}}\Ta^{\eps_0}_N\}_{N\in\N}$ is a Cauchy sequence in $Y_{-\frac12-3\eps}^{-\al, -\frac12-\eps}([0,1])$, $\muu_1 \otimes \mb P$-almost surely.
\end{proposition}

The main step in the proof of Proposition \ref{PROP:sto3} is the following pointwise moment estimate. For $x = (x^1, x^2) \in \T^2$, we denote by $\partial_{x^\l}$ for $\l \in \{1,2\}$ the derivative with respect to the $\l^{\text{th}}$ coordinate of $x$

\begin{proposition}\label{PROP:sto4}
Let $0<\be^2<4\pi$, $\eps_0 \in \{+ 1, -1\}$, $N_0 \in 2^{\N}$ and $(N, N_1, N_2) \in \N^3$ with $N_2 \ge N_1$. Then, the following bounds holds:
\begin{align}
\max_{\l \in \{1,2\}} \sup_{x \in \T^2}   \E_{\muu_1 \otimes \PP} \Big[\big|\big( \Box^{-\frac12-\eps} \, \partial_{x^\l}  (\P_{N_0} \ind_{[0,1]} \Ta^{\eps_0}_N) \big)(t,x)\big|^2\Big] & \les_{\eps} N_0^{\frac{\be^2}{2\pi} + \eps} \jb t^{4\eps}
\label{goal1}, \\
\max_{\l \in \{1,2\}} \sup_{x \in \T^2}   \E_{\muu_1 \otimes \PP} \Big[\big|\big( \Box^{-\frac12-\eps} \, \partial_{x^\l}  (\P_{N_0} \ind_{[0,1]}(\Ta^{\eps_0}_{N_1} - \Ta^{\eps_0}_{N_2}) \big)(t,x)\big|^2\Big] & \les_{\eps} N_0^{6} \, \jb t^{4\eps} N_1^{-\dl}
\label{goal1bbb}
\end{align}
for any small $\eps >0$ and small $\dl >0$, with implicit constants independent of $N$, $N_1$, $N_2$ and $N_0$. Here, $\ind_{[0,1]}$ is the indicator function of the time interval $[0,1]$ and $\P_{N_0}$ is as in \eqref{proj1}.
 \end{proposition}

\subsection{Proof of Proposition \ref{PROP:sto2}}
\label{SUBSEC:sto3}

In this subsection, we present a proof of Proposition \ref{PROP:sto2}.
We first state 
 a charge cancellation lemma adapted to the time-dependent damped wave setting. 
Given $N \in \N$, we introduce the potential
function $ \J_N $  by
\begin{align}
\J_N(t,x) =  \big( |t| + |x| + N^{-1} \big), \quad (t, x) \in \R_+\times \T^2.
\label{pot}
\end{align}

\noi
We state in the next lemma the key charge cancellation identity observed in \cite{HS, ORSW1} adapted to our setting.

\noi
\begin{lemma}\label{LEM:charge}
Let $N \in \N$,  $p \in \N$, and $\ld > 0$.
 Let $\{\eps_j\}_{j=1,...,2p} \in \{\pm 1 \}^{2p}$ be a collection of signs such that 
 $\eps_j = 1$ if $j$ is even and $\eps_j = -1$ if $j$ is odd. Then, the following estimate holds:
\begin{align}
\begin{split}
& \prod_{1 \le j < k \le 2p } \J_N \big(  z_j - z_k\big) ^{\eps_j \eps_k \ld}
 \les \max_{\s \in \mf S_p} \prod_{j = 1}^p \J_N \big( z_{2 j} -  z_{2\s(j)-1} \big)  ^{-\ld},
\end{split}
\label{charge1}
\end{align}

\noi
for any set of $2p$ space-time points $\{ z_j = ( t_j, x_j) 
\in \R \times \T^2
:  1 \le j \le 2p\}$, 
where $\mf S_p$ denotes the symmetric group on $\{1, \dots, p\}$.

\end{lemma}

\begin{proof}
The proof follows from Proposition \ref{PROP:cov} and a slight variation of the presentation in~\cite{ORSW1}. We omit details.

\end{proof}

We now present a proof of Proposition \ref{PROP:sto2}.

\begin{proof}[Proof of Proposition \ref{PROP:sto2}]
Fix  $0 < \be^2 < 6\pi$ and let $p \in \N$, finite $q \geq 1$ 
and $\al >  \frac{\be^2 }{4\pi} - \frac 12 $.
Without loss of generality, we assume $\al < 2$ in the following. In this proof, we fix $\eps_0 = +$ for convenience and write $\Ta_N$ for $\Ta^+_N$.

\medskip

\noi
{\bf $\bul$ Step $\1$: boundedness.}\quad Fix small  $\dl > 0$ (to be chosen later).
From \eqref{loc1},   Sobolev's inequality (with $q\dl >2$), 
and~\eqref{bessel2a}, 
 we have
\begin{align}
\begin{split}
 \|\Ta_N& \|_{L^{2p}(\muu_1 \otimes \PP) \Ld^{-\al, -\frac 12 +\eps}_\infty([0,1])} 
 =
\|\ind_{[0, 1]} \cdot \Ta_N\|_{L^{2p}(\muu_1 \otimes \PP) \Ld^{-\al, -\frac 12 +\eps}_\infty(\R)} \\
& \les 
\|\ind_{[0, 1]} \cdot \Ta_N\|_{L^{2p}(\muu_1 \otimes \PP) \Ld^{-\al+\dl, -\frac 12 +\eps+\dl}_q(\R)} \\
& \les \big\|
J^{(t)}_{\frac 12 - \eps-\dl}*_t 
 \jb{\nabla_x}^{-\al+\dl}
 (\ind_{[0, 1]} \cdot 
 \Ta_N)\big\|_{L^{2p}_\o L^q_{t, x}}\\
& \les 
\sum_{j = 1}^2\big\|
J^{(t), j }_{\frac 12 - \eps-\dl}*_t 
 \jb{\nabla_x}^{-\al+\dl}
 (\ind_{[0, 1]} \cdot 
 \Ta_N)\big\|_{L^{2p}(\muu_1 \otimes \PP) L^q_{t, x}}\\
 & =: \sum_{j = 1}^2 A_{N, j}, 
 \end{split}
\label{E0a}
\end{align}

\noi
where $*_t$ denotes a convolution in the temporal variable
and 
\begin{align}
J^{(t), 1 }_{\frac 12 - \eps-\dl}
= 
\ind_{|t| \le 3}\cdot J^{(t)}_{\frac 12 - \eps-\dl}
\qquad 
\text{and}\qquad 
J^{(t), 2}_{\frac 12 - \eps-\dl}
= 
\ind_{|t| > 3}\cdot J^{(t)}_{\frac 12 - \eps-\dl},
\label{E0b}
\end{align}
where $J^{(t)}_b$ is as in \eqref{bessel2a}.

We first estimate $A_{N,1}$ 
on the right-hand side of \eqref{E0a}. 
By Minkowski's inequality (with $2p \ge q$)
with   \eqref{E0b}, we have
\begin{align}
A_{N, 1} & \les 
\Big\|\|
J^{(t), 1 }_{\frac 12 - \eps-\dl}*_t 
 \jb{\nabla_x}^{-\al+\dl}
 (\ind_{[0, 1]} \cdot 
 \Ta_N) (t, x) 
 \|_{L^{2p}(\muu_1 \otimes \PP)}\Big\|_{L^q_t([-3, 4]; L^q_x)}.
\label{E0c}
\end{align}

\noi
Fix $t\in[-3,4]$ and $x\in\T^2$. 
Then, 
from \eqref{bessel0}
and 
\eqref{t1}, 
we have
 \begin{align}
 \begin{split}
\E_{\muu_1 \otimes \PP} & \Big[\big| 
J^{(t), 1 }_{\frac 12 - \eps-\dl}*_t 
\jb{\nabla}^{-\al+\dl}
 (\ind_{[0, 1]} \cdot 
\Ta_N)(t,x)\big|^{2p}\Big] \\
& = e^{p\be^2\s_N} \E_{\muu_1 \otimes \PP} \Bigg[ \bigg| 
\int_0^1\int_{\T^2} 
J_{\frac 12 - \eps - \dl}^{(t), 1} (t-s) J_{\al-\dl} (x-y) e^{i \be \Psi^{\textup{wave}}_N (s,y)} dy ds \bigg|^{2p} \Bigg]\\
& =   e^{p\be^2\s_N}
\int_{[0, 1]^{2p}}
\int_{(\T^{2})^{2p}} 
\E_{\muu_1 \otimes \PP} \bigg[  e^{i \be \sum_{j =1}^p  (\Psi^{\textup{wave}}_N (s_{2j},y_{2j}) - \Psi^{\textup{wave}}_N (s_{2j-1},y_{2j-1}))} \bigg] \\
& 
\hphantom{XXXXX}
\times  \prod_{k=1}^{2p} 
 J_{\frac 12 - \eps - \dl}^{(t), 1} (t-s_k)
  J_{\al-\dl} (x-y_k) \, d\vec y
  d\vec s, 
\end{split}
\label{E0}
\end{align}
\noi
where 
$d\vec s :=   ds_1\cdots d s_{2p}   $
and 
$d\vec y :=   dy_1\cdots d y_{2p}   $.
Noting that $ \sum_{j =1}^p  (\Psi^{\textup{wave}}_N (s_{2j},y_{2j}) - \Psi^{\textup{wave}}_N (s_{2j-1},y_{2j-1})) $ is a mean-zero Gaussian random variable, 
the explicit formula for the characteristic function of a Gaussian random variable yields
\begin{align}
\begin{split}
& \E \bigg[   e^{i \be \sum_{j =1}^p  (\Psi^{\textup{wave}}_N (s_{2j},y_{2j}) - \Psi^{\textup{wave}}_N (s_{2j-1},y_{2j-1}))} \bigg]\\
& \quad = 
e^{- \frac{\be^2}2 \E \big[| \sum_{j =1}^p  (\Psi^{\textup{wave}}_N (s_{2j},y_{2j}) - \Psi^{\textup{wave}}_N (s_{2j-1},y_{2j-1})) |^2\big]}  .
\end{split}
\label{E1}
     \end{align}

 Let  $\{\eps_j\}_{j=1,...,2p}$ be as in Lemma~\ref{LEM:charge}.
Then,  we can rewrite the expectation  in the exponent on the right-hand side 
of \eqref{E1} as 
\begin{align}
  \E\bigg[\Big|\sum_{j=1}^{2p}\eps_j\Psi^{\textup{wave}}_N(s_j,y_j)\Big|^2\bigg] 
  = \sum_{j,k=1}^{2p}\eps_j\eps_k \G_N(s_j-s_k,y_j-y_k),
  \label{E1b}
\end{align}

\noi
where  $\G_N$ is the space-time covariance defined in \eqref{cov}. 
From  \eqref{sig2}, we have $\G_N(0,0)=\s_N + O(1)$,
and thus
  \begin{equation} 
\eqref{E1} 
 \sim e^{- p \be^2 \sigma_N} e^{ - \be^2 \sum_{1\le j < k\le 2p} \eps_j\eps_k \G_N(s_j  - s_k,y_j-y_k)  } .
\label{E2}
\end{equation}

 \noi
 Then, from \eqref{E1}, \eqref{E2}, 
the two-sided bound \eqref{cov1} in Proposition \ref{PROP:cov}, 
 and  Lemma~\ref{LEM:charge}, we obtain
 \begin{equation}\label{E3}
 \begin{split}
 e^{p\be^2\s_N} 
 \E \bigg[  & e^{i \be \sum_{j =1}^p  (\Psi^{\textup{wave}}_N (s_{2j},y_{2j}) - \Psi^{\textup{wave}}_N (s_{2j-1},y_{2j-1}))} \bigg]\\
  & \sim  \prod_{1\le j < k \le 2p } \big(|s_j - s_k| + |y_j - y_k| + N^{-1}\big)^{\eps_j\eps_k \frac{\be^2 }{2 \pi}}\\
 &\les \max_{\s \in \mf S_p} \prod_{1\leq j \leq p}
  \big(|s_{2j} - s_{2\s(j)-1}| + |y_{2j} - y_{2\s(j)-1}| + N^{-1}\big)^{ -\frac{\be^2 }{2 \pi}}\\
&\le \sum_{\s \in \mf S_p} 
\prod_{1\leq j \leq p} 
\big(|s_{2j} - s_{2\s(j)-1}| + |y_{2j} - y_{2\tau(j)-1}| + N^{-1}\big)^{ -\frac{\be^2 }{2 \pi}}.
 \end{split}
 \end{equation}

 \noi
Hence,  from \eqref{E0} and \eqref{E3}
we obtain
 \begin{align}
\begin{split}
\E & \Big[\big| 
J^{(t), 1 }_{\frac 12 - \eps-\dl}*_t 
\jb{\nabla}^{-\al+\dl}
 (\ind_{[0, 1]} \cdot 
\Ta_N)(t,x)\big|^{2p}\Big] \\
 & \les \sum_{\s \in \mf S_p}
\int_{[0, 1]^{2p}}
 \int_{(\T^{2})^{2p}}   
  \prod_{1\leq j \leq p} \big(
  |s_{2j} - s_{2\s(j)-1}|  + 
  |y_{2j} - y_{2\s(j)-1}| + N^{-1}\big)^{ -\frac{\be^2}{2 \pi}} \\
& 
\hphantom{XXXX}
\times   \bigg(\prod_{k=1}^{2p}  
 |J_{\frac 12 - \eps - \dl}^{(t), 1} (t-s_k)|
|J_{\al-\dl} (x-y_k)|\bigg) d\vec yd\vec s.
\end{split}
\label{E3a}
 \end{align}

 In the following, we fix $\s \in \mf S_p$.
 Then,  it suffices to bound each pair of integrals:
 \begin{align*}
&   \int_0^1 \int_0^1
 \int_{\T^2}\int_{\T^2}\big(|s_j - s_k| + |y_{j} - y_{k}| + N^{-1}\big)^{ -\frac{\be^2 }{2 \pi}} \\
& \hphantom{XXXXX}
\times 
\bigg(\prod_{\l \in \{j, k\}}
 |J_{\frac 12 - \eps - \dl}^{(t), 1} (t-s_\l)|
 | J_{\al-\dl} (x-y_\l)|\bigg)  dy_{j}dy_{k}
 ds_j ds_k
 \end{align*}
  
 \noi
 for  an even integer $j=2,...,2p$ and $k=2\s(\tfrac{j}2)-1$.
  From \eqref{bessel2},  \eqref{bessel3}, and \eqref{E0b} with $0 < \al - \dl < 2$, 
  we can bound this integral by 
 \begin{align}
 \begin{split}
& \int_0^1 \int_0^1
 \int_{\T^2}  \int_{\T^2}\big(|s_j - s_k| + |y_{j} - y_{k}|+N^{-1}\big)^{ -\frac{\be^2 }{2 \pi}}\\
& 
\hphantom{XXXX}
\times \bigg(\prod_{\l \in \{j, k\}}
 |t-s_\l|^{-\frac 12 - \eps - \dl }
  |x-y_\l|^{\al-\dl-2}
\bigg) dy_{j}dy_{k} ds_j ds_k\\
&  \le 
\int_0^1 \int_0^1
 \int_{\T^2}\int_{\T^2}
 \big(|s_j -s_k|  + |y_{j} - y_{k}|\big)^{ -\frac{\be^2 }{2 \pi}}\\
& 
\hphantom{XXXX}
\times \bigg(\prod_{\l \in \{j, k\}}
 |t-s_\l|^{-\frac 12 - \eps - \dl }
  |x-y_\l|^{\al-\dl-2}
\bigg)
 dy_{j}dy_k
  ds_j ds_k, 
\end{split}
\label{E4}
  \end{align}
 uniformly in $N\in\N$.

\medskip
\noi
$\bullet$
{\bf Case 1:} 
 $|y_j-y_k|\sim |x-y_k|\gtrsim |x-y_j|$.
 \quad 
The symmetry allows us to handle the case
$|y_j-y_k|\sim |x-y_j|\gtrsim |x-y_k|$.

%

We first consider the case
$|s_j- s_k | \sim |t - s_k| \ges |t-s_j|$.
In this case, we have 
\begin{align}
\int_0^1 |t-s_k|^{-\frac 12 - \eps - \dl }
|t-s_k|^{-1 +2 \eps +3 \dl }
\int_{|t - s_j|\les |t - s_k|}
 |t-s_j|^{-\frac 12 - \eps - \dl } ds_j ds_k
 \les 1
 \label{E4b}
\end{align}

\noi
and thus 
\begin{align}
\begin{split}
\text{RHS of \eqref{E4}}
& \les \int_{\T^2} |x-y_k|^{\al-\dl-2-\frac{\be^2}{2\pi}
+ 1-2\eps - 3\dl} 
\int_{|x-y_j|\les |x-y_k|}
|x-  y_j|^{\al-\dl-2}
dy_jdy_k\\ 
&\les \int_{\T^2}|x-y_k|^{2\al - 1 -\frac{\be^2}{2\pi}
-2\eps - 5\dl
}dy_k
\les 1,
\end{split}
\label{E4a}
\end{align}

\noi
provided that  $\al > \frac{\beta^2}{4\pi} -\frac 12 + \eps$
(by choosing $\dl > 0$ sufficiently small).
By symmetry, 
the same conclusion holds when
$|s_j- s_k | \sim |t - s_j| \ges |t-s_k|$.

Next, we consider the case
$ |t-s_j| \sim |t - s_k| \gg |s_j- s_k |$.
In this case, we have 
\begin{align}
\int_0^1 |t-s_k|^{-1- 2\eps - 2\dl }
\int_{|s_j - s_k|\les |t - s_k|}
 |s_j-s_k|^{-1+ 2 \eps +3 \dl } ds_j ds_k
 \les 1.
 \label{E4c}
\end{align}

\noi
Then, \eqref{E4a} holds, 
provided that  $\al > \frac{\beta^2}{4\pi} -\frac 12 + \eps$
(by choosing $\dl > 0$ sufficiently small).

\medskip

\noi
$\bullet$
{\bf Case 2:} 
$|x-y_j|\sim |x-y_k|\gtrsim |y_j-y_k|$.
\quad 
From \eqref{E4b} and \eqref{E4c}, we have 
\begin{align}
\begin{split}
\text{RHS of \eqref{E4}}
& \les \int_{\T^2}|x-  y_k  |^{2(\al-\dl)-4} \int_{|y_j-y_k|\les |x-y_k|}
|y_j-y_k|^{-\frac{\be^2}{2\pi}+ 1-2\eps - 3\dl}dy_jdy_k\\
& \les \int_{\T^2}|x-y_k|^{2\al-1-\frac{\beta^2}{2\pi} - 2\eps - 5\dl}dy_k \les 1, 
\end{split}
\label{E45}
\end{align}

\noi
provided that  $\al > \frac{\beta^2}{4\pi} -\frac 12 + \eps$
and $\be^2 < 6\pi - 4\pi \eps$
(by choosing $\dl > 0$ sufficiently small).

\medskip

Hence, \eqref{E3a} and 
the estimates on \eqref{E4}, we obtain
\begin{align}
\E & \Big[\big| 
J^{(t), 1 }_{\frac 12 - \eps-\dl}*_t 
\jb{\nabla}^{-\al+\dl}
 (\ind_{[0, 1]} \cdot 
\Ta_N)(t,x)\big|^{2p}\Big] \les 1, 
\label{E44}
\end{align}

\noi
 uniformly in $N\in\N$, 
 provided that  
\begin{align}
\al > \frac{\beta^2}{4\pi} -\frac 12 + \eps
\qquad \text{and}
\qquad 
\be^2 < 6\pi - 4\pi \eps
\label{E5a}
\end{align}
(by choosing $\dl > 0$ sufficiently small).
Putting together \eqref{E0c} and \eqref{E44}, 
we conclude that 
\begin{align}
A_{N, 1} \les 1,
\label{E5}
\end{align}

\noi
uniformly in $N\in\N$, 
under the same condition.

\medskip

Let us now briefly discuss how to handle $A_{N, 2}$
on the right-hand side of \eqref{E0a}. We observe from \eqref{bessel3} that $|J_{\frac 12 - \eps - \dl}^{(t), 2} (t-s)| \les e^{-|t|}$ for $s \in [0,1]$. Hence, by proceeding as in the case of $A_{N,1}$, we have
 \begin{align}
\begin{split}
\E & \Big[\big| 
J^{(t), 2}_{\frac 12 - \eps-\dl}*_t 
\jb{\nabla}^{-\al+\dl}
 (\ind_{[0, 1]} \cdot 
\Ta_N)(t,x)\big|^{2p}\Big] \\
 & \les \sum_{\s \in \mf S_p}
\int_{[0, 1]^{2p}}
 \int_{(\T^{2})^{2p}}   
  \prod_{1\leq j \leq p} \big(
  |s_{2j} - s_{2\s(j)-1}|  + 
  |y_{2j} - y_{2\s(j)-1}| + N^{-1}\big)^{ -\frac{\be^2}{2 \pi}} \\
& 
\hphantom{XXXXXXX}
\times   \bigg(\prod_{k=1}^{2p}  
 |J_{\frac 12 - \eps - \dl}^{(t), 2} (t-s_k)|
|J_{\al-\dl} (x-y_k)|\bigg) d\vec yd\vec s \\
 & \les e^{-2p|t|} \sum_{\s \in \mf S_p}
\int_{[0, 1]^{2p}}
 \int_{(\T^{2})^{2p}}   
  \prod_{1\leq j \leq p} \big(
  |s_{2j} - s_{2\s(j)-1}|  + 
  |y_{2j} - y_{2\s(j)-1}| + N^{-1}\big)^{ -\frac{\be^2}{2 \pi}} \\
& 
\hphantom{XXXXXXX}
\times   \bigg(\prod_{k=1}^{2p} 
|J_{\al-\dl} (x-y_k)|\bigg) d\vec yd\vec s \\
& \les e^{-2 p |t|},
\end{split}
\label{E100}
 \end{align}
uniformly in $N\in\N$, 
provided that  $\al > \frac{\beta^2}{4\pi} -\frac 12 + \eps$
and $\be^2 < 6\pi - 4\pi \eps$
(by choosing $\dl > 0$ sufficiently small). This is due to the fact that the integrand on the right-hand-side of second-to-last line of \eqref{E100} is less singular than than that of \eqref{E4}.

Therefore, from~\eqref{E0b} and Minkowski's inequality, 
we have
\begin{align}
\begin{split}
A_{N, 2} & \les \big\| J^{(t), 2 }_{\frac 12 - \eps-\dl}*_t 
 \jb{\nabla_x}^{-\al+\dl}
 (\ind_{[0, 1]} \cdot 
 \Ta_N) \big\|_{L^{2p}(\muu_1 \otimes \PP) L^q_{x}L^q_t }\\
&\les \Big\|\|
J^{(t), 2 }_{\frac 12 - \eps-\dl}*_t 
 \jb{\nabla_x}^{-\al+\dl}
 (\ind_{[0, 1]} \cdot 
 \Ta_N) (t, x) 
 \|_{L^{2p}(\muu_1 \otimes \PP)}\Big\|_{L^q_t(\R; L^q_x)}\\
 & \les \| e^{-2p |t|}\|_{L^q_t}\\
 &  \les 1.
 \end{split}
\label{E6}
\end{align}
uniformly in $N\in\N$, 
provided that  $\al > \frac{\beta^2}{4\pi} -\frac 12 + \eps$
and $\be^2 < 6\pi - 4\pi \eps$.

Therefore, from \eqref{E0a}, \eqref{E5}, and \eqref{E6}, 
we obtain
\begin{align}
 \|\Ta_N \|_{L^{2p}(\muu_1 \otimes \PP) \Ld^{-\al, -\frac 12 +\eps}_\infty([0,1])} 
 \les 1.
 \label{E9}
\end{align}

\noi
uniformly in $N\in\N$, 
under the condition \eqref{E5a}.

\medskip

\noi
{\bf $\bul$ Step $\II$: convergence.}\quad
Next, we discuss  convergence of $\Ta_N$.
We first  estimate the contribution 
from 
$J^{(t), 1}_{\frac 12 - \eps-\dl}$
in~\eqref{E0b}.
Let $N_2 \geq N_1 \geq1$.
By 
repeating the  computation in Step $\1$ with $p=1$ and $\Ta_{N_1}-\Ta_{N_2}$ in place of $\Ta_N$, 
we have 
\begin{align}
\begin{split}
\E & \Big[ \big| 
J^{(t), 1}_{\frac 12 - \eps-\dl}
*_t  \jb{\nabla}^{-\al+ \dl}\big\{ \ind_{[0, 1]} ( \Ta_{N_1}-\Ta_{N_2})\big\}(t, x) \big|^{2} \Big]\\ 
& =  \int_{[0, 1]^2}
\int_{(\T^2)^2}
J^{(t), 1}_{\frac 12 - \eps-\dl}(t - s_1)
J^{(t)}_{\frac 12 - \eps-\dl}(t - s_2)
J_{\al-\dl} (x-y_1)J_{\al-\dl} (x-y_2) \\
&\hphantom{XXXX}
\times 
\E\bigg[\Big(e^{ \frac{\be^2}2 \sigma_{N_1}}e^{i \be \Psi_{N_1} (s_1,y_1)}
-e^{ \frac{\be^2}2 \sigma_{N_2}}e^{i \be \Psi_{N_2} (s_1,y_1)}\Big)\\
&\hphantom{XXXX}
\times \Big(e^{ \frac{\be^2}2 \sigma_{N_1}}
e^{-i \be \Psi_{N_1} (s_2, y_2)}-e^{ \frac{\be^2}2 \sigma_{N_2}}e^{-i \be \Psi_{N_2} (s_2, y_2)}\Big)\bigg] 
d\vec y d \vec s \\
&  = \sum_{j = 1}^2 
\int_{[0, 1]^2}
\int_{(\T^2)^2}
\bigg( \prod_{k = 1}^2 J^{(t), 1}_{\frac 12 - \eps-\dl}(t - s_k)
J_{\al-\dl} (x-y_k)\bigg)\\
&\hphantom{XXXX}
  \times \Big(e^{ \be^2 \G_{N_j}(s_1 - s_2 ,y_1 - y_2) }-e^{ \be^2 \G_{N_1, N_2}(s_1 - s_2 ,y_1 - y_2)}\Big)d\vec y d \vec s, 
\end{split}
\label{H1}
\end{align}

\noi
where $\G_{N_1, N_2}$ is as in \eqref{L7}.
Given $\dl > 0$, there exists $C_\dl > 0$ such that 
\begin{align}
| \log y | \le C_\dl  y^{-\dl}
\label{H3}
\end{align}

\noi
 for any $0 < y \les 1$.
Then, 
by  the fundamental theorem of calculus
and \eqref{L10} in Proposition \ref{PROP:cov}
with~\eqref{H3}, 
we have 
\begin{align}
\begin{split}
& \Big|e^{ \be^2 \G_{N_j}(t,x) }-e^{ \be^2 \G_{N_1, N_2}(t,x)}\Big|\\
&\quad = \bigg| \int_0^1\be^2 \exp\Big(\be^2\big(\tau \G_{N_j}(t,x) 
+(1-\tau)\G_{N_1, N_2}(t,x)\big)\Big)d\tau
\\
& \hphantom{XXXX}
\times 
\big(\G_{N_j}(t,x) - \G_{N_1, N_2}(t,x)\big)
\bigg|\\
&\quad \les \big(|t| + |x|+N_2^{-1}\big)^{-\frac{\be^2 }{2\pi}}
\Big\{ \big(|t| + |x|+N_2^{-1}\big)^{-\dl}\wedge\big(N_1^{-\frac 12 }|x|^{-\frac 12}\big)
+O(N_1^{-1})\Big\}\\
&\quad \les
N_1^{-\dl} |x|^{-2\dl}
\big(|t| + |x|+N_2^{-1}\big)^{-\frac{\be^2 }{2\pi}}.
\end{split}
\label{H2}
\end{align}

\noi
Hence, from 
proceeding as in Step $\1$ with \eqref{H2}, 
we obtain
\begin{align}
\begin{split}
&\E \Big[ \big| 
J^{(t), 1}_{\frac 12 - \eps-\dl}
*_t
\jb{\nabla}^{-\al + \dl}
\big\{ \ind_{[0, 1]} ( \Ta_{N_1}-\Ta_{N_2})\big\}(t, x) \big|^{2} \Big]
\\ &\quad \les 
N_1^{-\dl}
\int_{[0,1]^2}
 \int_{(\T^2)^2}
 |y_{j} - y_{k}|^{ -2\dl}
 \big(|s_j -s_k|  + |y_{j} - y_{k}|\big)^{ -\frac{\be^2 }{2 \pi}}\\
 & 
\hphantom{XXXX}
\times \bigg(\prod_{\l \in \{j, k\}}
 |t-s_\l|^{-\frac 12 - \eps - \dl }
  |x-y_\l|^{\al-\dl-2}
\bigg)
 d\vec y
  d\vec s\\
 & \quad \les  
N_1^{-\dl}
\end{split}
\label{H2a}
\end{align}

\noi
for any $N_2\geq N_1 \geq 1$ and  $(t , x) \in \R\times \T^2$, 
provided that
\eqref{E5a} holds
(and for $\dl > 0$ sufficiently small).

Fix  $p \geq 1$.
By interpolating \eqref{H2a} with \eqref{E44}, 
we have 
\begin{align*}
\E \Big[ \big| 
J^{(t), 1}_{\frac 12 - \eps-\dl}
*_t
\jb{\nabla}^{-\al + \dl}
\big\{ \ind_{[0, 1]} ( \Ta_{N_1}-\Ta_{N_2})\big\}(t, x) \big|^{p} \Big]
 \les  
N_1^{-\dl}
\end{align*}

\noi
for any $N_2\geq N_1 \geq 1$ and  $(t , x) \in \R\times \T^2$, 
provided that
\eqref{E5a} holds
(and for $\dl > 0$ sufficiently small).
A similar (but simpler) computation 
allows us to bound the contribution from 
$J^{(t), 2}_{\frac 12 - \eps-\dl}$ in~\eqref{E0b}, 
and 
therefore, we conclude that 
\[\|\Ta_{N_1}-\Ta_{N_2}\|_{L^{p}(\muu_1 \otimes \PP) 
 \Ld^{-\al, -\frac 12 +\eps}_\infty([0,1])} \le N_1^{-\tfrac{\dl}{p}}.\]

\noi
Namely, 
$\Ta_N$ is a Cauchy sequence in $L^p(\O;
 \Ld^{-\al, -\frac 12 +\eps}_\infty([0,1])$. This concludes the proof of Proposition \ref{PROP:sto2}.
\end{proof}

\begin{remark}\label{REM:heat1}\rm

In the case of the heat equation, 
the space-time covariance of the associated
stochastic convolution is given by 
\begin{align}
\G_N^\text{heat} (t_1-t_2, x_1-x_2) \approx - \frac{1}{2 \pi} \log \big( |t_1 - t_2 |^\frac12 + |x_1 - x_2| + N^{-1} \big).
\label{covv1}
\end{align}  

\noi
See, for example, \cite[Lemmas 3.7 and 3.8]{HS}.
Compare this with 
\eqref{cov1} in Proposition \ref{PROP:cov}.
By repeating the proof of Proposition \ref{PROP:sto2}, 
the main goal is then to bound
 \begin{align}
 \begin{split}
& \int_0^1 \int_0^1
 \int_{\T^2}\int_{\T^2}
 \big(|s_j -s_k|^\frac 12   + |y_{j} - y_{k}|\big)^{ -\frac{\be^2 }{2 \pi}}\\
& 
\hphantom{XXXX}
\times \bigg(\prod_{\l \in \{j, k\}}
 |t-s_\l|^{-\frac 12 - \eps - \dl }
  |x-y_\l|^{\al-\dl-2}
\bigg)
 dy_{j}dy_k
  ds_j ds_k, 
\end{split}
\label{E46}
  \end{align}

\noi
where there is an extra $\frac 12$-power
on $|s_j - s_k|$ as compared to \eqref{E4}. By adapting the computations in the proof above (see in particular Cases 1 and 2), one observes that $\{\Ta_N\}_{N \in \N}$ converges in the anisotropic space $L^p(\Omega, \Ld^{-\al, -b}_\infty([0,1]))$ for $\al, b>0$ if the condition
\[ \al + 2b > \frac{\be^2}{4\pi} \quad \text{and} \quad \be^2 < 8\pi \]
is met. See \cite[Theorem 2.1]{HS} for a construction of the imaginary Gaussian multiplicative chaos in isotropic spaces.
\end{remark}

\begin{remark}\label{REM:div}\rm

Let us now consider the case $\be^2 \ge 6 \pi$.
Given a test function  $\phi \in C^\infty_c(\R_+\times \T^2) \setminus \{0\}$, 
it follows from a slight modification of
the computation in the proof of 
Proposition~\ref{PROP:sto2}
that 
\begin{align*}
& \lim_{N \to \infty} \E\Bigg[\bigg|\int_{\R_+} \int_{\T^2 }
\phi(t, x) \Ta_N(t, x) dx dt\bigg|^2 \Bigg]\\
& \quad 
= \lim_{N \to \infty} \int_{(\R_+)^2} \int_{(\T^2)^2}
\phi(t_1, x_1) \cj{\phi(t_2, x_2)}
\E[\Ta_N (t_1, x_1) \cj{\Ta_N(t_2, x_2)}]
 dx_1 dx_2 dt_1 dt_2\\
& \quad \sim  
\lim_{N \to \infty} 
\int_{(\R_+)^2} \int_{(\T^2)^2}
\phi(t_1, x_1) \cj{\phi(t_2, x_2)}\\
& \hphantom{XXXXXXXXXX}
\times \big(|t_1 - t_2|  + |x_1 - x_2| + N^{-1}\big)^{ -\frac{\be^2 }{2 \pi}}
 dx_1 dx_2 dt_1 dt_2\\
& \quad = \infty
\end{align*}

\noi
for $\be^2 \ge 6 \pi$, 
since 
$\big(|t|  + |x|\big)^{ -\frac{\be^2 }{2 \pi}}$
is not locally integrable in this case.
In particular, 
the truncated imaginary Gaussian multiplicative chaos
$\Ta_N$ does not converge even as a space-time distribution
when $\be^2 \ge 6 \pi$.

\end{remark}

\subsection{A Sobolev type lemma}
\label{SUBSEC:sto4}
We first introduce some notations. Let $N \in \N$ and $\be \in \R$ with $0 < \be^2 < 4\pi$. We define the function $f_N=f_{N, \be}$ on $(\R \times \T^2)^{4}$ by 

\noi
\begin{align}
f_N( {\mbf z_1}, {\mbf z_2}) = \exp \Big(-\frac{\be^2}{2\pi} \G_N(t_1,t_2, y_1-y_2)\Big)
\label{def_f}
\end{align}

\noi
for every $\mbf z_j = (t_j, y_j) \in \R \times \T^2$, $1 =1 ,2$. Here, $\G_N$ is as in \eqref{cov}. 

Given a function $f : (\R \times \T^2)^2 \mapsto \R$ and $N_0 \in \N$, define the function $\mf F_{N_0}[f]$ on $(\R \times \T^2)^{2}$:

\noi
\begin{align}
\mf F_{N_0}[f]( z_1, z_2) = \int_{(\T^2)^{2}} dy_j \K_{N_0}(x_1 - y_1) \K_{N_0}(x_2 - y_2) f( \mbf z_1, \mbf z_2) dy_1 dy_2
 \label{def_F}
\end{align}

\noi
for any $z_j  = (t_j, x_j) \in \R \times \T^2$, $j =1,2$. In \eqref{def_F}, $\mbf z_j = (t_j, y_j)$ for any $j =1,2$ and $\K_{N_0}$ denotes the convolution kernel associated to the spatial frequency projection $\P_{N_0}$ defined in \eqref{proj1}.

%
%
%
%
%

Let $\ta: (0,4\pi) \to \R_+^*$ be the function given by

\noi
\begin{align}
\ta(\be^2) = \begin{cases}\frac32 - \frac{\be^2}{2\pi} & \quad \text{if $\be^2 \in [2\pi, 3\pi)$}, \\
2 - \frac{\be^2}{2\pi} & \quad \text{if $\be^2 \in [3\pi, 4\pi)$}. \end{cases}
\label{Ysgta}
\end{align}

Recall that for $x = (x^1, x^2) \in \T^2$, we denote by $\partial_{x^\l}$ for $\l \in \{1,2\}$ the derivative with respect to the $\l^{\text{th}}$ coordinate of $x$ Here, it will also be convenient to use the following notations: $|z|_+ = |t| + |x|_{\T^2}$ and $|z|_- = ||t| - |x|_{\T^2}|$ for a space-time point $z = (t,x) \in \R \times \T^2$.

The goal of this subsection is to bound the expression $\partial_{x_1^{\l}} \partial_{x_2^{\l}} \mf F_{N_0}[f_N]$ for $\l \in \{1,2\}$. Note that by moving the derivatives to the kernels and Proposition \ref{PROP:cov}, we get
\begin{align}
\begin{split}
& |\partial_{x_1^{\l}} \partial_{x_2^{\l}} \mf F_{N_0}[f_N] (z_1, z_2) | \\
& \qquad \qquad  = \Big| \int_{(\T^2)^{2}} dy_j \partial_{x_1^{\l}} \K_{N_0}(x_1 - y_1) \partial_{x_2^{\l}} \K_{N_0}(x_2 - y_2) f( \mbf z_1, \mbf z_2) dy_1 dy_2 \Big| \\
& \qquad \qquad \les N_0^2 \int_{(\T^2)^2} | \mbf z_1 - \mbf z_2|_+^{-\frac{\be^2}{2\pi}} \\
& \qquad \qquad \les N_0^2.
\end{split}
\label{Ysg300}
\end{align}
The bound \eqref{Ysg300} is too crude for our purposes, as we are only allowed a smaller power of $N_0$ in \eqref{goal1}. Alternatively, if we move the derivatives to the function $f_N$, Proposition \ref{PROP:cov2} gives the bound
\begin{align}
| \partial_{x_1^{\l}} \partial_{x_2^{\l}} \mf F_{N_0}[f_N] (\mbf z_1, \mbf z_2) | \les |x_1 - x_2|^{-\eps} | \mbf z_1 - \mbf z_2|_+^{-\frac12 - \frac{\be^2}{2\pi}} | \mbf z_1 - \mbf z_2 |_-^{-\frac32 -\eps}.
\label{Ysg301}
\end{align}
Unfortunately, the right-hand-side of \eqref{Ysg301} is not locally integrable. In the next lemma, we craft an interpolation argument by hand between the scenarios \eqref{Ysg300} and \eqref{Ysg301} which gives the appropriate power of $N_0$ allowed in the bound \eqref{goal1}. This argument can also be viewed as a ``Sobolev inequality" as we basically trade derivatives (i.e. powers of $N_0$) for integrability, which is lacking on the right-hand-side of \eqref{Ysg301}.
\begin{lemma}[potential-Sobolev argument] \label{LEM:der_F}
Fix $\be \in \R$ with $\be^2 \in [2\pi, 4\pi)$ and $0 < \kk_{\circ} = \kk_{\circ}(\be) \ll 1$ satisfying $\kk_\circ < \ta(\be^2)$, where $\ta$ is as in \eqref{Ysgta}, $N_0 \in  \N$ and $\l \in \{0,1\}$. Let $f_N$ be as in \eqref{def_f} and define $\mf F_{N_0} [f_N]$ as in \eqref{def_F}. Then, there exists an absolute constant $C>0$ such that the following pointwise estimates hold. 

\smallskip

\noi
\textup{(i)} If $\be^2 \in [2\pi, 3\pi)$, then we have
\begin{align}
| \partial_{x_1^{\l}} \partial_{x_2^{\l}} \mf F_{N_0}[f_N]( z_1, z_2) | \les_{\kk_{\circ}} N_0^{\frac{\be^2}{2\pi} + C\kk_{\circ}} \cdot |z_1 - z_2|_+^{-\frac12 - \frac{\be^2}{2\pi}+\kk_\circ} \, |z_1 - z_2|_-^{-\frac32 + \frac{\be^2}{2\pi}+\kk_\circ}
\label{goal2}
\end{align}
for any $z_j  = (t_j, x_j) \in [0,1] \times \T^2$, $j=1,2$ and uniformly in $N \in 1$.

\smallskip

\noi
\textup{(ii)} If $\be^2 \in [3\pi, 4\pi)$, then we have

\noi
\begin{align}
| \partial_{x_1^{\l}} \partial_{x_2^{\l}} \mf F_{N_0}[f_N]( z_1, z_2) | \les_{\kk_{\circ}} N_0^{\frac{\be^2}{2\pi} + C\kk_{\circ}} \cdot |z_1 - z_2|_+^{-2 + \kk_\circ} 
\label{goal2b}
\end{align}
for any $z_j  = (t_j, x_j) \in [0, 1] \times \T^2$, $j=1,2$ and uniformly in $N \in 1$.
\end{lemma}

In Subsection \ref{SEC:sing} below, we discuss the integrability of the functions on the right-hand-side of \eqref{goal2}-\eqref{goal2b} when convolved with convolution kernels of the operator $\Box^{-b}$, $b>\frac12$. It turns out that proving the relevant integrability results for the hyperbolic type singularities (i.e. the right-hand-side of \eqref{goal2}) is much more involved than those for the elliptic type singularities (i.e. the right-hand-side of \eqref{goal2b}) since dealing with singularities along light-cones requires a careful geometric analysis; see Lemma \ref{LEM:sing2}. This is why the estimates (i) and (ii) in Lemma \ref{LEM:der_F} above are rather surprising. Indeed, above $\be^2 = 3\pi$, our estimates do not see the hyperbolicity of the problem at hand even though this case corresponds, at the level of the dynamics, to a more singular equation \eqref{SdSG}.

\begin{proof} In this proof, we write $\partial_j$ for $\partial_{x^\l_j}$ for $j=1,2$ for the sake of notational convenience. Recall for $j=1,2$, we denote by $\mbf z_j = (t_j, y_j)$ the ``input" variables in the integrand of $\mf F_{N_0}[f_N]$ and by $z_j = (t_j, x_j)$ the ``output" variables (namely, the arguments of $\mf F_{N_0}[f_N]$ on the left-hand-side of \eqref{def_F}).

Fix $\be \in \R$ such that $0 < \be^2 < 4\pi$. Fix $0 < \kk_{\circ} \ll 1$ such that $\kk_\circ < \ta(\be^2)$. We first note that the inequalities \eqref{goal2} and \eqref{goal2b} are straightforward for $N_0 \les_{\kk_{\circ}} 1$ by \eqref{ker1}. Hence, in what follows we assume that $N_0 \gg_{\kk_{\circ}} 1$.

\medskip

\noi
{\bf $\bul$ Step \1: basic restrictions on input variables.}\quad Assume that $|x_{1} - y_{1}| > N_0^{-1+\kk_{\circ}}$ on the integrand of $\mf F_{N_0}[f_N]$. Then, we have 

\noi
\begin{align*}
 &| \partial_{1} \partial_2 \mf F_{N_0}[f_N]( z_1, z_2) | \\
& \quad \les \int_{(\T^2)^{2}} | \partial_{1} \K_{N_0}(x_1 - y_1)|  \ind_{|x_{1} - y_{1}| > N_0^{-1+\kk_{\circ}}} \, | \partial_{2} \K_{N_0}(x_{2}- y_{2})|  |f_N (\mbf z_1, \mbf z_2)| dy_1 dy_2.
\end{align*}

\noi
By \eqref{ker1}, the current assumption and Proposition \ref{PROP:cov}, we thus have

\noi
\begin{align*}
 &| \partial_{1} \partial_2 \mf F_{N_0}[f_N](z_1, z_2) | \\
& \qquad \les_A N_0^{6 - A\kk_{\circ}} \int_{(\T^2)^{2}} |\mbf z_1 - \mbf z_2|_+^{-\frac{\be^2}{2\pi}} dy_1 dy_2, \\
& \qquad \les_{\kk_0} 1,
\end{align*}

\noi
upon choosing $A$ large enough depending on $\kk_{\circ}$. The last bound implies \eqref{goal2} and \eqref{goal2b}. Note that by symmetry, we have a similar bound if $|x_2 - y_2 | > N_0^{-1+\kk_\circ}$. Thus, we may assume that the bound $|x_j - y_j| \le N_0^{-1+\kk_{\circ}}$ holds for any $j = 1,2$. This reduction allows us to compare $|z_1 - z_2|_+$ and $|\mbf z_1 - \mbf z_2|_+$ in the following case: if $\max\big( |z_1 - z_2|_+, |\mbf z_1 - \mbf z_2|_+ \big) > N_0^{-1 +2\kk_{\circ}}$ then we have 
\begin{align*}
\big| |z_1  - z_2|_+ - |\mbf z_1 - \mbf z_2|_+ \big| \le | x_1 - y_2+  y_1 - x_2 | \le 2 N_0^{-1+\kk_{\circ}},
\end{align*}
and hence
\begin{align}
|z_1- z_2|_+ \sim |\mbf z_1 - \mbf  z_2|_+,
\label{sob1}
\end{align}

\noi
for $N_0$ large enough, as claimed. Similarly, if $\max\big( |z_1 - z_2|_-, |\mbf z_1 - \mbf z_2|_- \big) > N_0^{-1 +2\kk_{\circ}}$, then we also have
\begin{align}
|z_1 - z_2|_- \sim |\mbf z_1 - \mbf z_2|_-.
\label{sobo104}
\end{align}

Now, we assume $|y_1 - y_2| \le N_0^{-c_1}$, with $c_1 = 10^{10} \cdot (2-\kk_{\circ} - \frac{\be^2}{2\pi})^{-1} \gg 1$ on the integrand of $\mf F_{N_0}[f_N]$. Then, we similarly get

\noi
\begin{align*}
 &| \partial_{1} \partial_2 \mf F_{N_0}[f_N]( z_1, z_2) |  \\
& \qquad \les N_0^{6}   \int_{(\T^2)^{2}} |\mbf z _1 - \mbf z_{2}|_+^{-\frac{\be^2}{2\pi}} \ind_{|y_{1} -y_{2}|\le N_0^{-c_1}} dy_1 dy_2  \\
& \qquad \les N_0^{6 -(2-\kk_{\circ} - \frac{\be^2}{2\pi})c_1} \int_{(\T^2)^{2}}  | \mbf z _{1} - \mbf z_{2}|_+^{-\frac{\be^2}{2\pi}}  |y_{1} -y_{2}|^{-2 +\kk_{\circ} + \frac{\be^2}{2\pi}} dy_1 dy_2  \\
& \qquad \les 1.
\end{align*}
The last bound implies \eqref{goal2} and \eqref{goal2b}. Thus, we henceforth assume that the condition $|y_{1} - y_{2}| > N_0^{-c_1}$ holds on the integrand of $\mf F_{N_0}[f_N]$.

Let us now assume that $|\mbf z_{1} - \mbf  z_{2}|_- \le N_0^{-c_2}$, with $c_2 = 10^{10} \cdot c_1 \gg 1$, holds on the integrand of $\mf F_{N_0}[f_N]$. Then, by the previous reduction and Proposition \ref{PROP:cov}, we have

\noi
\begin{align*}
 &| \partial_{1} \partial_2 \mf F_{N_0}[f_N]( z_1, z_2) |  \\
& \qquad \les N_0^{6}  \int_{(\T^2)^{2}}  | \mbf z _{1} - \mbf z_{2}|_+^{-\frac{\be^2}{2\pi}} \ind_{|y_1 - y_2| > N_0^{-c_1}} \ind_{|\mbf z_{1} - \mbf z_{2}|_- \le N_0^{-c_2}} d y_1 dy_2  \\
& \qquad \les N_0^{6 + \frac{\be^2}{2\pi}c_1 - \frac{c_2}{2}} \int_{(\T^2)^{2}}  | \mbf z_{1} - \mbf z_{2} |_-^{-\frac12} dy_1 dy_2  \\
& \qquad \les 1.
\end{align*}
The last bound implies \eqref{goal2} and \eqref{goal2b}. To sum up, we assume in the remaining part of the proof that the following conditions hold:

\noi
\begin{align}
\begin{split}
|x_1 - y_1| & \le N_0^{-1+\kk_{\circ}},\\
|x_2 - y_2| & \le N_0^{-1+\kk_{\circ}}, \\
|y_1 - y_2| & > N_0^{-c_1}, \\
|\mbf z_1 - \mbf z_2|_- & > N_0^{-c_2}
\end{split}
\label{sob1b}
\end{align}
for some constants $c_1, c_2 >0$. We denote by $\mc C$ the set pertaining to conditions \eqref{sob1b}. In what follows, we always assume that the indicator function $\ind_{\mc C}$ is included in the integrand of $\mf F[f_N]$, but we might omit to write it when it is not necessary.

\medskip

\noi
{\bf $\bul$ Step \II: case-by-case analysis on output variables.}\quad Assume ${|z_{1} - z_{2}|_+ \le N_0^{-1+2\kk_\circ}}$. By Proposition \ref{PROP:cov}, we have

\noi
\begin{align}
\begin{split}
&| \partial_{1} \partial_2 \mf F_{N_0}[f_N]( z_1, z_2) | \les \int_{(\T^2)^{2}} | \partial_{1} \K_{N_0}(x_1- y_{1})| | \partial_{2} \K_{N_0}(x_{2}- y_{2})|  | \mbf z_{1} - \mbf z_{2}|_+^{-\frac{\be^2}{2\pi}} d y_{1} dy_{2}.
\end{split}
\label{sob2}
\end{align}
From \eqref{ker1} and Lemma \ref{LEM:t0}, we deduce that

\noi
\begin{align}
\begin{split}
& \int_{(\T^2)^{2}} | \partial_{1} \K_{N_0}(x_1- y_{1})| | \partial_{2} \K_{N_0}(x_{2}- y_{2})|  | \mbf z_{1} - \mbf z_{2}|_+^{-\frac{\be^2}{2\pi}} d y_{1} dy_{2} \\
& \qquad \les   \int_{(\T^2)^{2}} | \partial_{1} \K_{N_0}(x_1- y_{1})| | \partial_{2} \K_{N_0}(x_{2}- y_{2})|  |y_{1} - y_{2}|^{-\frac{\be^2}{2\pi}} d y_{1} dy_{2} \\
& \qquad \les  N_0^{1+\frac{\be^2}{2\pi}} \int_{\T^2} | \partial_{1} \K_{N_0}(x_1- y_{1})| d y_{1} \\
& \qquad \les  N_0^{2+\frac{\be^2}{2\pi}}
\end{split}
\label{sob4}
\end{align}
Therefore, combining \eqref{sob2}, \eqref{sob4} and the condition $|z_{1} - z_{2}|_+ \le N_0^{-1+2\kk_\circ}$, gives
\begin{align}
| \partial_{1} \partial_2 \mf F_{N_0}[f_N]( z_1, z_2) | \les N_0^{\frac{\be^2}{2\pi} + C\kk_\circ} \cdot |z_1 - z_2|_+^{-2+\kk_\circ},
\label{sob4b}
\end{align}
which is acceptable since $|z_1 - z_2|_+ \ge |z_1 - z_2|_-$. Hence, we now assume $|z_1 - z_2|_+ > N_0^{-1+2\kk_\circ}$.

Assume $|z_1 - z_2|_- \le N_0^{-1 + 2 \kk_\circ}$ and $|z_1 - z_2|_+ >N_0^{-1+2\kk_\circ}$. In view of the discussion leading to \eqref{sob1}, we have $|z_1 - z_2|_+ \sim |\mbf z_1 - \mbf z_2|_+$. Then, by Propositions \ref{PROP:cov} and \ref{PROP:cov2} and the conditions \eqref{sob1b}, we have
\begin{align}
\begin{split}
&| \partial_{1} \partial_2 \mf F_{N_0}[f_N]( z_1, z_2) | \\
& \quad \les \int_{(\T^2)^{2}} \ind_{\mc C} (\mbf z_1, \mbf z_2)  \cdot| \partial_{1} \K_{N_0}(x_1- y_{1})| |\K_{N_0}(x_{2}- y_{2})| |\partial_2 f_N (\mbf z_1, \mbf z_2)| d y_{1} dy_{2} \\
& \quad \les_{\kk_{\circ}} \int_{(\T^2)^{2}} \ind_{\mc C} (\mbf z_1, \mbf z_2) \cdot | \partial_{1} \K_{N_0}(x_1- y_{1})| |\K_{N_0}(x_{2}- y_{2})| \\
& \qquad \qquad \qquad \qquad \qquad \times \Big( |y_1 - y_2|^{-\kk_\circ} + |\mbf z_1 - \mbf z_2|^{-\frac12-\frac{\be^2}{2\pi}}_+  |\mbf z_1 - \mbf z_2|_{-}^{-\frac12-\kk_0} \Big) d y_{1} dy_{2} \\
& \quad \les N_0^{C\kk_0} \cdot |z_1 - z_2|_+^{-\frac12 - \frac{\be^2}{2\pi}} \int_{(\T^2)^{2}} | \partial_{1} \K_{N_0}(x_1- y_{1})| |\K_{N_0}(x_{2}- y_{2})| |\mbf z_1 - \mbf z_2|_{-}^{-\frac12} d y_{1} dy_{2}.
\end{split}
\label{sobo100}
\end{align}
By Lemma \ref{LEM:green_wave0}, we have
\begin{align}
\int_{(\T^2)^{2}} | \partial_{1} \K_{N_0}(x_1- y_{1})| |\K_{N_0}(x_{2}- y_{2})| |\mbf z_1 - \mbf z_2|_{-}^{-\frac12} d y_{1} dy_{2} \les N_0^{\frac32 + \kk_0}.
\label{sobo101}
\end{align}
If $\be^2 \in [2\pi, 3\pi)$, then by \eqref{sobo100}, \eqref{sobo101} and the conditions $|z_1 - z_2|_- \le N_0^{-1+2\kk_\circ}$ and $|z_1 - z_2|_+ > N_0^{-1+2\kk_\circ}$, we get
\begin{align}
\begin{split}
| \partial_{1} \partial_2 \mf F_{N_0}[f_N]( z_1, z_2) | & \les N_0^{\frac32 + C \kk_\circ} \cdot |z_1 - z_2|_+^{-\frac12 - \frac{\be^2}{2\pi}} \\
& \les N_0^{\frac{\be^2}{2\pi}} \cdot |z_1 - z_2|_+^{-\frac12 - \frac{\be^2}{2\pi}+\kk_\circ} \, |z_1 - z_2|_-^{-\frac32 + \frac{\be^2}{2\pi}+\kk_\circ}.
\end{split}
\label{sobo102}
\end{align}
Otherwise, we have $\be^2 \in [3\pi, 4\pi)$. Therefore, from \eqref{sobo100}, \eqref{sobo101} and the condition ${|z_1 - z_2|_+ >N_0^{-1+2\kk_\circ}}$, we have
\begin{align}
\begin{split}
| \partial_{1} \partial_2 \mf F_{N_0}[f_N]( z_1, z_2) | & \les N_0^{\frac32 + C \kk_\circ} \cdot |z_1 - z_2|_+^{-\frac12 - \frac{\be^2}{2\pi}} \\
& \les N_0^{\frac{\be^2}{2\pi}} \cdot |z_1 - z_2|_+^{-2 + \kk_\circ}.
\end{split}
\label{sobo103}
\end{align}

Lastly, assume $|z_1 - z_2|_- > N_0^{-1 + 2 \kk_\circ}$ and $|z_1 - z_2|_+ >N_0^{-1+2\kk_\circ}$. Then, as in \eqref{sob1} and \eqref{sobo104}, we have that $|z_1 - z_2|_+ \sim |\mbf z_1 - \mbf z_2|_+$ and $|z_1 - z_2|_- \sim |\mbf z_1 - \mbf z_2|_-$. Thus, from Propositions \ref{PROP:cov} and \ref{PROP:cov2} and the conditions \eqref{sob1b}, we have
\begin{align}
\begin{split}
&| \partial_{1} \partial_2 \mf F_{N_0}[f_N]( z_1, z_2) | \\
& \quad \les \int_{(\T^2)^{2}} \ind_{\mc C} (\mbf z_1, \mbf z_2)  \cdot| \K_{N_0}(x_1- y_{1})| |\K_{N_0}(x_{2}- y_{2})| | \partial_1 \partial_2 f_N (\mbf z_1, \mbf z_2)| d y_{1} dy_{2} \\
& \quad \les_{\kk_{\circ}} \int_{(\T^2)^{2}} \ind_{\mc C} (\mbf z_1, \mbf z_2) \cdot | \K_{N_0}(x_1- y_{1})| |\K_{N_0}(x_{2}- y_{2})| \\
& \qquad \qquad \qquad \qquad \qquad \times \Big( |y_1 - y_2|^{-\kk_\circ} + |\mbf z_1 - \mbf z_2|^{-\frac12-\frac{\be^2}{2\pi}}_+  |\mbf z_1 - \mbf z_2|_{-}^{-\frac32-\kk_0} \Big) d y_{1} dy_{2} \\
& \quad \les N_0^{C\kk_0} \cdot |z_1 - z_2|^{-\frac12 - \frac{\be^2}{2\pi}} | z_1 -  z_2|_{-}^{-\frac32-\kk_0}.
\end{split}
\label{sobo105}
\end{align}
The estimates \eqref{goal2} and \eqref{goal2b} immediately follow from \eqref{sobo105} and the conditions ${|z_1 - z_2|_- > N_0^{-1 + 2 \kk_\circ}}$ and $|z_1 - z_2|_+ >N_0^{-1+2\kk_\circ}$.
\end{proof}

\subsection{Bounds on singular integrals}\label{SEC:sing}
In this subsection, we integrate the singularities output by Lemma \ref{LEM:der_F} against the kernel of $\Box^b$ for $b <-\frac12$ introduced in \eqref{Rhyp1}.

Let $b < -\frac12$, $s,s_1, s_2 >0$ and define
\begin{align}
\mf I^{+, b,s}(t) = \int_{[0,1]^2 \times (\R^2)^2} \prod_{j = 1}^2 dt_j dy_j \, \mf K_{b}(t-t_j, x_j) \cdot  \big(|t_1 - t_2| + |x_1 - x_2|_{\T^2} \big)^{-s}
\label{Ysgi1}
\end{align}
and 
\begin{align}
\begin{split}
\mf I^{-,b, s_1, s_2} (t)&  = \int_{[0,1]^2 \times (\R^2)^2} \prod_{j = 1}^2 dt_j dy_j \, \mf K_{b}(t-t_j, x_j) \\
& \qquad \quad \times \big(|t_1 - t_2| + |x_1 - x_2|_{\T^2} \big)^{-s_1} \big| |t_1 - t_2| -  |x_1 - x_2|_{\T^2}\big|^{-s_2}
\end{split}
\label{Ysgi2}
\end{align}
for $(t,x) \in \R \times \R^2$. Here, $\mf K_b$ is as in \eqref{kerhyp}.

Our main result in this subsection is the following quantitative estimate on the the integrals $\mf I^{+, b,s}(t)$ and $\mf I^{-,b, s_1, s_2} (t)$.

\begin{lemma}\label{LEM:sing1}
Fix $b < -\frac12$, $0 < s,s_1 < 2$ and $0 < s_2 < \frac12$ such that $s_1 + s_2 < 2$. Then, the following bounds hold: 
\begin{align}
\mf I^{+,b,s}(t)&  \les \jb t ^{-2 - 4b},\label{Ysg99} \\
\mf I^{-,b,s_1, s_2}(t)&\les  \jb t ^{-2 - 4b}. \label{Ysg100}
\end{align}
\end{lemma}

In order to prove the bounds \eqref{Ysg99} and \eqref{Ysg100}, we proceed with several spatial localizations of the integrands of $\mf I^{+,b,s}(t)$ and $\mf I^{-,b,s_1,s_2}(t)$ in what follows. For each $k \in 2\pi \Z^2$, $\ld >0$ and $r \in \Z_{\ge 0}$, we write
\begin{align}
\begin{split}
\T^2_k & = [-\pi, \pi)^2 + k, \\
A_r(\ld) & = \big\{ x \in \R^2 : \ld - \frac{r + 1}{100} \le |x|_{\R^2} < \ld- \frac{r}{100}\big\}.
\end{split}
\label{Ysgloc}
\end{align}

The most challenging part of the proof of Lemma \ref{LEM:sing1} is to estimate the contribution of the portion of the integrals $\mf I^{-,b,s}(t)$ and $\mf I^{-,b,s_1, s_2}(t)$ close to their respective elliptic and hyperbolic singularities. This is the purpose of the next lemma.

\begin{lemma}\label{LEM:sing2}
Fix $0 \le \eta_1, \eta_2, \eta_{1,2} < 2$ such that $\eta_1 + \eta_2 + \eta_{1,2} < 4$, $0\le \al_1, \al_2< 1$ and $0 < \al_{1,2} < \frac12$. We use the shorthand notation $\bar \al = (\al_1, \al_2, \al_{1,2})$ and $\bar \eta = (\eta_1, \eta_2, \eta_{1,2})$. For $(t,t_1,t_2) \in \R \times [0,1]^2$, $k_1, k_2, k_3 \in 2 \pi \Z^2$ with $|k_3|_{\R^2} \les 1$, let $f^{\bar \al}_{k_3,t,t_1,t_2}$ be the function
\[ f^{\bar \al}_{k_3,t,t_1,t_2}(x_1, x_2) =  \frac{ \ind_{||t_1 - t_2| - |x_1 - x_2 + k_3|_{\R^2}|\ll 1} \cdot \big| |t_1 - t_2| - |x_1 - x_2 + k_3|_{\R^2}\big|^{-\al_{1,2}}}{\big||t-t_1| - |x_1|_{\R^2}\big|^{\al_1} \big||t-t_2| - |x_2|_{\R^2}\big|^{\al_2}}  \]
and define $\mf O^{+,\bar \eta}_{k_1, k_2, k_3}$ and $\mf C^{-,\bar \al}_{k_1, k_2, k_3}(t,t_1,t_2)$ by
\begin{align}
\mf O^{+,\bar \eta}_{k_1, k_2, k_3} & = \int_{\T^2_{k_1} \times \T^2_{k_2}}  |x_1|^{-\eta_1}_{\R^2} |x_2|^{-\eta_2}_{\R^2} | x_1 - x_2 + k_3 |^{-\eta_{1,2}}_{\R^2} dx_1 dx_2, \label{Ysg101bb}  \\
\mf C^{-, \bar \al}_{k_1, k_2, k_3}(t,t_1,t_2) & = \int_{\T^2_{k_1} \times \T^2_{k_2}} \ind_{A_0(|t-t_1|)}(x_1) \ind_{A_0(|t-t_2|)}(x_2) \cdot f^{\bar \al}_{k_3,t,t_1,t_2} (x_1, x_2) dx_1 d x_2, \label{Ysg101cc}
\end{align}
where $A_0(|t-t_j|)$ for $j=1,2$ is as in \eqref{Ysgloc}. Then, we have
\begin{align}
& \sup_{\substack{k_1, k_2, k_3 \, \in 2 \pi \Z^2 \\ |k_3|_{\R^2} \les 1}}  \mf O^{+,\bar \eta}_{k_1, k_2, k_3}   < \infty, \label{Ysg102a} \\
& \sup_{\substack{k_1, k_2, k_3 \, \in 2 \pi \Z^2 \\ |k_3|_{\R^2} \les 1}} \mf C^{-,\bar \al}_{k_1, k_2, k_3}(t,t_1,t_2)  \les \o(t,t_1,t_2) |t_1-t_2|^{1-\al_{1,2}},
\label{Ysg102}
\end{align}
where $\o(t,t_1,t_2)$ is given by
\begin{align}
\o(t,t_1,t_2) = \begin{cases} |t-t_1| + |t-t_2| & \quad \text{for $|t| \le 10$}, \\
1 & \quad \text{for $|t| > 10$}.
\end{cases}
\label{Ysgfunc}
\end{align}
\end{lemma}

Obtaining \eqref{Ysg102} constitutes the most challenging part of Lemma \ref{LEM:sing2}. It essentially follows from bounding appropriately the volume of the intersection of transverse tubes in $\R^4$; see \eqref{Ysg105} below.

\begin{proof}In this proof, we write $|\cdot|$ for $|\cdot|_{\R^2}$. We only prove \eqref{Ysg102} as the proof of \eqref{Ysg102a} is much simpler and follows from arguments similar to those in the estimates \eqref{E4}-\eqref{E45} (which essentially correspond to the case $k_1=k_2=k_3=0$) in the proof of Proposition \ref{PROP:sto2}. 
%

We proceed with a multiscale decomposition procedure: we write
\begin{align}
\mf C^{-,\bar \al}_{k_1, k_2, k_3}(t,t_1,t_2) = \sum_{\substack{\mu_1, \mu_2, \mu_{1,2} \in 2^{\Z} \\ \mu_1, \mu_2, \mu_{1,2} \ll 1}} (\mu_1)^{-\al_1} (\mu_2)^{-\al_2} (\mu_{1,2})^{-\al_{1,2}} \cdot \mf C_{k_1, k_2, k_3}^- [\mu],
\label{Ysg103}
\end{align}
with $\mu = (\mu_1, \mu_2, \mu_{1,2})$ and $\mf C_{k_1, k_2, k_3}^- [\mu] (t,t_1,t_2)= | S^1_{k_1, k_2, k_3}[\mu](t,t_1,t_2) |$, where
\begin{align}
\begin{split}
S^1_{k_1, k_2, k_3}[\mu](t,t_1,t_2)&  = \Big\{ (y_1, y_2) \in \T^2_{k_1} \times \T^2_{k_2} : \frac{\mu_j}{2}\le \big||t- t_j| - |y_j|\big| < \mu_j, \, j = 1,2, \\
& \qquad \qquad \qquad \qquad  \, \frac{\mu_{1,2}}{2}\le \big||t_1- t_2| - |y_1 - y_2 + k_3 |\big| < \mu_{1,2} \Big\}.
\end{split}
\label{Ysg104}
\end{align}
The bound \eqref{Ysg102} reduces to proving 
\begin{align}
\sup_{\substack{k_1, k_2, k_3 \, \in 2\pi \Z^2 \\ |k_3| \les 1}} \mf C_{k_1, k_2, k_3}^{-,\bar \al}[\mu](t,t_1,t_2) \les \o(t,t_1, t_2) |t_1 - t_2|^{\frac12} \cdot \mu_{1} \mu_{2} (\mu_{1,2})^{\frac12},
\label{Ysg105}
\end{align}
where $\o(t,t_1,t_2)$ is as in \eqref{Ysgfunc}. Indeed summing \eqref{Ysg105} over $\mu_1, \mu_2, \mu_{1,2}$ and noting $\mu_{1,2} \les |t_1 - t_2|$ gives \eqref{Ysg102}. In what follows, we omit the dependence of all quantities in $(t, t_1, t_2)$ for notational convenience; i.e. we write $\mf C^{-, \bar \al}_{k_1, k_2, k_3}[\mu]$ for $\mf C^{-, \bar \al}_{k_1, k_2, k_3}[\mu](t, t_1, t_2)$ for instance.
\medskip

\noi
{\bf $\bul$ Case 1: $\mu_{1,2} \ges \max(\mu_1, \mu_2)$.}\quad Let us assume $t \ge t_1 \ge t_2$ in the following. The other cases follow from similar arguments upon changing signs in the expressions below. Then, for $(y_1, y_2) \in S^1_{k_1, k_2, k_3}[\mu]$, we have
\begin{align}
|y_1| - |y_2| = - |y_1 - y_2 + k_3 | + O(\mu_{1,2}).
\label{Ysg106}
\end{align}
Squarring \eqref{Ysg106} and doing some algebra then shows
\begin{align}
\jb{y_1, y_2} - \jb{y_1 - y_2, k_3} = \frac12 |k_3|^2 + |y_1| |y_2|  +O\big(|t_1-t_2|\mu_{1,2} + \mu_{1,2}^2\big).
\label{Ysg107}
\end{align}
%
%
%
From \eqref{Ysg107} and $\mu_{1,2} \les |t_1 - t_2|$, there exists a function $C=C(k_3, y_2, |y_1|)$ depending only on $y_2$ and $|y_1|$ such that 
\begin{align*}
\jb{y_1, y_2 - k_3} = C + O(|t_1 - t_2| \mu_{1,2})
\end{align*}
so that
\begin{align}
\cos(\angle(y_1, y_2 - k_3)) = C_1 + O\Big( \frac{|t_1 - t_2|\mu_{1,2}}{|y_1| |y_2 - k_3|}\Big),
\label{Ysg109}
\end{align}
where $C_1 = C_1(k_1, k_2, k_3, y_2, |y_1|)$ is another function depending only on $y_2$ and $|y_1|$. 

\medskip

\noi
{\bf $\bul$ Subcase 1.1: $k_3=0$.}\quad By \eqref{Ysg109}, there exists an interval $J_{1} = J_{1}(k_3,y_2, |y_1|)$ with
\begin{align}
|J_{1}| \les \frac{|t_1 - t_2|\mu_{1,2}}{|y_1| |y_2|}
\label{Ysg109b}
\end{align}
such that $\cos(\angle(y_1, y_2 - k_3)) \in J_{1}$. It is easy to observe via a Taylor expansion that
\begin{align}
\sup_{I}\big\{ \ta \in [0,2\pi] : \cos (\ta) \in I \big\} \les \eps^{\frac12},
\label{Ysg110}
\end{align}
where the supremum is taken over intervals $I \subset [-1,1]$ such that $| I | \les \eps$. Therefore, by \eqref{Ysg109}, \eqref{Ysg109b}, \eqref{Ysg110}, noting that $\angle(y_2, e_1)$, with $e_1 = (1,0)$, belongs to an interval $J_2 = J_2(k_2)$ of length $\les \jb{k_2} ^{-1}$ and switching to polar coordinates, we have
\begin{align}
\begin{split}
\mf C^{-, \bar \al}_{k_1, k_2, k_3}[\mu] & \les \int_{[0,2\pi]} \ind_{\ta_2 \in J_2(k_2)} d\ta_2 \int_{\R_+^2} \ind_{||t-t_1| - r_1| < \mu_1} \ind_{||t-t_2| - r_2| < \mu_{2}} r_1 r_2 dr_1 dr_2  \\
& \qquad \qquad \qquad \qquad \quad \times \int_{[0,2\pi]} \ind_{\cos \ta_1 \in J_{1}(k_3, r_1, r_2, \ta_2)} d\ta_1 \\
&  \les \big(|t_1 - t_2|\mu_{1,2}\big)^{\frac12} \int_{[0,2\pi]} \ind_{\ta_2 \in J_2(k_2)} d\ta_2 \\
& \qquad \qquad \qquad \quad \times \int_{\R_+^2} \ind_{||t-t_1| - r_1| < \mu_1} \ind_{||t-t_2| - r_2| < \mu_{2}} (r_1 r_2)^{\frac12}  dr_1 dr_2 \\
& \les \big( |t-t_1| |t-t_2| |t_1 - t_2|\big)^{\frac12} \, \jb{k_2}^{-1}  \cdot  \mu_1 \mu_2(\mu_{1,2})^{\frac12},
\end{split}
\label{Ysg111}
\end{align}
which shows \eqref{Ysg105} since $\jb{k_2} \sim \jb{t-t_1} \sim \jb{t-t_2}$.

\medskip

\noi
{\bf $\bul$ Subcase 1.2: $k_3\neq 0$ and $|k_2| \gg 1$.}\quad Then, since $y_2 \in \T^2_{k_2}$ and $|k_3| \les 1$, we have
\begin{align}
|y_2 - k_3| \sim |k_2|.
\label{Ysgg111a}
\end{align}
Therefore, from \eqref{Ysgg111a}, \eqref{Ysg109} and arguing as in \eqref{Ysg109b}-\eqref{Ysg111}, we get
\begin{align*}
\mf C^{-, \bar \al}_{k_1, k_2, k_3}[\mu] & \les \big( |t-t_1| |t_1 - t_2|\big)^{\frac12} |t - t_2| \, \jb{k_1}^{-\frac32} \cdot  \mu_1 \mu_2(\mu_{1,2})^{\frac12} \\
& \les |t_1 - t_2|^{\frac12} \cdot  \mu_1 \mu_2(\mu_{1,2})^{\frac12},
\end{align*}
as desired in \eqref{Ysg105} since $\jb{k_2} \sim \jb{t-t_1} \sim \jb{t-t_2} \gg 1$.

\medskip

\noi
{\bf $\bul$ Subcase 1.3: $k_3 \neq 0$ and $|k_2| \les 1$.}\quad Fix $0 < \dl \ll 1$. If we have
\[ |y_2 - y_1 - k_3| \le (1-\dl) \cdot |y_1|, \]
then we get
\begin{align} |y_2 - k_3| \ge |y_1 | -  |y_2 - y_1 - k_3| \ge \dl \cdot |y_1|.\label{Ysgg111}\end{align}
Hence, from \eqref{Ysgg111}, \eqref{Ysg109} and arguing as in \eqref{Ysg109b}-\eqref{Ysg111}, we get
\begin{align*}
\mf C^{-, \bar \al}_{k_1, k_2, k_3}[\mu] & \les_\dl |t_1 - t_2|^{\frac12} |t - t_2| \, \jb{k_2}^{-1} \cdot  \mu_1 \mu_2(\mu_{1,2})^{\frac12}\\
& \les | t - t_2 | |t_1 - t_2|^{\frac12} \cdot  \mu_1 \mu_2(\mu_{1,2})^{\frac12},
\end{align*}
as desired in \eqref{Ysg105} since $\jb{k_2} \sim 1$. Otherwise, we have 
\[ |y_1| < (1-\dl)^{-1} \cdot |y_2 - y_1 - k_3|, \]
which yields 
\begin{align*}
|y_2|&  \le |t-t_2| + \mu _2 \le |t- t_1| + |t_1 - t_2| + \mu_2\\
&  \le |y_1| + |t_1 - t_2| + \mu_1 +  \mu_2\\
& \le (1-\dl)^{-1} \cdot |y_2 - y_1 - k_3| + |t_1 - t_2| + \mu_1 + \mu_2 \\
& \le \big((1 - \dl)^{-1} + 1\big) |t_1 - t_2| + 10\max(\mu_1, \mu_2, \mu_{1,2})\\
& < 3,
\end{align*}
since $(t_1, t_2) \in [0,1]^2$ and $\mu_1, \mu_2, \mu_{1,2}, \dl \ll 1$. Noting, $|k_3| \ge 2\pi$ since $k_3 \neq 0$ and we have 
\begin{align}
|y_2 - k_3| > 1.
\label{Ysgg111b}
\end{align}
Thus, from \eqref{Ysgg111b}, \eqref{Ysg109} and arguing as in \eqref{Ysg109b}-\eqref{Ysg111}, we deduce that
\begin{align*}
\mf C^{-, \bar \al}_{k_1, k_2, k_3}[\mu]&  \les |t_1 - t_2|^{\frac12} |t-t_1|^{\frac12} |t - t_2| \, \jb{k_2}^{-1} \cdot  \mu_1 \mu_2(\mu_{1,2})^{\frac12} \\
& \les  |t-t_1|^{\frac12} |t - t_2| |t_1 - t_2|^{\frac12}  \cdot  \mu_1 \mu_2(\mu_{1,2})^{\frac12},
\end{align*}
as desired in \eqref{Ysg105} since $\jb{k_2} \sim 1$.
\medskip

\noi
{\bf $\bul$ Case 2: $\mu_{1,2} \ll \max(\mu_1, \mu_2)$.}\quad We for instance assume $\mu_2 = \max(\mu_1, \mu_{1,2})$. The case $\mu_1 = \max(\mu_2, \mu_{1,2})$ may be treated in a similar way. Then, by a change of variable, we have $|S^1_{k_1, k_2, k_3}[\mu]| \le |S^2_{k_1, k_2, k_3}[\mu]|$, where
\begin{align*}
S^2_{k_1, k_2, k_3}[\mu]&  = \Big\{ (y_1, y_2) \in \T^2_{k_1} \times \T^2_{0} : \frac{\mu_1}{2}\le \big||t- t_1| - |y_1|\big| < \mu_1, \\
& \qquad \qquad \qquad \qquad \qquad \, \frac{\mu_{1,2}}{2}\le \big||t_1- t_2| - |y_2|\big| < \mu_{1,2},\\
& \qquad \qquad \qquad \qquad \qquad \, \frac{\mu_{2}}{2}\le \big||t- t_2| - |y_1 - y_2 + k_3|\big| < \mu_{2} \Big\}.
\end{align*}
As in \eqref{Ysg106}-\eqref{Ysg107}, we have that
\begin{align}
\jb{y_1, y_2} - \jb{y_1 - y_2, k_3} = \frac12 |k_3|^2 + |y_1| |y_2| + O\big(|t-t_2|\mu_{2} + \mu_{2}^2\big).
\label{Ysg114}
\end{align}
Therefore, from \eqref{Ysg114} and as in \eqref{Ysg109}, we have
\begin{align}
\cos(\angle(y_2, y_1 +k_3)) =  C_2 + O\Big(\frac{|t - t_2| \mu_{2}}{|y_2| |y_1 + k_3|}\Big),
\label{Ysg115}
\end{align}
where $C_2 = C_2(k_3, y_1, |y_2|)$ is a function depending only on $y_1$ and $|y_2|$. 

\medskip

\noi
{\bf $\bul$ Subcase 2.1: $k_3 = 0$.}\quad From \eqref{Ysg115}, there exists an interval $J_{3} = J_{3}(k_3,y_1, |y_2|)$ with 
\begin{align}
 |J_3| \les \frac{|t - t_2| \mu_{2}}{|y_2| |y_1|}
\label{Ysg116a}
\end{align}
such that $\cos(\angle(y_2, y_1 + k_3 )) \in J_3$. From \eqref{Ysg110}, \eqref{Ysg116a} and since $\angle(y_1,e_1)$, with $e_1 = (1,0$), belongs to an interval $J_4(k_1)$ of length $\les \jb{k_1}^{-1}$, we have that
\begin{align}
\begin{split}
\mf C^{-, \bar \al}_{k_1, k_2, k_3}[\mu] & \les \int_{[0,2\pi]} \ind_{\ta_1 \in J_4(k_1)} d\ta_2 \int_{\R_+^2} \ind_{||t-t_1| - r_1| < \mu_1} \ind_{||t_1-t_2| - r_2| < \mu_{1,2}} r_1 r_2 dr_1 dr_2  \\
& \qquad \qquad \qquad \qquad \quad \times \int_{[0,2\pi]} \ind_{\cos \ta_1 \in J_3(k_3, r_1, r_2, \ta_1)} d\ta_2 \\
&  \les \big(|t - t_2|\mu_{2}\big)^{\frac12} \int_{[0,2\pi]} \ind_{\ta_2 \in J_4(k_1)} d\ta_2  \\
& \qquad \qquad \qquad \times \int_{\R_+^2} \ind_{||t-t_1| - r_1| < \mu_1} \ind_{||t_1-t_2| - r_2| < \mu_{1,2}} (r_1 r_2)^{\frac12} dr_1 dr_2 \\
& \les \big( |t- t_1| |t-t_2| |t_1 - t_2|\big)^{\frac12} \, \jb{k_1}^{-1} \cdot \mu_1 \mu_{1,2}(\mu_{2})^{\frac12},
\end{split}
\label{Ysg116}
\end{align}
which shows \eqref{Ysg105} since $\jb{k_1} \sim \jb{t-t_1} \sim \jb{t-t_2}$.

\medskip

\noi
{\bf $\bul$ Subcase 2.2: $k_3 \neq 0$ and $|k_1| \gg 1$.}\quad As in Subcase 2.1 above, we have
\begin{align}
|y_1 + k_3| \sim |k_1|.
\label{Ysgg117}
\end{align}
Hence, from \eqref{Ysgg117}, \eqref{Ysg115} and arguing as in \eqref{Ysg116a}-\eqref{Ysg116}, we get
\begin{align*}
\mf C^{-, \bar \al}_{k_1, k_2, k_3}[\mu]&  \les \big(|t_1 - t_2| |t-t_2|\big)^{\frac12} |t - t_1| \jb{k_1}^{-\frac32} \cdot  \mu_1 \mu_{1,2}(\mu_{2})^{\frac12} \\
& \les |t_1 - t_2|^{\frac12} \cdot \mu_1 \mu_{1,2}(\mu_{2})^{\frac12},
\end{align*}
as desired in \eqref{Ysg105} since $\jb{k_1} \sim \jb{t-t_1} \sim \jb{t-t_2} \gg 1$.

\medskip

\noi
{\bf $\bul$ Subcase 2.3: $k_3 \neq 0$ and $|k_1| \les 1$.}\quad Fix $0 < \dl \ll 1$. If we have
\[ |y_2| \le (1-\dl) \cdot |y_1 - y_2 + k_3|,\]
then we deduce 
\begin{align}
\begin{split}
|y_1 + k_3|& \ge |y_1 - y_2 + k_3| - |y_2| \\
&  \ge \dl \cdot |y_1 - y_2 + k_3| \\
& \ges_{\dl}  |t - t_2|
\end{split}
\label{Ysgg117b}
\end{align}
Therefore, from \eqref{Ysgg117b}, \eqref{Ysg115} and arguing as in \eqref{Ysg116a}-\eqref{Ysg116}, we have that
\begin{align*}
\mf C^{-, \bar \al}_{k_1, k_2, k_3}[\mu]&  \les |t_1 - t_2|^{\frac12} |t - t_1| \, \jb{k_1}^{-1} \cdot  \mu_1 \mu_{1,2}(\mu_{2})^{\frac12} \\
& \les |t_1 - t_2|^{\frac12} |t - t_1| \cdot \mu_1 \mu_{1,2}(\mu_{2})^{\frac12},
\end{align*}
as desired in \eqref{Ysg105} since $\jb{k_1} \sim 1$. Otherwise, we have 
\[  |y_1 - y_2 + k_3| < (1-\dl)^{-1} \cdot |y_2|,\]
\begin{align*}
|y_1|&  \le |t-t_1| + \mu _1 \le |t- t_2| + |t_1 - t_2| + \mu_1\\
&  \le |y_1 - y_2 + k_3| + |t_1 - t_2| + \mu_1 +  \mu_2\\
& \le (1-\dl)^{-1} \cdot |y_2| + |t_1 - t_2| + \mu_1 + \mu_2 \\
& \le \big((1 - \dl)^{-1} + 1\big) |t_1 - t_2| + 10\max(\mu_1, \mu_2, \mu_{1,2})\\
& < 3,
\end{align*}
since $(t_1, t_2) \in [0,1]^2$ and $\mu_1, \mu_2, \mu_{1,2}, \dl \ll 1$. Hence, $|k_3 | \ge 2\pi$ since $k_3 \neq 0$ and we have 
\begin{align}
|y_1 + k_3| > 1.
\label{Ysgg118}
\end{align}
Therefore, from \eqref{Ysgg118}, \eqref{Ysg115} and arguing as in \eqref{Ysg116a}-\eqref{Ysg116}, we have that
\begin{align*}
\mf C^{-, \bar \al}_{k_1, k_2, k_3}[\mu]&  \les \big(|t-t_2| |t_1 - t_2|\big)^{\frac12} |t - t_1| \, \jb{k_1}^{-1} \cdot  \mu_1 \mu_{1,2}(\mu_{2})^{\frac12} \\
& \les  \big(|t-t_2| |t_1 - t_2|\big)^{\frac12} |t - t_1| \cdot \mu_1 \mu_{1,2}(\mu_{2})^{\frac12},
\end{align*}
as desired in \eqref{Ysg105} since $\jb{k_1} \sim 1$.

%
%
%
\end{proof}

We now prove Lemma \ref{LEM:sing1}.

\begin{proof}[Proof of Lemma \ref{LEM:sing1}]
We have
\begin{align}
\mf I^{+,b,s}(t) & = \sum_{k_1, k_2 \in 2\pi \Z^2} \sum_{r_1, r_2 \in \Z_{\ge0}} \mf I^{+,b,s}_{r_1, r_2, k_1, k_2}(t), \label{Ysg100b} \\
\mf I^{-,b,s_1,s_2}(t) & = \sum_{k_1, k_2 \in 2\pi \Z^2} \sum_{r_1, r_2 \in \Z_{\ge0}} \mf I^{-,b,s_1,s_2}_{r_1, r_2, k_1, k_2}(t) \label{Ysg100c}
\end{align}
for any $t \in \R_+$, where
\begin{align*}
\mf I^{+,b,s}_{r_1, r_2, k_1, k_2}(t)&  = \int_{[0,1]^2 \times (\R^2)^2} \prod_{j = 1}^2 dt_jdy_j \, \mf K_{b}(t-t_j, x_j) \ind_{\T^2_{k_j}}(x_j) \ind_{A_{r_j} (|t-t_j|)}(x_j)dx_1 dx_2 \\
& \qquad \qquad \qquad \times \big(|t_1 - t_2| + |x_1 - x_2|_{\T^2} \big)^{-s}
\end{align*}
and
\begin{align*}
\mf I^{-,b,s_1,s_2}_{r_1, r_2, k_1, k_2}(t)&  = \int_{[0,1]^2 \times (\R^2)^2} \prod_{j = 1}^2 dt_jdy_j \, \mf K_{b}(t-t_j, x_j) \ind_{\T^2_{k_j}}(x_j) \ind_{A_{r_j} (|t-t_j|)}(x_j)dx_1 dx_2 \\
& \qquad \qquad \quad \times \big(|t_1 - t_2| + |x_1 - x_2|_{\T^2} \big)^{-s_1} \big| |t_1 - t_2| -  |x_1 - x_2|_{\T^2}\big|^{-s_2} .
\end{align*}
Note that the sums in \eqref{Ysg100b} and \eqref{Ysg100c} are finite for $t$ fixed: in view of the spatial localization in the kernel \eqref{kerhyp}, we have $|k_j|_{\R^2} \les \jb t$ and $r_j \le 100 (1+|t|)$ for $j=1,2$. We claim the following estimates on the localized integrals defined above:
\begin{align}
\sup_{k_1, k_2 \in 2\pi \Z^2} \mf I^{+,b,s}_{r_1, r_2, k_1, k_2}(t) & \les \jb{r_1}^{-\frac32-b} \jb{r_2}^{-\frac32-b} \jb{t}^{-3 -2b}, \label{Ysg118} \\
\sup_{k_1, k_2 \in 2\pi \Z^2} \mf I^{-,b,s_1,s_2}_{r_1, r_2, k_1, k_2}(t)& \les \jb{r_1}^{-\frac32-b} \jb{r_2}^{-\frac32-b} \jb{t}^{-3 -2b}. \label{Ysg119}
\end{align}
for any $t\in\R_+$ and $(r_1, r_2) \in (\Z_{\ge0})^2$. Let us show how the bounds \eqref{Ysg118} and \eqref{Ysg119} imply \eqref{Ysg99} and \eqref{Ysg100}. Assume the estimate \eqref{Ysg118}. Then, for any fixed $r \le 100 (1+|t|)$, a volume packing argument shows that 
\[ \# \big\{ k \in 2\pi \Z^2 : A_{r}(|t-t'|) \cap \T^2_k \neq \emptyset \big\} \les 1 +  |t-t'| - \frac{r}{100} \les \jb{t},\]
uniformly in $|t'| \le 1$. Hence, we have
\begin{align}
\# \{ (k_1, k_2) \in (2\pi \Z^2)^2 :  \mf I^{+,b,s}_{r_1, r_2, k_1, k_2}(t)  \neq 0 \} \les \jb t ^2.
 \label{Ysg120}
 \end{align}
for each fixed $(r_1, r_2) \in (\Z_{\ge0})^2$ and $t \in \R$, with an implicit constant which is uniform in the parameters $r_1, r_2$ and $t$. Therefore, from \eqref{Ysg100b}, \eqref{Ysg118} and \eqref{Ysg120}, we have that
\begin{align*}
\mf I^{+,b,s}(t) & \les \jb t^{-1 - 2b} \cdot \sum_{\substack{r_1 \in \Z_{\ge 0} \\ 0 \le r_1 \le 100 (1+|t|)}} \, \sum_{\substack{r_2 \in \Z_{\ge 0} \\ 0 \le r_2 \le 100(1+|t|) }}  \jb{r_1}^{-\frac32-b} \jb{r_2}^{-\frac32-b} \\
& \les \jb t^{-2 - 4b},
\end{align*}
which is exactly \eqref{Ysg99}. Similarly, one shows that \eqref{Ysg100} follows from \eqref{Ysg119} together with \eqref{Ysg100c}. 

We first prove the simpler bound \eqref{Ysg118} and start with the case $(r_1, r_2) = (0,0)$. First, we rewrite $|x_1 - x_2|_{\T^2}$ in the integrand of $\mf I^{+,b,s}_{r_1, r_2, k_1, k_2}(t)$ by using our spatial localizations. Let $(x_1,x_2) \in \T^2_{k_1} \times \T^2_{k_2}$ and write $x_1 = y_1 + k_1$ and $x_2 = y_2 + k_2$ with $(y_1, y_2) \in [-\pi, \pi)^2$. Thus, by definition of the norm $|\cdot|_{\T^2}$, we have
\begin{align}
\begin{split}
|x_1 - x_2|_{\T^2} & = \min_{k_3 \, \in 2 \pi \Z^2} |y_1 - y_2 +  k_3 + k_2 - k_1|_{\R^2} \\
& = \min_{k_3 \, \in B(k_1 - k_2, 4\pi) \cap 2 \pi \Z^2 } |y_1 - y_2 + k_3 + k_2 - k_1|_{\R^2} \\
& = \min_{k_3 \, \in B(k_1-k_2, 4\pi) \cap 2 \pi \Z^2 } |x_1 - x_2 + k_3|_{\R^2},
\end{split}
\label{Ysg121}
\end{align}
since $|y_1 - y_2|_{\R^2} \in [-2\pi, 2\pi)^2 \subset B(0,3\pi)$. Fix $k_1, k_2 \in 2\pi \Z^2$. Note that by definition of the annulus $A_r(\ld)$ in \eqref{Ysgloc}, we have
\begin{align}
\sup_{(t,y) \in \R \times \R^2} \int_{[0,1]} \big| |t-t'| - |y|_{\R^2}\big|^{-\frac32-b} \ind_{A_0(|t-t'|)}(y) \, dt' < \infty
\label{Ysg122}
\end{align}
for $b < -\frac12$. If $|t| \les 1$ (which implies $|k_1|_{\R^2}, |k_2|_{\R^2} \les 1$), then by using the estimates \eqref{Ysg122} and \eqref{Ysg121} together with \eqref{Ysg102a} in Lemma \ref{LEM:sing2}, we find 
\begin{align}
\begin{split}
\mf I^{+, b,s}_{0,0,k_1,k_2}(t) & \les \int_{\T^2_{k_1} \times \T^2_{k_2}} |x_1|_{\R^2}^{-\frac32 -b}  |x_2|_{\R^2}^{-\frac32 -b} |x_1 - x_2|_{\T^2}^{-s} d x_1 dx_2 \\
&  \qquad \quad \times  \int_{[0,1]} \big| |t-t_1| - |x_1|_{\R^2}\big|^{-\frac32-b}\ind_{A_0(|t-t_1|)}(x_1)\, dt_1  \\
& \qquad \qquad \quad \times  \int_{[0,1]}   \big| |t-t_2| - |x_2|_{\R^2}\big|^{-\frac32-b}  \ind_{A_0(|t-t_2|)}(x_2)  \, dt_2\\
& \les  \int_{\T^2_{k_1} \times \T^2_{k_2}} |x_1|_{\R^2}^{-\frac32 -b}  |x_2|_{\R^2}^{-\frac32 -b} |x_1 - x_2|_{\T^2}^{-s} d x_1 dx_2 \\
& \les \sum_{k_3 \, \in B(k_1-k_2, 4\pi) \cap 2 \pi \Z^2} \mf O^{+, \frac32+b, \frac32+b, s}_{k_1, k_2, k_3} \\
& \les 1,
\end{split}
\label{Ysg123}
\end{align}
since the set $B(k_1-k_2, 4\pi) \cap 2 \pi \Z^2$ has at most $10$ elements and $2 b + s <1$ (which is always true as $b < -\frac12$ and $s<2$). Here, $ \mf O^{+, \eta_1, \eta_2, \eta_{1,2}}_{k_1, k_2, k_3}$ is as in \eqref{Ysg101bb}. Similarly, if $|t| \gg 1$, we have
\begin{align}
\begin{split}
\mf I^{+, b,s}_{0,0,k_1,k_2}(t) & \les \jb{t}^{-3 - 2b} \int_{\T^2_{k_1} \times \T^2_{k_2}}  |x_1 - x_2|_{\T^2}^{-s} d x_1 dx_2 \\
&  \qquad \quad \times   \int_{[0,1]} \big| |t-t_1| - |x_1|_{\R^2}\big|^{-\frac32-b}\ind_{A_0(|t-t_1|)}(x_1) dt_1  \\
& \qquad \qquad \quad \times  \int_{[0,1]}   \big| |t-t_2| - |x_2|_{\R^2}\big|^{-\frac32-b}  \ind_{A_0(|t-t_2|)}(x_2)  dt_2\\
& \les  \jb{t}^{-3 - 2b} \sum_{k_3 \, \in B(k_1-k_2, 4\pi) \cap 2 \pi \Z^2} \mf O^{+, 0,0, s}_{k_1, k_2, k_3} \\
& \les  \jb{t}^{-3 - 2b}.
\end{split}
\label{Ysg124}
\end{align}
Therefore, combining \eqref{Ysg123} and \eqref{Ysg124} yields
\begin{align}
\mf I^{+, b,s}_{0,0,k_1, k_2}(t) \les \jb{t}^{-3 - 2b}
\label{Ysg125}
\end{align}
for all $t \in \R$, as desired in \eqref{Ysg118}. Now, assume $r_1, r_2 \ge 1$. Then, by definition of the annulus $A_r(\ld)$ in \eqref{Ysgloc} and \eqref{Ysg121}, we have
\begin{align}
\begin{split}
\mf I^{+, b,s}_{r_1, r_2, k_1, k_2}(t) & \les  \jb{r_1}^{-\frac32-b} \jb{r_2}^{-\frac32-b} \int_{\T^2_{k_1} \times \T^2_{k_2}}  |x_1 - x_2|_{\T^2}^{-s} d x_1 dx_2 \\
&  \qquad \quad \times   \int_{[0,1]^2}  |t-t_1|^{-\frac32-b} |t-t_2|^{-\frac32-b} dt_1 dt_2 \\
& \les \jb{r_1}^{-\frac32-b} \jb{r_2}^{-\frac32-b}  \jb{t}^{-3 - 2b} \sum_{k_3 \, \in B(k_1-k_2, 4\pi) \cap 2 \pi \Z^2} \mf O^{+, 0,0, s}_{k_1, k_2, k_3} \\
& \les  \jb{r_1}^{-\frac32-b} \jb{r_2}^{-\frac32-b}  \jb{t}^{-3 - 2b},
\end{split}
\label{Ysg126}
\end{align}
as desired in \eqref{Ysg118}. Thus \eqref{Ysg125} and \eqref{Ysg126} proves \eqref{Ysg118} in the cases $(r_1, r_2) = 0$ and $r_1, r_2 \ge 1$. The mixed case $r_1 = 0$ or $r_2 = 0$ and $(r_1, r_2) \neq 0$ is treated via similar arguments; we omit details.

We now turn our attention to \eqref{Ysg119}. Consider the contribution of $\big||t_1 - t_2| - |x_1 - x_2|_{\T^2}\big| \ges 1$ to the integrand of $\mf I^{-,b,s_1,s_2}_{r_1, r_2, k_1, k_2}(t)$. Then, $\mf I^{-,b,s_1,s_2}_{r_1, r_2, k_1, k_2}(t)$ basically reduces to $\mf I^{+,b,s_1}_{r_1, r_2, k_1, k_2}(t)$ and the bound \eqref{Ysg119} follows from \eqref{Ysg118}. Thus, in what follows, we only need to bound $\mf I^{-,b,s_1,s_2}_{r_1, r_2, k_1, k_2}(t)$ under the assumption $\big||t_1 - t_2| - |x_1 - x_2|_{\T^2}\big| \ll 1$.

We consider the case $(r_1, r_2) = (0,0)$. From \eqref{Ysg121} together with \eqref{Ysg102} in Lemma \ref{LEM:sing2}, we have
\begin{align}
\begin{split}
\mf I^{-,b,s_1,s_2}_{0,0, k_1, k_2}(t) & \les \int_{[0,1]^2}  |t-t_1|^{-\frac32-b} |t-t_2|^{-\frac32-b} |t_1 - t_2|^{-s_1} dt_1 dt_2 \\
& \qquad  \times \int_{\T^2_{k_1} \times \T^2_{k_2}}  \ind_{A_0(|t-t_1|)}(x_1) \ind_{A_0(|t-t_2|)}(x_2)   \ind_{||t_1 - t_2| - |x_1 - x_2|_{\T^2}|\ll 1}  \\
& \qquad \qquad \qquad \quad \times \frac{\big| |t_1 - t_2| - |x_1 - x_2 |_{\T^2}\big|^{-s_2}}{\big||t-t_1| - |x_1|_{\R^2}\big|^{\frac32+b} \big||t-t_2| - |x_2|_{\R^2}\big|^{\frac32+b}}  dx_1 dx_2 \\
& \les \sum_{k_3 \, \in B(k_1-k_2, 4\pi) \cap 2 \pi \Z^2}  \int_{[0,1]^2}  |t-t_1|^{-\frac32-b} |t-t_2|^{-\frac32-b} |t_1 - t_2|^{-s_1}  \\
& \qquad \qquad \qquad \qquad \qquad \qquad \times \mf C^{-, \frac32+ b, \frac32+b, s_2}_{k_1, k_2, k_3}(t,t_1,t_2) dt_1 dt_2 \\
& \les \int_{[0,1]^2}  |t-t_1|^{-\frac32-b} |t-t_2|^{-\frac32-b} |t_1 - t_2|^{1-s_1 -s_2} \, \o(t,t_1,t_2)  dt_1 dt_2,
\end{split}
\label{Ysg127}
\end{align}
where the function $\o$ is as in \eqref{Ysgfunc}. Here, we used the fact that $B(k_1-k_2, 4\pi) \cap 2 \pi \Z^2$ has at most $10$ elements along with the conditions $\frac32+b < 1$ and $s_2 < \frac12$. A simple computation then shows
\begin{align}
\int_{[0,1]^2}  |t-t_1|^{-\frac32-b} |t-t_2|^{-\frac32-b} |t_1 - t_2|^{1-s_1 -s_2} \, \o(t,t_1,t_2)  dt_1 dt_2 \les \jb t^{-3-2b}
\label{Ysg128}
\end{align}
for $\frac32+b<1$ and $s_1 + s_2 < 2$. Thus, by combining \eqref{Ysg127} and \eqref{Ysg128}, we deduce
\begin{align}
\mf I^{-,b,s_1,s_2}_{0,0, k_1, k_2}(t) \les \jb t^{-3-2b}.
\label{Ysg129}
\end{align}
We now treat the case when $r_1, r_2 \ge 1$. Assume $|t| \les 1$. By proceeding as in \eqref{Ysg123}, we have
\begin{align}
\begin{split}
\mf I^{-,b,s_1,s_2}_{r_1, r_2, k_1, k_2}(t) & \les \jb{r_1}^{-\frac32-b} \jb{r_2}^{-\frac32-b} \int_{\T^2_{k_1} \times \T^2_{k_2}} |x_1|_{\R^2}^{-\frac32-b}  |x_2|_{\R^2}^{-\frac32-b} |x_1 - x_2|^{-s_1}_{\T^2} dx_1 dx_2 \\
& \qquad \qquad \qquad \quad \times \int_{[0,1]^2} \big| |t_1 - t_2| - |x_1 - x_2 |_{\T^2}\big|^{-s_2} dt_1 dt_2 \\
& \les \jb{r_1}^{-\frac32-b} \jb{r_2}^{-\frac32-b} \int_{\T^2_{k_1} \times \T^2_{k_2}} |x_1|_{\R^2}^{-\frac32-b}  |x_2|_{\R^2}^{-\frac32-b} |x_1 - x_2|^{-s_1}_{\T^2} dx_1 dx_2 \\
& \les  \jb{r_1}^{-\frac32-b} \jb{r_2}^{-\frac32-b},
\end{split}
\label{Ysg130}
\end{align}
where we used the conditions $b<-\frac12$, $s_1 < 2$ and $s_2 < \frac12$. Similarly, by arguing as in \eqref{Ysg124}, we have the bound
\begin{align}
\mf I^{-,b,s_1,s_2}_{r_1, r_2, k_1, k_2}(t)  \les \jb{r_1}^{-\frac32-b} \jb{r_2}^{-\frac32-b} \jb t^{-3-2b}
\label{Ysg131}
\end{align}
for $|t| \gg 1$. Thus \eqref{Ysg129}, \eqref{Ysg130} and \eqref{Ysg131} show \eqref{Ysg119} in the cases $(r_1, r_2) = 0$ and $r_1, r_2 \ge 1$. The mixed case $r_1 = 0$ or $r_2 = 0$ and $(r_1, r_2) \neq 0$ may be treated via similar arguments and we omit details.
\end{proof}

\subsection{Proofs of Propositions \ref{PROP:sto3} and \ref{PROP:sto4}}

We first start with the proof of Proposition \ref{PROP:sto4} which is a consequence of the results in Subsections \ref{SUBSEC:sto4} and \ref{SEC:sing}.

\begin{proof}[Proof of Proposition \ref{PROP:sto4}] Let $\be \in \R$ with $\be^2 \in [2\pi, 4\pi)$, $N_0 \in 2^{\Z}$, $(N,N_1, N_2) \in \N^3$ with $N_1 \ge N_2$ and $\l \in \{1,2\}$. Fix $\eps >0$ and $0 < \kk_{\circ} \ll \eps$ as in Lemma \ref{LEM:der_F}. Let $(t,x) \in \R \times \T^2 \cong \R \times [-\pi, \pi) ^2$.

We first prove \eqref{goal1}. From the properties of the operator $\Box_{\T^2}^{b}$ in Subsection \ref{SUBSEC:hyp} (see in particular \eqref{Ysgker}), we have
\begin{align}
\begin{split}
& \big|\big( \Box^{-\frac12-\eps} \, \partial_{x^\l}  (\P_{N_0} \ind_{[0,1]} \Ta^{\eps_0}_N) \big)(t,x)\big|^2 \\
& \qquad = \int_{( [0,1] \times \R^2)^2} \mf K_{-\frac12-\eps}(t-t_1, x-x_1)  \mf K_{-\frac12-\eps}(t-t_2, x-x_2) \\
& \qquad \qquad \qquad \quad \times \partial_{x_1^\l} (\P_{N_0} \Ta^{\eps_0}_N)(t_1,x_1) \cdot \cj{\partial_{x_2^\l} (\P_{N_0} \Ta^{\eps_0}_N)}(t_2,x_2)  \, dt_1 dt_2 dx_1 dx_2,
\label{Ysg132}
\end{split}
\end{align}
where $\Ta^{\eps_0}_N(t,\cdot)$ is interpreted as a $2\pi$-periodic function on $\R^2$. Recalling $\P_{N_0}$ has convolution kernel $\K_{N_0}$ and by \eqref{E1}-\eqref{E1b} and the smoothness of $\Ta^{\eps_0}_N(t,\cdot)$, we then have
\begin{align}
\begin{split}
& \E_{\muu_1 \otimes \PP} \Big[  \partial_{x_1^\l} (\P_{N_0} \Ta^{\eps_0}_N)(t_1,x_1) \cdot \cj{\partial_{x_2^\l} (\P_{N_0} \Ta^{\eps_0}_N)}(t_2,x_2) \Big] \\
& \qquad \qquad = \E_{\mu \otimes \PP} \Big[  \partial_{x_1^\l}  \partial_{x_2^\l}  \big\{\P_{N_0} \Ta^{\eps_0}_N(t_1,x_1) \cdot \cj{\P_{N_0} \Ta^{\eps_0}_N}(t_2,x_2)\big\} \Big] \\
&  \qquad \qquad = \partial_{x_1^\l}  \partial_{x_2^\l} \, \E_{\mu \otimes \PP} \Big[\P_{N_0} \Ta^{\eps_0}_N(t_1,x_1) \cdot \cj{\P_{N_0} \Ta^{\eps_0}_N}(t_2,x_2)\Big] \\
& \qquad \qquad = \partial_{x_1^\l}  \partial_{x_2^\l} \mf F_{N_0}[f_N](t_1, x_1, t_2, x_2),
\end{split}
\label{Ysg133}
\end{align}
where $\mf F_{N_0}$ and $f_N$ are as in \eqref{def_F} and \eqref{def_f}, respectively. 

If $\be^2 \in [2\pi, 3\pi)$, then by \eqref{Ysg132}, \eqref{Ysg133}, Lemma \ref{LEM:der_F} (i), we have
\begin{align}
 \E_{\muu_1 \otimes \PP} \Big[\big|\big( \Box^{-\frac12-\eps} \, \partial_{x^\l}  (\P_{N_0} \ind_{[0,1]} \Ta^{\eps_0}_N) \big)(t,x)\big|^2\Big] \les N_0^{\frac{\be^2}{2\pi} + C \kk_{\circ}} \cdot \mf I^{-,b,s_1, s_2}(t),
\end{align}
where $\mf I^{-,b,s_1, s_2}(t)$ is as in \eqref{Ysgi2} and with
\begin{align*}
b & = - \frac12-\eps, \\
s_1 & = -\frac12 - \frac{\be^2}{2\pi} - \kk_\circ, \\
s_2 & = -\frac32 + \frac{\be^2}{2\pi} - \kk_{\circ}.
\end{align*}
Note that $(b,s_1, s_2)$ satisfies the conditions in the statement of Lemma \ref{LEM:sing1}. Therefore, by \eqref{Ysg99} in Lemma \ref{LEM:sing1}, we have
\begin{align}
 \E_{\muu_1 \otimes \PP} \Big[\big|\big( \Box^{-\frac12-\eps} \, \partial_{x^\l}  (\P_{N_0} \ind_{[0,1]} \Ta^{\eps_0}_N) \big)(t,x)\big|^2\Big] \les N_0^{\frac{\be^2}{2\pi} + C \kk_{\circ}} \jb{t}^{4\eps},
\end{align}
as claimed in \eqref{goal1}.

The case $\be^2 \in [3\pi, 4\pi)$ follows in a similar fashion by using Lemma \ref{LEM:der_F} (ii) and \eqref{Ysg100} in Lemma \ref{LEM:sing1}. This proves \eqref{goal1}.

We now prove \eqref{goal1bbb}. By proceeding as in \eqref{Ysg132}-\eqref{Ysg133}, we have
\begin{align}
\begin{split}
& \E_{\muu_1 \otimes \PP} \Big[\big|\big( \Box^{-\frac12-\eps} \, \partial_{x^\l}  (\P_{N_0} \ind_{[0,1]}(\Ta^{\eps_0}_{N_1} - \Ta^{\eps_0}_{N_2}) \big)(t,x)\big|^2\Big] \\
& \qquad = \int_{( [0,1] \times \R^2)^2} dt_1 dt_2 dx_1 dx_2, \mf K_{-\frac12-\eps}(t-t_1, x-x_1)  \mf K_{-\frac12-\eps}(t-t_2, x-x_2) \\
& \qquad \qquad \quad  \times \partial_{x_1^\l}  \partial_{x_2^\l} \, \E_{\mu \otimes \PP} \Big[\P_{N_0}(\Ta_{N_1}^{\eps_0} - \Ta^{\eps_0}_{N_2})(t_1,x_1) \cdot \cj{\P_{N_0}(\Ta_{N_1}^{\eps_0} - \Ta^{\eps_0}_{N_2})}(t_2,x_2)\Big].
\end{split}
\label{Ysg400} 
\end{align}
Next, by arguing as in \eqref{H1}-\eqref{H2}, we have that
\begin{align}
\begin{split}
& \Big|\partial_{x_1^\l}  \partial_{x_2^\l} \, \E_{\mu \otimes \PP} \Big[\P_{N_0}(\Ta_{N_1}^{\eps_0} - \Ta^{\eps_0}_{N_2})(t_1,x_1) \cdot \cj{\P_{N_0}(\Ta_{N_1}^{\eps_0} - \Ta^{\eps_0}_{N_2})}(t_2,x_2)\Big]\Big| \\
& \le \sum_{j=1}^2 \int_{(\T^2)^2} \big| \partial_{x_1^\l} \mc K_{N_0}(x_1 - y_1) \big| \big| \partial_{x_2^\l} \mc K_{N_0}(x_2 - y_2)\big| dy_1 dy_2 \\
& \qquad \qquad \qquad \qquad \times   \Big|e^{ \be^2 \G_{N_j}(t_1 - t_2 ,y_1 - y_2) }-e^{ \be^2 \G_{N_1, N_2}(t_1 - t_2 ,y_1 - y_2)}\Big| \\
& \les N_1^{-\dl} N_0^{6} \int_{(\T^2)^2} |y_1 - y_2|^{-\frac{\be^2}{2\pi} - \dl} dy_1 dy_2 \\
& \les  N_1^{-\dl} N_0^{6},
\end{split}
\label{Ysg401}
\end{align}
where $\dl = \dl(\be)$ is small enough so that $\frac{\be^2}{2\pi} + \dl < 2$. Therefore, from \eqref{Ysg400} and \eqref{Ysg401}, we deduce
\begin{align*}
 & \E_{\muu_1 \otimes \PP} \Big[\big|\big( \Box^{-\frac12-\eps} \, \partial_{x^\l}  (\P_{N_0} \ind_{[0,1]}(\Ta^{\eps_0}_{N_1} - \Ta^{\eps_0}_{N_2}) \big)(t,x)\big|^2\Big] \\
 & \qquad  \les N_1^{-\dl} N_0^{6} \big\| \mf K_{-\frac12-\eps}(t-t', y)\big\|_{L_{t',y}^1([0,1] \times \R^2)}^2 \\
 & \qquad \les N_1^{-\dl} N_0^{6} \, \jb t^{4\eps},
\end{align*}
proving \eqref{goal1bbb} as claimed.
\end{proof}

Next, we present a proof of Proposition \ref{PROP:sto3}.

 \begin{proof}[Proof of Proposition \ref{PROP:sto3}]
Fix $0 < T \le 1$. From the definition of restriction norms \eqref{loc1}, our goal is to prove the following bounds:
\begin{align}
 \big\|\qf_{-\frac12-\eps} \P_{\textup{hi}}\Q^{\textup{hi,hi}}(\ind_{[0,1]} \Ta^{\eps_0}_N)\big\|_{L^2(\muu_1 \otimes \PP) Y_{-\frac12-3\eps}^{-\al,-\frac12-\eps}} & \les 1, \label{Ysg150g1} \\
  \big\|\qf_{-\frac12-\eps} \P_{\textup{hi}}\Q^{\textup{hi,hi}}(\ind_{[0,1]} (\Ta^{\eps_0}_{N_1}-\Ta^{\eps_0}_{N_2}))\big\|_{L^2(\muu_1 \otimes \PP) Y_{-\frac12-3\eps}^{-\al,-\frac12-\eps}} & \les N^{-\dl} \label{Ysg150g2}
\end{align}
for any integers $N_2 \ge N_1 \ge 1$ and with some implicit constants are independent of $N_1$ and $N_2$.
 
We first prove \eqref{Ysg150g1}. From \eqref{proj4a} and \eqref{proj4} and Plancherel's identity in space (in order to move around spatial derivatives), we have
\begin{align}
\begin{split}
& \big\|\qf \P_{\textup{hi}}\Q^{\textup{hi,hi}}(\ind_{[0,1]} \Ta^{\eps_0}_N)\big\|_{Y_{-\frac12-3\eps}^{-\al,-\frac12-\eps}} \\
& \qquad \les \sum_{\substack{N_0, R \, \in 2^{\N} \\ N_0 \sim R}}  N_0^{-\al-1} \cdot \big\| \mathbf{T}_{R} \P_{N_0} \qf_{-\frac12-\eps} \P_{\textup{hi}}\Q^{\textup{hi,hi}}(\ind_{[0,1]} \Ta^{\eps_0}_N)\big\|_{Y_{-\frac12-3\eps}^{1,-\frac12-\eps}} \\
& \qquad \les \sum_{\substack{N_0, R \, \in 2^{\N} \\ N_0 \sim R}} N_0^{-\al-1} \cdot \big\| \jb t^{-\frac12-3\eps} \big( ||\dt| -|\nb||^{-\frac12-\eps} \qf_{-\frac12-\eps} \mathbf{T}_{R} \P_{N_0} |\nb| (\ind_{[0,1]} \Ta^{\eps_0}_N)\big)\big\|_{L^{2}_{t,x}}.
\end{split}
\label{Ysg151}
\end{align} 
Let $\ld \in C^{\infty}_c(\R; \R)$ be a bump function such that
\begin{align*}
\ld(\tau) = \begin{cases} 1 \quad & \text{for $10^{-10} \le |\tau| \le 10^{10}$} \\
0 \quad & \text{for $|\tau| < 10^{-10}$ or $|\tau| > 10^{10}$},
\end{cases}
\end{align*}
such that 
\begin{align*}
 \eta\Big(\frac{\tau}{R}\Big) \phi\Big(\frac{n}{N_0}\Big) \cdot \ld\Big(\frac{|\tau| + |n|}{N_0}\Big) =  \eta\Big(\frac{\tau}{R}\Big) \phi\Big(\frac{n}{N_0}\Big)
\end{align*}
for any $(\tau, n) \in \R \times \Z^2$ and $N_0 \sim R$. Here, the bump functions $\phi$ and $\eta$ are as in \eqref{eta1} and $\eta\big(\tau / R \big)$ and $\phi\big(n/N_0\big)$ are the symbols of the multipliers $\mathbf{T}_{R}$ and $\P_{N_0}$, respectively. Set $\wt \ld(\tau) = |\tau|^{\frac12+\eps} \ld(\tau)$ and let $T_{N_0}$ be the Fourier multiplier on $\R \times \T^2$ given by
\[ \ft{T_{N_0}u}(\tau, n) = \wt \ld \Big(\frac{|\tau|+|n|}{N_0}\Big) \cdot \ft u (\tau,n), \quad (\tau,n) \in \R \times \Z^2. \]
With these notations, we continue
\begin{align}
\begin{split}
& \big\| \jb t^{-\frac12-3\eps} \big( ||\dt| -|\nb||^{-\frac12-\eps} \qf_{-\frac12-\eps} \mathbf{T}_{R} \P_{N_0} |\nb| (\ind_{[0,1]} \Ta^{\eps_0}_N)\big)\big\|_{L^{2}_{t,x}} \\
& \qquad \quad = N_0^{\frac12+\eps} \big\| \jb t^{-\frac12-3\eps} \big( T_{N_0} \, \big||\dt|^2 -|\nb|^2\big|^{-\frac12-\eps} \qf_{-\frac12-\eps} \mathbf{T}_{R} \P_{N_0} |\nb|  (\ind_{[0,1]} \Ta^{\eps_0}_N)\big)\big\|_{L^{2}_{t,x}} \\
& \qquad \quad = N_0^{\frac12+\eps} \big\| \jb t^{-\frac12-3\eps} \big( T_{N_0} \mathbf{T}_{R} \, \Box^{-\frac12-\eps} \P_{N_0} |\nb|  (\ind_{[0,1]} \Ta^{\eps_0}_N)\big)\big\|_{L^{2}_{t,x}} \\
& \qquad \quad \les N_0^{\frac12+\eps} \big\| \jb t^{-\frac12-3\eps} \big(\Box^{-\frac12-\eps} |\nb|  \P_{N_0} (\ind_{[0,1]} \Ta^{\eps_0}_N)\big)\big\|_{L^{2}_{t,x}}.
\end{split}
\label{Ysg152}
\end{align}
Here, we used Corollary \ref{COR:AP} twice to the operators $T_{N_0}$ and $\mathbf{T}_R$ and the fact that Fourier multipliers $|\nb|$ and $\P_{N_0}$ commute in the last inequality. 

Now, by taking the $L^2(\mu \otimes \PP)$-norm with H\"older's inequality and Proposition \ref{PROP:sto4}, we deduce that
\begin{align}
\begin{split}
& \big\| \jb t^{-\frac12-3\eps} \big(\Box^{-\frac12-\eps} |\nb|  \P_{N_0} (\ind_{[0,1]} \Ta^{\eps_0}_N)\big)\big\|_{L^2(\muu_1 \otimes \PP) L^{2}_{t,x}} \\
& \qquad = \big\| \jb t^{-\frac12-3\eps} \big\|\Box^{-\frac12-\eps} |\nb|  \P_{N_0} (\ind_{[0,1]} \Ta^{\eps_0}_N)\big\|_{L^2_x L^2(\muu_1 \otimes \PP)}\big\|_{L^{2}_{t}} \\
& \qquad \les N_0^{\frac{\be^2}{4\pi} + \eps} \cdot \| \jb t^{-\frac12-\eps}\|_{L^{2}_{t}} \\
& \qquad \les N_0^{\frac{\be^2}{4\pi} + \eps}.
\end{split}
\label{Ysg153}
\end{align}
Hence, from \eqref{Ysg151}, \eqref{Ysg152}, \eqref{Ysg153} and taking the $L^2(\mu \otimes \PP)$-norm, we have
\begin{align*}
\big\|\qf \P_{\textup{hi}}\Q^{\textup{hi,hi}}(\ind_{[0,1]} \Ta^{\eps_0}_N)\big\|_{L^2(\muu_1 \otimes \PP)Y_{-\frac12-3\eps}^{-\al,-\frac12-\eps}} \les \sum_{\substack{N_0, R \, \in 2^{\N} \\ N_0 \sim R}}  N_0^{-\al-\frac12 + \frac{\be^2}{4\pi} +2\eps} \les 1
\end{align*}
when $\al >  \frac{\be^2}{4\pi} -\frac12 + 2\eps$. This proves \eqref{Ysg150g1}.

We now focus on \eqref{Ysg150g2}. By proceeding as in \eqref{Ysg153} and interpolating \eqref{goal1} and \eqref{goal1bbb}, we get
\begin{align}
 \big\| \jb t^{-\frac12-3\eps} \big(\Box^{-\frac12-\eps} |\nb|  \P_{N_0} (\ind_{[0,1]} (\Ta_{N_1} - \Ta^{\eps_0}_{N_2}))\big)\big\|_{L^2(\muu_1 \otimes \PP) L^{2}_{t,x}} \les N_0^{\frac{\be^2}{4\pi} + 2\eps} N_1^{-\ta_1}
 \label{Ysg154}
\end{align}
for some small constant $\ta>0$. Therefore the bound \eqref{Ysg150g2} follows from \eqref{Ysg154} and arguments similar to the proof of the estimate \eqref{Ysg150g1} above.
 \end{proof}

\section{Well-posedness}\label{SEC:WP}

In this section, we prove Theorem \ref{THM:main}.

\subsection{A deterministic global well-posedness result} Here, we prove well-posedness on the time interval $[0,1]$ for the model equation:

\noi
\begin{align}
 v(t) = \wt{\mc{U}}(t)(u_0, v_0)  -   \wt \I \big( e^{i v} e^{i \Psi} \cdot  \Theta \big)(t), \quad t \in \R
\label{E}
\end{align}
for $\Ta$ and $\Psi$ two distributions and where $\wt{\mc{U}}$ and $\wt \I$ is as in \eqref{lin2} and \eqref{duha2}.

Next, we prove that \eqref{E} is well-posed on $[0,1]$. Recall the definition of the space $\mc Z^{\al,\eps}([0,1])$ in \eqref{normZ}.

\noi
\begin{proposition}\label{PROP:gwp}
Let $0 < \al < \frac{3\sqrt{241}-41}{244}$, $\eps = \eps(\al) >0$ a small real number and $\dl = \al+10\eps$. Then, the equation \eqref{E} is well-posed on $[0,1]$. More precisely, given an enhanced data set $(u_0, v_0,\Psi, \Ta)$ belonging to 
\[ \mc X^{\al, \dl, \eps}([0,1]) := \mc H^{\frac12+\dl}(\T^2) \times \big(\Ld_{\infty}^{1-\eps, 0} \cap \Ld^{0,\frac12-\eps}_{\infty}\big)([0,1]) \times \mc Z^{\al, \eps}([0,1]),\]
there exists a unique solution $v$ to \eqref{E} in the class $X^{\frac12+\dl, \frac12 + \frac \eps2}([0,1])$. Furthermore, the solution map
\[ (u_0, v_0, \Psi, \Ta) \in \mc X^{\al, \dl, \eps}([0,1]) \mapsto v \in X^{\frac12+\dl, \frac12 + \frac \eps2}([0,1])\]
is Lipschitz continuous.
\end{proposition}

\begin{proof}
Define the map $\G_{\Psi, \Ta}$ via 

\noi
\begin{align*}
\G_{\Psi, \Ta}(v)(t) = \wt{\mc{U}}(t)(u_0, v_0) - \wt \I \big( e^{i v} e^{i \Psi}  \Theta \big)(t).
\end{align*}

 We start by proving a difference estimate for $\G_{\psi, \Ta}$. Let $0<T\le 1$. From Lemmas \ref{LEM:lin}, \ref{LEM:inho}, \ref{LEM:timeloc} and \ref{LEM:restri} (ii), we have that 
\begin{align}
\begin{split}
\|\G_{\Psi, \Ta}(v_1) - \G_{\Psi, \Ta}(v_2)\|_{X^{\frac12+\dl, \frac12+\frac \eps2}_T} & \les  \|(e^{iv_1} - e^{iv_2}) e^{i\Psi} \Ta \|_{X^{-\frac12+\dl,-\frac12+ \frac \eps 2}_T} \\
&  \les  T^{\frac \eps 2} \|\ind_{[0,T]}(t)(e^{iv_1} - e^{iv_2}) e^{i\Psi} \Ta \|_{X^{-\frac12+\dl,-\frac12+\eps}} \\
& \les  T^{\frac \eps 2} \sup_{\substack{w \in X^{\frac12-\dl, \frac12-\eps} \\ \|w\|_{X^{\frac12-\dl, \frac12-\eps}} \le 1 } } A(w; v_1, v_2,\Psi, \Ta),
\end{split}
\label{p1}
\end{align}

\noi
where $A$ is given by

\noi
\begin{align*}
A(w; v_1, v_2, \Psi, \Ta) = \int_{\R \times \T^2} \tilde u  w \cdot \wt \Ta \, dt dx,
\end{align*} 

\noi
with 
\[\tilde u = \ld(t) F(v_1, v_2)  e^{i \wt \Psi},\]
where $F(v_1, v_2) = \ind_{[0,T]}(t)(e^{iv_1} - e^{iv_2})$, $\wt \Psi = \ind_{[0,1]}(t) \Psi$, $\ld \in C_c^{\infty}(\R; \R)$ such that $\ld \equiv 1$ on $[0,1]$ and $\wt \Ta = \ind_{[0,1]}(t) \Ta$. Recall the definitions of the multiplier $\mf q_{b}$ in \eqref{Rhyp2} and the space-time localizations in \eqref{proj4a}, \eqref{proj4a}, \eqref{proj4}, \eqref{proj4} and \eqref{proj4}. By Plancherel's identity and duality, we have
\noi
\begin{align*}
A(w; v_1, v_2, \Psi, \Ta) &= \int_{\R \times \T^2} \P_{\textup{lo}}\Q^{\textup{hi,hi}}(\tilde u w) \cdot \P_{\textup{lo}} \Q^{\textup{hi,hi}}\big( \wt \Ta \big)  +  \int_{\R \times \T^2} \P_{\textup{hi}}\mf q_{-\frac12-\eps}^{-1} \Q^{\textup{hi,hi}}( \tilde u w) \cdot \P_{\textup{hi}} \mf q_{-\frac12-\eps} \Q^{\textup{hi,hi}}\big( \wt \Ta\big) \\
&\qquad \qquad  + \int_{\R \times \T^2}\Q^{\textup{lo,hi}}(\tilde u w) \cdot \Q^{\textup{lo,hi}}\big(\wt \Ta \big)     + \int_{\R \times \T^2}\Q^{\textup{hi,lo}}( \tilde u w) \cdot \Q^{\textup{hi,lo}}\big(\wt \Ta\big)  \\
& \les \| \tilde u w\|_{L^1_{t,x}} \|\P_{\textup{lo}}\Q^{\textup{hi,hi}}(\tilde \Ta)\|_{L^{\infty}_{t,x}} + \big\|\mf q_{-\frac12-\eps}^{-1} \P_{\textup{hi}} \Q^{\textup{hi,hi}}(\tilde u w) \big\|_{Y_{\frac12+3\eps}^{\al,\frac12+\eps}} \|\mf q_{-\frac12-\eps} \Q^{\textup{hi,hi}}\big(\wt \Ta \big)\|_{Y_{-\frac12-3\eps}^{-\al,-\frac12-\eps}} \\
& \qquad \qquad \qquad  + \|\Q^{\textup{lo,hi}}(\tilde u w)\|_{\Ld^{\al,\frac12-2\eps}_1} \|\Q^{\textup{lo,hi}}\big(\wt \Ta \big)\|_{\Ld^{-\al,-\frac12+2\eps}_\infty} \\
& \qquad \qquad \qquad +  \|\Q^{\textup{hi,lo}}( \tilde u w)\|_{\Ld^{\frac12+\al,0}_{1+\eps}} \|\Q^{\textup{lo,hi}}\big(\wt \Ta \big)\|_{\Ld^{-\frac12-\al,0}_{\frac{1+\eps}{\eps}}} \\
& \les \Big( \|\tilde u w\|_{L^1_{t,x}}+ \big\|\mf q_{-\frac12-\eps}^{-1} \P_{\textup{hi}}  \Q^{\textup{hi,hi}}(\tilde u w) \big\|_{Y^{\al,\frac12+\eps}_{\frac12+3\eps}} +\|\Q^{\textup{lo,hi}}(\tilde u w)\|_{\Ld^{\al,\frac12-2\eps}_1} \\
& \qquad \qquad \qquad + \|\Q^{\textup{hi,lo}}(\tilde u w)\|_{\Ld^{\al + \frac12,0}_{1+\eps}} \Big) \cdot \|\Ta\|_{\mathcal Z^{\al,\eps}([0,1])}.
\end{align*}

\noi
Hence, by Propositions \ref{PROP:prod1}, \ref{PROP:prod2} and \ref{PROP:prod3}, we get

\noi
\begin{align}
\begin{split}
A(w; v_1, v_2, \Psi, \Ta)& \les \Big( \big\|F(v_1, v_2) e^{i \be \wt \Psi} \big\|_{\Ld^{\frac12+\dl_1,0}_{\frac{3}{2(1-\dl_1)}}}  +\big\| F(v_1, v_2) e^{i \be \wt \Psi} \big\|_{\Ld^{0,\frac12-\eps}_{\frac{3}{2+\dl_2}}} \\
& \qquad \qquad \quad  + \big\| F(v_1, v_2) e^{i \be \wt \Psi} \big\|_{L^2_{t,x}}  \Big) \|\Ta\|_{\mathcal Z^{\al,\eps}([0,1])},
\end{split}
\label{p3}
\end{align}
where $\dl_1 := \al + 5\eps$ and $\dl_2 = \al + 15\eps$. From the product estimate (Lemma \ref{LEM:prod}), the fractional chain rule (Lemma \ref{LEM:fcr}) and H\"older's and Sobolev's inequalities, we have
\begin{align}
\begin{split}
& \big\|F(v_1, v_2) e^{i \be \wt \Psi} \big\|_{L_t^{\frac{3}{2(1-\dl_1)}}W_x^{\frac12+\dl_1,\frac{3}{2(1-\dl_1)}}} \\ 
& \qquad \les \|F(v_1, v_2)\|_{L_t^{\frac{3}{2(1-\dl_1)}}W_x^{\frac12+\dl_1,\frac{3}{2(1-\dl_1)-3\eps}}} \big\| e^{i \be \wt \Psi} \big\|_{L_t^{\frac{3}{2(1-\dl_1)}}W_x^{\frac12+\dl_1,\frac{1}{\eps}}} \\
& \qquad \les \|F(v_1, v_2)\|_{L_t^{\frac{3}{2(1-\dl_1)}}W_x^{\frac12+\dl_1+\eps,\frac{3}{2(1-\dl_1)}}} \| \Psi\|_{\Ld_{\infty}^{1-\eps,0}([0,1])} \\
& \qquad \les \|F(v_1, v_2)\|_{L_t^{\frac{3}{2(1-\dl)}}W_x^{\frac12+\dl,\frac{3}{2(1-\dl)}}} \| \Psi\|_{\Ld_{\infty}^{1-\eps,0}([0,1])}.
\end{split}
\label{p33}
\end{align}
Similarly, we also have that
\begin{align}
\begin{split}
\big\| F(v_1, v_2) e^{i \be \wt \Psi} \big\|_{\Ld^{0,\frac12-\eps}_{\frac{3}{2+\dl_2}}} & \les \big\| F(v_1, v_2) \big\|_{\Ld^{0,\frac12-\eps}_{\frac{3}{2+\dl}}} \|\Psi\|_{\Ld_\infty^{0,\frac12-\eps}([0,1])}, \\
 \big\| F(v_1, v_2) e^{i \be \wt \Psi} \big\|_{L^2_{t,x}} & \les \| F(v_1, v_2)\|_{L^2_{t,x}} \|\Psi\|_{L^{\infty}([0,1] \times \T^2)}.
 \end{split}
 \label{p35}
\end{align}

By the mean value theorem, we may write $F(v_1,v_2)$ as 

\noi
\begin{align*}
F(v_1,v_2) = \ind_{[0,T]}(t) (v_1-v_2) \int_{0}^1 e^{i(v_1 + s(v_2-v_1)} ds =: \ind_{[0,1]}(t) (v_1-v_2) G(v_1,v_2).
\end{align*}

\noi
Hence, by Lemma \ref{LEM:prod}, Lemma \ref{LEM:fcr}, H\"older's and Sobolev's inequality in time and \eqref{stri1}, we have that 

\noi
\begin{align}
\begin{split}
\|F(v_1, v_2)\|_{\Ld^{\frac12+\dl,0}_{\frac{3}{2(1-\dl)}}} & =  \|(v_1-v_2) G(v_1,v_2)\|_{L_t^{\frac{3}{2(1-\dl)}}([0,T]) W_x^{\frac12+\dl,\frac{3}{2(1-\dl)}}} \\
&  \les \|v_1-v_2\|_{L_t^{\infty}([0,T]) H_x^{\frac12+\dl}} \|G(v_1,v_2)\|_{L_t^{\frac{3}{2(1-\dl)}}([0,T]) L_x^{\frac{6}{1-4\dl}}} \\
& \qquad  +\|G(v_1,v_2)\|_{L_t^{\infty}([0,T]) H_x^{\frac12+\dl}} \|v_1-v_2\|_{L_t^{\frac{3}{2(1-\dl)}}([0,T]) L_x^{\frac{6}{1-4\dl}}}  \\
& \les \|v_1-v_2\|_{X^{\frac12+\dl, \frac12+\frac{\eps}{2}}_T} \big(1+ \|v_1\|_{X^{\frac12+\dl, \frac12+\frac{\eps}{2}}_T} + \|v_2\|_{X^{\frac12+\dl, \frac12+\frac{\eps}{2}}_T}\big).
\end{split}
\label{p5}
\end{align}

\noi
Similarly, by Lemma \ref{LEM:prod}, Lemma \ref{LEM:fcr}, H\"older's inequality, \eqref{stri1} and Lemma \ref{LEM:restri} (ii), we have

\noi
\begin{align}
\begin{split}
\|F(v_1, v_2)\|_{\Ld^{0,\frac12-\eps}_{\frac{3}{2+\dl}}} & =  \| \ind_{[0,T]}(t) (v_1-v_2) G(v_1,v_2)\|_{L_x^{\frac{3}{2+\dl}} W_t^{\frac12-\eps,\frac{3}{2+\dl}}} \\
&  \les \|\ind_{[0,T]}(t)(v_1-v_2)\|_{L_x^{2} H_t^{\frac12-\eps}} \|G(v_1,v_2)\|_{L_x^{\frac{6}{1+2\dl}} L_t^{\frac{6}{1+2\dl}}([0,T])} \\
& \qquad  +\|\ind_{[0,T]}(t) G(v_1,v_2)\|_{L_x^{2} H_t^{\frac12+\eps}} \|v_1-v_2\|_{L_x^{\frac{6}{1+2\dl}} L_t^{\frac{6}{1+2\dl}}([0,T])}  \\
& \les \|v_1-v_2\|_{X^{\frac12+\dl, \frac12+\frac{\eps}{2}}_T} \big(1+ \|v_1\|_{X^{\frac12+\dl, \frac12+\frac{\eps}{2}}_T} + \|v_2\|_{X^{\frac12+\dl, \frac12+\frac{\eps}{2}}_T}\big).
\end{split}
\label{p6}
\end{align}
We immediately have

\noi
\begin{align}
\|F(v_1, v_2)\|_{L^2_{t,x}} \les \|v_1-v_2\|_{X^{\frac12+\dl, \frac12+\frac{\eps}{2}}_T}. \label{p7}
\end{align}

\noi
Thus, by combining \eqref{p1}, \eqref{p3}, \eqref{p33}, \eqref{p35}, \eqref{p5}, \eqref{p6}, \eqref{p7}, we deduce that

\noi
\begin{align}
\begin{split}
& \|\G_{\Psi, \Ta}(v_1) - \G_{\Psi, \Ta}(v_2)\|_{X^{\frac12+\dl, \frac12+\frac \eps2}_T} \\
& \qquad \quad \les T^{\frac \eps2} \|v_1-v_2\|_{X^{\frac12+\dl, \frac12+\frac{\eps}{2}}_T} \big(1 + \|v_1\|_{X^{\frac12+\dl, \frac12+\frac{\eps}{2}}_T} + \|v_2\|_{X^{\frac12+\dl, \frac12+\frac{\eps}{2}}_T}\big) \\
& \qquad \qquad \qquad \qquad \qquad \times \|\Psi\|_{(\Ld_{\infty}^{1-\eps, 0} \cap \Ld^{0,\frac12-\eps}_{\infty})([0,1])} \|\Ta\|_{\mathcal Z^{\al,\eps}([0,1])}.
\end{split}
\label{p8}
\end{align}
By arguing as in the proof of \eqref{p8},\footnote{the proof is in fact easier as we do not have to use product estimates as in \eqref{p5} and \eqref{p6}.} we get the following a priori estimate on $\G_{\psi, \Ta}(v)$:
\noi
\begin{align}
\begin{split}
& \|\G_{\Psi, \Ta}(v)\|_{X^{\frac12+\dl, \frac12+\frac \eps2}_T} \\
& \qquad \quad \les \|(u_0, v_0)\|_{\mc H ^{\frac12 + \dl}} + T^{\frac \eps2} \|v\|_{X^{\frac12+\dl, \frac12+\frac{\eps}{2}}_T} \|\Psi\|_{(\Ld_{\infty}^{1-\eps, 0} \cap \Ld^{0,\frac12-\eps}_{\infty})([0,1])} \|\Ta\|_{\mathcal Z^{\al,\eps}([0,1])}.
\end{split}
\label{p9}
\end{align}
Hence, from \eqref{p8} and \eqref{p9}, the standard Banach fixed point argument yields a unique local solution $v$ to \eqref{E} on the time interval $[0,T]$, with
\begin{align}
 T \sim \big(\|\Psi\|_{(\Ld_{\infty}^{1-\eps, 0} \cap \Ld^{0,\frac12-\eps}_{\infty})([0,1])} \|\Ta\|_{\mathcal Z^{\al,\eps}([0,1])}\big)^{-\frac\eps2}.
\label{p10}
\end{align}

The claimed regularity of the map $(u_0, v_0, \Psi, \Ta) \in \mc X^{\al, \dl, \eps} \mapsto v$ is established via similar estimates.

By reiterating the local-in-time argument in above, noting that the local existence time \eqref{p10} does not depend on the initial data, and gluing the local solutions thus obtained by using Lemma \ref{LEM:gluing} therefore yields a unique global solution on $[0,1]$, as claimed.
%
%
%
%
\end{proof}

\subsection{Proof of Theorem \ref{THM:main}}
In this subsection, we combine the results in the previous sections and prove our main theorem.

\begin{proof}[Proof of Theorem \ref{THM:main}] Let $\be \in \R$ with
\[0 < \be^2 < 2\pi\Big(1 + \frac{3 \sqrt{241} - 41}{122}\Big).\]
Then, there exists $\al = \al(\be)>0$ and $\eps = \eps(\al) >0$ such that 
\begin{align}
\frac{\be^2}{4\pi} - \frac12 + 100 \eps < \al < \frac{3\sqrt{241}-41}{244}.
\label{cond_al}
\end{align}
Furthermore, we may choose $\eps$ small enough so that the estimates in Subsection \ref{SUBSEC:bilin} hold.

 For each $N \in \N$ and $\eps_0 \in \{+1,-1\}$, let $\Ta^{\eps_0}_N$ be as in \eqref{t1}. Recall the definitions of the truncated stochastic convolutions in \eqref{Psi_trunc1} and \eqref{Psi_trunc2}. Set $\Psi_N := \Psi^{\textup{KG}}_N-\Psi^{\textup{wave}}_N$. Then, by Lemma \ref{LEM:diff_psi} and Proposition \ref{PROP:sto_main} with \eqref{cond_al}, there exists $(\Psi, \Ta^{\eps_0}) \in \big(\Ld_{\infty}^{1-\eps, 0} \cap \Ld^{0,\frac12-\eps}_{\infty}\big)([0,1]) \times \mc Z^{\al, \eps}([0,1])$ such that
\begin{align*}
(\Psi_N, \Ta^{\eps_0}_N) \too (\Psi, \Ta^{\eps_0}) \quad \text{ in }  \big(\Ld_{\infty}^{1-\eps, 0} \cap \Ld^{0,\frac12-\eps}_{\infty}\big)([0,1]) \times \mc Z^{\al, \eps}([0,1]),
\end{align*}
$\muu_1 \otimes \PP$-almost surely as $N \to \infty$. 

Therefore, by Proposition \ref{PROP:gwp}, there exists $(v, \dt v) \in C([0,1]; \mc H^{\frac12+\dl}(\T^2)$, $\dl = \al + 10\eps$, solving the equation
\begin{align*}
 v = -  \sum_{\eps_0, \eps_1 \in \{+,-\}} c_{\eps_0, \eps_1} \Pii_{\le N} \I \Big( e^{i \eps_1 \be v} e^{i \be (\Psi^{\textup{KG}} - \Psi^{\textup{wave}})} \cdot  \Theta^{\eps_0}\Big),
\end{align*}
where the constants $c_{\eps_0, \eps_1}$ are as in \eqref{vN2}, such that the solution $v_N$ to \eqref{vN1} satisfies
\begin{align*}
(v_N, \dt v_N) \too (v, \dt v) \quad \text{ in } C([0,1]; \mc H^{\frac12+\dl}(\T^2),
\end{align*}
$\muu_1 \otimes \PP$-almost surely as $N \to \infty$. Let $u_N = \Psi^{\textup{KG}} + v_N$\footnote{Here, $\Psi^{\textup{KG}} = \Psi^{\textup{wave}} + \Psi$, where $\Psi^{\textup{wave}}$ is the distributional limit of the sequence $\{\Psi^{\textup{wave}}_N\}$ provided by Lemma \ref{LEM:psi}.} be the solution to \eqref{RSdSGN} and $u := \Psi^{\textup{KG}} + v$. Then, we deduce from the above and Lemma \ref{LEM:psi} that $(u_N, \dt u_N)$ converges to $(u,\dt u)$ in $C([0,1];\mc  H^{0-}(\T^2))$ $\muu_1 \otimes \PP$-almost surely as $N \to \infty$.

From Lemma \ref{LEM:Gibbs}, we get that $(u_N, \dt u_N)$ converges to $(u, \dt u)$ in $C([0,1]; \mc H^{0-}(\T^2))$ $\rhoo \otimes \PP$-almost surely as $N \to \infty$. Moreover, in view of Lemma \ref{LEM:invariance}, the law of $(u(t), \dt u(t))$ is given by $\rhoo$ for each $t \in [0,1]$.

Since $\operatorname{Law}(u(1), \dt u(1)) = \rhoo$, we may extend reiterate the above argument and extend $(u, \dt u)$ to the time interval $[1,2]$. Iterating this process gives a stochastic process $(u,\dt u) \in C(\R_+;\mc  H^{0-}(\T^2))$ such that $\operatorname{Law}(u(t), \dt u(t)) = \rhoo$ for each $t \ge 0$. This concludes the proof.
\end{proof}

\begin{ackno}\rm
First and foremost, the author thanks Tadahiro Oh for numerous useful discussions and comments on an earlier version of this manuscript. The author is also grateful to Jonathan Hickman and Hrit Roy for discussions on the cone multiplier problem, and to Tristan Robert for helpful insights on the hyperbolic Riesz potential and for pointing out the reference~\cite{SKM}. The author further thanks Roland Bauerschmidt, David Beltran, Bjoern Bringmann, Yu Deng, Pawel Duch, Martin Hairer, Herbert Koch, and Hao Shen for fruitful conversations.
This work was partially supported by the European Research Council (grant no.~864138 ``SingStochDispDyn”) and by the Chair of Probability and PDEs at EPFL.
\end{ackno}

\end{document}